\documentclass[10pt]{amsart}
\usepackage{epic,eepic}
\usepackage[mathscr]{eucal} 
\usepackage{comment}

\numberwithin{equation}{section}
\newtheorem{thm}{Theorem}[section]
\newtheorem{lem}[thm]{Lemma}
\newtheorem{prop}[thm]{Proposition}
\newtheorem{cor}[thm]{Corollary}
\newtheorem{conj}[thm]{Conjecture}
\newtheorem*{claim}{Claim}
\newtheorem*{claim1}{Claim 1}
\newtheorem*{claim2}{Claim 2}
\theoremstyle{definition}
\newtheorem{defn}[thm]{Definition}
\newtheorem{exmp}[thm]{Example}
\newtheorem{rem}[thm]{Remark}

\title[Periodicities of T-systems and Y-systems]
{Periodicities of T-systems and Y-systems}

\author[R.\  Inoue]{Rei Inoue}
\address{ R.\  Inoue:
Faculty of Pharmaceutical Sciences, Suzuka University of Medical Science,
Suzuka, 513-8670, Japan}

\author[O.\ Iyama]{Osamu Iyama}
\address{ O.\ Iyama:
 Graduate School of Mathematics, Nagoya University,
Nagoya, 464-8604, Japan}

\author[A.\ Kuniba]{\\Atsuo Kuniba}
\address{ A.\ Kuniba:
Institute of Physics,
University of Tokyo,
Tokyo, 153-8902, Japan}

\author[T.\ Nakanishi]{Tomoki Nakanishi}
\address{ T.\ Nakanishi:
 Graduate School of Mathematics, Nagoya University,
Nagoya, 464-8604, Japan}

\author[J.\ Suzuki]{Junji Suzuki}
\address{ J.\ Suzuki :
Department of Physics, Faculty of Science,
Shizuoka University,
Ohya, 836, Japan}

\begin{document}

\begin{abstract}
The unrestricted T-system
is a family of relations
in the
 \linebreak 
 Grothendieck ring
of the category of the
 finite-dimensional modules
of the Yangian or
the quantum affine algebra
associated with a complex simple Lie
algebra.
The unrestricted T-system
admits a reduction called
the restricted T-system.
In this paper we formulate 
the periodicity conjecture for the restricted T-systems,
which is the counterpart of the
known and partially proved
periodicity conjecture for the restricted Y-systems.
Then, we partially prove the conjecture
by various methods:
the cluster algebra and cluster category method
for the simply laced case,
the determinant method
for types $A$ and $C$,
and the direct method
for types $A$, $D$, and $B$ (level 2).
\end{abstract}

\maketitle

\tableofcontents

\section{Introduction}
The {\it Y-system}
 was introduced as a system of functional relations
concerning the solutions of the thermodynamic
Bethe ansatz equations for the factorizable scattering
theory and the solvable lattice models
\cite{Z,KP,KN,RTV}.
It was conjectured that the solutions of the Y-system
has the periodicity \cite{Z,RTV,KNS1}.
Fomin and Zelevinsky
proved it
for a special case (level 2 case in our terminology)
\cite{FZ3} by the cluster algebra approach \cite{FZ1,FZ2,FZ4}.
Since then, a remarkable link has been 
established between cluster algebras
and cluster categories of the
quiver representations (See \cite{BMRRT,BMR,CC,CK1,CK2,Kel2}
and references therein).
Based on this categorification method,
  Keller recently proved the periodicity
of the Y-system for more general case \cite{Kel2,Kel3}.

Meanwhile, it has been known  that the Y-system is
 related to another
systems of relations called the {\it T-system}
and the {\it Q-system} \cite{KP,
KNS1}.
The T-system
is a family of relations
in the Grothendieck ring
of the category of the
 finite-dimensional modules
of the Yangian $Y(\mathfrak{g})$ 
 or
the quantum affine algebra
$U_q(\hat{\mathfrak{g}})$
associated with a complex simple Lie
algebra ${\mathfrak{g}}$  \cite{KNS1,N3,Her1,Her2}.
As a discrete dynamical system,
the $T$-system can be also viewed as a discrete analogue
of the Toda field equation \cite{KOS, KLWZ}.
The Q-system
is a degenerated version
of the T-system and plays an important role in the
algebraic Bethe ansatz method \cite{Ki,KR,HKOTY,KNT}.
As a side remark, it may be worth mentioning at this point
that `T' stands for {\em Transfer matrix},
while `Q' does for {\em Quantum  character} \cite{Ki2}
in the original literature.

As a more recent development,
a connection between the Q-systems
and cluster algebras is clarified by \cite{Ked,DiK}.
Also, a connection between 
the T-systems (or $q$-characters) and cluster algebras  is studied
while seeking a natural categorification of cluster algebras
by abelian monoidal categories
\cite{HL}.

Having these results as a background,
we make three simple observations:

\begin{itemize}
\item[(1)] There are actually {\em two classes\/} of the Y-systems
(resp.\ T-systems);
namely, the {\em unrestricted\/} and {\em restricted\/} Y-systems
(resp.\ T-systems).
The latter is obtained by a certain reduction from the former.
The periodicity property above mentioned is for
the {\em restricted\/} Y-systems.

\item[(2)] The cluster algebra structure is simpler
in the T-systems than the Y-systems.

\item[(3)] The representation theory of
  quantum affine algebras is more directly connected
with the T-systems than the Y-systems.
\end{itemize}

These observations motivate us to ask if there is a similar periodicity
property for the {\em restricted T-systems\/};
and, indeed, there is.

In this paper, we formulate 
the periodicity conjecture for the restricted T-systems,
which is the counterpart of the
known and partially proved
 periodicity conjecture for the restricted Y-systems.
Then, we partially prove the conjecture
by various methods.
We remark that the restricted T-systems
are relations in
certain {\it quotients\/} of the Grothendieck ring $\mathrm{Rep}\, U_q(\hat{\mathfrak{g}})$,
while the T-systems  studied in \cite{HL} are relations
in certain {\it subrings} of $\mathrm{Rep}\, U_q(\hat{\mathfrak{g}})$.
Accordingly, the correspondence between the T-systems for the simply laced case
and  cluster algebras
considered here and the one in \cite{HL}
are close but  slightly different.
We also note that the correspondence between
the {\em unrestricted\/} T-systems for the simply laced case
and cluster algebras
is described in \cite[AppendixB]{DiK}.

Let us explain the outline of the paper,
whose contents could be roughly divided into three parts.

In the first part (Section 2)
we introduce the {\em unrestricted T-systems\/}
together with their associated rings,
which we call the {\em unrestricted T-algebras}.
Then, we establish an isomorphism between
a subring of the unrestricted T-algebra and the
Grothendieck ring of
the category of the finite-dimensional representations
of an untwisted quantum affine algebra
(Corollary \ref{cor:qch1}).
The relation between 
the unrestricted T and Y-algebras
is also given (Theorem \ref{thm:TtoY1}).
They provide the representation theoretical
background of the periodicity problem we are going to discuss.

In the second part (Sections 3--7)
we introduce the {\em level $\ell$ restricted T-systems\/}
together with their associated rings,
which we call the {\em level $\ell$ restricted T-algebras},
where $\ell$ is an integer greater than or equal to two.
Then, we formulate the {\em periodicity conjecture\/}
(Conjecture \ref{conj:Tperiod1})
of the restricted T-systems
in terms of the restricted T-algebras.
This is the main claim of the paper.
Conjecture \ref{conj:Tperiod1} 
is completely parallel to that of
the restricted Y-systems
(Conjecture \ref{conj:Yperiod1}).
A detailed summary of
our methods and  results
concerning Conjecture \ref{conj:Tperiod1}
is given in Section 3.4.
In brief, we study and partially
prove the periodicity conjecture by
three independent methods:
the {\em cluster algebra/category method\/}
for the simply laced case
in Section 4,
the {\em determinant method\/}
for types $A$ and $C$ 
in Sections 5 and 6,
and the {\em direct method\/}
for types $A$, $D$, and $B$ (level 2)
in Section 7.
In particular, for the simply laced case,
the relation between the restricted T-algebras
and cluster algebras is clarified in Section 4.
 For the cluster category method, we follow the ideas of Keller [Kel2]
based on Amiot's generalized cluster categories [A].

In the third part (Sections \ref{sect:lev10} and \ref{sect:twisted})
we treat  the extensions of the above periodicity property
to two classes of T and Y-systems.
In Section \ref{sect:lev10} 
we formulate and prove
the periodicity property for the restricted T
and Y-systems
at {\em levels $1$ and $0$\/}.
In Section \ref{sect:twisted}
we formulate the periodicity property also
for the restricted T and Y-systems associated
with the {\em twisted\/} quantum affine algebras.
It turns out that their periodicity property reduces to 
that of the untwisted case.
We remark that the nonsimply laced Y-systems studied in
\cite{FZ2,Kel2} are identified with 
certain reductions of the restricted
Y-systems belonging to this class (Remark \ref{rem:nonsim}).

We conclude the paper with a brief remark
(Section \ref{sect:qch}) on a formal correspondence between the
periodicity of the T-systems and the
$q$-character of the quantum affine algebras
{\em at roots of unity\/}.
This suggests that there is some further connection
between the representation theories of quivers
and the quantum affine algebras at roots of unity
behind this periodicity phenomena, possibly through
the works of \cite{N1,N2}.
The relation between the restricted T and Y-algebras and cluster algebras
for the nonsimply laced case will be discussed in
 a separate publication.

{\em Acknowledgments.} We thank
Sergey Fomin,
David Hernandez,
Bernhard Keller,
Anatol Kirillov,
Bernard Leclerc,
Hyohe Miyachi,
Roberto Tateo,
and Andrei Zelevinsky
for discussions and communications.

\section{Unrestricted T and Y-systems
}

\label{sect:unrest}

In this section  we introduce the
{\em unrestricted T and Y-systems\/} of \cite{KNS1}
as a background of the periodicity problem.
We also introduce the associated algebras,
which we call the {\em unrestricted T and Y-algebras}.
They are closely connected 
to the Grothendieck ring
of the category of the
 finite-dimensional 
$Y(\mathfrak{g})$--modules or $U_q(\hat{\mathfrak{g}})$-modules.
The content of this
section is  rather independent of the rest
 of the paper.

\subsection{Unrestricted T-systems}
\label{subsect:unrestrictedT}

Throughout the paper, a `ring' means a
{\em commutative ring (algebra over $\mathbb{Z}$) with identity element}.
For a ring $R$, $R^{\times}$ denote the set of
all the invertible elements of $R$.
The set of all the positive integers is denoted by $\mathbb{N}$.

Let $X_r$ be a Dynkin diagram of finite type
with rank $r$,
and $I=\{1,\dots, r\}$ be the enumeration
of the vertices of $X_r$ as 
Figure \ref{fig:Dynkin}.
We follow \cite{Ka} {\em except for $E_6$\/},
for which we choose the one
naturally corresponding to 
the enumeration of the twisted affine diagram
$E^{(2)}_6$ in Section \ref{sect:twisted}.

Let $C=(C_{ab})$, $C_{ab}=2(\alpha_a,\alpha_b)/(\alpha_a,\alpha_a)$,
 be the Cartan matrix of $X_r$.
We set  numbers $t$ and $t_a$ ($a\in I$) by
\begin{align}
\label{eq:t1}
t=
\begin{cases}
1 & \text{$X_r$: simply laced},\\
2 & \text{$X_r=B_r$, $C_r$, $F_4$},\\
3 & \text{$X_r=G_2$},
\end{cases}
\quad t_a=
\begin{cases}
1 & \text{$X_r$: simply laced},\\
1 & \text{$X_r$: nonsimply laced, $\alpha_a$: long root},\\
t & \text{$X_r$: nonsimply laced, $\alpha_a$: short root}.
\end{cases}
\end{align}

\begin{figure}
\begin{picture}(283,185)(-15,-175)
%
\put(0,0){\circle{6}}
\put(20,0){\circle{6}}
\put(80,0){\circle{6}}
\put(100,0){\circle{6}}
\put(45,0){\circle*{1}}
\put(50,0){\circle*{1}}
\put(55,0){\circle*{1}}
\drawline(3,0)(17,0)
\drawline(23,0)(37,0)
\drawline(63,0)(77,0)
\drawline(83,0)(97,0)
\put(-30,-2){$A_r$}
\put(-2,-15){\small $1$}
\put(18,-15){\small $2$}
\put(70,-15){\small $ r-1$}
\put(98,-15){\small $r$}
%
\put(180,0){
\put(0,0){\circle{6}}
\put(20,0){\circle{6}}
\put(80,0){\circle{6}}
\put(100,0){\circle{6}}
\put(45,0){\circle*{1}}
\put(50,0){\circle*{1}}
\put(55,0){\circle*{1}}
\drawline(3,0)(17,0)
\drawline(23,0)(37,0)
\drawline(63,0)(77,0)
\drawline(82,-2)(98,-2)
\drawline(82,2)(98,2)
\drawline(87,6)(93,0)
\drawline(87,-6)(93,0)
\put(-30,-2){$B_r$}
\put(-2,-15){\small $1$}
\put(18,-15){\small $2$}
\put(70,-15){\small $ r-1$}
\put(98,-15){\small $r$}
}
%

\put(0,-40){
\put(0,0){\circle{6}}
\put(20,0){\circle{6}}
\put(80,0){\circle{6}}
\put(100,0){\circle{6}}
\put(45,0){\circle*{1}}
\put(50,0){\circle*{1}}
\put(55,0){\circle*{1}}
\drawline(3,0)(17,0)
\drawline(23,0)(37,0)
\drawline(63,0)(77,0)
\drawline(82,-2)(98,-2)
\drawline(82,2)(98,2)
\drawline(87,0)(93,-6)
\drawline(87,0)(93,6)
\put(-30,-2){$C_r$}
\put(-2,-15){\small $1$}
\put(18,-15){\small $2$}
\put(70,-15){\small $ r-1$}
\put(98,-15){\small $r$}
}

%
\put(180,-40){
\put(0,0){\circle{6}}
\put(20,0){\circle{6}}
\put(80,0){\circle{6}}
\put(98,9){\circle{6}}
\put(98,-9){\circle{6}}
\put(45,0){\circle*{1}}
\put(50,0){\circle*{1}}
\put(55,0){\circle*{1}}
\drawline(3,0)(17,0)
\drawline(23,0)(37,0)
\drawline(63,0)(77,0)
\drawline(82.8,1.5)(95.5,7.8)
\drawline(82.8,-1.5)(95.5,-7.8)
\put(-30,-2){$D_r$}
\put(-2,-15){\small $1$}
\put(18,-15){\small $2$}
\put(83,15){\small $ r-1$}
\put(70,-15){\small $ r-2$}
\put(98,-20){\small $r$}
}
%
\put(0,-80){
\put(0,0){\circle{6}}
\put(20,0){\circle{6}}
\put(40,0){\circle{6}}
\put(60,0){\circle{6}}
\put(80,0){\circle{6}}
\put(40,20){\circle{6}}
\drawline(3,0)(17,0)
\drawline(23,0)(37,0)
\drawline(43,0)(57,0)
\drawline(40,3)(40,17)
\drawline(63,0)(77,0)
\put(-30,-2){$E_6$}
\put(-2,-15){\small $1$}
\put(18,-15){\small $2$}
\put(38,-15){\small $3$}
\put(58,-15){\small $5$}
\put(78,-15){\small $6$}
\put(50,18){\small $4$}
}
%
\put(180,-80){
\put(0,0){\circle{6}}
\put(20,0){\circle{6}}
\put(40,0){\circle{6}}
\put(60,0){\circle{6}}
\put(80,0){\circle{6}}
\put(100,0){\circle{6}}
\put(40,20){\circle{6}}
\drawline(3,0)(17,0)
\drawline(23,0)(37,0)
\drawline(43,0)(57,0)
\drawline(40,3)(40,17)
\drawline(63,0)(77,0)
\drawline(83,0)(97,0)
\put(-30,-2){$E_7$}
\put(-2,-15){\small $1$}
\put(18,-15){\small $2$}
\put(38,-15){\small $3$}
\put(58,-15){\small $4$}
\put(78,-15){\small $5$}
\put(98,-15){\small $6$}
\put(50,18){\small $7$}
}
%
\put(0,-120){
\put(0,0){\circle{6}}
\put(20,0){\circle{6}}
\put(40,0){\circle{6}}
\put(60,0){\circle{6}}
\put(80,0){\circle{6}}
\put(100,0){\circle{6}}
\put(120,0){\circle{6}}
\put(80,20){\circle{6}}
\drawline(3,0)(17,0)
\drawline(23,0)(37,0)
\drawline(43,0)(57,0)
\drawline(80,3)(80,17)
\drawline(63,0)(77,0)
\drawline(83,0)(97,0)
\drawline(103,0)(117,0)
\put(-30,-2){$E_8$}
\put(-2,-15){\small $1$}
\put(18,-15){\small $2$}
\put(38,-15){\small $3$}
\put(58,-15){\small $4$}
\put(78,-15){\small $5$}
\put(98,-15){\small $6$}
\put(118,-15){\small $7$}
\put(90,18){\small $8$}
}
%
\put(0,-160){
\put(0,0){\circle{6}}
\put(20,0){\circle{6}}
\put(40,0){\circle{6}}
\put(60,0){\circle{6}}
\drawline(3,0)(17,0)
\drawline(43,0)(57,0)
\drawline(22,-2)(38,-2)
\drawline(22,2)(38,2)
\drawline(27,6)(33,0)
\drawline(27,-6)(33,0)
\put(-30,-2){$F_4$}
\put(-2,-15){\small $1$}
\put(18,-15){\small $2$}
\put(38,-15){\small $3$}
\put(58,-15){\small $4$}
}
%
\put(180,-160){
\put(0,0){\circle{6}}
\put(20,0){\circle{6}}
\drawline(3,0)(17,0)
\drawline(2,-2)(18,-2)
\drawline(2,2)(18,2)
\drawline(7,6)(13,0)
\drawline(7,-6)(13,0)
\put(-30,-2){$G_2$}
\put(-2,-15){\small $1$}
\put(18,-15){\small $2$}
}
\end{picture}
\caption{The Dynkin diagrams $X_r$ and their enumerations.}
\label{fig:Dynkin}
\end{figure}
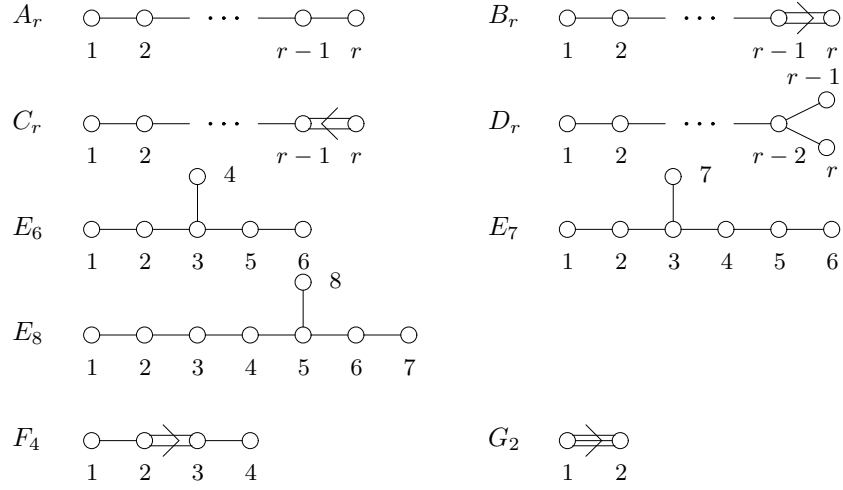

Let $U$ be either the complex plane $\mathbb{C}$, or
the cylinder $\mathbb{C}_{\xi}:= \mathbb{C}/
(2\pi \sqrt{-1}/\xi) \mathbb{Z}$ for some
 $\xi\in \mathbb{C}
\setminus 2\pi \sqrt{-1}\mathbb{Q}$.

\begin{defn}
The {\it unrestricted T-system $\mathbb{T}(X_r)$ of type $X_r$}
is the following system of relations for
a family of variables
 $T=\{T^{(a)}_m(u) \mid a\in I, m\in \mathbb{N},
u\in U \}$,
where 
$T^{(0)}_m (u)=T^{(a)}_0 (u)= 1$ if they occur
in the right hand sides in the relations:

(Here and throughout the paper,
$2m$ (resp.\ $2m+1$) in the left hand sides, for example,
represents elements  $ 2,4,\dots$
(resp.\ $1,3,\dots$).)

For simply laced $X_r$,
\begin{align}
\label{eq:TA1}
T^{(a)}_m(u-1)T^{(a)}_m(u+1)
=
T^{(a)}_{m-1}(u)T^{(a)}_{m+1}(u)
+
\prod_{b\in I: C_{ab}=-1}
T^{(b)}_{m}(u).
\end{align}

{\allowdisplaybreaks
For $X_r=B_r$,
\begin{align}
\label{eq:TB1}
T^{(a)}_m(u-1)T^{(a)}_m(u+1)
&=
T^{(a)}_{m-1}(u)T^{(a)}_{m+1}(u)\\
&\qquad
+T^{(a-1)}_{m}(u)T^{(a+1)}_{m}(u)
\quad
 (1\leq a\leq r-2),\notag\\
T^{(r-1)}_m(u-1)T^{(r-1)}_m(u+1)
&=
T^{(r-1)}_{m-1}(u)T^{(r-1)}_{m+1}(u)
+
T^{(r-2)}_{m}(u)T^{(r)}_{2m}(u),\notag\\
T^{(r)}_{2m}\left(u-\textstyle\frac{1}{2}\right)
T^{(r)}_{2m}\left(u+\textstyle\frac{1}{2}\right)
&=
T^{(r)}_{2m-1}(u)T^{(r)}_{2m+1}(u)\notag\\
&\qquad
+
T^{(r-1)}_{m}\left(u-\textstyle\frac{1}{2}\right)
T^{(r-1)}_{m}\left(u+\textstyle\frac{1}{2}\right),
\notag\\
T^{(r)}_{2m+1}\left(u-\textstyle\frac{1}{2}\right)
T^{(r)}_{2m+1}\left(u+\textstyle\frac{1}{2}\right)
&=
T^{(r)}_{2m}(u)T^{(r)}_{2m+2}(u)
+
T^{(r-1)}_{m}(u)T^{(r-1)}_{m+1}(u).
\notag
\end{align}
}

{\allowdisplaybreaks
For $X_r=C_r$,
\begin{align}
\label{eq:TC1}
T^{(a)}_m\left(u-\textstyle\frac{1}{2}\right)
T^{(a)}_m\left(u+\textstyle\frac{1}{2}\right)
&=
T^{(a)}_{m-1}(u)T^{(a)}_{m+1}(u)\\
&\qquad
+T^{(a-1)}_{m}(u)T^{(a+1)}_{m}(u)
\quad
 (1\leq a\leq r-2),\notag\\
T^{(r-1)}_{2m}\left(u-\textstyle\frac{1}{2}\right)
T^{(r-1)}_{2m}\left(u+\textstyle\frac{1}{2}\right)
&=
T^{(r-1)}_{2m-1}(u)T^{(r-1)}_{2m+1}(u)\notag\\
&\qquad
+
T^{(r-2)}_{2m}(u)
T^{(r)}_{m}\left(u-\textstyle\frac{1}{2}\right)
T^{(r)}_{m}\left(u+\textstyle\frac{1}{2}\right)
,\notag\\
T^{(r-1)}_{2m+1}\left(u-\textstyle\frac{1}{2}\right)
T^{(r-1)}_{2m+1}\left(u+\textstyle\frac{1}{2}\right)
&=
T^{(r-1)}_{2m}(u)T^{(r-1)}_{2m+2}(u)\notag\\
&\qquad
+
T^{(r-2)}_{2m+1}(u)
T^{(r)}_{m}(u)T^{(r)}_{m+1}(u),
\notag\\
T^{(r)}_{m}(u-1)
T^{(r)}_{m}(u+1)
&=
T^{(r)}_{m-1}(u)T^{(r)}_{m+1}(u)
+
T^{(r-1)}_{2m}(u).
\notag
\end{align}
}

{\allowdisplaybreaks
For $X_r=F_4$,
\begin{align}
\label{eq:TF1}
T^{(1)}_m(u-1)T^{(1)}_m(u+1)
&=
T^{(1)}_{m-1}(u)T^{(1)}_{m+1}(u)
+T^{(2)}_{m}(u),\\
T^{(2)}_m(u-1)T^{(2)}_m(u+1)
&=
T^{(2)}_{m-1}(u)T^{(2)}_{m+1}(u)
+
T^{(1)}_{m}(u)T^{(3)}_{2m}(u),\notag\\
T^{(3)}_{2m}\left(u-\textstyle\frac{1}{2}\right)
T^{(3)}_{2m}\left(u+\textstyle\frac{1}{2}\right)
&=
T^{(3)}_{2m-1}(u)T^{(3)}_{2m+1}(u)\notag\\
&\qquad
+
T^{(2)}_{m}\left(u-\textstyle\frac{1}{2}\right)
T^{(2)}_{m}\left(u+\textstyle\frac{1}{2}\right)
T^{(4)}_{2m}(u),\notag\\
T^{(3)}_{2m+1}\left(u-\textstyle\frac{1}{2}\right)
T^{(3)}_{2m+1}\left(u+\textstyle\frac{1}{2}\right)
&=
T^{(3)}_{2m}(u)T^{(3)}_{2m+2}(u)
+
T^{(2)}_{m}(u)T^{(2)}_{m+1}(u)
T^{(4)}_{2m+1}(u),\notag\\
\notag
T^{(4)}_{m}\left(u-\textstyle\frac{1}{2}\right)
T^{(4)}_{m}\left(u+\textstyle\frac{1}{2}\right)
&=
T^{(4)}_{m-1}(u)T^{(4)}_{m+1}(u)
+T^{(3)}_m(u).
\notag
\end{align}
}

For $X_r=G_2$,
\begin{align}
\label{eq:TG1}
T^{(1)}_m(u-1)T^{(1)}_m(u+1)
&=
T^{(1)}_{m-1}(u)T^{(1)}_{m+1}(u)
+
T^{(2)}_{3m}(u),\\
T^{(2)}_{3m}\left(u-\textstyle\frac{1}{3}\right)
T^{(2)}_{3m}\left(u+\textstyle\frac{1}{3}\right)
&=
T^{(2)}_{3m-1}(u)T^{(2)}_{3m+1}(u)\notag\\
&\qquad
+
T^{(1)}_{m}\left(u-\textstyle\frac{2}{3}\right)
T^{(1)}_m(u)
T^{(1)}_{m}\left(u+\textstyle\frac{2}{3}\right),\notag\\
T^{(2)}_{3m+1}\left(u-\textstyle\frac{1}{3}\right)
T^{(2)}_{3m+1}\left(u+\textstyle\frac{1}{3}\right)
&=
T^{(2)}_{3m}(u)T^{(2)}_{3m+2}(u)\notag\\
&\qquad+
T^{(1)}_{m}\left(u-\textstyle\frac{1}{3}\right)
T^{(1)}_{m}\left(u+\textstyle\frac{1}{3}\right)
T^{(1)}_{m+1}(u),\notag\\
T^{(2)}_{3m+2}\left(u-\textstyle\frac{1}{3}\right)
T^{(2)}_{3m+2}\left(u+\textstyle\frac{1}{3}\right)
&=
T^{(2)}_{3m+1}(u)T^{(2)}_{3m+3}(u)\notag\\
&\qquad+
T^{(1)}_{m}(u)
T^{(1)}_{m+1}\left(u-\textstyle\frac{1}{3}\right)
T^{(1)}_{m+1}\left(u+\textstyle\frac{1}{3}\right).\notag
\end{align}
\end{defn}

The choice of the
domain $U=\mathbb{C}_{\xi}$ of the parameter
$u$ effectively imposes an additional periodic condition:
\begin{align}
\label{eq:imperiod1}
T^{(a)}_m(u)=
T^{(a)}_m(\textstyle u+ \frac{2\pi \sqrt{-1}}{\xi}).
\end{align}
By the assumption, we have $2\pi \sqrt{-1}/\xi\notin
\mathbb{Q}$ so that
it is compatible with the relations $\mathbb{T}(X_r)$.

\begin{rem}
Originally, the system $\mathbb{T}(X_r)$ was introduced
in \cite{KNS1}
as a family of relations
in the ring of commuting {\em transfer matrices} of the
solvable lattice models.
For example, for $X_r=A_r$, 
the relations in (\ref{eq:TA1}) are the Jacobi identities
among the Jacobi-Trudi-type  determinantal expression
of the transfer matrices in \cite{BR}.
The T-system is a natural affinization of
the {\em Q-system} of \cite{Ki,KR} (see Appendix \ref{subsect:Q}),
and the idea behind the both systems
was the existence of a conjectured family
of exact sequences among
the {\em Kirillov-Reshetikhin modules}
\cite{KR,CP1,KNS1} of
the Yangian $Y(\mathfrak{g})$
and/or the untwisted quantum affine algebra
$U_q(\hat{\mathfrak{g}})$ 
associated with the complex simple Lie algebra
$\mathfrak{g}$ of type $X_r$ \cite{D1,D2,J}.
The choice $U=\mathbb{C}$ corresponds to the $Y(\mathfrak{g})$ case,
while the choice
$U=\mathbb{C}_{\xi}$ corresponds to the $U_q(\hat{\mathfrak{g}})$ case
as explained below.
For $U_q(\hat{\mathfrak{sl}}_2)$, the existence 
of such exact sequences  was known by \cite{CP1}.
Later this conjecture was proved for $U_q(\hat{\mathfrak{g}})$
 by \cite{N3,Her1}. See Theorem \ref{thm:qch1} (2).
\end{rem}

It is useful to introduce the rings
associated  with $\mathbb{T}(X_r)$.
\begin{defn}
The {\em unrestricted T-algebra $\EuScript{T}(X_r)$
of type $X_r$} is the ring with generators
$T^{(a)}_m(u)^{\pm 1}$ ($a\in I, m\in \mathbb{N},
u\in U $)
and the relations $\mathbb{T}(X_r)$.
(Here we also assume the relation
$T^{(a)}_m(u)T^{(a)}_m(u)^{-1}=1$ implicitly.
We do not repeat this remark in the forthcoming
similar definitions.)
Also, we define the ring $\EuScript{T}^{\circ}(X_r)$
as the subring of $\EuScript{T}(X_r)$
generated by 
$T^{(a)}_m(u)$ ($a\in I, m\in \mathbb{N},
u\in U $).
\end{defn}

We write all the relations in $\mathbb{T}(X_r)$ in a
 unified manner as follows:
\begin{align}
\label{eq:ta}
\textstyle
T^{(a)}_m(u-\frac{1}{t_a})T^{(a)}_m(u+\frac{1}{t_a})
=
T^{(a)}_{m-1}(u)T^{(a)}_{m+1}(u)
+M^{(a)}_m(u),
\end{align}
where $M^{(a)}_m(u)$ is the second term of the right hand side
of each relation, and $t_a$ is the number 
in (\ref{eq:t1}).
Then, define $S_{amu}(T)\in \mathbb{Z}[T]$ ($a\in I, m\in \mathbb{N},
 u\in U)$ by
\begin{align}
\label{eq:sa}
\textstyle
S_{amu}(T)=
T^{(a)}_m(u-\frac{1}{t_a})T^{(a)}_m(u+\frac{1}{t_a})
-
T^{(a)}_{m-1}(u)T^{(a)}_{m+1}(u)
-M^{(a)}_m(u),
\end{align}
so that all the relations in $\mathbb{T}(X_r)$ are written
in the form $S_{amu}(T)=0$.
Let  $I(\mathbb{T}(X_r))$ denote the ideal of
$\mathbb{Z}[T]$ generated by $S_{amu}(T)$'s.
We consider the natural embedding
 $\mathbb{Z}[T]\subset \mathbb{Z}[T^{\pm1}]$.

We use the following description of the ring  $\EuScript{T}^{\circ}(X_r)$:

\begin{lem}
\label{lem:Tc1}
(1) There is a ring isomorphism:
\begin{align}
 \EuScript{T}^{\circ}(X_r)
\simeq
\mathbb{Z}[T]/
(\mathbb{Z}[T^{\pm1}]I(\mathbb{T}(X_r))\cap \mathbb{Z}[T]).
\end{align}
\par
(2) For $P(T)\in \mathbb{Z}[T]$, the following conditions are equivalent:
\begin{itemize}
\item[(i)] $P(T)\in \mathbb{Z}[T^{\pm1}]I(\mathbb{T}(X_r))$.
\item[(ii)] There is a nonzero monomial $M(T)\in \mathbb{Z}[T]$
such that $M(T)P(T)\in 
\break 
 I(\mathbb{T}(X_r))$.
\end{itemize}
\end{lem}

Let us clarify the relation between the ring $\EuScript{T}^{\circ}(X_r)$
and the Grothendieck ring $\mathrm{Rep}\, U_q(\hat{\mathfrak{g}})$
of the category of the
type 1 finite-dimensional $U_q(\hat{\mathfrak{g}})$-modules
\cite{CP2}.

Choose  $\hbar\in
\mathbb{C}\setminus 2\pi\sqrt{-1}\mathbb{Q}$
arbitrarily.
We set the deformation parameter $q$ of
$U_q(\hat{\mathfrak{g}})$ as $q=e^{\hbar}\in \mathbb{C}^{\times}$,
 so that 
$q$ is {\em not a root of unity}.

Let
\begin{align}
\chi_q:\mathrm{Rep}\, U_q(\hat{\mathfrak{g}})
\rightarrow 
\mathbb{Z}[Y^{\pm 1}_{i,a}]_{i\in I, a\in \mathbb{C^{\times}}}
\end{align}
be the {\em $q$-character map\/} of $U_q(\hat{\mathfrak{g}})$ 
in \cite{FR,FM},
which is an injective ring homomorphism as shown in \cite{FR}.
{}From now on, we employ the  parametrization of the variables
$Y_{a,q^{t u}}$ ($a\in I$,
$u\in \mathbb{C}_{t\hbar}$)
instead of $Y_{i,a}$ ($i\in I$, $a\in \mathbb{C}^{\times}$)
in \cite{FR,FM}, where $t$ is the number in (\ref{eq:t1}).

The {\em $q$-character ring\/} $\mathrm{Ch}\, U_q(\hat{\mathfrak{g}})$
of $U_q(\hat{\mathfrak{g}})$
is defined to be
 $\mathrm{Im}\, \chi_q$.
Thus,  $\mathrm{Ch}\, U_q(\hat{\mathfrak{g}})$
is an integral domain and
isomorphic to $\mathrm{Rep}\, U_q(\hat{\mathfrak{g}})$.

\begin{defn}
A {\em Kirillov-Reshetikhin module} $W^{(a)}_m(u)$ 
($a\in I$, $m\in \mathbb{N}$,
$u\in \mathbb{C}_{t\hbar}$) of $U_q(\hat{\mathfrak{g}})$ is
the irreducible finite dimensional
$U_q(\hat{\mathfrak{g}})$-module 
with highest weight monomial
\begin{align}
\label{eq:KR1}
m_+=
\prod_{j=1}^m Y_{a,q^{tu}q_a^{m+1-2j}},
\end{align}
where $q_a=q^{t/t_a}$.
Especially, $W^{(a)}_1(u)$
($a\in I$, $u\in \mathbb{C}_{t\hbar}$) is  called a {\em fundamental
module}.
\end{defn}

\begin{rem}
The above $W^{(a)}_m(u)$ corresponds to
$W^{(a)}_{m,q^{tu}q_a^{-m+1}}$ in \cite{N3,Her1}.
\end{rem}

The following fact is well known:
\begin{thm}[{Frenkel-Reshetikhin \cite[Corollary\ 2]{FR}}]
\label{thm:FR}
The ring $\mathrm{Ch}\, U_q(\hat{\mathfrak{g}})$
is freely generated by the fundamental characters
$ \chi_q(W^{(a)}_1(u))$ ($a\in I,u\in \mathbb{C}_{t\hbar}$).
\end{thm}

Correspondingly, we choose the domain $U$ of the parameter $u$
for the T-system $\mathbb{T}(X_r)$ as
$U=\mathbb{C}_{t\hbar}$.
Here is an alternative description of
$\mathrm{Ch}\, U_q(\hat{\mathfrak{g}})$
by the $q$-characters of the Kirillov-Reshetikhin modules and the
T-system $\mathbb{T}(X_r)$.

\begin{thm}
\label{thm:qch1}
Let
$\widetilde{T}=\{\widetilde{T}^{(a)}_m(u):=\chi_q(W^{(a)}_m(u))
\mid
a\in I, m\in \mathbb{N},
u\in \mathbb{C}_{t\hbar}
\}$
 be the family of the
$q$-characters of the Kirillov-Reshetikhin modules
of $U_q(\hat{\mathfrak{g}})$.
Then,
\par
(1) The family $\widetilde{T}$ generates the ring
 $\mathrm{Ch}\, U_q(\hat{\mathfrak{g}})$.
\par
(2) (Nakajima \cite{N3}, Hernandez \cite{Her1})
The family $\widetilde{T}$ satisfies the T-system $\mathbb{T}(X_r)$
in $\mathrm{Ch}\, U_q(\hat{\mathfrak{g}})$
(by replacing $T^{(a)}_m(u)$  in $\mathbb{T}(X_r)$
 with $\widetilde{T}^{(a)}_m(u)$).
\par
(3) For any $P(T)\in \mathbb{Z}[T]$,
the relation $P(\widetilde{T})=0$ holds
 in $\mathrm{Ch}\, U_q(\hat{\mathfrak{g}})$
if and only if there is a nonzero monomial $M(T)\in \mathbb{Z}[T]$
such that  $M(T)P(T) \in I(\mathbb{T}(X_r))$.
\end{thm}

\begin{proof}
(1) This is a corollary of Theorem \ref{thm:FR}.

(2) This was proved by  \cite[Thorem 1.1]{N3}
(for simply laced case)
and by \cite[Theorem 3.4]{Her1} (including non-simply laced case).

(3) The `if' part follows from (2) and  that
$\mathrm{Ch}\, U_q(\hat{\mathfrak{g}})$ is an integral domain.
Let us show the `only if' part.
 To begin with, we introduce the {\em height}
 of  $T^{(a)}_m(u)$,
$\mathrm{ht}\,  T^{(a)}_m(u)$, by
\begin{alignat}{2}
\label{eq:height1}
&\text{for simply laced $X_r$,}
&\quad
\mathrm{ht}\,  T^{(a)}_m(u)&=m,
\notag\\
&\text{for $X_r=B_r$,}
&\quad
\mathrm{ht}\,  T^{(a)}_m(u)&=
\begin{cases}
2m-1 & a=1,\dots,r-1,\notag\\
m & a=r,
\end{cases}
\\
&\text{for $X_r=C_r$,}
&\quad
\mathrm{ht}\,  T^{(a)}_m(u)&=
\begin{cases}
m & a=1,\dots,r-1,\\
2m-1 & a=r,
\end{cases}
\\
&\text{for $X_r=F_4$,}
&\quad
\mathrm{ht}\,  T^{(a)}_m(u)&=
\begin{cases}
2m-1 & a=1,2,\notag\\
m & a=3,4,
\end{cases}
\\
&\text{for $X_r=G_2$,}
&\quad
\mathrm{ht}\,  T^{(a)}_m(u)&=
\begin{cases}
3m-2 & a=1,\notag\\
m & a=2.
\end{cases}
\end{alignat}
Then, the following facts can be easily checked by inspection.
\begin{itemize}
\item[(a)] $\mathrm{ht}\,  T^{(a)}_m(u)=1$ if and only if $m=1$.

\item[(b)] The variable $ T^{(a)}_m(u)$ ($m\geq 2$)
occurs in $ S_{a,m-1,u}(T)$,
and $\mathrm{ht}\,  T^{(a)}_m(u)$ is greater than the heights
of the other variables occurring in $S_{a,m-1,u}(T)$.
\end{itemize}
For a polynomial $P(T)\in \mathbb{Z}[T]$, we define $\mathrm{ht}\,  P(T)$ by
the greatest height of
all the generators $T^{(a)}_m(u)$ occurring in $P(T)$.

Now suppose that there is a nontrivial
relation $P(\widetilde{T})= 0$ in $\mathrm{Ch}\, U_q(\hat{\mathfrak{g}})$
for $P(T) \in \mathbb{Z}[T]$,
and that $h:=\mathrm{ht}\,  P(T)\geq 2$.
Let $S$ be the set of all the triplets
$(a,m,u)$ such that
 $T^{(a)}_m(u)$ is of height $h$ and occurs in $P(T)$.
Let $M_h(T)=\prod_{(a,m,u)\in S}
T^{(a)}_{m-2}(u)$.
Then, thanks to  (a) and (b),
there is some $Q(T)\in \mathbb{Z}[T]$ with $\mathrm{ht}\,  Q(T) < h$
such that $Q(T)\equiv M_h(T)P(T)$ mod $I(\mathbb{T}(X_r))$.
Furthermore, by (2), $Q(\widetilde{T})=0$
in $\mathrm{Ch}\, U_q(\hat{\mathfrak{g}})$.
Repeat it until 
the relation reduces to the form
$Q(\widetilde{T})=0$ with  $\mathrm{ht}\,  Q(T)=1$ or $0$.
However, the former does not occur,
since it contradicts  Theorem \ref{thm:FR}.
Therefore, we have $\mathrm{ht}\,  Q(T)=0$,
i.e., $Q(T)=0$,  which
proves the claim.
\end{proof}

\begin{cor}
\label{cor:qch1}
The ring
 $\EuScript{T}^{\circ}(X_r)$
with $U=\mathbb{C}_{t\hbar}$
is isomorphic to $\mathrm{Rep}\, U_q(\hat{\mathfrak{g}})$
by the correspondence $T^{(a)}_m(u)\mapsto W^{(a)}_m(u)$.
\end{cor}
\begin{proof}
It follows from Theorem \ref{thm:qch1} and Lemma \ref{lem:Tc1} that
\begin{align}
\mathrm{Rep}\, U_q(\hat{\mathfrak{g}})
\simeq
\mathrm{Ch}\, U_q(\hat{\mathfrak{g}})
\simeq
\mathbb{Z}[T]/
(\mathbb{Z}[T^{\pm1}]I(\mathbb{T}(X_r))\cap \mathbb{Z}[T])
\simeq
\EuScript{T}^{\circ}(X_r).
\end{align}
\end{proof}

It should be proved
in the same way, using the character by \cite{Kn},
 that
 $\EuScript{T}^{\circ}(X_r)$ with $U=\mathbb{C}$
is isomorphic to $\mathrm{Rep}\, Y(\mathfrak{g})$.
In Appendix \ref{subsect:Q}
we give parallel results for
the ring associated with the Q-system
and
$\mathrm{Rep}\, U_q(\mathfrak{g})$.

\subsection{Unrestricted Y-systems}

\begin{defn}
The {\it unrestricted Y-system $\mathbb{Y}(X_r)$ of type $X_r$}
is the following system of relations for 
a family of variables
 $Y=\{Y^{(a)}_m(u)
\mid
a\in I,\ m\in \mathbb{N},
u\in U
\}$,
where 
$Y^{(0)}_m (u)=Y^{(a)}_0 (u)^{-1}=0$ if they occur
in the right hand sides in the relations:

For simply laced $X_r$,
\begin{align}
\label{eq:YA1}
Y^{(a)}_m(u-1)Y^{(a)}_m(u+1)
=
\frac{
\prod_{b\in I: C_{ab}=-1}
(1+Y^{(b)}_{m}(u))
}
{
(1+Y^{(a)}_{m-1}(u)^{-1})(1+Y^{(a)}_{m+1}(u)^{-1})}.
\end{align}

For $X_r=B_r$,
\begin{align}
Y^{(a)}_m(u-1)Y^{(a)}_m(u+1)
&=
\frac{
(1+Y^{(a-1)}_{m}(u))(1+Y^{(a+1)}_{m}(u))
}
{
(1+Y^{(a)}_{m-1}(u)^{-1})(1+Y^{(a)}_{m+1}(u)^{-1})
}\\
&\hskip130pt
 (1\leq a\leq r-2),\notag\\
Y^{(r-1)}_m(u-1)Y^{(r-1)}_m(u+1)
&=
\frac{
\begin{array}{l}
\textstyle
(1+Y^{(r-2)}_{m}(u))
(1+Y^{(r)}_{2m-1}(u))
(1+Y^{(r)}_{2m+1}(u))\\
\textstyle
\quad\times(1+Y^{(r)}_{2m}\left(u-\frac{1}{2}\right))
(1+Y^{(r)}_{2m}\left(u+\frac{1}{2}\right))
\end{array}
}
{
(1+Y^{(r-1)}_{m-1}(u)^{-1})(1+Y^{(r-1)}_{m+1}(u)^{-1})
},
\notag
\displaybreak
 \\
Y^{(r)}_{2m}\left(u-\textstyle\frac{1}{2}\right)
Y^{(r)}_{2m}\left(u+\textstyle\frac{1}{2}\right)
&=
\frac{1+Y^{(r-1)}_{m}(u)}
{
(1+Y^{(r)}_{2m-1}(u)^{-1})(1+Y^{(r)}_{2m+1}(u)^{-1})
},\notag
\\
Y^{(r)}_{2m+1}\left(u-\textstyle\frac{1}{2}\right)
Y^{(r)}_{2m+1}\left(u+\textstyle\frac{1}{2}\right)
&=
\frac{1}{(1+Y^{(r)}_{2m}(u)^{-1})(1+Y^{(r)}_{2m+2}(u)^{-1})}.
\notag
\end{align}

For $X_r=C_r$,
{\allowdisplaybreaks
\begin{align}
\label{eq:YC1}
Y^{(a)}_m\left(u-\textstyle\frac{1}{2}\right)
Y^{(a)}_m\left(u+\textstyle\frac{1}{2}\right)
&=
\frac{
(1+Y^{(a-1)}_{m}(u))(1+Y^{(a+1)}_{m}(u))
}{
(1+Y^{(a)}_{m-1}(u)^{-1})(1+Y^{(a)}_{m+1}(u)^{-1})
}\\
&\hskip130pt
 (1\leq a\leq r-2),\notag\\
Y^{(r-1)}_{2m}\left(u-\textstyle\frac{1}{2}\right)
Y^{(r-1)}_{2m}\left(u+\textstyle\frac{1}{2}\right)
&=
\frac{
(1+Y^{(r-2)}_{2m}(u))(1+Y^{(r)}_{m}(u))
}{
(1+Y^{(r-1)}_{2m-1}(u)^{-1})(1+Y^{(r-1)}_{2m+1}(u)^{-1})
},\notag\\
Y^{(r-1)}_{2m+1}\left(u-\textstyle\frac{1}{2}\right)
Y^{(r-1)}_{2m+1}\left(u+\textstyle\frac{1}{2}\right)
&=
\frac{
1+Y^{(r-2)}_{2m+1}(u)
}{
(1+Y^{(r-1)}_{2m}(u)^{-1})(1+Y^{(r-1)}_{2m+2}(u)^{-1})
},\notag\\
Y^{(r)}_{m}(u-1)
Y^{(r)}_{m}(u+1)
&=
\frac{
\begin{array}{l}
\textstyle
(1+Y^{(r-1)}_{2m+1}(u))
(1+Y^{(r-1)}_{2m-1}(u))\\
\textstyle
\quad\times(1+Y^{(r-1)}_{2m}\left(u-\frac{1}{2}\right))
(1+Y^{(r-1)}_{2m}\left(u+\frac{1}{2}\right))
\end{array}
}
{
(1+Y^{(r)}_{m-1}(u)^{-1})(1+Y^{(r)}_{m+1}(u)^{-1})
}.
\notag
\end{align}
}

For $X_r=F_4$,
\begin{align}
Y^{(1)}_m(u-1)Y^{(1)}_m(u+1)
&=
\frac{
1+Y^{(2)}_{m}(u)
}{
(1+Y^{(1)}_{m-1}(u)^{-1})(1+Y^{(1)}_{m+1}(u)^{-1})
},\\
Y^{(2)}_m(u-1)Y^{(2)}_m(u+1)
&=
\frac{
{
\begin{array}{l}
\textstyle
(1+Y^{(1)}_{m}(u))
(1+Y^{(3)}_{2m-1}(u))
(1+Y^{(3)}_{2m+1}(u))\\
\textstyle
\quad\times(1+Y^{(3)}_{2m}\left(u-\frac{1}{2}\right))
(1+Y^{(3)}_{2m}\left(u+\frac{1}{2}\right))
\end{array}
}
}
{
(1+Y^{(2)}_{m-1}(u)^{-1})(1+Y^{(2)}_{m+1}(u)^{-1})
},
\notag\\
Y^{(3)}_{2m}\left(u-\textstyle\frac{1}{2}\right)
Y^{(3)}_{2m}\left(u+\textstyle\frac{1}{2}\right)
&=
\frac{
(1+Y^{(2)}_{m}(u))(1+Y^{(4)}_{2m}(u))
}{
(1+Y^{(3)}_{2m-1}(u)^{-1})(1+Y^{(3)}_{2m+1}(u)^{-1})
},\notag\\
Y^{(3)}_{2m+1}\left(u-\textstyle\frac{1}{2}\right)
Y^{(3)}_{2m+1}\left(u+\textstyle\frac{1}{2}\right)
&=
\frac{
1+Y^{(4)}_{2m+1}(u)
}{
(1+Y^{(3)}_{2m}(u)^{-1})(1+Y^{(3)}_{2m+2}(u)^{-1})
},\notag\\
Y^{(4)}_{m}\left(u-\textstyle\frac{1}{2}\right)
Y^{(4)}_{m}\left(u+\textstyle\frac{1}{2}\right)
&=
\frac{
1+Y^{(3)}_{m}(u)
}{
(1+Y^{(4)}_{m-1}(u)^{-1})(1+Y^{(4)}_{m+1}(u)^{-1})
}.\notag
\end{align}

For $X_r=G_2$,
\begin{align}
\label{eq:YG1}
Y^{(1)}_m(u-1)Y^{(1)}_m(u+1)
&=
\frac{
{
\begin{array}{l}
\textstyle
(1+Y^{(2)}_{3m-2}(u))
(1+Y^{(2)}_{3m+2}(u))\\
\textstyle
\times(1+Y^{(2)}_{3m-1}\left(u-\frac{1}{3}\right))
(1+Y^{(2)}_{3m-1}\left(u+\frac{1}{3}\right))\\
\textstyle
\times(1+Y^{(2)}_{3m+1}\left(u-\frac{1}{3}\right))
(1+Y^{(2)}_{3m+1}\left(u+\frac{1}{3}\right))\\
\textstyle
\times(1+Y^{(2)}_{3m}\left(u-\frac{2}{3}\right))
(1+Y^{(2)}_{3m}\left(u+\frac{2}{3}\right))\\
\times (1+Y^{(2)}_{3m}\left(u\right))
\end{array}
}
}
{
(1+Y^{(1)}_{m-1}(u)^{-1})(1+Y^{(1)}_{m+1}(u)^{-1})
},
\\
Y^{(2)}_{3m}\left(u-\textstyle\frac{1}{3}\right)
Y^{(2)}_{3m}\left(u+\textstyle\frac{1}{3}\right)
&=
\frac{1+Y^{(1)}_m(u)}
{
(1+Y^{(2)}_{3m-1}(u)^{-1})(1+Y^{(2)}_{3m+1}(u)^{-1})
},
\notag\\
Y^{(2)}_{3m+1}\left(u-\textstyle\frac{1}{3}\right)
Y^{(2)}_{3m+1}\left(u+\textstyle\frac{1}{3}\right)
&=
\frac{1}
{
(1+Y^{(2)}_{3m}(u)^{-1})(1+Y^{(2)}_{3m+2}(u)^{-1})
},
\notag\\
Y^{(2)}_{3m+2}\left(u-\textstyle\frac{1}{3}\right)
Y^{(2)}_{3m+2}\left(u+\textstyle\frac{1}{3}\right)
&=
\frac{1}
{
(1+Y^{(2)}_{3m+1}(u)^{-1})(1+Y^{(2)}_{3m+3}(u)^{-1})
}.
\notag
\end{align}
\end{defn}

The choice of the
domain $U=\mathbb{C}_{\xi}$ of the parameter
$u$ effectively imposes an additional periodic condition:
\begin{align}
\label{eq:imperiod12}
Y^{(a)}_m(u)=
Y^{(a)}_m(\textstyle u+ \frac{2\pi \sqrt{-1}}{\xi}).
\end{align}

\begin{defn}
The {\em unrestricted Y-algebra $\EuScript{Y}(X_r)$
of type $X_r$} is the ring with generators
$Y^{(a)}_m(u)^{\pm 1}$, $(1+Y^{(a)}_m(u))^{-1}$
 ($a\in I, m\in \mathbb{N},
u\in U $)
and the relations $\mathbb{Y}(X_r)$.
\end{defn}

The system $\mathbb{Y}(X_r)$ is introduced by \cite{KN}.
See also Remark \ref{rem:ysystem1} for the origin of the $Y$-systems.

Though the T-systems and Y-systems arose
in different contexts with different motivations,
there is a simple and remarkable connection between them
as described below.
Recall that $M^{(a)}_m(u)$ is defined in
(\ref{eq:ta}).

\begin{thm}
\label{thm:TtoY1}
(1) There is a ring homomorphism
\begin{align}
\label{eq:phi1}
\varphi: \EuScript{Y}(X_r) \rightarrow
\EuScript{T}(X_r)
\end{align}
defined by 
\begin{align}
\label{eq:TtoY1}
Y^{(a)}_m(u)\mapsto
\frac{M^{(a)}_m(u)}
{T^{(a)}_{m-1}(u)T^{(a)}_{m+1}(u)},
\end{align}
or, equivalently, by either of
\begin{align}
\label{eq:TtoY2}
1+Y^{(a)}_m(u)&\mapsto
\frac{T^{(a)}_{m}(u-\textstyle\frac{1}{t_a})
T^{(a)}_{m}(u+\textstyle\frac{1}{t_a})}
{T^{(a)}_{m-1}(u)T^{(a)}_{m+1}(u)},\\
\label{eq:TtoY3}
1+Y^{(a)}_m(u)^{-1}&\mapsto
\frac{T^{(a)}_{m}(u-\textstyle\frac{1}{t_a})
T^{(a)}_{m}(u+\textstyle\frac{1}{t_a})}
{M^{(a)}_m(u)},
\end{align}
where $T^{(a)}_0(u)=1$.
\par
(2) There is a ring homomorphism
\begin{align}
\label{eq:psi}
\psi:
{\EuScript{T}}(X_r)
\rightarrow
\EuScript{Y}(X_r)
\end{align}
such that $\psi\circ\varphi = \mathrm{id}_{\EuScript{Y}(X_r)}$.
\par
(3) ${\EuScript{Y}}(X_r)$ is isomorphic to
a subring and a quotient ring of ${\EuScript{T}}(X_r)$.
\end{thm}

The homomorphism $\varphi$ is canonical, while
 $\psi$  is neither unique nor canonical.

\begin{proof}
(3) is a corollary of (1) and (2).
We prove (1) and (2).
Here, we concentrate on the case $U=\mathbb{C}$.
The modification of the proof
 for the case $U=\mathbb{C}_{\xi}$ is straightforward.

{\em (i) The case $X_r$ is simply laced.} 
(1) For simplicity, let us write the image $\varphi(Y^{(a)}_m(u))$ as
$Y^{(a)}_m(u)$.
Then, the relation (\ref{eq:YA1}) is shown as follows:
\begin{align}
\label{eq:yty1}
\begin{split}
&\ Y^{(a)}_m(u-1)Y^{(a)}_m(u+1)\\
=&\
\frac{
\prod_{b:C_{ab}=-1}
T^{(b)}_{m}(u-1)T^{(b)}_{m}(u+1)
}
{
T^{(a)}_{m-1}(u-1)T^{(a)}_{m+1}(u-1)T^{(a)}_{m-1}(u+1)T^{(a)}_{m+1}(u+1)
}\\
=&\,
\frac{
\prod_{b:C_{ab}=-1}
(T^{(b)}_{m-1}(u)T^{(b)}_{m+1}(u)
+
\prod_{c:C_{bc}=-1}T^{(c)}_{m}(u))
}
{
T^{(a)}_{m-2}(u)T^{(a)}_{m}(u)+\prod_{b:C_{ab}=-1}
T^{(b)}_{m-1}(u)
}
\\
&\quad \times
\frac{1}{T^{(a)}_{m}(u)T^{(a)}_{m+2}(u)+\prod_{b:C_{ab}=-1}
T^{(b)}_{m+1}(u)
}
\\
=&\ \frac{
\prod_{b:C_{ab}=-1}(1+Y^{(b)}_{m}(u))
}
{
(1+Y^{(a)}_{m-1}(u)^{-1})(1+Y^{(a)}_{m+1}(u)^{-1})
}.
\end{split}
\end{align}
We remark that the above calculation is valid also at $m=1$
by formally setting $T^{(0)}_{-1}(u)=0$.
\par
(2)
Below we define the image $\psi(T^{(a)}_m(u))$
($a\in I, \ m\in \mathbb{N},\
u\in \mathbb{C}$)
in three steps,
then we show that they satisfy $\mathbb{T}(X_r)$.
For simplicity, let us write the image $\psi(T^{(a)}_m(u))$ as
$T^{(a)}_m(u)$.

Step 1. We arbitrary choose
$T^{(a)}_1(u)\in \EuScript{Y}(X_r)^{\times}$
($a\in I$) for each $u\in \mathbb{C}$
in the region $-1\leq \mathrm{Re}\, u < 1$.

Step 2.
We define $T^{(a)}_1(u)$ ($a\in I$)
for the rest of the region $-2\leq \mathrm{Re}\, u < 2$
by
\begin{align}
\label{eq:ty4}
T^{(a)}_{1}(u\pm 1) =
\left(1+Y^{(a)}_1(u)^{-1}\right)
\frac{M^{(a)}_1(u)}
{T^{(a)}_{1}(u\mp 1)}.
\end{align}
We repeat it to 
define $T^{(a)}_1(u)$ ($a\in I$) for all $u\in \mathbb{C}$.

Step 3.
For each $a$, we recursively
define $T^{(a)}_m(u)$ ($m \geq 2$, $u\in \mathbb{C}$)
by
\begin{align}
\label{eq:ty5}
T^{(a)}_{m+1}(u) =
\frac{1}{1+Y^{(a)}_m(u)}
\frac{T^{(a)}_m(u-1)T^{(a)}_m(u+1)}
{T^{(a)}_{m-1}(u)},
\end{align}
where $T^{(a)}_0(u)=1$.

\begin{claim}
 The family $T$ defined above
satisfies the following relations in $\EuScript{Y}(X_r)$:
\begin{align}
\label{eq:yt4}
{1+Y^{(a)}_m(u)}&=
\frac{T^{(a)}_{m}(u-1)T^{(a)}_{m}(u+1)}
{T^{(a)}_{m-1}(u)T^{(a)}_{m+1}(u)},\\
\label{eq:yt3}
1+Y^{(a)}_m(u)^{-1}&=
\frac
{T^{(a)}_{m}(u-1)T^{(a)}_{m}(u+1)}
{M^{(a)}_m(u)}. 
\end{align}
\end{claim}

The relation (\ref{eq:yt4}) clearly holds
by  (\ref{eq:ty5}).
The relation (\ref{eq:yt3}) is shown by
the induction on $m$,
where the $m=1$ case is true by (\ref{eq:ty4}).
%

Now,
 taking the inverse sum of
(\ref{eq:yt4}) and (\ref{eq:yt3}),
we obtain (\ref{eq:TA1}).
Therefore, $\psi$ is a ring homomorphism.
Furthermore,
taking the ratio of
(\ref{eq:yt4}) and (\ref{eq:yt3}),
we obtain
$Y^{(a)}_m(u)=
{M^{(a)}_m(u)}/
({T^{(a)}_{m-1}(u)T^{(a)}_{m+1}(u)}).
$
This proves $\psi\circ \varphi = \mathrm{id}_{\EuScript{Y}(X_r)}$.

{\em (ii) The case $X_r$ is nonsimply laced.} 
(1) This  can be proved one by one
with similar calculations to (\ref{eq:yty1}), 
though they are slightly more complicated.
\par
(2)
Below we define the image $\psi(T^{(a)}_m(u))$
($a\in I, \ m\in \mathbb{N},\
u\in \mathbb{C}$)
in three steps,
then we show that they satisfy $\mathbb{T}(X_r)$.
For simplicity, let us write the image $\psi(T^{(a)}_m(u))$ as
$T^{(a)}_m(u)$, again.

Step 1. First, we arbitrary choose
$T^{(a)}_1(u)\in \EuScript{Y}(X_r)^{\times}$ 
($a\in I$) for each $u\in \mathbb{C}$
in the region
$-\frac{1}{t_a} \leq \mathrm{Re}\, u < \frac{1}{t_a}$.
Next,
for each $a$ with $t_a=2$ (resp.\ $t_a=3$,
which occurs only for $X_r=G_2$ and $a=2$
)
we define $T_2^{(a)}(u)$ 
(resp.\ $T_2^{(a)}(u)$  and $T_3^{(a)}(u)$)
in the region
$-\frac{1}{t_a} \leq \mathrm{Re}\, u < \frac{1}{t_a}$
by
\begin{align}\label{eq:ta2}
  T_{m+1}^{(a)}(u) = \frac{
M_{m}^{(a)}(u)}
{Y_{m}^{(a)}(u) T_{m-1}^{(a)}(u)},
\end{align}
or, more explicitly,
\begin{align}\label{ta=23}
  T_2^{(a)}(u) = \frac{T_1^{(a-1)}(u) T_1^{(a+1)}(u)}{Y_1^{(a)}(u)}, \quad
  T_3^{(2)}(u) = \frac{T_1^{(1)}(u-\frac{1}{3}) T_1^{(1)}(u+\frac{1}{3})}
                      {Y_2^{(2)}(u) T_1^{(2)}(u)},
\end{align}
where $T_1^{(0)}(u) =T_1^{(r+1)}(u)= 1$.

Step 2.
Let $t$ be the number in (\ref{eq:t1}).
First, we define $T^{(a)}_1(u)$ ($a\in I$)
for the rest of the region $-\frac{1}{t_a}-\frac{1}{t}
\leq \mathrm{Re}\, u < \frac{1}{t_a}+\frac{1}{t}$
by
\begin{align}\label{m=1}
{\textstyle
  T_1^{(a)}(u \pm \frac{1}{t_a}) }
  =
  \left( 1+Y_1^{(a)}(u)^{-1} \right)
 \frac{M_1^{(a)}(u)}
     {T_1^{(a)}(u \mp \frac{1}{t_a})}.
\end{align} 
Next, 
we define $T^{(a)}_m(u)$ ($t_a=2,3$; $m=2,\dots,t_a$)
for the rest of the region $-\frac{1}{t_a}-\frac{1}{t}
\leq \mathrm{Re}\, u < \frac{1}{t_a}+\frac{1}{t}$
by (\ref{eq:ta2}).
We repeat it to 
define $T^{(a)}_m(u)$ ($a\in I$; $m=1,\dots,t_a$)
 for all $u\in \mathbb{C}$.

Step 3. For each $a$, we recursively define $T_m^{(a)}(u)$
($m  > t_a$, $u\in \mathbb{C}$) by
\begin{align}\label{T-Y1'}
  T_{m+1}^{(a)}(u) 
  = 
  \frac{1}{1+Y_m^{(a)}(u)}
  \frac{T_m^{(a)}(u-\frac{1}{t_a}) T_m^{(a)}(u+\frac{1}{t_a})}
       {T_{m-1}^{(a)}(u)},                              
\end{align}
where $T_0^{(a)}(u) = 1$.

\begin{claim1}
The family $T$ defined above satisfies the
following relations in $\EuScript{Y}(X_r)$:
\begin{align}
  \label{taT-Y2}
  1+Y_{m}^{(a)}(u) 
  &= 
  \frac{T_m^{(a)}(u-\frac{1}{t_a}) T_m^{(a)}(u+\frac{1}{t_a})}
       {T_{m-1}^{(a)}(u)T_{m+1}^{(a)}(u)}
\quad(t_a = 2,3;\ m=1,\ldots,t_a-1),\\
  \label{taT-Y1}
  1+Y_m^{(a)}(u)^{-1} 
  &= 
  \frac{T_m^{(a)}(u-\frac{1}{t_a}) T_m^{(a)}(u+\frac{1}{t_a})}
       {M_m^{(a)}(u)}
\quad(t_a = 2,3;\ m=2,\ldots,t_a).
\end{align}
\end{claim1}

The relation \eqref{taT-Y2} for $m=1$ is an
immediate consequence of  \eqref{eq:ta2} and \eqref{m=1}.
The relation \eqref{taT-Y1} for $m=2$ is verified one by one.
For $t_a=3$,
\eqref{taT-Y2} for $m=2$ is an
immediate consequence of  \eqref{eq:ta2} and \eqref{taT-Y1}
for $m=2$;
and \eqref{taT-Y1} for $m=3$ is verified by
\eqref{eq:ta2} and \eqref{m=1}.

\begin{claim2}
The family $T$ defined above satisfies the 
following relations in $\EuScript{Y}(X_r)$ for any $(a,m,u)$:
\begin{align}
  \label{T-Y2}
  1+Y_m^{(a)}(u) 
  &= 
  \frac{T_m^{(a)}(u-\frac{1}{t_a}) T_m^{(a)}(u+\frac{1}{t_a})}
  {T_{m-1}^{(a)}(u) T_{m+1}^{(a)}(u)},\\
  \label{T-Y1}
  1+Y_m^{(a)}(u)^{-1} 
  &= 
  \frac{T_m^{(a)}(u-\frac{1}{t_a}) T_m^{(a)}(u+\frac{1}{t_a})}
       {M_m^{(a)}(u)}.
\end{align}
\end{claim2}

The relation \eqref{T-Y2} holds for any $(a,m,u)$ because of 
\eqref{T-Y1'} and \eqref{taT-Y2}.
The relation \eqref{T-Y1} holds for $m=1,\ldots,t_a$ because of 
\eqref{m=1} and \eqref{taT-Y1}.
Then, one can verify \eqref{T-Y1} by the induction on $m$
one by one.
\par
The rest of the argument is the same as 
for
the simply laced case.
\end{proof}

\begin{rem}
The transformation (\ref{eq:TtoY1}) first appeared in \cite{KP}
for the simplest case $X_r=A_1$,
and generalized in \cite{KNS1} for general $X_r$.
The analogous transformation  plays
an important role also in the approach by
 cluster algebras with coefficients  \cite{FZ4}.
\end{rem}

\subsection{Regular solutions of
T and Y-systems}

In application,
we usually consider solutions of $\mathbb{T}(X_r)$ and
$\mathbb{Y}(X_r)$ in a particular ring.

\begin{defn}
\label{defn:regular1}
Let $R$ be a ring.
\par
(i) A family   $T=\{T^{(a)}_m(u)\in R
\mid
a\in I,\ m\in \mathbb{N},
u\in U
\}$ satisfying $\mathbb{T}(X_r)$ is called
a {\em solution of the T-system $\mathbb{T}(X_r)$ in $R$}.
We say a solution $T$ of $\mathbb{T}(X_r)$ in $R$ 
is {\em regular} if $T^{(a)}_m(u)\in R^{\times}$
for any $(a,m,u)$.
\par
(ii) A family   $Y=\{Y^{(a)}_m(u)\in R
\mid
a\in I,\ m\in \mathbb{N},
u\in U
\}$ satisfying $\mathbb{Y}(X_r)$ is called
a {\em solution of the Y-system $\mathbb{Y}(X_r)$ in $R$}.
We say a solution $Y$ of\/ $\mathbb{Y}(X_r)$ in $R$ 
is {\em regular} if $Y^{(a)}_m(u)$, $1+Y^{(a)}_m(u)\in R^{\times}$
for any $(a,m,u)$.
\end{defn}

\begin{rem}
Actually any solution of $\mathbb{Y}(X_r)$  is
regular, because, for any $Y^{(a)}_m(u)$,
there is a relation among $\mathbb{Y}(X_r)$ such that
 $(1+Y^{(a)}_m(u)^{-1})^{-1}$ appears
in the right hand side.
However, this is not always true for the
restricted Y-system we shall discuss in Section
\ref{sect:restricted}.
Therefore, it is convenient to introduce the above
definition so that
the unrestricted/restricted T/Y-systems can be treated in a unified manner.
\end{rem}

Clearly, there is a one-to-one correspondence between
the regular solutions of $\mathbb{T}(X_r)$
(resp.\ $\mathbb{Y}(X_r)$) in $R$
and the ring homomorphisms $f:\EuScript{T}(X_r)\rightarrow R$
(resp.\ $f:\EuScript{Y}(X_r)\rightarrow R$).

As a corollary of Theorem
\ref{thm:TtoY1} (2), we obtain
\begin{cor}
For any ring $R$, the map
\begin{align}
\varphi^*: \mathrm{Hom}\, (\EuScript{T}(X_r),R)
\rightarrow 
\mathrm{Hom}\, (\EuScript{Y}(X_r),R),
\end{align}
induced from the homomorphism $\varphi$ in (\ref{eq:phi1}), is surjective.
Namely, for any regular solution $Y$ of $\mathbb{Y}(X_r)$ in $R$,
there exists  some
regular solution $T$ of $\mathbb{T}(X_r)$ in $R$
such that $Y$ is expressed by $T$ as
\begin{align}
\label{eq:TtoY5}
Y^{(a)}_m(u)=
\frac{M^{(a)}_m(u)}
{T^{(a)}_{m-1}(u)T^{(a)}_{m+1}(u)}.
\end{align}
\end{cor}

\section{Restricted T and Y-systems,
and their periodicities
}
\label{sect:restricted}

In this section we state the main claims of the paper.
We first introduce the {\em restricted\/}
T and Y-systems together with the associated algebras.
Then, the conjectures
and the results concerning their periodicity property
are presented.

\subsection{Restricted T and Y-systems}

Let $t_a$ ($a\in I$) be the numbers in (\ref{eq:t1}).

\begin{defn}
\label{defn:RT}
Fix an integer $\ell \geq 2$.
The {\it level $\ell$ restricted T-system $\mathbb{T}_{\ell}(X_r)$
of type $X_r$
(with the unit boundary condition)}
is the system of relations
(\ref{eq:TA1})--(\ref{eq:TG1}) naturally restricted to
a family of variables $T=\{T^{(a)}_m(u)
\mid
a\in I ; m=1,\dots,t_a\ell-1;
u\in U
\}$,
where 
$T^{(0)}_m (u)=T^{(a)}_0 (u)=1$,
and furthermore,  $T^{(a)}_{t_a\ell}(u)=1$
(the {\em unit boundary condition\/}) if they occur
in the right hand sides in the relations.
\end{defn}

\begin{defn}
\label{defn:RT2}
The {\em level $\ell$ restricted T-algebra $\EuScript{T}_{\ell}(X_r)$
of type $X_r$} is the ring with generators
$T^{(a)}_m(u)^{\pm 1}$
 ($a\in I; m=1,\dots,t_a\ell-1; 
u\in U $)
and the relations $\mathbb{T}_{\ell}(X_r)$.
Also, we define the ring $\EuScript{T}^{\circ}_{\ell}(X_r)$
as the subring of $\EuScript{T}_{\ell}(X_r)$
generated by 
$T^{(a)}_m(u)$ 
 ($a\in I; m=1,\dots,t_a\ell-1; u\in U $).
\end{defn}

\begin{rem}
\label{rem:rt}
The notion of the {\em level $\ell$ restriction\/} originates from
a class of solvable lattice model, called
the level $\ell$ restricted solid-on-solid (RSOS) model
associated with the R-matrix of
$U_q(\hat{\mathfrak{g}})$ at
a $2t(h^{\vee}+\ell)$th root of unity
\cite{ABF,JMO,Pas,BR}.
The {\it level $\ell$ restricted T-system} was
introduced in \cite{KNS1},
where, instead of the condition  $T^{(a)}_{t_a\ell}(u)=1$ above,
a slightly weaker condition $T^{(a)}_{t_a\ell+1}(u)=0$ was imposed.
We hope that no serious confusion occurs by
 calling $\mathbb{T}_{\ell}(X_r)$ also
as {\it level $\ell$ restricted T-system} for simplicity.
We impose the unit boundary condition here  to ensure
 the periodicity property we are going to discuss.
(Actually,
 this is not the only choice of the boundary condition showing the periodicity,
but we do not discuss this point in the paper.)
\end{rem}

\begin{prop}
The ring $\EuScript{T}^{\circ}_{\ell}(X_r)$
is isomorphic to a quotient of $\EuScript{T}^{\circ}(X_r)$.
\end{prop}

\begin{proof}
First we
 note that the ring $\EuScript{T}^{\circ}(X_r)$ is freely generated by
$T^{(a)}_1(u)$ ($a\in I,u\in U$).
This is true for
$U=\mathbb{C}_{t\hbar}$
 by Theorem \ref{thm:FR},
and so is for any choice of $U$,
since nontrivial relations exist
only among $T^{(a_i)}_{m_i}(u_i)$'s with $u_i-u_j \in\mathbb{R}$.
So we have a ring homomorphism
\begin{align}
\label{eq:piel11}
\pi_{\ell}:
\EuScript{T}^{\circ}(X_r)
\rightarrow
\EuScript{T}^{\circ}_{\ell}(X_r)
\end{align}
uniquely
determined by the condition $\pi_{\ell}(T^{(a)}_1(u))=T^{(a)}_1(u)$
($a\in I,u\in U$).
We claim that $\pi_{\ell}(T^{(a)}_m(u))=T^{(a)}_m(u)$
for any $m=1,\dots,t_a\ell-1$, from which
the surjectivity of $\pi_{\ell}$ follows.
The claim can be shown
by the induction on the height of
$T^{(a)}_m(u)$ in \eqref{eq:height1}.
Namely, suppose that the claim holds for
any  $T^{(b)}_{k}(v)$
such that $\mathrm{ht}\,T^{(b)}_{k}(v)$ is smaller than
$\mathrm{ht}\,T^{(a)}_m(u)$.
Let $S_{amu}(T)$ be the one in \eqref{eq:sa}.
Then, $S_{a,m-1,u}(T)=0$ in  
$\EuScript{T}^{\circ}(X_r)$ and
$\EuScript{T}^{\circ}_{\ell}(X_r)$;
hence,
$\pi_{\ell}(S_{a,m-1,u}(T))
=S_{a,m-1,u}(T)$ in $\EuScript{T}^{\circ}_{\ell}(X_r)$.
The claim follows from this and the induction hypothesis.
\end{proof}

Similarly,

\newpage

\begin{defn}
\label{defn:RY}
Fix an integer $\ell \geq 2$.
The {\it level $\ell$ restricted Y-system $\mathbb{Y}_{\ell}(X_r)$
of type $X_r$}
is the system of relations
(\ref{eq:YA1})--(\ref{eq:YG1})
naturally restricted to
a family of variables $Y=\{Y^{(a)}_m(u)
\mid
a\in I; m=1,\dots,t_a\ell-1; 
u\in U
\}$,
where 
$Y^{(0)}_m (u)=Y^{(a)}_0 (u)^{-1}=0$,
and furthermore, $Y^{(a)}_{t_a\ell}(u)^{-1}=0$
 if they occur
in the right hand sides in the relations.
\end{defn}

\begin{defn}
\label{defn:RY2}
The {\em level $\ell$ restricted Y-algebra $\EuScript{Y}_{\ell}(X_r)$
of type $X_r$} is the ring with generators
$Y^{(a)}_m(u)^{\pm 1}$, $(1+Y^{(a)}_m(u))^{-1}$
($a\in I; m=1,\dots,t_a\ell-1; 
u\in U $)
and the relations $\mathbb{Y}_{\ell}(X_r)$.
\end{defn}

\begin{rem}
\label{rem:ysystem1}
The system $\mathbb{Y}_{\ell}(X_r)$ was introduced by
\cite{Z} for simply laced $X_r$ and $\ell=2$
to characterize the solutions of the thermodynamic
Bethe ansatz equations for the
factorizable scattering theories.
Then, it was extended to the general case by \cite{KN}
based on the thermodynamic treatment of \cite{Ku}.
See also \cite[Appendix B]{KNS1}.
For  simply laced $X_r$, it was also given by \cite{RTV}
independently.
\end{rem}

For any ring $R$,
one can define the {\em regular solutions of $\mathbb{T}_{\ell}(X_r)$
and $\mathbb{Y}_{\ell}(X_r)$ in $R$} in the same
way as Definition \ref{defn:regular1}.
Again, they are identified with the elements 
in $\mathrm{Hom}\, (\EuScript{T}_{\ell}(X_r),R)$
and $\mathrm{Hom}\, (\EuScript{Y}_{\ell}(X_r),R)$.

The restrictions of T-systems and Y-systems are
{\em partly\/} compatible
in view of Theorem \ref{thm:TtoY1}. Namely,

\begin{prop}
\label{prop:TtoY2}
The correspondence (\ref{eq:TtoY1}),
with $T^{(a)}_0(u)=T^{(a)}_{t_a\ell}(u)=1$,
defines a ring homomorphism
\begin{align}
\varphi_{\ell}: \EuScript{Y}_{\ell}(X_r) \rightarrow
\EuScript{T}_{\ell}(X_r).
\end{align}
\end{prop}

\begin{proof}
Due to Theorem \ref{thm:TtoY1},
we have only to check  the compatibility between the  boundary
conditions, $T^{(a)}_{t_a\ell}(u)=1$ and
 $Y^{(a)}_{t_a\ell}(u)^{-1}=0$.
For simply laced $X_r$,
this can be seen by formally
setting $T^{(a)}_{\ell+1}(u)=0$
at $m=\ell-1$
in (\ref{eq:yty1}).
The nonsimply laced case is similar.
\end{proof}

Unfortunately, the properties (2) and (3) in
Theorem \ref{thm:TtoY1} 
do not necessarily hold for general $X_r$ and $\ell$.
\begin{exmp}
\label{exmp:surj}
(1)
{\em The case $X_r=A_2$  and $\ell=2$.}
Two systems,
\begin{align}
\mathbb{T}_{2}(A_2):\ 
T^{(1)}_1(u-1)T^{(1)}_1(u+1)&=1+T^{(2)}_1(u),\\
T^{(2)}_1(u-1)T^{(2)}_1(u+1)&=1+T^{(1)}_1(u),\notag\\
\mathbb{Y}_{2}(A_2):\ 
Y^{(1)}_1(u-1)Y^{(1)}_1(u+1)&=1+Y^{(2)}_1(u),\\
Y^{(2)}_1(u-1)Y^{(2)}_1(u+1)&=1+Y^{(1)}_1(u),\notag
\end{align}
are identical;
moreover,
we have
$\varphi_{2}:Y^{(1)}_1(u)\mapsto T^{(2)}_1(u)$,
$Y^{(2)}_1(u)\mapsto T^{(1)}_1(u)$.
Thus, $\varphi_{2}$ is bijective.

(2)
{\em The case $X_r=A_3$  and $\ell=2$.}
We have
$\varphi_{2}:Y^{(1)}_1(u)\mapsto T^{(2)}_1(u)$,
$Y^{(3)}_1(u)\mapsto T^{(2)}_1(u)$.
Thus, $\varphi_{2}$ is {\em not\/} injective.

(3)
{\em The case $X_r=C_2$  and $\ell=2$.}
We have
$\varphi_{2}:Y^{(1)}_1(u)\mapsto T^{(2)}_1(u)/T^{(1)}_2(u)$,
$Y^{(1)}_3(u)\mapsto T^{(2)}_1(u)/T^{(1)}_2(u)$.
Thus, $\varphi_{2}$ is {\em not\/} injective.
\end{exmp}

However, at least  for
$A_r$, one can resolve this incompatibility by
modifying the  boundary condition of
${\mathbb{T}}_{\ell}(X_r)$
while keeping the periodicity (Proposition \ref{prop:TtoY4}).

There are some isomorphisms among the restricted
T-algebras or Y-algebras.
\begin{exmp}[Level-rank duality]
\label{exmp:lrd}
The rings 
$\EuScript{T}_{\ell}(A_{r-1})$
and 
$\EuScript{T}_{r}(A_{\ell-1})$
are isomorphic under the correspondence
$T^{(a)}_m(u) \leftrightarrow
T^{(m)}_a(u)$.
The rings 
$\EuScript{Y}_{\ell}(A_{r-1})$
and 
$\EuScript{Y}_{r}(A_{\ell-1})$
are isomorphic under the correspondence
$Y^{(a)}_m(u) \leftrightarrow
Y^{(m)}_a(u)^{-1}$.

\end{exmp}

\subsection{T and Y-systems with discrete spectral parameter}

So far, we assume that the spectral parameter $u$
 takes values in $U=\mathbb{C}$ or $\mathbb{C}_{\xi}$.
In the original context of T and Y-systems,
the analyticities of $T^{(a)}_m(u)$ and $Y^{(a)}_m(u)$ with respect
to $u$ are of fundamental importance \cite{Z,KP,KN,RTV,KNS2}.

However, from the algebraic point of view,
it is possible to discretize the parameter $u$
by choosing $U= \frac{1}{t}\mathbb{Z}$,
where $t$ is the number in (\ref{eq:t1}).
There are at least two reasons why we are interested in
such a discretization.

Firstly, by regarding $u$ as `discretized time',
the T and Y-systems have their own interests as
{\em discrete dynamical systems}.
For example, $\mathbb{T}(A_r)$ is
a discrete analogue of the Toda field equation
and a particular case
of the Hirota's bilinear difference equation \cite{Hi1,Hi2,
KOS,KLWZ}.
See [KLWZ] for more information.

Secondly, the periodicities of the restricted T and Y-systems,
which are the subjects of the paper, concern only with the algebraic
aspect of the T and Y-systems; therefore,
 it is adequate to discuss the periodicities in discretized systems.

{\em From now on till the end of Section \ref{sect:lev10},
 we assume $U=\frac{1}{t}\mathbb{Z}$ for
all the T-systems and Y-systems.
}

\subsection{Periodicity Conjecture for restricted T and Y-systems}
\label{subsect:period}

For $X_r$, let $h^{\vee}$ be the dual Coxeter number of $X_r$
as listed below.
\begin{align}
\begin{tabular}{c|ccccccccc}
  $X_r$ & $A_r$ & $B_r$ & $C_r$ & $D_r$ & $E_6$ & $E_7$ & $E_8$ &
$F_4$ & $G_2$ \\
  \hline
 $h^{\vee}$& $r+1$& $2r-1$&$r+1$&$2r-2$&$12$&$18$&$30$&$9$&$4$\\
\end{tabular}
\end{align}
For simply laced $X_r$, $h^{\vee}$ equals to the Coxeter number
$h$ of $X_r$.

Let $\omega$ be the involution on the set $I$ such that
$\omega(a)=a$ {\em except for the following  cases
(in our enumeration)\/}:
\begin{alignat}{2}
\label{eq:omega1}
& \omega(a)=r+1-a\quad (a\in I)&&X_r=A_r,\notag \\
&  \omega(r-1)=r,\ \omega(r)=r-1&&X_r=D_r\ \text{($r$: odd)}, \\
& \omega(1)=6,\  \omega(2)=5,\  \omega(5)=2,\
 \omega(6)=1&\quad& X_r=E_6.\notag
\end{alignat}
(Caution: For $X_r=D_r$ ($r$: even),  $\omega(a)=a$ ($a\in I$).)
The involution $\omega$ is related to 
 the {\em longest element\/} $\omega_0$ in the Weyl group of 
type $X_r$
by $\omega_0(\alpha_a)=-\alpha_{\omega(a)}$ \cite{B}
 (cf.\  \cite[Proposition 2.5]{FZ3}).

Now let us give the main claim of the paper.

\begin{conj}
\label{conj:Tperiod1}
The following relations hold in  $\EuScript{T}_{\ell}(X_r)$:

(1) Half-periodicity: $T^{(a)}_m(u+h^\vee+\ell)=
T^{(\omega(a))}_{t_a\ell-m}(u)$.

(2) Periodicity: $T^{(a)}_m(u+2(h^\vee+\ell))=
T^{(a)}_{m}(u)$.
\end{conj}

We may sometimes refer to (2) also as {\em full-periodicity}
in contrast to (1).
Of course, the full-periodicity follows from the half-periodicity.

This is the counterpart of the
(already conjectured and partially proved)
 periodicity property  for
the restricted Y-systems
in various contexts;
here we present it in the parallel form
to Conjecture \ref{conj:Tperiod1}:

\begin{conj}
\label{conj:Yperiod1}
The following relations hold in  $\EuScript{Y}_{\ell}(X_r)$:

(1) Half-periodicity: $Y^{(a)}_m(u+h^\vee+\ell)=
Y^{(\omega(a))}_{t_a\ell-m}(u)$.

(2) Periodicity: $Y^{(a)}_m(u+2(h^\vee+\ell))=
Y^{(a)}_{m}(u)$.
\end{conj}

\begin{rem}
One can rephrase these periodicity properties 
as those of the regular
solutions of the corresponding T and Y-systems in an arbitrary ring $R$.
For example, suppose that
Conjecture \ref{conj:Tperiod1} (1) is true.
Then, for any regular solution $T$ of $\mathbb{T}_{\ell}(X_r)$
in $R$,
the equality $T^{(a)}_m(u+h^\vee+\ell)=
T^{(\omega(a))}_{t_a\ell-m}(u)$ holds in $R$.
The converse is also true by setting $R=\EuScript{T}_{\ell}(X_r)$.
This remark will be applicable  to any
 periodicity statement in the rest of the paper as well.
\end{rem}

Let us summarize the known and/or related  results
 on Conjectures \ref{conj:Tperiod1}
and \ref{conj:Yperiod1} so far.

(i) Conjecture \ref{conj:Yperiod1} was 
initially given by
 \cite{Z} for simply laced $X_r$ and $\ell =2$,
then generalized by
 \cite{RTV} for simply laced $X_r$ and $\ell\geq 2$
(including half-periodicity),
and by \cite[Appendix B]{KNS1} for general $X_r$ and $\ell\geq 2$
(full-periodicity).
The established so far are as follows:

\begin{itemize}
\item[(a)]
It was proved 
for $X_r=A_r$ and $\ell =2$ by
  Gliozzi-Tateo
\cite{GT}
via three-dimensional geometry.
The same case was also proved by
 Frenkel-Szenes
\cite{FS} with the  explicit solution given.


\item[(b)] It was proved 
for simply laced $X_r$ and $\ell =2$ by Fomin-Zelevinsky
\cite{FZ3} via
the cluster algebra method.

\item[(c)] It was proved
for $X_r=A_r$ and $\ell\geq 2$ by Volkov
\cite{V} via
the determinant method.

\item[(d)] 
The full-periodicity was proved
for simply laced $X_r$ and  $\ell> 2$ by 
Keller \cite{Kel2} via
the cluster algebra/category method.
\end{itemize}

\noindent
We emphasize that
the nonsimply laced Y-systems treated in \cite{FZ3,Kel2}
are {\em different\/} from ours,
and their nonsimply laced
Y-systems are identified with certain reductions of 
our Y-systems associated with
the {\em twisted\/} quantum affine algebras.
See Remark \ref{rem:nonsim}.
In particular, there has been no systematic result
on Conjecture \ref{conj:Yperiod1}
 for the nonsimply laced case so far.
The same remark applies to the T-systems as well.

(ii)
Conjecture \ref{conj:Tperiod1} appeared in
\cite{CGT} for simply laced $X_r$,
while the one for nonsimply laced $X_r$ seems new
in the literature.
The following related results are already known:
\begin{itemize}
\item[(a)] For simply laced $X_r$,
we will see that the ring $\EuScript{T}^{\circ}_{\ell}(X_r)$ 
is isomorphic to (a subring of) a certain cluster algebra.
The periodicity property of the
corresponding cluster algebra is known
for $\ell=2$ (including half-periodicity)
by Fomin-Zelevinsky \cite{FZ2,FZ3},
and, for $\ell> 2$ (full-periodicity only)
by \cite{Kel2}.
A more precise account
will be given in Section \ref{sect:cluster}.

\item[(b)]
For $X_r=A_r$ and  $\ell\geq 2$,
Conjecture \ref{conj:Tperiod1}
follows from a more general theorem
by Henriques \cite{Hen} proved by
the graph theoretical method.
The same case was also proved essentially   by  \cite{V}
while proving Conjecture \ref{conj:Yperiod1}.
 We will give a detailed account of the latter method
 in Section\ \ref{sect:TA}.
\end{itemize}

(iii) 
Though Conjectures \ref{conj:Tperiod1} and \ref{conj:Yperiod1} are
tightly connected to each other in view of  the map
$\varphi_{\ell}$ in Proposition \ref{prop:TtoY2},
one is not the consequence of the other,
 in general.
However, at least for simply laced $X_r$, they are unified
as the periodicity property of the 
{\em $F$-polynomials\/} of the corresponding cluster algebra
{\em with coefficients\/} \cite{FZ4}.

\begin{rem}
Recall that the choice $U=\mathbb{C}_{\xi}$ for the domain
of the parameter $u$ of the unrestricted T-algebra
$\EuScript{T}(X_r)$ imposes the period
$2 \pi \sqrt{-1}/\xi$ in (\ref{eq:imperiod1}),
where $\xi$ is taken from $\mathbb{C}
\setminus 2\pi \sqrt{-1}\mathbb{Q}$
to avoid the incompatibility with  the relations
 $\mathbb{T}(X_r)$.
The level $\ell$ restricted T-algebra
$\EuScript{T}_{\ell}(X_r)$ has
an additional period $2 (h^{\vee}+\ell)$.
This means the choice
$\xi =\pi \sqrt{-1}/(h^{\vee}+\ell)$
is {\em compatible\/} with
 the relations $\mathbb{T}_{\ell}(X_r)$.
In the context of the $q$-character,
we made the identification
$\xi=t\hbar$, where $q=e^{\hbar}$.
Then, the above choice corresponds
to $q=\exp(\pi \sqrt{-1}/t(h^{\vee}+\ell))$,
namely, {\em $q$ is a primitive $2t(h^{\vee}+\ell)${\em th}
root of unity.}
This is natural in view of the
origin of the level $\ell$ restriction
in Remark \ref{rem:rt}.
We make a further remark on the implication of the
periodicity of $\EuScript{T}_{\ell}(X_r)$ for the $q$-character
in Section \ref{sect:qch}.
\end{rem}

\subsection{Summary of methods and results}

In the following, we will study and partially prove
Conjecture \ref{conj:Tperiod1}
by three independent methods.
This is a good point to
 outline the methods and the  results.

{\em 1. Cluster algebra/category method
 applied to
 $\EuScript{T}_{\ell}(X_r)$  with simply laced $X_r$:}
(Section \ref{sect:cluster})

This is actually more than a method
to prove Conjecture \ref{conj:Tperiod1},
since it includes the identification of the ring
$\EuScript{T}^{\circ}_{\ell}(X_r)$ as a (subring of) cluster algebra.

In the simplest case $\ell =2$,
 the ring $\EuScript{T}^{\circ}_{2}(X_r)$
is isomorphic to the tensor square of the
 cluster algebra $\mathcal{A}_{Q}$
of type $X_r$ (Proposition \ref{prop:CT1}).
The ring $\mathcal{A}_{Q}$
 is a cluster algebra of {\em finite type}
and particularly well studied.
In particular, the periodicity property of $\mathcal{A}_{Q}$ is proved
in \cite{FZ2,FZ3} by making use of the piecewise-linear
modification of the simple reflections
acting on the set of the almost positive roots
$\Phi_{\geq -1}$ of type $X_r$.
The periodicity of  $\EuScript{T}_{2}(X_r)$
 is its immediate corollary (Corollary \ref{cor:SL1}).

For the case $\ell> 2$, 
 the ring $\EuScript{T}^{\circ}_{\ell}(X_r)$
is isomorphic to the tensor square of a
subring of the cluster algebra 
$\mathcal{A}_{Q\square Q'}$,
where $Q\square Q'$
is the {\em square product\/} of quivers
recently introduced by \cite{Kel2, HL}
(Proposition \ref{prop:CT2}).
The cluster algebra $\mathcal{A}_{Q\square Q'}$ 
 is {\em not\/} of finite type;
nevertheless, it still admits the periodicity along the
{\em bipartite belt\/} of \cite{FZ4}.
The periodicity
of $\mathcal{A}_{Q\square Q'}$ 
 is studied in \cite{Kel2}, in a more
general situation {\em with coefficients\/},
using the categorification by the 2-Calabi-Yau category
associated with the tensor product of the path algebras
of quivers $Q$ and $Q'$.
The full-periodicity of  $\EuScript{T}_{\ell}(X_r)$
 is its immediate corollary.
Furthermore, this cluster categorical
approach can be adapted for the
half-periodicity.
Thus, we obtain the desired periodicity for
$\EuScript{T}_{\ell}(X_r)$ (Corollary \ref{cor:SL3}).

{\em 2. Determinant method
 applied to
 $\EuScript{T}_{\ell}(A_r)$ and  $\EuScript{T}_{\ell}(C_r)$:}
(Sections \ref{sect:TA}/\ref{sect:TC})

The method seeks a  manifestly periodic
expression of $T^{(a)}_m(u)$ as
a minor  of a matrix $M$
over $\EuScript{T}_{\ell}(X_r)$ of infinitely-many
finite columns
{\em with periodicity}.
It was introduced by \cite{V} to prove
the periodicity of the regular solutions of
$\mathbb{Y}_{\ell}(A_r)$ in $\mathbb{C}$.

Such a determinant expression (without periodicity)
is known for the unrestricted T-system $\mathbb{T}(A_r)$
by \cite[Eq.\ (2.25)]{KLWZ},
where the relation (\ref{eq:TA1}) of  $\mathbb{T}(A_r)$ is regarded
as the Hirota's bilinear difference equation.
Then, the  existence of such a determinant expression is viewed
as a discrete analogue of the well-known relation between
 the Hirota's bilinear equation and the Grassmannians \cite{S}.
Remarkably, the restriction of the T-system to 
$\mathbb{T}_{\ell}(A_r)$ is compatible with
this determinant expression by imposing the
{\em periodicity\/} on the matrix $M$ (Proposition \ref{prop:Deta1}).
This forces the desired periodicity for
 $\EuScript{T}_{\ell}(A_r)$ (Theorem \ref{thm:TAperiod}).

Since the method takes advantage of the bilinearity
of the relation of $\mathbb{T}(A_r)$,
it does not seem applicable to $X_r$
other than $A_r$.
A pleasant surprise is that it is still applicable for 
$C_r$
through the relation between
  $\EuScript{T}(C_r)$ and
a certain variant of $\EuScript{T}(A_{2r+1})$
\cite{KOSY}.
(Note that this is different from the usual `folding' relation
between $C_r$ and $A_{2r-1}$.)
This relation is compatible with the restriction,
and induces the relation between
$\EuScript{T}_{\ell}(C_r)$ and $\widehat{\EuScript{T}}_{2\ell}(A_{2r+1})$,
where $\widehat{\EuScript{T}}_{2\ell}(A_{2r+1})$
is a variant of  $\EuScript{T}_{2\ell}(A_{2r+1})$
(Proposition \ref{prop:bij1}).
Since $\widehat{\EuScript{T}}_{2\ell}(A_{2r+1})$
admits the determinant expression, the desired periodicity
for $\EuScript{T}_{\ell}(C_r)$ is obtained (Corollary
 \ref{cor:TCperiod}).
This is the first main result concerning
Conjecture \ref{conj:Tperiod1} for the nonsimply laced case.

At this moment the method is applicable only for
these two cases,
since a similar relation 
between  $\EuScript{T}(X_r)$ and
a certain variant of $\EuScript{T}(A_{r'})$
is not known for the other types $X_r$.

{\em 3. Direct method  applied to
 $\EuScript{T}_{2}(A_r)$,  $\EuScript{T}_{2}(D_r)$,
and  $\EuScript{T}_{2}(B_r)$:} (Section \ref{sect:direct})

The method seeks a  manifestly periodic
Laurent polynomial expression of $T^{(a)}_m(u)$ 
in terms of the `initial variables'
by considering the T-system
as a discrete dynamical system.
At least for the above three cases,
we can directly find such an expression
with the aid of computer,
and verify  that it indeed satisfies the T-system.

The problem to express the  cluster variables in terms of
the initial cluster is a much studied subject
(e.g.\ \cite{CC,FZ4,YZ}, etc.).
The first two cases,  $\EuScript{T}_{2}(A_r)$ and
$\EuScript{T}_{2}(D_r)$, should be obtained as the specialization
of those more general expressions.
Our goal here is to prove the periodicity
for $\EuScript{T}_{2}(B_r)$,
which is the first nontrivial result for $B_r$.
(Let us repeat that this is different from the tensor square 
of the cluster 
algebra of type $B_r$.)

\section{Cluster algebra/category method:
 $\EuScript{T}_{\ell}(X_r)$ with simply laced $X_r$}
\label{sect:cluster}

In this section,
we study the periodicity of
$\EuScript{T}_{\ell}(X_r)$
for simply laced $X_r$.
We establish the relation between
the ring
$\EuScript{T}^{\circ}_{\ell}(X_r)$
and cluster algebras \cite{FZ1,FZ2}.
Then, the  periodicity of $\EuScript{T}_{\ell}(X_r)$
reduces to that
of the corresponding cluster algebra.
For $\ell =2$, the periodicity
of the corresponding cluster algebra  is known by \cite{FZ2,FZ3}.
For $\ell > 2$, the full-periodicity
of the corresponding cluster algebra is recently 
shown by \cite{Kel2} using the cluster categorical method.
We prove the half-periodicity 
for $\ell >2$ as well by adapting this categorical method.
See \cite{Kel2} for a comprehensive review of
cluster algebras and cluster categories.

\subsection{Cluster algebra}

For a finite quiver $Q$
 without loops or 2-cycles
with vertex set, say, $I=\{1,\dots,n\}$
and an $I$-tuple of variables $x=\{x_1,\dots,x_n\}$,
we define a {\em cluster algebra (with trivial coefficients)}
 $\mathcal{A}_Q$
\cite{FZ1,FZ2},
which is  a $\mathbb{Z}$-subalgebra
of the field $\mathbb{Q}(x_1,\dots,x_n)$, as follows:

(1) We start from the pair (`initial seed')
  $(Q,x)$,
where $Q$ and $x$ are as above.

(2) For each $k=1$, \dots, $n$,
we define another pair (`seed')  $(R,y)
=\mu_k(Q,x)$
of a quiver $R$
 without loops or 2-cycles
with vertex set $I$
and an $I$-tuple $y=\{y_1,\dots,y_n\}$,
$y_i\in \mathbb{Q}(x_1,\dots,x_n)$,
called the {\em mutation of $(Q,x)$ at $k$},
where $y$ is given by the following
{\em exchange relation\/},
\begin{align}
\label{eq:ex1}
y_i=
\begin{cases}
x_i & i\neq k,\\
\displaystyle
\frac{1}{x_k}
\left({\prod_{\mathrm{arrows}\ j \rightarrow k
\ \mathrm{of}\  Q}}
x_j
+
{\prod_{\mathrm{arrows}\ k \rightarrow j
\ \mathrm{of}\  Q}}
x_j\right)
& i = k,
\end{cases}
\end{align}
while $R=\mu_k(Q)$ is obtained from $Q$ by the following {\em mutation rule}:
\begin{itemize}
\item[(i)] For each $i\rightarrow k \rightarrow j$ of $Q$,
create a new arrow $i\rightarrow j$.
\item[(ii)] Replace each $i\rightarrow k$ of $Q$ with $k\rightarrow i$,
and $k\rightarrow j$ of $Q$ with $j\rightarrow k$, respectively.
\item[(iii)] Remove a maximal disjoint collection of 2-cycles
of the resulting quiver after (i) and (ii).
\end{itemize}

(3) Iterate the mutation for every  new seed
at every $k$, and collect all the (possibly infinite
number of) seeds.
For any  seed $(R,y)$,
$y$ is called a {\em cluster}
and each element  $y_i$ of  $y$ is
called a {\em cluster variable}.

(4) The cluster algebra
$\mathcal{A}_Q$ is the $\mathbb{Z}$-subalgebra
of the field $\mathbb{Q}(x_1,\dots,x_n)$
generated by all the cluster variables.

Due to the {\em Laurent phenomenon\/} \cite{FZ1},
$\mathcal{A}_Q$ is a subring of
$\mathbb{Z}[x_1^{\pm1},\dots,x_n^{\pm1}]$.

\subsection{Level 2 case}
\label{subsect:level2}
Here we study the periodicity of
$\EuScript{T}_{2}(X_r)$ for simply laced $X_r$.
Since the case $X_r=A_1$ is trivial,
we assume $X_r\neq A_1$.

Let $X_r$ $(\neq A_1)$ be
a simply laced Dynkin diagram,
and $I=I_+ \sqcup I_-$
be a bipartite decomposition of the vertex set $I$ of $X_r$;
namely,
$C_{ab}=0$ for any $a, b\in I_{\pm}$ with $a\neq b$.
We set $\varepsilon(a)=\pm$ for $a\in I_{\pm}$.

Recall that the ring
$\EuScript{T}_{2}(X_r)$
has the generators $T=\{T^{(a)}_1(u)^{\pm1}\mid
a\in I, u\in \mathbb{Z}\}$ and the relations
$\mathbb{T}_2(X_r)$:
\begin{align}
\label{eq:TA4}
T^{(a)}_1(u-1)T^{(a)}_1(u+1)
=
1
+
\prod_{b\in I: C_{ab}=-1}
T^{(b)}_{1}(u).
\end{align}
Let 
$\EuScript{T}^{\circ}_{2}(X_r)_{\pm}$
 be
the subring of 
$\EuScript{T}^{\circ}_{2}(X_r)$
generated by $T^{(a)}_1(u)$
($a\in I, u\in \mathbb{Z})$ such that
$\varepsilon(a)(-1)^u = \pm $,
where we identify $+$ and $-$  with $1$ and  $-1$,
respectively.
Since the relation (\ref{eq:TA4})
closes among those $T^{(a)}_1(u)$
with fixed parity $\varepsilon(a)(-1)^u$,
we have
\begin{align}
\label{eq:Tfact1}
\EuScript{T}^{\circ}_{2}(X_r)&\simeq
\EuScript{T}^{\circ}_{2}(X_r)_+\otimes_{\mathbb{Z}}
\EuScript{T}^{\circ}_{2}(X_r)_-,
\quad
\EuScript{T}^{\circ}_{2}(X_r)_+\simeq
\EuScript{T}^{\circ}_{2}(X_r)_-.
\end{align}

Let $Q=Q(X_r)$ be the {\it alternating quiver}
such that
$X_r$ is the  underlying graph,
$a\in I_+$ is a source, and $a\in I_-$ is a sink of $Q$.
We introduce an $I$-tuple of variables
$x=\{x_a\}_{a\in I}$,
and define $\mathcal{A}_Q$ to be 
 the cluster algebra 
with initial seed $(Q,x)$.

Following \cite{FZ2,FZ3},
we introduce 
composed mutations
$\mu_{\pm}=\prod_{a\in I_{\pm}}\mu_a$
and 
$\mu=\mu_-\mu_+$
for $\mathcal{A}_Q$.
We set $x=x(0)$, and define clusters
$x(u)=\{x_a(u)\}_{a\in I}$ ($u\in \mathbb{Z}$) of 
$\mathcal{A}_Q$ by the following sequence of the
mutations:
\begin{align}
\cdots
 \overset{\mu_+}{\longleftrightarrow}
(Q^{\mathrm{op}},x(-1))
 \overset{\mu_-}{\longleftrightarrow}
(Q,x(0))
\overset{\mu_+}{\longleftrightarrow}
(Q^{\mathrm{op}},x(1))
\overset{\mu_-}{\longleftrightarrow}
(Q,x(2))
\overset{\mu_+}{\longleftrightarrow}
\cdots,
\end{align}
where $Q^{\mathrm{op}}$ is the {\em opposite quiver} of $Q$,
i.e., the quiver obtained from $Q$ by reversing all the arrows.
In particular,
\begin{align}
\label{eq:xa1}
x_a(u+1)&= x_a(u)
\quad \text{if $\varepsilon(a)(-1)^u = -$},\\
\label{eq:qq1}
\begin{split}
(Q,x(2k))&=\mu^{k}(Q,x(0))\quad (k\in \mathbb{Z}),\\
(Q^{\mathrm{op}},x(2k+1))&=\mu_+\mu^{k}(Q,x(0))\quad (k\in \mathbb{Z}).
\end{split}
\end{align}
Furthermore, any cluster variable
of $\mathcal{A}_Q$ occurs in $x(u)$ for some $u\in \mathbb{Z}$,
due to Theorems 1.9 and 3.1 of \cite{FZ2}. (This is not true
for a general finite quiver $Q$.)

\begin{lem} [{\cite[Eq.\ (8.12)]{FZ4}}]
\label{lem:CA1}
The family
$\{ x_a(u)\mid a\in I, u\in \mathbb{Z}\}$
satisfies the T-system $\mathbb{T}_2(X_r)$ in $\mathcal{A}_Q$; namely,
\begin{align}
\label{eq:TA5}
x_a(u-1)x_a(u+1)
&=
1
+
\prod_{b\in I: C_{ab}=-1}
x_{b}(u).
\end{align}
\begin{proof}
For example, suppose that $a\in I_+$ and $u$ is odd.
Then,
\begin{align}
x_a(u-1)= \mu_+(x_a(u))=
\frac{1}{x_a(u)}\left(1+
\prod_{b\in I: C_{ab}=-1}
x_b(u)\right)
\end{align}
by (\ref{eq:ex1}), and $x_a(u) = x_a (u+1)$ by (\ref{eq:xa1}).
The other cases are similar.
\end{proof}
\end{lem}

Now let us describe the relation 
between the rings, $\EuScript{T}_2(X_r)$ and $\mathcal{A}_Q$.
Define a  ring homomorphism
$f:\mathcal{A}_Q\rightarrow 
\EuScript{T}_2(X_r)$ 
as the restriction of the ring
homomorphism
$\mathbb{Z}[x^{\pm1}_a]_{a\in I}
\rightarrow
\EuScript{T}_2(X_r)$
given by
\begin{align}
\begin{split}
f:x_a^{\pm1}&\mapsto
\begin{cases}
T^{(a)}_1(0)^{\pm1} & a\in I_+,\\
T^{(a)}_1(1)^{\pm1}& a\in I_-.
\end{cases}
\end{split}
\end{align}
Then, we have (see Figure \ref{fig:cluster1})
\begin{lem}
\label{lem:xt1}
For the above homomorphism 
$f:\mathcal{A}_Q\rightarrow 
\EuScript{T}_2(X_r)$,
\begin{align}
\label{eq:xT1}
f:x_a(u)&\mapsto
\begin{cases}
T^{(a)}_1(u) & \varepsilon(a)(-1)^u=+,\\
T^{(a)}_1(u+1)& \varepsilon(a)(-1)^u=-.
\end{cases}
\end{align}
\end{lem}
\begin{proof}
For $u=0$, (\ref{eq:xT1}) holds by the definition of $f$.
Then, one can prove (\ref{eq:xT1})
by the induction on $\pm u$ 
with  (\ref{eq:TA4}), (\ref{eq:xa1}), and   (\ref{eq:TA5}).
\end{proof}

\begin{prop}
\label{prop:CT1}
The ring
$\EuScript{T}^{\circ}_2(X_r)$
is isomorphic to $\mathcal{A}_Q\otimes_{\mathbb{Z}}
\mathcal{A}_Q$.
\end{prop}
\begin{proof}
It follows from Lemma \ref{lem:xt1}
that the image $f(\mathcal{A}_Q)$ is
$\EuScript{T}^{\circ}_2(X_r)_+$.
Furthermore the inverse correspondence $g:
\EuScript{T}^{\circ}_2(X_r)_+\rightarrow 
\mathcal{A}_Q$,
$T^{(a)}_1(u)\mapsto x_a(u)$
defines a homomorphism
by  Lemma \ref{lem:CA1}
and the fact that 
$\mathcal{A}_Q$ is an integral domain.
Therefore,
\begin{align}
\EuScript{T}^{\circ}_2(X_r)_+
\simeq 
\mathcal{A}_Q,
\end{align}
and we obtain the assertion.
\end{proof}

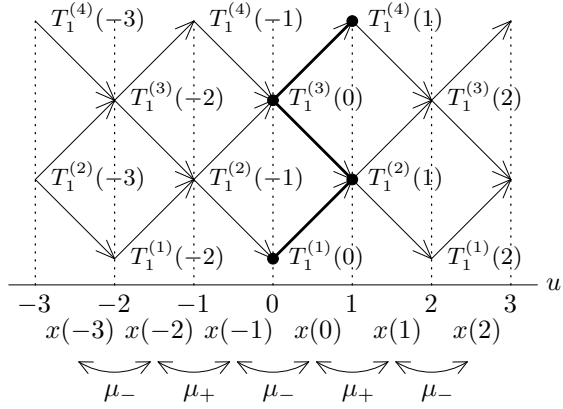
\begin{figure}
\begin{picture}(200,150)(-100,-55)
{
\put(0,0){\Thicklines\line(1,1){30}}
\put(30,30){\Thicklines\line(-1,1){30}}
\put(0,60){\Thicklines\line(1,1){30}}
}
\drawline(-90,30)(-30,90)
\drawline(90,30)(30,90)
\drawline(-60,0)(30,90)
\drawline(60,0)(-30,90)
\drawline(0,0)(90,90)
\drawline(0,0)(-90,90)
\drawline(60,0)(90,30)
\drawline(-60,0)(-90,30)
\drawline(83,27)(90,30)
\drawline(87,23)(90,30)
\drawline(23,27)(30,30)
\drawline(27,23)(30,30)
\drawline(-37,27)(-30,30)
\drawline(-33,23)(-30,30)
\drawline(53,57)(60,60)
\drawline(57,53)(60,60)
\drawline(-7,57)(0,60)
\drawline(-3,53)(0,60)
\drawline(-67,57)(-60,60)
\drawline(-63,53)(-60,60)
\drawline(83,87)(90,90)
\drawline(87,83)(90,90)
\drawline(23,87)(30,90)
\drawline(27,83)(30,90)
\drawline(-37,87)(-30,90)
\drawline(-33,83)(-30,90)
\drawline(53,3)(60,0)
\drawline(57,7)(60,0)
\drawline(-7,3)(0,0)
\drawline(-3,7)(0,0)
\drawline(-67,3)(-60,0)
\drawline(-63,7)(-60,0)
\drawline(83,33)(90,30)
\drawline(87,37)(90,30)
\drawline(23,33)(30,30)
\drawline(27,37)(30,30)
\drawline(-37,33)(-30,30)
\drawline(-33,37)(-30,30)
\drawline(53,63)(60,60)
\drawline(57,67)(60,60)
\drawline(-7,63)(0,60)
\drawline(-3,67)(0,60)
\drawline(-67,63)(-60,60)
\drawline(-63,67)(-60,60)
\drawline(-100,-10)(100,-10)
\dottedline{3}(-90,-10)(-90,90)
\dottedline{3}(-60,-10)(-60,90)
\dottedline{3}(-30,-10)(-30,90)
\dottedline{3}(90,-10)(90,90)
\dottedline{3}(60,-10)(60,90)
\dottedline{3}(30,-10)(30,90)
\dottedline{3}(0,-10)(0,90)
\put(100,-12){ $u$}
\put(-6,-20){ $0$}
\put(24,-20){ $1$}
\put(54,-20){ $2$}
\put(84,-20){ $3$}
\put(-40,-20){ $-1$}
\put(-70,-20){ $-2$}
\put(-100,-20){ $-3$}
\put(8,-30){$x(0)$}
\put(38,-30){$x(1)$}
\put(68,-30){$x(2)$}
\put(-26,-30){$x(-1)$}
\put(-56,-30){$x(-2)$}
\put(-86,-30){$x(-3)$}
\put(0,-18){\arc{50}{1}{2.14}}
\put(30,-18){\arc{50}{1}{2.14}}
\put(60,-18){\arc{50}{1}{2.14}}
\put(-30,-18){\arc{50}{1}{2.14}}
\put(-60,-18){\arc{50}{1}{2.14}}
\drawline(8,-39.2)(13.5,-39.1)
\drawline(10,-43.5)(13.5,-39.1)
\drawline(38,-39.2)(43.5,-39.1)
\drawline(40,-43.5)(43.5,-39.1)
\drawline(68,-39.2)(73.5,-39.1)
\drawline(70,-43.5)(73.5,-39.1)
\drawline(-22,-39.2)(-16.5,-39.1)
\drawline(-20,-43.5)(-16.5,-39.1)
\drawline(-52,-39.2)(-46.5,-39.1)
\drawline(-50,-43.5)(-46.5,-39.1)
\drawline(-8,-39.2)(-13.5,-39.1)
\drawline(-10,-43.5)(-13.5,-39.1)
\drawline(-38,-39.2)(-43.5,-39.1)
\drawline(-40,-43.5)(-43.5,-39.1)
\drawline(-68,-39.2)(-73.5,-39.1)
\drawline(-70,-43.5)(-73.5,-39.1)
\drawline(22,-39.2)(16.5,-39.1)
\drawline(20,-43.5)(16.5,-39.1)
\drawline(52,-39.2)(46.5,-39.1)
\drawline(50,-43.5)(46.5,-39.1)
\put(56,-52){$\mu_-$}
\put(26,-52){$\mu_+$}
\put(-4,-52){$\mu_-$}
\put(-34,-52){$\mu_+$}
\put(-64,-52){$\mu_-$}
\put(6,-2){\small$T^{(1)}_1(0)$}
\put(66,-2){\small$T^{(1)}_1(2)$}
\put(-54,-2){\small$T^{(1)}_1(-2)$}
\put(36,28){\small$T^{(2)}_1(1)$}
\put(-24,28){\small$T^{(2)}_1(-1)$}
\put(-84,28){\small$T^{(2)}_1(-3)$}
\put(6,58){\small$T^{(3)}_1(0)$}
\put(66,58){\small$T^{(3)}_1(2)$}
\put(-54,58){\small$T^{(3)}_1(-2)$}
\put(36,88){\small$T^{(4)}_1(1)$}
\put(-24,88){\small$T^{(4)}_1(-1)$}
\put(-84,88){\small$T^{(4)}_1(-3)$}
\put(0,0){\circle*{4}}
\put(30,30){\circle*{4}}
\put(0,60){\circle*{4}}
\put(30,90){\circle*{4}}
\end{picture}
\caption{Relation of
$x(u)$ and $T^{(a)}_1(u)$
for $X_r=A_4$.
The thick quiver corresponds to the initial seed
$(Q,x(0))$ of the cluster algebra $\mathcal{A}_Q$,
where we take $I_+=\{1,3\}$, $I_-=\{2,4\}$.}
\label{fig:cluster1}
\end{figure}

Thanks to the isomorphism,
the periodicity of
$\EuScript{T}_{2}(X_r)$
is reduced to the known periodicity of $\mathcal{A}_Q$.
By the correspondence (\ref{eq:xT1}),
it is easy to check that
the periodicity of $\EuScript{T}_{2}(X_r)$
is translated  as
\begin{align}
\label{eq:xp1}
\begin{split}
&\text{half-periodicity:}\quad x_a(u+h+2)=x_{\omega(a)}(u),\\
&\text{periodicity:}\hskip16.5pt x_a(u+2( h+2))=x_a(u).
\end{split}
\end{align}

Recall that the Coxeter  number $h=h^{\vee}$ of $X_r$ is
 odd if and only if $X_r=A_r$ ($r$: even);
furthermore, the involution $\omega: I \rightarrow I$
induces a quiver isomorphism 
$\omega: Q
\rightarrow Q^{\mathrm{op}}$
 if $X_r=A_r$ ($r$: even),
and 
$\omega: Q
\rightarrow Q$ otherwise.
For a pair of seeds $(R,y)$ and $(R',y')$,
we write 
$(R,y)\overset{\nu}{=}(R',y')$
if $\nu: I \rightarrow I$
is a bijection 
which induces an quiver isomorphism
 $R\rightarrow R'$ 
and $y'_{\nu(a)}=y_a$ for any $a\in R$.
The following periodicity of $\mathcal{A}_Q$
is
 due to Theorems 1.9 and 3.1 of \cite{FZ2},
and Propositions 2.5 and 2.6 of \cite{FZ3}.

\begin{thm}[Fomin-Zelevinsky {\cite{FZ2,FZ3}}]
\label{thm:cluster1}
 The following equalities hold
for $\mathcal{A}_Q$ ($u$: even):

\par
(1) Half-periodicity:
\par
(i) For $X_r$  other than $A_r$ ($r$: even),
where $h$ is even,
\begin{align}
\mu^{\frac{h+2}{2}}(Q,x(u)) \overset{\omega}{=} (Q,x(u)).
\end{align}
\par
(ii) For $X_r=A_r$ ($r$: even), where $h$ is odd,
\begin{align}
\mu_+ \mu^{\frac{h+1}{2}}(Q,x(u))
\overset{\omega}{=} (Q,x(u)).
\end{align}
\par
(2) Periodicity:
For any $X_r$,
\begin{align}
\mu^{h+2}(Q,x(u)) \overset{\mathrm{id}}{=} (Q,x(u)).
\end{align}
\end{thm}

\begin{cor}
\label{cor:SL1}
The following relations hold in  $\EuScript{T}_{2}(X_r)$
for any simply laced $X_r$:

(1) Half-periodicity: $T^{(a)}_1(u+h+2)=
T^{(\omega(a))}_{1}(u)$.

(2) Periodicity: $T^{(a)}_1(u+2(h+2))=
T^{(a)}_{1}(u)$.
\end{cor}
\begin{proof}
The relations in (\ref{eq:xp1})
immediately follow from
Theorem \ref{thm:cluster1}
and (\ref{eq:qq1}).
\end{proof}

\subsection{Alternative proof of Theorem \ref{thm:cluster1}\
by cluster category}
\label{subsect:cc1}

Here we present
an alternative proof
of Theorem \ref{thm:cluster1}
based on the categorification of
$\mathcal{A}_Q$ by the cluster category $\mathcal{C}_Q$,
in the spirit of \cite{Kel2}.
The definitions and results here will be also used
to prove the periodicity for
the levels greater than two in Section \ref{subsect:cc2}.

Let $Q$ be the alternating quiver whose underlying graph is
simply  laced $X_r$ other than $A_1$ as in Section \ref{subsect:level2}.
Let $K$ be an algebraically closed field and $KQ$ be the path algebra
 of $Q$ \cite{ARS,ASS}.
We denote by $\mathcal{D}_Q=\mathcal{D}^{\rm b}(\mathrm{mod}\, KQ)$
 the bounded derived category of
finite dimensional $KQ$-modules.
Then $\mathcal{D}_Q$ forms a $K$-linear triangulated category with
 the suspension functor $[1]$.
We denote by $D$ the $K$-dual.
The autoequivalence
\begin{align}
\tau:=D(KQ)[-1]\stackrel{{\bf L}}{\otimes}_{KQ}-:\mathcal{D}_Q
\to\mathcal{D}_Q
\end{align}
is called the
\emph{Auslander-Reiten translation} and plays an
 important role in representation theory of $KQ$ \cite{ARS,ASS,Ha}.

Now we define another autoequivalence
 of $\mathcal{D}_Q$ by $F:=\tau^{-1}\circ[1]$.
Then the \emph{cluster category of $Q$} \cite{BMRRT} is
 defined as the orbit category
\begin{align}
\label{eq:cq1}
\mathcal{C}_Q:=\mathcal{D}_Q/F,
\end{align}
which means that $\mathcal{C}_Q$ has the same objects with $\mathcal{D}_Q$,
 and the morphism space is given by
\begin{align}
\mathrm{Hom}_{\mathcal{C}_Q}(X,Y):=\bigoplus_{i\in{\bf Z}}
\mathrm{Hom}_{\mathcal{D}_Q}(X,F^i(Y))
\end{align}
for any $X,Y\in\mathcal{C}_Q$.
Then $\mathcal{C}_Q$ forms a triangulated category with the suspension functor $[1]$, and the natural functor $\mathcal{D}_Q\to\mathcal{C}_Q$ is
 a triangle functor \cite{Kel1}.

For $a\in I$, we denote by $e_a$ the path of length zero in $Q$.
Define $KQ$-modules by
\begin{align}
P_a:=(KQ)e_a, \quad P:=KQ=\bigoplus_{a\in I}P_a.
\end{align}
The following description of indecomposable objects in $\mathcal{C}_Q$ follows from Gabriel's Theorem \cite{ASS}
and Fomin-Zelevinsky's description of finite type cluster algebras \cite{FZ2}.

\begin{thm}
\label{thm:Xbij}
There exists a bijection
\begin{align}
X:\{\mbox{indecomposable objects in $\mathcal{C}_Q$}\}/\simeq\
 \to\{\mbox{cluster variables in $\mathcal{A}_Q$}\}
\end{align}
satisfying $X_{P_a}=x_a(0)$ for any $a\in I$.
\end{thm}

We say that an object $T=\bigoplus_{a\in I}T_a\in\mathcal{C}_Q$ is \emph{cluster tilting} if 
\begin{itemize}
\item[(1)] each $T_a$ is indecomposable and mutually non-isomorphic,
\item[(2)] $\mathrm{Hom}_{\mathcal{C}_Q}(T_a,T_b[1])=0$ holds for any $a,b\in I$.
\end{itemize}
For a cluster tilting object $T\in\mathcal{C}_Q$, we denote by $Q_T$ the quiver of the endomorphism ring $\mathrm{End}_{\mathcal{C}_Q}(T)$ \cite{ARS,ASS}.

We give two important examples of cluster tilting objects.

\begin{exmp}\label{P and U}
(1) The $KQ$-module
\begin{align}
P=\bigoplus_{a\in I}P_a
\end{align}
gives a cluster tilting object.
We have $Q_{P}=Q$, since $\mathrm{End}_{\mathcal{C}_Q}(P)=KQ$.
\par
(2) We define $KQ$-modules  by
\begin{align}
U_a:=\begin{cases}
\tau^{-1}(P_a)&a\in I_+,\\
P_a&a\in I_-,
\end{cases}
\quad U:=\bigoplus_{a\in I}U_a.
\end{align}
Then $U$ is a tilting $KQ$-module,
 and so it gives a cluster tilting object in $\mathcal{C}_Q$.
We have $Q_U=Q^{\rm op}$, since $\mathrm{End}_{\mathcal{D}_Q}(U)
\simeq KQ^{\rm op}$.
\end{exmp}

The mutation of cluster tilting objects
 is introduced in \cite[Theorem 5.1]{BMRRT}.

\begin{thm}
\label{thm:mut1}
Let $T=\bigoplus_{a\in I}T_a\in\mathcal{C}_Q$ be a cluster tilting object.
 For any $a\in I$,
there exists a unique indecomposable object $T^*_a\in\mathcal{C}_Q$ which
 is not isomorphic to $T_a$ such that $(T/T_a)\oplus T^*_a$ is a cluster tilting object.
\end{thm}

We call the above $(T/T_a)\oplus T^*_a$
the   \emph{cluster tilting mutation of $T$ at $a$},
and denote it by $\mu_a(T)$.
We have the following key observation \cite[Theorem 6.1]{BMR}.

\begin{thm}\label{X map}
(1) We have a bijection
\begin{align}
\widetilde{X}:\{\mbox{cluster tilting objects in $\mathcal{C}_Q$}\}/\simeq\ 
\to\{\mbox{seeds in $\mathcal{A}_Q$}\}
\end{align}
defined by
\begin{align}
T=\bigoplus_{a\in I}T_a\mapsto(Q_T,\{X_{T_a}\}_{a\in I}).
\end{align}
\par
(2) We have $\widetilde{X}\circ\mu_a=\mu_a\circ\widetilde{X}$ for any $a\in I$.
\end{thm}

We have
\begin{equation}\label{initial}
\widetilde{X}_P=(Q,x(0)).
\end{equation}
We number elements of $I_+$ and $I_-$ as $\{a_1,\cdots,a_s\}$ and $\{a_{s+1},\cdots,a_r\}$ respectively.
Define composed cluster tilting mutations by
\begin{align}
\mu_+:=\mu_{a_s}\cdots\mu_{a_1},\quad
\mu_-:=\mu_{a_r}\cdots\mu_{a_{s+1}},\quad \mu:=\mu_-\mu_+.
\end{align}
For a cluster tilting object $T=\bigoplus_{a\in I}
T_a$, let $[T]_a$ denote $T_a$;
similarly,
for a seed $(R,y)$ of $\mathcal{A}_Q$,
 let $[(R,y)]_a$ denote $y_a$.
By Theorem \ref{X map} (2), we have the following
 relationship between the seed mutation and the cluster tilting mutation,
\begin{equation}\label{CC and mu}
[\mu^k(\widetilde{X}_T)]_a=X_{[\mu^k(T)]_a},\quad
[\mu_+\mu^k(\widetilde{X}_T)]_a=X_{[\mu_+\mu^k(T)]_a},
\end{equation}
for any cluster tilting object $T\in\mathcal{C}_Q$, $k\in\mathbb{Z}$
 and $a\in I$.
The following observation is a key result.

\begin{prop}\label{mu and tau}
For any $k\in \mathbb{Z}$ and $a\in I$, the following assertions hold.
\par
(1) $[\mu^k(P)]_a\simeq\tau^{-k}(P_a)$ in $\mathcal{C}_Q$.
\par
(2) $[\mu_+\mu^k(P)]_a\simeq\tau^{-k}(U_a)$ in $\mathcal{C}_Q$.
\end{prop}

\begin{proof}
Since cluster tilting mutation commutes with any autoequivalence of $\mathcal{C}_Q$, we
 have only to show the assertion for $k=1$ for (1) and  $k=0$ for (2).

(2) We put $I^i:=\{a_1,\cdots,a_i\}$ and
\[T^i:=(\bigoplus_{a\in I^i}\tau^{-1}(P_a))\oplus(\bigoplus_{a\in I\backslash I^i}P_a).\]
Then $T^i$ is a tilting $KQ$-module by \cite{ASS}.
Thus it is a cluster tilting object by \cite{BMRRT}.
Since $T^{i-1}$ and $T^i$ have the same indecomposable direct summands except $P_{a_i}$, we have $\mu_{a_i}(T^{i-1})=T^i$.
In particular, we have $\mu_+(P)=U$.

(1)  By a similar argument to (2), we have $\mu(P)=\mu_-(U)=\tau^{-1}(P)$.
\end{proof}

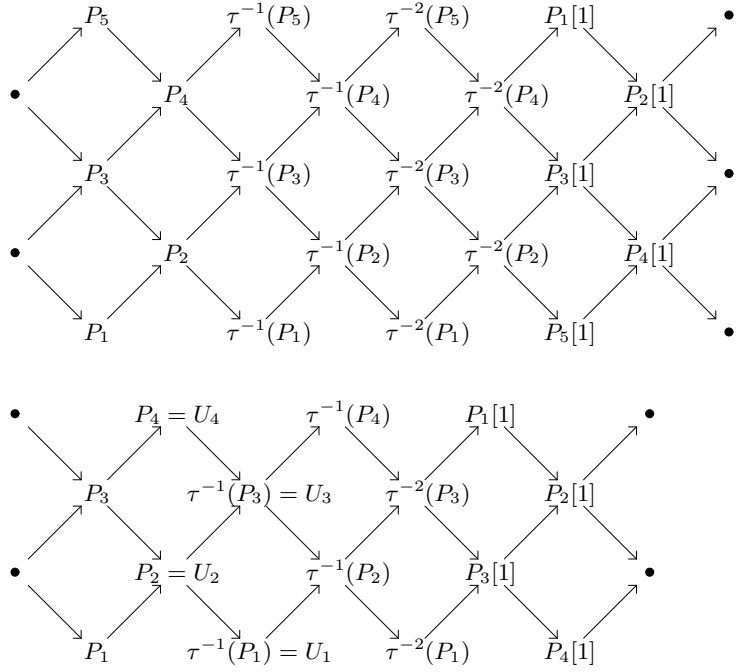
\begin{figure}[t]
\begin{picture}(270,125)(0,0)
\put(0,30){\circle*{3}}
\put(0,90){\circle*{3}}

\put(-4,-3)
{
\put(30,0){\small$P_1$}
\put(30,60){\small$P_3$}
\put(30,120){\small$P_5$}
}

\put(56,-3)
{
\put(0,30){\small$P_2$}
\put(0,90){\small$P_4$}
}
\put(50,-3)
{
\put(30,0){\small$\tau^{-1}(P_1)$}
\put(30,60){\small$\tau^{-1}(P_3)$}
\put(30,120){\small$\tau^{-1}(P_5)$}
}

\put(110,-3)
{
\put(0,30){\small$\tau^{-1}(P_2)$}
\put(0,90){\small$\tau^{-1}(P_4)$}
\put(30,0){\small$\tau^{-2}(P_1)$}
\put(30,60){\small$\tau^{-2}(P_3)$}
\put(30,120){\small$\tau^{-2}(P_5)$}
}

\put(170,-3)
{
\put(0,30){\small$\tau^{-2}(P_2)$}
\put(0,90){\small$\tau^{-2}(P_4)$}
}
\put(170,-3)
{
\put(30,0){\small$P_5[1]$}
\put(30,60){\small$P_3[1]$}
\put(30,120){\small$P_1[1]$}
}

\put(230,-3)
{
\put(0,30){\small$P_4[1]$}
\put(0,90){\small$P_2[1]$}
}
\put(240,0)
{
\put(30,0){\circle*{3}}
\put(30,60){\circle*{3}}
\put(30,120){\circle*{3}}
}

\put(0,30)
{
\drawline(5,5)(25,25)
\drawline(25,25)(22,25)
\drawline(25,25)(25,22)
}

\put(0,90)
{
\drawline(5,5)(25,25)
\drawline(25,25)(22,25)
\drawline(25,25)(25,22)
}

\put(30,0)
{
\drawline(5,5)(25,25)
\drawline(25,25)(22,25)
\drawline(25,25)(25,22)
}

\put(30,60)
{
\drawline(5,5)(25,25)
\drawline(25,25)(22,25)
\drawline(25,25)(25,22)
}

\put(60,30)
{
\drawline(5,5)(25,25)
\drawline(25,25)(22,25)
\drawline(25,25)(25,22)
}

\put(60,90)
{
\drawline(5,5)(25,25)
\drawline(25,25)(22,25)
\drawline(25,25)(25,22)
}

\put(90,0)
{
\drawline(5,5)(25,25)
\drawline(25,25)(22,25)
\drawline(25,25)(25,22)
}

\put(90,60)
{
\drawline(5,5)(25,25)
\drawline(25,25)(22,25)
\drawline(25,25)(25,22)
}

\put(120,30)
{
\drawline(5,5)(25,25)
\drawline(25,25)(22,25)
\drawline(25,25)(25,22)
}

\put(120,90)
{
\drawline(5,5)(25,25)
\drawline(25,25)(22,25)
\drawline(25,25)(25,22)
}

\put(150,0)
{
\drawline(5,5)(25,25)
\drawline(25,25)(22,25)
\drawline(25,25)(25,22)
}

\put(150,60)
{
\drawline(5,5)(25,25)
\drawline(25,25)(22,25)
\drawline(25,25)(25,22)
}

\put(180,30)
{
\drawline(5,5)(25,25)
\drawline(25,25)(22,25)
\drawline(25,25)(25,22)
}

\put(180,90)
{
\drawline(5,5)(25,25)
\drawline(25,25)(22,25)
\drawline(25,25)(25,22)
}

\put(210,0)
{
\drawline(5,5)(25,25)
\drawline(25,25)(22,25)
\drawline(25,25)(25,22)
}

\put(210,60)
{
\drawline(5,5)(25,25)
\drawline(25,25)(22,25)
\drawline(25,25)(25,22)
}

\put(240,30)
{
\drawline(5,5)(25,25)
\drawline(25,25)(22,25)
\drawline(25,25)(25,22)
}

\put(240,90)
{
\drawline(5,5)(25,25)
\drawline(25,25)(22,25)
\drawline(25,25)(25,22)
}

\put(0,0)
{
\drawline(25,5)(5,25)
\drawline(25,5)(22,5)
\drawline(25,5)(25,8)
}

\put(0,60)
{
\drawline(25,5)(5,25)
\drawline(25,5)(22,5)
\drawline(25,5)(25,8)
}

\put(30,30)
{
\drawline(25,5)(5,25)
\drawline(25,5)(22,5)
\drawline(25,5)(25,8)
}

\put(30,90)
{
\drawline(25,5)(5,25)
\drawline(25,5)(22,5)
\drawline(25,5)(25,8)
}

\put(60,0)
{
\drawline(25,5)(5,25)
\drawline(25,5)(22,5)
\drawline(25,5)(25,8)
}

\put(60,60)
{
\drawline(25,5)(5,25)
\drawline(25,5)(22,5)
\drawline(25,5)(25,8)
}

\put(90,30)
{
\drawline(25,5)(5,25)
\drawline(25,5)(22,5)
\drawline(25,5)(25,8)
}

\put(90,90)
{
\drawline(25,5)(5,25)
\drawline(25,5)(22,5)
\drawline(25,5)(25,8)
}

\put(120,0)
{
\drawline(25,5)(5,25)
\drawline(25,5)(22,5)
\drawline(25,5)(25,8)
}

\put(120,60)
{
\drawline(25,5)(5,25)
\drawline(25,5)(22,5)
\drawline(25,5)(25,8)
}

\put(150,30)
{
\drawline(25,5)(5,25)
\drawline(25,5)(22,5)
\drawline(25,5)(25,8)
}

\put(150,90)
{
\drawline(25,5)(5,25)
\drawline(25,5)(22,5)
\drawline(25,5)(25,8)
}

\put(180,0)
{
\drawline(25,5)(5,25)
\drawline(25,5)(22,5)
\drawline(25,5)(25,8)
}

\put(180,60)
{
\drawline(25,5)(5,25)
\drawline(25,5)(22,5)
\drawline(25,5)(25,8)
}

\put(210,30)
{
\drawline(25,5)(5,25)
\drawline(25,5)(22,5)
\drawline(25,5)(25,8)
}

\put(210,90)
{
\drawline(25,5)(5,25)
\drawline(25,5)(22,5)
\drawline(25,5)(25,8)
}

\put(240,0)
{
\drawline(25,5)(5,25)
\drawline(25,5)(22,5)
\drawline(25,5)(25,8)
}

\put(240,60)
{
\drawline(25,5)(5,25)
\drawline(25,5)(22,5)
\drawline(25,5)(25,8)
}
\end{picture}

\vskip30pt

\begin{picture}(270,95)(0,-5)

\put(0,30){\circle*{3}}
\put(0,90){\circle*{3}}

\put(-4,-3)
{
\put(30,0){\small$P_1$}
\put(30,60){\small$P_3$}
}

\put(45,-3)
{
\put(0,30){\small$P_2=U_2$}
\put(0,90){\small$P_4=U_4$}
}
\put(35,-3)
{
\put(30,0){\small$\tau^{-1}(P_1)=U_1$}
\put(30,60){\small$\tau^{-1}(P_3)=U_3$}
}

\put(110,-3)
{
\put(0,30){\small$\tau^{-1}(P_2)$}
\put(0,90){\small$\tau^{-1}(P_4)$}
\put(30,0){\small$\tau^{-2}(P_1)$}
\put(30,60){\small$\tau^{-2}(P_3)$}
}

\put(170,-3)
{
\put(0,30){\small$P_3[1]$}
\put(0,90){\small$P_1[1]$}
}
\put(170,-3)
{
\put(30,0){\small$P_4[1]$}
\put(30,60){\small$P_2[1]$}
}

\put(240,0)
{
\put(0,30){\circle*{3}}
\put(0,90){\circle*{3}}
}

\put(0,30)
{
\drawline(5,5)(25,25)
\drawline(25,25)(22,25)
\drawline(25,25)(25,22)
}

\put(30,0)
{
\drawline(5,5)(25,25)
\drawline(25,25)(22,25)
\drawline(25,25)(25,22)
}

\put(30,60)
{
\drawline(5,5)(25,25)
\drawline(25,25)(22,25)
\drawline(25,25)(25,22)
}

\put(60,30)
{
\drawline(5,5)(25,25)
\drawline(25,25)(22,25)
\drawline(25,25)(25,22)
}

\put(90,0)
{
\drawline(5,5)(25,25)
\drawline(25,25)(22,25)
\drawline(25,25)(25,22)
}

\put(90,60)
{
\drawline(5,5)(25,25)
\drawline(25,25)(22,25)
\drawline(25,25)(25,22)
}

\put(120,30)
{
\drawline(5,5)(25,25)
\drawline(25,25)(22,25)
\drawline(25,25)(25,22)
}

\put(150,0)
{
\drawline(5,5)(25,25)
\drawline(25,25)(22,25)
\drawline(25,25)(25,22)
}

\put(150,60)
{
\drawline(5,5)(25,25)
\drawline(25,25)(22,25)
\drawline(25,25)(25,22)
}

\put(180,30)
{
\drawline(5,5)(25,25)
\drawline(25,25)(22,25)
\drawline(25,25)(25,22)
}

\put(210,0)
{
\drawline(5,5)(25,25)
\drawline(25,25)(22,25)
\drawline(25,25)(25,22)
}

\put(210,60)
{
\drawline(5,5)(25,25)
\drawline(25,25)(22,25)
\drawline(25,25)(25,22)
}

\put(0,0)
{
\drawline(25,5)(5,25)
\drawline(25,5)(22,5)
\drawline(25,5)(25,8)
}

\put(0,60)
{
\drawline(25,5)(5,25)
\drawline(25,5)(22,5)
\drawline(25,5)(25,8)
}

\put(30,30)
{
\drawline(25,5)(5,25)
\drawline(25,5)(22,5)
\drawline(25,5)(25,8)
}

\put(60,0)
{
\drawline(25,5)(5,25)
\drawline(25,5)(22,5)
\drawline(25,5)(25,8)
}

\put(60,60)
{
\drawline(25,5)(5,25)
\drawline(25,5)(22,5)
\drawline(25,5)(25,8)
}

\put(90,30)
{
\drawline(25,5)(5,25)
\drawline(25,5)(22,5)
\drawline(25,5)(25,8)
}

\put(120,0)
{
\drawline(25,5)(5,25)
\drawline(25,5)(22,5)
\drawline(25,5)(25,8)
}

\put(120,60)
{
\drawline(25,5)(5,25)
\drawline(25,5)(22,5)
\drawline(25,5)(25,8)
}

\put(150,30)
{
\drawline(25,5)(5,25)
\drawline(25,5)(22,5)
\drawline(25,5)(25,8)
}

\put(180,0)
{
\drawline(25,5)(5,25)
\drawline(25,5)(22,5)
\drawline(25,5)(25,8)
}

\put(180,60)
{
\drawline(25,5)(5,25)
\drawline(25,5)(22,5)
\drawline(25,5)(25,8)
}

\put(210,30)
{
\drawline(25,5)(5,25)
\drawline(25,5)(22,5)
\drawline(25,5)(25,8)
}

\end{picture}

\caption{The Auslander-Reiten quivers of $\mathcal{D}_Q$
for $X_r=A_5$ (the above) and $A_4$ (the below).}
\label{fig:AR2}
\end{figure}

Now we are ready to prove Theorem \ref{thm:cluster1} (2)
(full-periodicity).
We use the following classical periodicity result.

\begin{prop}\label{tau periodicity}
(1) $\tau^{-h}(X)$ is isomorphic to $X[2]$ for any $X\in\mathcal{D}_Q$.
\par
(2) $\tau^{-h-2}(X)$ is isomorphic to $X$ for any $X\in\mathcal{C}_Q$.
\end{prop}

\begin{proof}
(1)  This follows from the structure of
 the Auslander-Reiten quiver of $\mathcal{D}_Q$ \cite{G,Ha}.
 See Figure \ref{fig:AR2}.

(2)  By (1), we have $\tau^{-h-2}(X)\simeq\tau^{-2}(X[2])
\simeq F^{2}(X)\simeq X$.
\end{proof}

\begin{thm}[Theorem \ref{thm:cluster1}  (2)]
$[\mu^{h+2}(Q,x(0))]_a=x_a(0)$.
\end{thm}

\begin{proof}
We have $[\mu^{h+2}(P)]_a
\stackrel{\mathrm{Prop.}\ref{mu and tau}(1)}{=}\tau^{-h-2}(P_a)
\stackrel{\mathrm{Prop.}\ref{tau periodicity}(2)}{=}P_a$ for any $a\in I$.
Applying $\widetilde{X}$, we have
$[\mu^{h+2}(Q,x(0))]_a\stackrel{\eqref{initial}}{=}
[\mu^{h+2}(\widetilde{X}_P)]_a
\stackrel{\eqref{CC and mu}}{=}X_{[\mu^{h+2}(P)]_a}
=X_{P_a}=x_a(0)$.
\end{proof}

Next we prove Theorem \ref{thm:cluster1} (1)
(half-periodicity).
We divide the proof into two cases.

\medskip\noindent
{\bf (Case 1) $h$ is even.}

In this case, the map $\omega:I\to I$ induces the quiver
 automorphism $\omega:Q\to Q$.
Thus $\omega$ induces an automorphism $\omega:KQ\to KQ$ of
 our $K$-algebra $KQ$,
and we have an autoequivalence
\begin{align}
\omega:\mathcal{D}_Q\to\mathcal{D}_Q.
\end{align}
of categories.
We have $\omega(P_a)=P_{\omega(a)}$ for any $a\in I$.

\begin{prop}\label{tau periodicity 2}
(1) $\tau^{-\frac{h}{2}}(X)\simeq\omega(X[1])$
 for any $X\in\mathcal{D}_Q$.
\par
(2) $\tau^{-\frac{h+2}{2}}(X)\simeq\omega(X)$
 for any $X\in\mathcal{C}_Q$.
\end{prop}

\begin{proof}
(1)  This follows from the structure of
 the Auslander-Reiten quiver of $\mathcal{D}_Q$ \cite{G,Ha}.
See Figure \ref{fig:AR2}.

(2)  By (1), we have $\tau^{-\frac{h+2}{2}}(X)\simeq\tau^{-1}(\omega( X[1]))
\simeq F(\omega (X))\simeq\omega(X)$.
\end{proof}

\begin{thm}[Theorem \ref{thm:cluster1} (1-i)]
$[\mu^{\frac{h+2}{2}}(Q,x(0))]_a=x_{\omega(a)}(0)$.
\end{thm}

\begin{proof}
We have $[\mu^{\frac{h+2}{2}}(P)]_a
\stackrel{\mathrm{Prop.}\ref{mu and tau}(1)}{=}
\tau^{-\frac{h+2}{2}}(P_a)
\stackrel{\mathrm{Prop.}\ref{tau periodicity 2}(2)}{=}P_{\omega(a)}$.
Applying $\widetilde{X}$, we have
$[\mu^{\frac{h+2}{2}}(Q,x(0))]_a\stackrel{\eqref{initial}}{=}
[\mu^{\frac{h+2}{2}}(\widetilde{X}_P)]_a
\stackrel{\eqref{CC and mu}}{=}X_{[\mu^{\frac{h+2}{2}}(P)]_a}=
X_{P_{\omega(a)}}=x_{\omega(a)}(0)$.
\end{proof}

\medskip\noindent
{\bf (Case 2) $h$ is odd.}

In this case, the map $\omega:I\to I$ induces the quiver isomorphism $\omega:Q\to Q^{\rm op}$.
Thus $\omega$ induces an isomorphism $\omega:KQ\to KQ^{\rm op}$ of $K$-algebras,
and we have an equivalence
\begin{align}
\omega:\mathcal{D}_Q\to\mathcal{D}_{Q^{\rm op}}
\end{align}
of categories.
On the other hand, the tilting $KQ$-module $U$ in Example
 \ref{P and U} (2) induces an equivalence
\begin{align}
U\stackrel{{\bf L}}{\otimes}_{KQ^{\rm op}}-:
\mathcal{D}_{Q^{\rm op}}\to\mathcal{D}_Q.
\end{align}
Composing them, we have an autoequivalence 
\begin{align}
r:\mathcal{D}_Q\xrightarrow{\omega}\mathcal{D}_{Q^{\rm op}}
\xrightarrow{U\stackrel{{\bf L}}{\otimes}_{KQ^{\rm op}}-}\mathcal{D}_Q.
\end{align}
We have $r(P_{a})=U_{\omega(a)}$
and $r(U_{a})=\tau^{-1}(P_{\omega(a)})$
 for any $a\in I$ (see Figure \ref{fig:AR2});
 hence, $r^2(X)\simeq \tau^{-1}(X)$ for any $X\in \mathcal{D}_Q$.

\begin{prop}\label{r periodicity}
(1) $r\tau^{-\frac{h-1}{2}}(X)\simeq X[1]$ for any $X\in\mathcal{D}_Q$.
\par
(2) $\tau^{-\frac{h+1}{2}}(X)\simeq r^{-1}(X)$ for any $X\in\mathcal{C}_Q$.
\end{prop}

\begin{proof}
(1)  This follows from the structure of
 the Auslander-Reiten quiver of $\mathcal{D}_Q$ \cite{G,Ha}.
See Figure \ref{fig:AR2}.

(2)  By (1), we have $r\tau^{-\frac{h+1}{2}}(X)\simeq\tau^{-1}(X[1])\simeq X$,
{}from which the claim follows.
\end{proof}

\begin{thm}[Theorem \ref{thm:cluster1} (1-ii)]
$[\mu_+\mu^{\frac{h+1}{2}}(Q,x(0))]_a=x_{\omega(a)}(0)$.
\end{thm}

\begin{proof}
We have $[\mu_+\mu^{\frac{h+1}{2}}(P)]_a
\stackrel{\mathrm{Prop.}\ref{mu and tau}(2)}{=}
\tau^{-\frac{h+1}{2}}(U_a)
\stackrel{\mathrm{Prop.}\ref{r periodicity}(2)}{=}r^{-1}(U_a)=P_{\omega(a)}$.
Applying $\widetilde{X}$, we have
$[\mu_+\mu^{\frac{h+1}{2}}(Q,x(0))]_a\stackrel{\eqref{initial}}{=}
[\mu_+\mu^{\frac{h+1}{2}}(\widetilde{X}_P)]_a\stackrel{\eqref{CC and mu}}{=}
X_{[\mu_+\mu^{\frac{h+1}{2}}(P)]_a}
\allowbreak
=
X_{P_{\omega(a)}}=x_{\omega(a)}(0)$.
\end{proof}

\subsection{Level greater than two case}
\label{subsect:lg2}
Here we study the periodicity of
$\EuScript{T}_{\ell}(X_r)$ for simply laced $X_r$
and $\ell >2$.
Since the case $X_r=A_1$ reduces to
$\EuScript{T}_{2}(A_{\ell-1})$
by the level-rank duality (Example \ref{exmp:lrd}),
we continue to assume $X_r\neq A_1$.

First, let us establish a connection
between the ring
$\EuScript{T}^{\circ}_{\ell}(X_r)$ 
and the cluster algebra
$\mathcal{A}_{Q\square Q'}$  considered in \cite{Kel2,HL}.
In doing that, we slightly generalize the problem
and consider a pair of simply laced Dynkin diagrams
$X_r$, $X'_{r'}$ ($\neq A_1$).
For $X_r$ (resp.\ $X'_{r'}$),
let $C$, $h$, $I$, $I_{\pm}$, $\varepsilon$, $Q$, $\omega$
(resp.\ $C'$, $h'$, $I'$, $I'_{\pm}$, $\varepsilon'$,
 $Q'$, $\omega'$)
be the same as Section \ref{subsect:level2}.

\begin{defn}
The {\it T-system $\mathbb{T}(X_r,X'_{r'})$
of type $(X_r,X'_{r'})$}
is the following system of relations for
a family of variables $T=\{T_{a,b}(u)
\mid
a\in I,\ b\in I',\ u\in \mathbb{Z}
\}$:
\begin{align}
\label{eq:TXX1}
T_{a,b}(u-1)T_{a,b}(u+1)
=
\prod_{k\in I': C'_{bk}=-1}
T_{a,k}(u)
+
\prod_{k\in I: C_{ak}=-1}
T_{k,b}(u).
\end{align}
\end{defn}

\begin{defn}
The {\em T-algebra $\EuScript{T}(X_r,X'_{r'})$
of type $(X_r,X'_{r'})$} is the ring with generators
$T_{a,b}(u)^{\pm 1}$ ($a\in I, b\in I',
u\in \mathbb{Z} $)
and the relations $\mathbb{T}(X_r,X'_{r'})$.
Also, we define the ring $\EuScript{T}^{\circ}(X_r,X'_{r'})$
as the subring of $\EuScript{T}(X_r,X'_{r'})$
generated by 
$T_{a,b}(u)$ ($a\in I, b\in I',u\in \mathbb{Z} $).
\end{defn}

The system $\mathbb{T}(X_r,X'_{r'})$ is the counterpart
of the $Y$-system of \cite{RTV} studied in \cite{Kel2},
and $\EuScript{T}_{\ell}(X_r)=\EuScript{T}(X_r,A_{\ell-1})$
by identifying $T^{(a)}_m(u)$ with $T_{a,m}(u)$.
We are going to show the following 
periodicity of $\EuScript{T}(X_r,X'_{r'})$:
\begin{align}
\label{eq:xpt2}
\begin{split}
&\text{half-periodicity:}
\quad
T_{a,b}(u+h+h')=T_{\omega(a),\omega'(b)}(u),
\\
&\text{periodicity:}
\quad\hskip6.8pt
T_{a,b}(u+2(h+h'))=T_{a,b}(u).
\end{split}
\end{align}

Let 
$\EuScript{T}^{\circ}(X_r,X'_{r'})_{\pm}$
 be
the subring of 
$\EuScript{T}^{\circ}(X_r,X'_{r'})$
generated by $T_{a,b}(u)$
($a\in I, b \in I', u\in \mathbb{Z})$ such that
$\varepsilon(a)\varepsilon'(b)(-1)^u = \pm $.
Then, we have
\begin{align}
\label{eq:Tfact2}
\EuScript{T}^{\circ}(X_r,X'_{r'})&\simeq
\EuScript{T}^{\circ}(X_r,X'_{r'})_+\otimes_{\mathbb{Z}}
\EuScript{T}^{\circ}(X_r,X'_{r'})_-,
\quad
\EuScript{T}^{\circ}(X_r,X'_{r'})_+\simeq
\EuScript{T}^{\circ}(X_r,X'_{r'})_-.
\end{align}

To describe the corresponding cluster algebra
to $\EuScript{T}^{\circ}(X_r,X'_{r'})$,
we introduce two kinds of quivers,
$Q\,\square\, Q'$ and $Q\otimes Q'$.

\begin{defn}[\cite{Kel2}]
(i) The {\em square product} $Q\, \square\, Q'$ of $Q$ and $Q'$
is the quiver
obtained from the product $Q\times Q'$ by
reversing all the arrows in the full subquivers
$\{a \}\times Q'$ ($a$: sink of $Q$)
and $Q\times \{b \}$ ($b$: source of $Q'$).
\par
(ii) The {\em tensor product} $Q\otimes Q'$ of $Q$ and $Q'$
is the quiver
obtained from the product $Q\times Q'$ by
adding an arrow $(a_2,b_2)\rightarrow (a_1,b_1)$
for each pair of arrows  $a_1\rightarrow a_2$ of $Q$
and  $b_1\rightarrow b_2$ of $Q'$
\end{defn}

\begin{exmp}
Since $a\in I_+$ is a source of $Q$, in our convention,
the ordinary product $Q\times Q'$ consists
of the following type of squares
\begin{align}
Q\times Q':\quad
\begin{matrix}
(+-)& \rightarrow & (--)\\
\uparrow &&\uparrow\\
(++)&\rightarrow & (-+)
\end{matrix},
\end{align}
where $(+-)$, for example, represents
a vertex $(a,b)$ of $Q\times Q'$ with $a\in I_+$, $b\in I'_-$.
Correspondingly,
$Q\,\square\, Q'$
and
$Q\,\otimes\, Q'$ consist of the following types of squares
\begin{align}
\label{eq:square1}
Q\,\square\, Q':
\quad
\begin{matrix}
(+-)& \rightarrow & (--)\\
\uparrow &&\downarrow\\
(++)&\leftarrow & (-+)
\end{matrix},
\quad\quad
Q\otimes Q':
\quad
\begin{matrix}
(+-)& \rightarrow & (--)\\
\uparrow &\swarrow&\uparrow\\
(++)&\rightarrow & (-+)
\end{matrix}.
\end{align}
\end{exmp}

Using these diagrams, one can easily check that
\begin{align}
\begin{split}
Q\,\square\, Q' & = Q^{\mathrm{op}}\,\square\, Q'^{\mathrm{op}},\\
(Q\,\square\, Q')^{\mathrm{op}} & = Q^{\mathrm{op}}\,\square\, Q'
= Q\,\square\, Q'^{\mathrm{op}},\\
(Q\otimes Q')^{\mathrm{op}} & = Q^{\mathrm{op}}\otimes Q'^{\mathrm{op}}.
\end{split}
\end{align}
We define  composed mutations,
\begin{align}
\mu_{\pm\pm}=\prod_{(a,b)\in I_{\pm}\times I'_{\pm}}
\mu_{a,b},
\quad
\mu_{\pm\mp}=\prod_{(a,b)\in I_{\pm}\times I'_{\mp}}
\mu_{a,b}.
\end{align}
where $\mu_{a,b}$ is the mutation at $(a,b)$.
Then,  the following cycle of mutations of quivers occurs:
(the `eyeglass diagram')
\begin{align}
\label{eq:mutation1}
\begin{matrix}
\hskip10pt
   Q\otimes Q'^{\mathrm{op}}\hskip55pt Q\otimes Q'
\hskip10pt
\\
 {}^{\mu_{++}}\
\begin{picture}(8,12)(0,-1)
\drawline(0,0)(8,8)
\drawline(6,8)(8,8)
\drawline(8,6)(8,8)
\drawline(2,0)(0,0)
\drawline(0,2)(0,0)
\end{picture}
\hskip30pt
\begin{picture}(8,10)
\drawline(0,8)(8,0)
\drawline(0,8)(2,8)
\drawline(0,8)(0,6)
\drawline(8,0)(6,0)
\drawline(8,0)(8,2)
\end{picture}
\ {}^{\mu_{--}}
\ \
 {}^{\mu_{+-}}\ 
\begin{picture}(8,10)
\drawline(0,0)(8,8)
\drawline(6,8)(8,8)
\drawline(8,6)(8,8)
\drawline(2,0)(0,0)
\drawline(0,2)(0,0)
\end{picture}
\hskip30pt
\begin{picture}(8,10)
\drawline(0,8)(8,0)
\drawline(0,8)(2,8)
\drawline(0,8)(0,6)
\drawline(8,0)(6,0)
\drawline(8,0)(8,2)
\end{picture}
\ {}^{\mu_{-+}}
\\
Q\,\square\, Q'
\hskip40pt
 (Q\,\square\,Q')^{\mathrm{op}}
\hskip40pt
 Q\,\square\, Q'.\\
 {}^{\mu_{--}}
\
\begin{picture}(8,10)
\drawline(0,8)(8,0)
\drawline(0,8)(2,8)
\drawline(0,8)(0,6)
\drawline(8,0)(6,0)
\drawline(8,0)(8,2)
\end{picture}
\hskip30pt
\begin{picture}(8,10)
\drawline(0,0)(8,8)
\drawline(6,8)(8,8)
\drawline(8,6)(8,8)
\drawline(2,0)(0,0)
\drawline(0,2)(0,0)
\end{picture}
\ {}^{\mu_{++}}
\ \
 {}^{\mu_{-+}}\
\begin{picture}(8,10)
\drawline(0,8)(8,0)
\drawline(0,8)(2,8)
\drawline(0,8)(0,6)
\drawline(8,0)(6,0)
\drawline(8,0)(8,2)
\end{picture}
\hskip30pt
\begin{picture}(8,12)(0,-1)
\drawline(0,0)(8,8)
\drawline(6,8)(8,8)
\drawline(8,6)(8,8)
\drawline(2,0)(0,0)
\drawline(0,2)(0,0)
\end{picture}
\ {}^{ \mu_{+-}}
\\
\hskip10pt
 Q^{\mathrm{op}}\otimes Q'
\hskip55pt
Q^{\mathrm{op}}\otimes Q'^{\mathrm{op}}
\end{matrix}
\end{align}
We further define  composed mutations \cite{Kel2}
\begin{align}
\label{eq:mupm}
\begin{split}
\mu_{-}&=\mu_{+-}\mu_{--},
\quad
\mu_{+}=\mu_{++}\mu_{-+},\\
\mu_{\otimes}&= \mu_{ -}\mu_{+}=
\mu_{+-}\mu_{--}\mu_{++}\mu_{-+}.
\end{split}
\end{align}
In particular, $\mu_{\otimes}$ preserves $Q\otimes Q'$.

\begin{rem}
(1) The mutations $\mu_{\pm\pm}$, $\mu_{\pm\mp}$,
and $\mu_{\otimes}$ here
correspond to $\mu_{\mp,\mp}$, $\mu_{\mp,\pm}$,
and the inverse of $\mu_{\otimes}$ in \cite{Kel2}.
This is  due to our  convention of 
the assignment of $+/-$ for the sources/sinks of $Q$ and $Q'$,
and not essential at all.
\par
(2) Instead of (\ref{eq:mupm}), one may set
$\mu_{+}=\mu_{+-}\mu_{++}$,
$\mu_{-}=\mu_{--}\mu_{-+}$,
and $\mu_{\otimes}=\mu_{+}\mu_{-}$.
This is again a matter of choice.
\quad

\end{rem}

We introduce the $I\times I'$-tuple of variables
$x=\{ x_{a,b} \mid a\in I, b\in I'\}$,
and define $\mathcal{A}_{Q\square Q'}$
to be the  cluster algebras
with initial seeds $(Q\,\square\,Q',x)$.
We set $x=x(0)$, and define clusters
$x(u)$, $z(u)$, $\overline{z}(u)$ ($u\in \mathbb{Z}$) of
$\mathcal{A}_{Q\square Q'}$ by the following sequence of mutations:
($u$: even)
\begin{align}
\label{eq:mutation2}
\begin{matrix}
   (Q\otimes Q'^{\mathrm{op}},z(u+1))\hskip35pt (Q\otimes Q',z(u+2))\\
 {}^{\mu_{++}}\
\begin{picture}(8,12)(0,-1)
\drawline(0,0)(8,8)
\drawline(6,8)(8,8)
\drawline(8,6)(8,8)
\drawline(2,0)(0,0)
\drawline(0,2)(0,0)
\end{picture}
\hskip60pt
\begin{picture}(8,10)
\drawline(0,8)(8,0)
\drawline(0,8)(2,8)
\drawline(0,8)(0,6)
\drawline(8,0)(6,0)
\drawline(8,0)(8,2)
\end{picture}
\ {}^{\mu_{--}}
\ \
 {}^{\mu_{+-}}\ 
\begin{picture}(8,10)
\drawline(0,0)(8,8)
\drawline(6,8)(8,8)
\drawline(8,6)(8,8)
\drawline(2,0)(0,0)
\drawline(0,2)(0,0)
\end{picture}
\hskip60pt
\begin{picture}(8,10)
\drawline(0,8)(8,0)
\drawline(0,8)(2,8)
\drawline(0,8)(0,6)
\drawline(8,0)(6,0)
\drawline(8,0)(8,2)
\end{picture}
\ {}^{\mu_{-+}}\\
\hskip10pt
(Q\,\square\, Q',x(u)) \hskip35pt    ((Q\,\square\,Q')^{\mathrm{op}},x(u+1))
\hskip35pt
 (Q\,\square\, Q',x(u+2))\\
 {}^{\mu_{--}}\
\begin{picture}(8,10)
\drawline(0,8)(8,0)
\drawline(0,8)(2,8)
\drawline(0,8)(0,6)
\drawline(8,0)(6,0)
\drawline(8,0)(8,2)
\end{picture}
\hskip60pt
\begin{picture}(8,10)
\drawline(0,0)(8,8)
\drawline(6,8)(8,8)
\drawline(8,6)(8,8)
\drawline(2,0)(0,0)
\drawline(0,2)(0,0)
\end{picture}
\ {}^{\mu_{++}}
\ \ 
 {}^{\mu_{-+}}\ 
\begin{picture}(8,10)
\drawline(0,8)(8,0)
\drawline(0,8)(2,8)
\drawline(0,8)(0,6)
\drawline(8,0)(6,0)
\drawline(8,0)(8,2)
\end{picture}
\hskip60pt
\begin{picture}(8,12)(0,-1)
\drawline(0,0)(8,8)
\drawline(6,8)(8,8)
\drawline(8,6)(8,8)
\drawline(2,0)(0,0)
\drawline(0,2)(0,0)
\end{picture}
\ {}^{\mu_{+-}}
\\
\hskip20pt
 (Q^{\mathrm{op}}\otimes Q',\overline{z}(u+1))\hskip35pt
(Q^{\mathrm{op}}\otimes Q'^{\mathrm{op}},\overline{z}(u+2))\\
\end{matrix} .
\end{align}
In particular,
\begin{align}
\label{eq:xab1}
&x_{a,b}(u+1)= x_{a,b}(u)
\quad \text{if $\varepsilon(a)\varepsilon'(b)(-1)^u = -$},\\
\label{eq:xa3}
&z_{a,b}(u)=
\begin{cases}
 x_{a,b}(u) & a\in I_+,\\
 x_{a,b}(u-1) & a\in I_-,
\end{cases}
\quad
\overline{z}_{a,b}(u)=
\begin{cases}
 x_{a,b}(u-1) & a\in I_+,\\
 x_{a,b}(u) & a\in I_-,
\end{cases}
\\
\label{eq:xa4}
\begin{split}
&(Q\otimes Q',z(2k))=\mu^{k}_{\otimes}(Q\otimes Q',z(0))\quad
 (k\in \mathbb{Z}),\\
&(Q\otimes Q'{}^{\mathrm{op}},z(2k+1))
=\mu_+\mu^{k}_{\otimes}(Q\otimes Q',z(0))\quad
 (k\in \mathbb{Z}).
\end{split}
\end{align}

\begin{lem} 
\label{lem:xab1}
The family 
 $\{ x_{a,b}(u)\mid a\in I, b\in I',
u\in \mathbb{Z}\}$
satisfies the T-system $\mathbb{T}(X_r,X'_{r'})$ in
 $\mathcal{A}_{Q\square Q'}$; namely,
\begin{align}
\label{eq:TA6}
x_{a,b}(u-1)x_{a,b}(u+1)
&=
\prod_{k\in I': C'_{bk}=-1}
x_{a,k}(u)
+
\prod_{k\in I: C_{ak}=-1}
x_{k,b}(u).
\end{align}
\end{lem}
\begin{proof}
The proof is the same as Lemma \ref{lem:CA1}
by replacing $\mu_+$ (resp.\ $\mu_-$) therein
with $\mu_{--}\mu_{++}$
(resp.\ $\mu_{-+}\mu_{+-}$).

\end{proof}

Define a ring homomorphism
$f:\mathcal{A}_{Q\square Q'}\rightarrow 
\EuScript{T}(X_r,X'_{r'})$
as the restriction of the ring homomorphism
$\mathbb{Z}[x_{a,b}^{\pm1}]_{(a,b)\in I\times I'}
\rightarrow \EuScript{T}(X_r,X'_{r'})$
given by
\begin{align}
\begin{split}
f:x_{a,b}^{\pm1}&\mapsto
\begin{cases}
T_{a,b}(0)^{\pm1} & \varepsilon(a)\varepsilon'(b)=+,\\
T_{a,b}(1)^{\pm1}& \varepsilon(a)\varepsilon'(b)=-.
\end{cases}
\end{split}
\end{align}
Then, as Lemma \ref{lem:xt1}, we have
\begin{lem}
\label{lem:xt2}
For the above homomorphism 
$f:\mathcal{A}_{Q\square Q'}\rightarrow 
\EuScript{T}(X_r,X'_{r'})$,
\begin{align}
\label{eq:xT2}
f:x_{a,b}(u)&\mapsto
\begin{cases}
T_{a,b}(u) & \varepsilon(a)\varepsilon'(b)(-1)^u=+,\\
T_{a,b}(u+1)& \varepsilon(a)\varepsilon'(b)(-1)^u=-.
\end{cases}
\end{align}
\end{lem}

Let $\mu_{\square}:=\mu_{-+}\mu_{+-}
\mu_{--}\mu_{++}$.
We define
${\mathcal{A}}^{\mathrm{t}}_{Q\square Q'}$
to be the subring of 
$\mathcal{A}_{Q\square Q'}$
generated by
$x_{a,b}(u)$ ($a\in I, b\in I',u\in 2\mathbb{Z}$),
i.e., the cluster variables belonging to
the seeds $\mu_{\square}^k(Q\square Q',x(0))$
($k\in \mathbb{Z}$).
We call ${\mathcal{A}}^{\mathrm{t}}_{Q\square Q'}$
 the {\em translation subalgebra
 of
${\mathcal{A}}_{Q\square Q'}$
with respect to
$\mu_{\square}$}.
The ring ${\mathcal{A}}^{\mathrm{t}}_{Q\square Q'}$
 is no longer a cluster algebra.
We remark that, by (\ref{eq:xab1}),
the ring ${\mathcal{A}}^{\mathrm{t}}_{Q\square Q'}$
coincides with  the subring of 
$\mathcal{A}_{Q\square Q'}$
generated by
$x_{a,b}(u)$ ($a\in I, b\in I',u\in \mathbb{Z}$),
which are the cluster variables belonging
to the  `bipartite belt' in \cite[Section 8]{FZ4}.

\begin{prop}
\label{prop:CT2}
The ring
$\EuScript{T}^{\circ}(X_r,X'_{r'})$
is isomorphic to  ${\mathcal{A}}^{\mathrm{t}}_{Q\square Q'}
\otimes_{\mathbb{Z}}
{\mathcal{A}}^{\mathrm{t}}_{Q\square Q'}$.
\end{prop}

\begin{proof}
By Lemma \ref{lem:xt2},
the restriction $f:{\mathcal{A}}^{\mathrm{t}}_{Q\square Q'}
\rightarrow \EuScript{T}^{\circ}(X_r,X'_{r'})_+$ is surjective.
Furthermore, the inverse correspondence $g:
\EuScript{T}^{\circ}(X_r,X'_{r'})_+\rightarrow 
{\mathcal{A}}^{\mathrm{t}}_{Q\square Q'}$,
$T_{a,b}(u)\mapsto x_{a,b}(u)$
defines a homomorphism
by Lemma \ref{lem:xab1} and the fact that
${\mathcal{A}}^{\mathrm{t}}_{Q\square Q'}$
is an integral domain.
Therefore,
\begin{align}
\EuScript{T}^{\circ}(X_r,X'_{r'})_+
\simeq
{\mathcal{A}}^{\mathrm{t}}_{Q\square Q'},
\end{align}
and we obtain the assertion.
\end{proof}

\begin{rem} Hernandez-Leclerc \cite{HL} also study
the relation between 
the cluster algebra $\mathcal{A}_{Q\square Q'}$
with $X'_{r'}=A_{r'}$ and the T-system
in view of the categorification of 
$\mathcal{A}_{Q\square Q'}$ by a subcategory
of the category of the finite-dimensional
$U_q(\hat{\mathfrak{g}})$-modules.
\end{rem}

Now let us turn to the periodicity problem.
Let $\mathcal{A}_{Q\otimes Q'}$ be the cluster algebra
with initial seed $(Q\otimes Q', z)$,
$z=\{ z_{a,b} \mid a\in I, b\in I'\}$.
Two cluster algebras
$\mathcal{A}_{Q\otimes Q'}$ and $\mathcal{A}_{Q\square Q'}$
coincide by setting $z=z(0)$ in (\ref{eq:mutation2}).
A crucial observation  made by Keller \cite{Kel2} is    that
the periodicity of $\mathcal{A}_{Q\square Q'}$
is more transparent 
in  `$\otimes$-picture'
than  `$\square$-picture' from the cluster categorical point of view.
By (\ref{eq:xa3}) and 
(\ref{eq:xT2}),
it is easy to check that
the periodicity (\ref{eq:xpt2}) of $\EuScript{T}(X_r,X'_{r'})$
is translated  as
\begin{align}
\label{eq:xpt1}
\begin{split}
&\text{half-periodicity:}
\quad
z_{a,b}(u+h+h')=
\begin{cases}
z_{\omega(a),\omega'(b)}(u)& \text{$h$: even},\\
\overline{z}_{\omega(a),\omega'(b)}(u)& \text{$h$: odd},
\end{cases}
\\
&\text{periodicity:}
\quad\hskip6.8pt
z_{a,b}(u+2(h+h'))=z_{a,b}(u).
\end{split}
\end{align}

The following 
periodicity of ${\mathcal{A}}_{Q\otimes Q'}$
is immediately obtained from the results in  \cite{Kel2}.

\begin{thm}[Keller {\cite{Kel2}}]
\label{thm:cluster2}
 The following equality holds
for ${\mathcal{A}}_{Q\otimes Q'}$
 ($u$: even):
\begin{align}
 \mu^{h+h'}_{\otimes}(Q\otimes Q',z(u))\overset{\mathrm{id}}{=}
 (Q\otimes Q',z(u)).
\end{align}
\end{thm}
\begin{proof}
The $F$-polynomials
 of ${\mathcal{A}}_{Q\otimes Q'}$  relevant 
to the above
periodicity  are expressed in terms of the triangulated
category $\mathcal{C}_{KQ\otimes KQ'}$
 \cite[Theorem 7.13 (c)]{Kel2};
furthermore,  $\mathcal{C}_{KQ\otimes KQ'}$
has the desired periodicity  \cite[Proposition 8.5]{Kel2}.
Therefore, the  cluster variable $z(u)$ has the same periodicity
by \cite[Corollay 6.3, Proposition 6.9]{FZ4}.
\end{proof}

As a refinement of Theorem \ref{thm:cluster2},
we also show
the half-periodicity of ${\mathcal{A}}_{Q\otimes Q'}$.
\begin{thm}
\label{thm:cluster3}
 The following equalities hold for
 ${\mathcal{A}}_{Q\otimes Q'}$
 ($u$: even):

(1) For $(h,h')=(\text{even},\text{even})$,
\begin{align}
\mu_{\otimes}^{\frac{h+h'}{2}}
(Q\otimes Q',z(u))
\overset{\omega\times \omega'}{=}
 (Q\otimes Q', z(u)).
\end{align}

(2) For $(h,h')=(\text{odd},\text{odd}\/)$,
\begin{align}
\mu_{\otimes}^{\frac{h+h'}{2}}
(Q\otimes Q',z(u))
\overset{\omega\times \omega'}{=}
(Q^{\mathrm{op}}\otimes Q'{}^{\mathrm{op}},
\overline{z}(u)).
\end{align}

(3) For $(h,h')=(\text{even},\text{odd}\/)$,
\begin{align}
\mu_{+}\mu_{\otimes}^{\frac{h+h'-1}{2}}
(Q\otimes Q',z(u))
\overset{\omega\times \omega'}{=}
 (Q\otimes Q', z(u)).
\end{align}

(4) For $(h,h')=(\text{odd},\text{even}\/)$,
\begin{align}
\mu_{+}\mu_{\otimes}^{\frac{h+h'-1}{2}}
(Q\otimes Q',z(u))
\overset{\omega\times \omega'}{=}
(Q^{\mathrm{op}}\otimes Q'{}^{\mathrm{op}},
\overline{z}(u)).
\end{align}
\end{thm}

The proof of Theorem \ref{thm:cluster3}
is given in the next subsection.

\begin{cor}
\label{cor:SL2}
The following relations hold in
$\EuScript{T}(X_r,X'_{r'})$:
\par
(1) Half-periodicity:
$T_{a,b}(u+h+h')=
T_{\omega(a),\omega'(b)}(u)$.
\par
(2)
Periodicity:
$T_{a,b}(u+2(h+h'))=
T_{a,b}(u)$.
\end{cor}
\begin{proof}
The relations in (\ref{eq:xpt1})
immediately follow from 
Theorems \ref{thm:cluster2}, \ref{thm:cluster3},
and (\ref{eq:xa4}).
\end{proof}

By Corollaries \ref{cor:SL1} and \ref{cor:SL2},
we obtain the main result of this section:
\begin{cor}
\label{cor:SL3}
The following relations hold in
$\EuScript{T}_{\ell}(X_r)$
for any simply laced $X_r$ and any $\ell \geq 2$:
\par
(1) Half-periodicity:
$T^{(a)}_m(u+h+\ell)=
T^{(\omega(a))}_{\ell-m}(u)$.
\par
(2)
Periodicity:
$T^{(a)}_m(u+2(h+\ell))=
T^{(a)}_{m}(u)$.
\end{cor}

\subsection{Proof of Theorems
 \ref{thm:cluster2} and  \ref{thm:cluster3}\ by cluster category}
\label{subsect:cc2}
Here we prove Theorem \ref{thm:cluster3}
by adapting the method of \cite[Theorem 8.2]{Kel2}
for our situation.
In the course we also include a proof of
Theorem  \ref{thm:cluster2} without using the $F$-polynomials
 for the reader's convenience.
We present the proof as parallel as possible to the level 2 case
in Section \ref{subsect:cc1}.

Let $Q$ and $Q'$ continue to be
 the alternating quivers in Section \ref{subsect:lg2}
whose underlying graphs are 
simply laced $X_r$ and $X'_{r'}$ other than $A_1$ respectively.
For $Q$ (resp.\ $Q'$), let $KQ$, $\mathcal{D}_Q$,
$\tau:\mathcal{D}_Q\to\mathcal{D}_Q$
(resp.\ 
$KQ'$, $\mathcal{D}_{Q'}$,
$\tau':\mathcal{D}_{Q'}\to\mathcal{D}_{Q'}$)
to be the ones in Section \ref{subsect:cc1}.

We denote the tensor product $\otimes_K$ by $\otimes$ simply,
and we define a finite dimensional $K$-algebra $A$ by
\begin{align}
A:=KQ\otimes KQ'.
\end{align}
Let $\mathcal{D}_A=\mathcal{D}^{\rm b}(\mathrm{mod}\, A)$
be  the bounded derived category of
finite dimensional $A$-modules, and
\begin{align}
\tau_{\otimes}:=DA[-1]\stackrel{{\bf L}}{\otimes}_{A}-:
\mathcal{D}_A\to\mathcal{D}_A
\end{align}
be the Auslander-Reiten translation for $\mathcal{D}_A$.
Now we define another autoequivalence of $\mathcal{D}_A$ by $F:=\tau_{\otimes}^{-1}\circ[1]$.
Later we need the following easy observation.

\begin{lem}\label{F}
$\tau(X)\otimes\tau'(Y)\simeq F^{-1}(X\otimes Y)$ in $\mathcal{D}_A$ for any $X\in\mathcal{D}_Q$ and $Y\in\mathcal{D}_{Q'}$.
\end{lem}

\begin{proof}
We have
$\tau(X)\otimes\tau'(Y)\simeq (D(KQ)\stackrel{{\bf L}}{\otimes}_{KQ}X[-1])\otimes(D(KQ')\stackrel{{\bf L}}{\otimes}_{KQ'}Y[-1])
\simeq(D(KQ)\otimes D(KQ'))[-2]\stackrel{{\bf L}}{\otimes}_A(X\otimes Y)=F^{-1}(X\otimes Y)$.
\end{proof}

The orbit category
\begin{align}
\mathcal{D}_A/F
\end{align}
has the same objects with $\mathcal{D}_A$, and the morphism space is given by
\begin{align}
\mathrm{Hom}_{(\mathcal{D}_A/F)}(X,Y):=\bigoplus_{i\in{\bf Z}}
\mathrm{Hom}_{\mathcal{D}_A}(X,F^i(Y))
\end{align}
for any $X,Y\in\mathcal{D}_A/F$.
In contrast to (\ref{eq:cq1}),
 $\mathcal{D}_A/F$  is no longer a triangulated category in general.
However, based on the works of Keller \cite{Kel1,Kel4},
Amiot \cite[Sect.\ 4]{A}
 constructed
a triangulated hull $\mathcal{C}_A$ of $\mathcal{D}_A/F$,
which is a 2-Calabi-Yau triangulated category with a fully faithful functor $\mathcal{D}_A/F\to\mathcal{C}_A$ satisfying a certain universal property.
(Here we need the fact that the functor $H^0(F-)$ is nilpotent on $\mathrm{mod}\, A$, which follows from Lemma \ref{F}.)
We call $\mathcal{C}_A$ the {\em (generalized) cluster category
of $A$.}

\medskip
We say that an object $T=\bigoplus_{(a,b)\in I\times I'}T_{a,b}
\in\mathcal{C}_A$ is \emph{cluster tilting} if 
\begin{itemize}
\item[(1)] each $T_{a,b}$ is indecomposable and mutually non-isomorphic,
\item[(2)] $\mathrm{add}\, T = \{
X\in \mathcal{C}_A
\mid \mathrm{Hom}_{\mathcal{C}_A}(T,X[1])=0\}$.
\end{itemize}
(To simplify our proof, we assume that the index set of
direct summands of $T$ is $I\times I'$.
This does not affect 
the definition essentially due to
Example \ref{QQ} and \cite[Theorem 2.4]{DeK}.)
For a cluster tilting object $T\in\mathcal{C}_A$, we denote by $Q_T$ the quiver of the endomorphism ring $\mathrm{End}_{\mathcal{C}_A}(T)$ \cite{ARS,ASS}.

\begin{exmp}[cf.\ Example \ref{P and U}]\label{QQ}
For $a\in I$ and $b\in I'$, we denote by $e_a$ and $e'_b$ the paths of length zero in $Q$ and $Q'$ respectively.
As in Example \ref{P and U},
 we define $KQ$-modules $P$, $U$ and $KQ'$-modules $P'$, $U'$ by
\begin{alignat}{2}
P_a:&=(KQ)e_a,\quad& P:&=KQ=\bigoplus_{a\in I}P_a,\\
U_a:&=\begin{cases}
\tau^{-1}(P_a)&a\in I_+,\\
P_a&a\in I_-,
\end{cases}\quad &U:&=\bigoplus_{a\in I}U_a,\\
P'_b:&=(KQ')e'_b,\quad &P':&=KQ'=\bigoplus_{b\in I'}P'_b,\\
U'_b:&=\begin{cases}
{\tau'}^{-1}(P'_b)&b\in I'_+,\\
P'_b&b\in I'_-,
\end{cases}
\quad &U':&=\bigoplus_{b\in I'}U'_b.
\end{alignat}
Then
\begin{align}
P\otimes P'=\bigoplus_{(a,b)\in I\times I'}P_a\otimes P'_b
\end{align}
is a cluster tilting object in $\mathcal{C}_A$ with $Q_{P\otimes P'}=Q\otimes Q'$
\cite[Theorem 4.10]{A}.
\end{exmp}

Again we can define the mutation of cluster tilting objects as
 follows \cite[Theorem 5.3]{IY}.

\begin{thm}[cf.\ Theorem \ref{thm:mut1}]
Let $T=\bigoplus_{(a,b)\in I\times I'}T_{a,b}\in\mathcal{C}_A$ be
 a cluster tilting object. For any $(a,b)\in I\times I'$,
there exists a unique indecomposable object $T^*_{a,b}\in\mathcal{C}_A$
 which is not isomorphic to $T_{a,b}$ such that $(T/T_{a,b})\oplus T^*_{a,b}$
 is a cluster tilting object.
\end{thm}

We call the above $(T/T_{a,b})\oplus T^*_{a,b}$ the
 \emph{cluster tilting mutation of\/ $T$ at $(a,b)$},
and denote it by $\mu_{a,b}(T)$.

We number elements of $I_-\times I'_+$ as $\{c_1,\cdots,c_s\}$
 ($s:=|I_-\times I'_+|$), and define a composed cluster tilting mutation
$\mu_{-+}$  by
\begin{equation}
\mu_{-+}:=\mu_{c_s}\cdots\mu_{c_1}.
\end{equation}
Similarly we define
\begin{equation}
\mu_{++},\ \mu_{--},\ \mu_{+-}
\end{equation}
by using $I_+\times I'_+$, $I_-\times I'_-$ and $I_+\times I'_-$ respectively.
We further define
\begin{equation}
\mu_\otimes:=\mu_{+-}\mu_{--}\mu_{++}\mu_{-+},\quad\mu_{+}:=\mu_{++}\mu_{-+}.
\end{equation}
Thus we have a numbering $\{c_1,\cdots,c_{rr'}\}$ of the 
elements of $I\times I'$
 such that $\mu_\otimes=\mu_{c_{rr'}}\cdots\mu_{c_1}$.

For a cluster tilting object $T=\bigoplus_{(a,b)\in I\times I'}T_{a,b}$,
let $[T]_{a,b}$ denote $T_{a,b}$.
We have the following key observation.

\begin{prop}[cf.\ Proposition \ref{mu and tau}]\label{mu and tau1}
(1) We have a diagram
\begin{align}
\label{eq:mutation3}
\begin{matrix}
\hskip10pt
P\otimes{\tau'}^{-k}(P')
\hskip30pt
P\otimes{\tau'}^{-k}(U')
\hskip30pt
P\otimes{\tau'}^{-k-1}(P')\\
\begin{picture}(8,12)(0,-1)
\drawline(0,8)(8,0)
\drawline(0,8)(2,8)
\drawline(0,8)(0,6)
\drawline(8,0)(6,0)
\drawline(8,0)(8,2)
\end{picture}
\ {}^{\mu_{-+}}
\ \ 
 {}^{\mu_{++}}\
\begin{picture}(8,10)
\drawline(0,0)(8,8)
\drawline(6,8)(8,8)
\drawline(8,6)(8,8)
\drawline(2,0)(0,0)
\drawline(0,2)(0,0)
\end{picture}
\hskip30pt
\begin{picture}(8,10)
\drawline(0,8)(8,0)
\drawline(0,8)(2,8)
\drawline(0,8)(0,6)
\drawline(8,0)(6,0)
\drawline(8,0)(8,2)
\end{picture}
\ {}^{\mu_{--}}
\ \
 {}^{\mu_{+-}}\ 
\begin{picture}(8,10)
\drawline(0,0)(8,8)
\drawline(6,8)(8,8)
\drawline(8,6)(8,8)
\drawline(2,0)(0,0)
\drawline(0,2)(0,0)
\end{picture}
\\
\hskip7pt
V^k \hskip82pt W^k
\\
\begin{picture}(8,10)
\drawline(0,0)(8,8)
\drawline(6,8)(8,8)
\drawline(8,6)(8,8)
\drawline(2,0)(0,0)
\drawline(0,2)(0,0)
\end{picture}
\ {}^{\mu_{+-}}
\ \ 
 {}^{\mu_{--}}
\
\begin{picture}(8,10)
\drawline(0,8)(8,0)
\drawline(0,8)(2,8)
\drawline(0,8)(0,6)
\drawline(8,0)(6,0)
\drawline(8,0)(8,2)
\end{picture}
\hskip30pt
\begin{picture}(8,10)
\drawline(0,0)(8,8)
\drawline(6,8)(8,8)
\drawline(8,6)(8,8)
\drawline(2,0)(0,0)
\drawline(0,2)(0,0)
\end{picture}
\ {}^{\mu_{++}}
\ \
 {}^{\mu_{-+}}\
\begin{picture}(8,12)(0,-1)
\drawline(0,8)(8,0)
\drawline(0,8)(2,8)
\drawline(0,8)(0,6)
\drawline(8,0)(6,0)
\drawline(8,0)(8,2)
\end{picture}
\\
\hskip12pt
U\otimes{\tau'}^{-k}(U')
\hskip25pt
U\otimes{\tau'}^{-k-1}(P')
\hskip25pt
U\otimes{\tau'}^{-k-1}(U')\\
\end{matrix}
\end{align}
of composed cluster tilting mutations for any $k\in\mathbb{Z}$, where
\begin{align}
V^k:&=(\bigoplus_{(a,b)\in I_-\times I'_+}P_a\otimes{\tau'}^{-k-1}(P'_b))
\oplus
(\bigoplus_{(a,b)\notin I_-\times I'_+}P_a\otimes{\tau'}^{-k}(P'_b)),\\
W^k:&=(\bigoplus_{(a,b)\notin I_+\times I'_-}P_a\otimes{\tau'}^{-k-1}(P'_b))
\oplus
(\bigoplus_{(a,b)\in I_+\times I'_-}P_a\otimes{\tau'}^{-k}(P'_b)).
\end{align}
\par
(2) We have
\begin{align}
\begin{split}
 [\mu_\otimes^k(P\otimes P')]_{a,b}&=P_a\otimes{\tau'}^{-k}(P'_b),\\
 [\mu_+\mu_\otimes^k(P\otimes P')]_{a,b}&=P_a\otimes{\tau'}^{-k}(U'_b)
\end{split}
\end{align}
for any $k\in\mathbb{Z}$ and $(a,b)\in I\times I'$.
\end{prop}

\begin{proof}
(2) is an immediate consequence of (1).

(1) We only show $\mu_{-+}(P\otimes P')=V^0$ since other cases are shown similarly.
We put $J^{\ell}:=\{c_i\ |\ 1\le i\le \ell\}$.
Define $T^{\ell}\in\mathcal{D}_A$ by
\[T^{\ell}:=(\bigoplus_{(a,b)\in J^{\ell}}P_a\otimes{\tau'}^{-1}(P'_b))\oplus(\bigoplus_{(a,b)\in(I\times I')\backslash J^{\ell}}P_a\otimes P'_b)\]
for any $0\le\ell\le s=|I_-\times I'_+|$.
Clearly we have $T^0=P\otimes P'$ and $T^s=V^0$.

The following observation is crucial.

\begin{prop}\label{endomorphism algebra}
(1) $T^{\ell}$ is a tilting $A$-module for any $0\le\ell\le s$.
\par
(2) The algebra $B:=\mathrm{End}_A(T^{\ell})$ has global dimension at most two.
\par
(3) For any simple $B$-modules $S$ and $S'$, we have $\mathrm{Ext}^1_B(S,S)=0$,
 and either $\mathrm{Ext}^1_B(S,S')=0$ or $\mathrm{Ext}^2_B(S,S')=0$ holds.
\end{prop}

\begin{proof}
(1) Since $Q'$ is not of type $A_1$, our $T^{\ell}$ is an $A$-module.
Clearly $T^{\ell}$ has projective dimension at most one, and the number of
indecomposable direct summands of $T^{\ell}$ is $rr'$.
Thus we have only to show $\mathrm{Ext}^1_A(T^{\ell},T^{\ell})=0$.
Fix $(a,b),(c,d)\in I\times I'$, and we shall show
\begin{align}
\label{desired}
\mathrm{Hom}_{\mathcal{D}_A}(T^{\ell}_{a,b},T^{\ell}_{c,d}[1])=0.
\end{align}
We have a general equality
\begin{align}
\label{1}
\mathrm{Hom}_{\mathcal{D}_A}(X\otimes X',(Y\otimes Y')[i])
=\bigoplus_{j+j'=i}\mathrm{Hom}_{\mathcal{D}_Q}(X,Y[j])\otimes\mathrm{Hom}_{\mathcal{D}_{Q'}}(X',Y'[j'])
\end{align}
for any $X,Y\in\mathcal{D}_Q$ and $X',Y'\in\mathcal{D}_{Q'}$.
We also have
\begin{align}\label{2}
\begin{split}
&\mathrm{Hom}_{\mathcal{D}_Q}(P_a,P_c[j])=0\\
&\qquad \mbox{if ``$j\neq0$''
 or ``$j=0$, $a\neq c$ and there is no arrow $a\to c$ in $Q$''.}
\end{split}
\end{align}
We divide the proof of \eqref{desired} into four cases.

(i) Assume $(a,b),(c,d)\notin J^{\ell}$.
Then \eqref{desired} follows from \eqref{1}, \eqref{2} and 
\begin{align}
\mathrm{Hom}_{\mathcal{D}_{Q'}}(P'_b,P'_d[j'])=0\quad\mbox{if $j'\neq0$}.
\end{align}

(ii) Assume $(a,b),(c,d)\in J^{\ell}$.
Then \eqref{desired} follows from \eqref{1}, \eqref{2} and 
\begin{align}
\mathrm{Hom}_{\mathcal{D}_{Q'}}({\tau'}^{-1}(P'_b),{\tau'}^{-1}(P'_d)[j'])\simeq \mathrm{Hom}_{\mathcal{D}_{Q'}}(P'_b,P'_d[j'])=0\quad\mbox{if $j'\neq0$}.
\end{align}

(iii) Assume $(a,b)\in J^{\ell}$ and $(c,d)\notin J^{\ell}$. We have
\begin{align}
\begin{split}
&\mathrm{Hom}_{\mathcal{D}_{Q'}}({\tau'}^{-1}(P'_b),P'_d[j'])\simeq D\mathrm{Hom}_{\mathcal{D}_{Q'}}(P'_d[j'],P'_b[1])=0\\
&\qquad \mbox{if ``$j'\neq 1$'' or ``$j'=1$, $b\neq d$
 and there is no arrow $d\to b$ in $Q'$''.}
\end{split}
\end{align}
Since $a\in I_-$, $b\in I'_+$ and $(a,b)\neq(c,d)$ hold,
 either ``$a\neq c$ and there is no arrow $a\to c$ in $Q$'' 
 or ``$b\neq d$ and there is no arrow $d\to b$ in $Q'$''.
Thus \eqref{desired} follows from \eqref{1} and \eqref{2}.

(iv) Assume $(a,b)\notin J^{\ell}$ and $(c,d)\in J^{\ell}$.
Since $d\in I'_+$ and $Q'$ is not type $A_1$, we have
\begin{align}
\mathrm{Hom}_{\mathcal{D}_{Q'}}(P'_b,{\tau'}^{-1}(P'_d)[j'])=0\quad\mbox{if $j'\neq0$}.
\end{align}
Thus \eqref{desired} follows from \eqref{1} and \eqref{2}.

(2) \& (3) It is enough to prove the following claim:

\begin{claim}
Every simple $B$-module $S$ has a projective resolution
\begin{align}
0\to B_2\to B_1\to B_0\to S\to0
\end{align}
such that $B_1$ and $B_0\oplus B_2$ do not have any non-zero common direct summand.
\end{claim}

Let $(a,b)\in I\times I'$ be the one
such that $S$ is the top of the indecomposable projective
 $B$-module $\mathrm{Hom}_A(T^{\ell},T^{\ell}_{a,b})$.
We divide the proof of Claim into four cases.

(i) Assume $a\in I_+$.
Then the radical of the $A$-module $P_a\otimes P'_b$ is $\bigoplus_{d\to b}P_a\otimes P'_d$, which belongs to $\mathrm{add}\, T^{\ell}$.
Here the direct sum is taken over all arrows in $Q'$ with target $b$.
Thus we have a projective resolution
\begin{align}
0\to\mathrm{Hom}_A(T^{\ell},\bigoplus_{d\to b}P_a\otimes P'_d)
\to\mathrm{Hom}_A(T^{\ell},P_a\otimes P'_b)\to S\to0,
\end{align}
and Claim follows.

(ii) Assume $(a,b)\in(I_-\times I'_+)\backslash J^{\ell}$.
Then the radical of the $A$-module $P_a\otimes P'_b$ is $\bigoplus_{c\to a}P_c\otimes P'_b$, which belongs to $\mathrm{add}\, T^{\ell}$.
Thus we have a projective resolution
\begin{align}
0\to\mathrm{Hom}_A(T^{\ell},\bigoplus_{c\to a}P_c\otimes P'_b)
\to\mathrm{Hom}_A(T^{\ell},P_a\otimes P'_b)\to S\to0,
\end{align}
and Claim follows.

(iii) Assume $(a,b)\in J^{\ell}$.
Applying $(P_a\otimes-)$ to the Auslander-Reiten sequence $0\to P'_b\to \bigoplus_{b\to d}P'_d\to{\tau'}^{-1}(P'_b)\to0$ of $KQ'$-modules,
we have an exact sequence 
\begin{align}
0\to P_a\otimes P'_b\to\bigoplus_{b\to d}P_a\otimes P'_d\xrightarrow{f}P_a\otimes{\tau'}^{-1}(P'_b)\to0
\end{align}
of $A$-modules whose middle term belongs to $\mathrm{add}\, T^{\ell}$.
Clearly any morphism $T^{\ell}\to P_a\otimes{\tau'}^{-1}(P'_b)$ which is not a split epimorphism factors through $f$.
Moreover the left term $P_a\otimes P'_b$ does not belong to $\mathrm{add}\, T^{\ell}$, but its radical $\bigoplus_{c\to a}P_c\otimes P'_b$ belongs to $\mathrm{add}\, T^{\ell}$.
Consequently we have a projective resolution
\begin{align}
\begin{split}
&0\to\mathrm{Hom}_A(T^{\ell},\bigoplus_{c\to a}P_c\otimes P'_b)
\to\mathrm{Hom}_A(T^{\ell},\bigoplus_{b\to d}P_a\otimes P'_d)\\
&\qquad
\xrightarrow{f}\mathrm{Hom}_A(T^{\ell},P_a\otimes{\tau'}^{-1}(P'_b))\to S\to0
\end{split}
\end{align}
and Claim follows.

(iv) Assume $(a,b)\in I_-\times I'_-$.
Taking a tensor product of exact sequences $0\to\bigoplus_{c\to a}P_c\to P_a$ and $0\to\bigoplus_{d\to b,\ (a,d)\notin J^{\ell}}P'_d\to P'_b$, we have an exact sequence
\begin{align}
0\to\bigoplus_{\def\arraystretch{.5}\begin{array}{c}
{\scriptstyle c\to a}\\
{\scriptstyle d\to b}\\
{\scriptstyle (a,d)\notin J^{\ell}}
\end{array}}P_c\otimes P'_d
\to(\bigoplus_{\def\arraystretch{.5}\begin{array}{c}
{\scriptstyle d\to b}\\
{\scriptstyle (a,d)\notin J^{\ell}}
\end{array}}P_a\otimes P'_d)\oplus(\bigoplus_{c\to a}P_c\otimes P'_b)\xrightarrow{f}P_a\otimes P'_b
\end{align}
of $A$-modules whose terms belong to $\mathrm{add}\, T^{\ell}$.
Clearly any morphism $T^{\ell}\to P_a\otimes P'_b$ which is not a split epimorphism factors through $f$.
Thus we have a projective resolution
\begin{align}
\begin{split}
&0\to\mathrm{Hom}_A(T^{\ell},\bigoplus_{\def\arraystretch{.5}\begin{array}{c}
{\scriptstyle c\to a}\\
{\scriptstyle d\to b}\\
{\scriptstyle (a,d)\notin J^{\ell}}
\end{array}}
P_c\otimes P'_d)\\
&\qquad
\to\mathrm{Hom}_A(T^{\ell},(\bigoplus_{\def\arraystretch{.5}\begin{array}{c}
{\scriptstyle d\to b}\\
{\scriptstyle (a,d)\notin J^{\ell}}
\end{array}}P_a\otimes P'_d)\oplus(\bigoplus_{c\to a}P_c\otimes P'_b))\\
&\qquad\qquad
\xrightarrow{f}\mathrm{Hom}_A(T^{\ell},P_a\otimes P'_b)\to S\to0
\end{split}
\end{align}
and Claim follows.
\end{proof}

By Proposition \ref{endomorphism algebra} (1) and (2) together with \cite[Thm.\ 4.10]{A},
$T^\ell$ is a cluster tilting object in $\mathcal{C}_A$ for any $1\le\ell\le s$.
Since $T^{\ell-1}$ and $T^{\ell}$ have the same indecomposable direct summands except $T^\ell_{c_\ell}$, we have $\mu_{c_\ell}(T^{\ell-1})=T^\ell$.
Consequently, we have $\mu_{-+}(P\otimes P')=\mu_{-+}(T^0)=T^s=V^0$.
\end{proof}

We define a set $\mathbf{T}$ of cluster tilting objects by
\[\mathbf{T}=\{\mu_{c_i}\cdots\mu_{c_1}\mu_\otimes^u(P\otimes P')\ |\ u\in\mathbf{Z},\ 0\le i<rr'\}.\]
We have the following result by \cite[Proposition 8.3]{Kel2}.

\begin{prop}\label{cluster structure}
The quiver $Q_T$ has no loops and 2-cycles for any $T\in\mathbf{T}$.
\end{prop}

\begin{proof}
This is a consequence of Proposition \ref{endomorphism algebra} (3)
 and \cite[Proposition 4.16]{A}.
\end{proof}

The following result is crucial in our proof.

\begin{thm}[cf.\ Theorems \ref{thm:Xbij} and \ref{X map}]
\label{X map1}
There exists a map
\begin{align}
\begin{split}
&X:\{\mbox{indecomposable direct summands of objects in $\mathbf{T}$}\}/\simeq
\\
&\qquad  \to\{\mbox{cluster variables in $\mathcal{A}_{Q\otimes Q'}$}\}
\end{split}
\end{align}
such that we have a map
\begin{align}
\widetilde{X}:\mathbf{T}\to\{\mbox{seeds in $\mathcal{A}_{Q\otimes Q'}$}\}
\end{align}
defined by
\begin{align}
T=\bigoplus_{(a,b)\in I\times I'}T_{a,b}\mapsto
(Q_T,\{X_{T_{a,b}}\}_{(a,b)\in I\times I'})
\end{align}
satisfying the following conditions.
\begin{itemize}
\item[(1)] $\widetilde{X}_{P\otimes P'}=(Q\otimes Q',z(0))$.
\item[(2)] If $T$ and $\mu_{a,b}(T)$ belong to $\mathbf{T}$,
 then $\widetilde{X}_{\mu_{a,b}(T)}=\mu_{a,b}(\widetilde{X}_T)$.
\end{itemize}
\end{thm}

\begin{proof}
The assertion follows from Proposition \ref{cluster structure},
 \cite[Theorem 4]{Pal} and \cite[Theorem I.1.6]{BIRS}.
\end{proof}

By Theorem \ref{X map1}, (\ref{eq:mutation2}),
and (\ref{eq:mutation3}), we also obtain
\begin{align}
\begin{split}
\widetilde{X}_{U\otimes U'}&=(Q^{\mathrm{op}}\otimes
 Q'{}^{\mathrm{op}},\overline{z}(0)),\\
\widetilde{X}_{P\otimes U'}&=(Q\otimes
 Q'{}^{\mathrm{op}},z(1)),\quad
\widetilde{X}_{U\otimes P'}=(Q^{\mathrm{op}}\otimes
 Q',\overline{z}(-1)).
\end{split}
\end{align}

Now we are ready to prove Theorem \ref{thm:cluster2} (full-periodicity).
We use the following periodicity result by \cite[Propoistion 8.5]{Kel2}.

\begin{prop}\label{Tau periodicity}
$X\otimes{\tau'}^{-h-h'}(Y)\simeq X\otimes Y$ in $\mathcal{C}_A$ for any $X\in\mathcal{D}_Q$ and $Y\in\mathcal{D}_{Q'}$.
\end{prop}

\begin{proof}
We have
$X\otimes{\tau'}^{-h-h'}(Y)\simeq F^{-h}(X\otimes{\tau'}^{-h-h'}(Y))
\stackrel{{\rm Lem.\ref{F}}}{\simeq}\tau^h(X)\otimes{\tau'}^{-h'}(Y)
\allowbreak
\stackrel{{\rm Prop.\ref{tau periodicity}(1)}}{\simeq}X[-2]\otimes Y[2]\simeq X\otimes Y$.
\end{proof}

For a seed $(R,y)$ of
 $\mathcal{A}_{Q\otimes Q'}$, let $[(R,y)]_{a,b}$ denote $y_{a,b}$.

\begin{thm}[Theorem \ref{thm:cluster2}]
$[\mu^{h+h'}_\otimes(Q\otimes Q',z(0))]_{a,b}=z_{a,b}(0)$.
\end{thm}

\begin{proof}
We have $[\mu^{h+h'}_\otimes(P\otimes P')]_{a,b}\stackrel{{\rm Prop. \ref{mu and tau1}(2)}}{=}P_a\otimes{\tau'}^{-h-h'}(P'_b)
\stackrel{\rm Prop. \ref{Tau periodicity}}{=}P_a\otimes P'_b$ for any $(a,b)\in I\times I'$.
Applying $\widetilde{X}$, we have
$[\mu^{h+h'}_\otimes(Q\otimes Q',z(0))]_{a,b}\stackrel{{\rm Thm. \ref{X map1}(1)}}{=}
[\mu^{h+h'}_\otimes(\widetilde{X}_{P\otimes P'})]_{a,b}
\stackrel{{\rm Thm. \ref{X map1}(2)}}{=}X_{[\mu^{h+h'}_\otimes(P\otimes P')]_{a,b}}
=X_{P_a\otimes P'_b}=z_{a,b}(0)$.
\end{proof}

Next we prove Theorem \ref{thm:cluster3} (half-periodicity).
We divide the proof  into four cases.

Recall the following facts in Section \ref{subsect:cc1}:
When $h$ is even, the map $\omega:I\to I$ induces
 the quiver automorphism $\omega:Q\to Q$, the $K$-algebra
 automorphism $\omega:KQ\to KQ$ and an autoequivalence
 $\omega:\mathcal{D}_Q\to\mathcal{D}_Q$.
We have $\omega(P_a)=P_{\omega(a)}$ for any $a\in I$.
When $h$ is odd, the map $\omega:I\to I$ induces the
 quiver isomorphism $\omega:Q\to Q^{\rm op}$, the
 $K$-algebra isomorphism $\omega:KQ\to KQ^{\rm op}$ and an autoequivalence
\begin{align}
r:\mathcal{D}_{Q}\xrightarrow{\omega}\mathcal{D}_{Q^{\rm op}}
\xrightarrow{U\stackrel{{\bf L}}{\otimes}_{KQ{}^{\rm op}}-}\mathcal{D}_{Q}.
\end{align}
We have $r(P_a)=U_{\omega(a)}$ for any $a\in I$,
and $r^2(X)\simeq \tau^{-1}(X)$ for any $X\in \mathcal{D}_Q$.

We also 
define $\omega':\mathcal{D}_{Q'}\to\mathcal{D}_{Q'}$
for even $h'$,
and
 $r':\mathcal{D}_{Q'}\to\mathcal{D}_{Q'}$ for odd $h'$
in the same way.

\medskip\noindent
{\bf (Case 1) Both $h$ and  $h'$ are even.}

\begin{prop}\label{Tau periodicity 2}
$X\otimes{\tau'}^{-\frac{h+h'}{2}}(Y)\simeq\omega(X)\otimes\omega'(Y)$
 in $\mathcal{C}_A$ for any $X\in\mathcal{D}_Q$ and $Y\in\mathcal{D}_{Q'}$.
\end{prop}

\begin{proof}
We have
$X\otimes{\tau'}^{-\frac{h+h'}{2}}(Y)\simeq F^{-\frac{h}{2}}(X\otimes
{\tau'}^{-\frac{h+h'}{2}}(Y))
\stackrel{{\rm Lem. \ref{F}}}{\simeq}\tau^{\frac{h}{2}}(X)\otimes
{\tau'}^{-\frac{h'}{2}}(Y)
\allowbreak
\stackrel{{\rm Prop.\ref{tau periodicity 2}(1)}}{\simeq}\omega^{-1}(X[-1])\otimes\omega'(Y[1])
\simeq\omega(X)\otimes\omega'(Y)$.
\end{proof}

\begin{thm}[Theorem \ref{thm:cluster3} (1)]
$[\mu^{\frac{h+h'}{2}}_\otimes(Q\otimes Q',z(0))]_{a,b}=
z_{\omega(a),\omega'(b)}(0)$.
\end{thm}

\begin{proof}
We have $[\mu^{\frac{h+h'}{2}}_\otimes(P\otimes P')]_{a,b}\stackrel{{\rm Prop. \ref{mu and tau1}(2)}}{=}
P_a\otimes{\tau'}^{-\frac{h+h'}{2}}(P'_b)\stackrel{{\rm Prop. \ref{Tau periodicity 2}}}{=}
P_{\omega(a)}\otimes P'_{\omega'(b)}$.
Applying $\widetilde{X}$, we have
$[\mu^{\frac{h+h'}{2}}_\otimes(Q\otimes Q',z(0))]_{a,b}
\stackrel{{\rm Thm. \ref{X map1}(1)}}{=}
[\mu^{\frac{h+h'}{2}}_\otimes(\widetilde{X}_{P\otimes P'})]_{a,b}
\allowbreak
\stackrel{{\rm Thm. \ref{X map1}(2)}}{=}
X_{[\mu^{\frac{h+h'}{2}}_\otimes(P\otimes P')]_{a,b}}=
X_{P_{\omega(a)}\otimes P'_{\omega'(b)}}=z_{\omega(a),\omega'(b)}(0)$.
\end{proof}

\medskip\noindent
{\bf (Case 2)  Both $h$ and  $h'$ are odd.}

\begin{prop}\label{Tau periodicity 4}
$X\otimes{\tau'}^{-\frac{h+h'}{2}}(Y)\simeq r(X)\otimes r'{}^{}(Y)$ in $\mathcal{C}_A$ for any $X\in\mathcal{D}_Q$ and $Y\in\mathcal{D}_{Q'}$.
\end{prop}

\begin{proof}
We have
$X\otimes{\tau'}^{-\frac{h+h'}{2}}(Y)
\simeq F^{-\frac{h-1}{2}}(X\otimes{\tau'}^{-\frac{h+h'}{2}}(Y))
\stackrel{{\rm Lem. \ref{F}}}{\simeq}\tau^{\frac{h-1}{2}}(X)\otimes
{\tau'}^{-1}
{\tau'}^{-\frac{h'-1}{2}}(Y)
\stackrel{{\rm Prop.\ref{r periodicity}(1)}}{\simeq}r(X[-1])\otimes 
{\tau'}^{-1}{r'}^{-1}(Y[1])\simeq r(X)\otimes r'{}^{}(Y)$.
\end{proof}

\begin{thm}[Theorem \ref{thm:cluster3} (2)]
$[\mu^{\frac{h+h'}{2}}_\otimes(Q\otimes Q',z(0))]_{a,b}
=\overline{z}_{\omega(a),\omega'(b)}(0)$.
\end{thm}

\begin{proof}
We have $[\mu^{\frac{h+h'}{2}}_\otimes(P\otimes P')]_{a,b}\stackrel{{\rm Prop. \ref{mu and tau1}(2)}}{=}
P_a\otimes {\tau'}^{-\frac{h+h'}{2}}(P'_b)\stackrel{{\rm Prop. \ref{Tau periodicity 4}}}{=}r(P_a)\otimes r'(P'_b)=
U_{\omega(a)}\otimes U'_{\omega'(b)}$.
Applying $\widetilde{X}$, we have
$[\mu^{\frac{h+h'}{2}}_\otimes(Q\otimes Q',z(0))]_{a,b}\stackrel{{\rm Thm. \ref{X map1}(1)}}{=}
[\mu^{\frac{h+h'}{2}}_\otimes(\widetilde{X}_{P\otimes P'})]_{a,b}
\stackrel{{\rm Thm. \ref{X map1}(2)}}{=}X_{[\mu^{\frac{h+h'}{2}}_\otimes(P\otimes P')]_{a,b}}=
X_{U_{\omega(a)}\otimes U'_{\omega'(b)}}=\overline{z}_{\omega(a),\omega'(b)}(0)$.
\end{proof}

\medskip\noindent
{\bf (Case 3) $h$ is even, and  $h'$ is odd.}

\begin{prop}\label{Tau periodicity 3}
$X\otimes{\tau'}^{-\frac{h+h'-1}{2}}(Y)\simeq\omega(X)\otimes {r'}^{-1}(Y)$ in $\mathcal{C}_A$ for any $X\in\mathcal{D}_Q$ and $Y\in\mathcal{D}_{Q'}$.
\end{prop}

\begin{proof}
We have
$X\otimes{\tau'}^{-\frac{h+h'-1}{2}}(Y)\simeq F^{-\frac{h}{2}}(X\otimes{\tau'}^{-\frac{h+h'-1}{2}}(Y))
\stackrel{{\rm Lem. \ref{F}}}{\simeq}\tau^{\frac{h}{2}}(X)\otimes{\tau'}^{-\frac{h'-1}{2}}(Y)
\stackrel{{\rm Prop.\ref{tau periodicity 2}(1),
\ref{r periodicity}(1)}}{\simeq}\omega^{-1}(X[-1])\otimes {r'}^{-1}(Y[1])
\simeq\omega(X)\otimes {r'}^{-1}(Y)$.
\end{proof}

\begin{thm}[Theorem \ref{thm:cluster3} (3)]
$[\mu_{+}\mu^{\frac{h+h'-1}{2}}_\otimes(Q\otimes Q',z(0))]_{a,b}=z_{\omega(a),\omega'(b)}(0)$.
\end{thm}

\begin{proof}
We have $[\mu_{+}\mu^{\frac{h+h'-1}{2}}_\otimes(P\otimes P')]_{a,b}\stackrel{{\rm Prop. \ref{mu and tau1}(2)}}{=}
P_a\otimes {\tau'}^{-\frac{h+h'-1}{2}}(U'_b)\stackrel{{\rm Prop. \ref{Tau periodicity 3}}}{=}P_{\omega(a)}\otimes {r'}^{-1}(U'_{b})=P_{\omega(a)}\otimes P'_{\omega'(b)}$.
Applying $\widetilde{X}$, we have
$[\mu_{+}\mu^{\frac{h+h'-1}{2}}_\otimes(Q\otimes Q',z(0))]_{a,b}
\allowbreak
\stackrel{{\rm Thm. \ref{X map1}(1)}}{=}
[\mu_{+}\mu^{\frac{h+h'-1}{2}}_\otimes(\widetilde{X}_{P\otimes P'})]_{a,b}
\stackrel{{\rm Thm. \ref{X map1}(2)}}{=}X_{[\mu_{+}\mu^{\frac{h+h'-1}{2}}_\otimes(P\otimes P')]_{a,b}}=
X_{P_{\omega(a)}\otimes P'_{\omega'(b)}}
\allowbreak
=z_{\omega(a),\omega'(b)}(0)$.
\end{proof}

\medskip\noindent
{\bf (Case 4)   $h$ is odd, and  $h'$ is even.}

\begin{prop}\label{Tau periodicity 5}
$X\otimes{\tau'}^{-\frac{h+h'-1}{2}}(Y)\simeq r(X)\otimes\omega'(Y)$
 in $\mathcal{C}_A$ for any $X\in\mathcal{D}_Q$ and $Y\in\mathcal{D}_{Q'}$.
\end{prop}

\begin{proof}
We have
$X\otimes{\tau'}^{-\frac{h+h'-1}{2}}(Y)\simeq F^{-\frac{h-1}{2}}
(X\otimes{\tau'}^{-\frac{h+h'-1}{2}}(Y))
\stackrel{{\rm Lem. \ref{F}}}{\simeq}\tau^{\frac{h-1}{2}}(X)
\otimes{\tau'}^{-\frac{h'}{2}}(Y)
\stackrel{{\rm Prop.\ref{tau periodicity 2}
(1), \ref{r periodicity}(1)}}{\simeq}r(X[-1])\otimes\omega'(Y[1])
\simeq r(X)\otimes\omega'(Y)$.
\end{proof}

\begin{thm}[Theorem \ref{thm:cluster3} (4)]
 $[\mu_{+}\mu^{\frac{h+h'-1}{2}}_\otimes(Q\otimes Q',z(0))]_{a,b}=
\overline{z}_{\omega(a),\omega'(b)}(0)$.
\end{thm}

\begin{proof}
We have $[\mu_{+}\mu^{\frac{h+h'-1}{2}}_\otimes(P\otimes P')]_{a,b}
\stackrel{{\rm Prop. \ref{mu and tau1}(2)}}{=}
P_a\otimes {\tau'}^{-\frac{h+h'-1}{2}}(U'_b)
\stackrel{{\rm Prop. \ref{Tau periodicity 5}}}{=}r(P_a)\otimes
\omega(U'_{b})=U_{\omega(a)}\otimes U'_{\omega'(b)}$.
Applying $\widetilde{X}$, we have
$[\mu_{+}\mu^{\frac{h+h'-1}{2}}_\otimes(Q\otimes Q',z(0))]_{a,b}
\allowbreak
\stackrel{{\rm Thm. \ref{X map1}(1)}}{=}
[\mu_{+}\mu^{\frac{h+h'-1}{2}}_\otimes(\widetilde{X}_{P\otimes P'})]_{a,b}
\stackrel{{\rm Thm. \ref{X map1}(2)}}{=}
X_{[\mu_{+}\mu^{\frac{h+h'-1}{2}}_\otimes(P\otimes P')]_{a,b}}=
X_{U_{\omega(a)}\otimes U'_{\omega'(b)}}
\allowbreak
=\overline{z}_{\omega(a),\omega'(b)}(0)$.
\end{proof}

\section{Determinant method I: $\EuScript{T}_{\ell}(A_r)$
}

\label{sect:TA}

Volkov \cite{V} proved
the periodicity of
 the regular
solutions of $\mathbb{Y}_\ell
(A_r)$ in  $\mathbb{C}$ for any $\ell\geq 2$
 by constructing
the manifestly periodic
determinant expression.
In the process of the proof,
he essentially proved the periodicity
of the regular
solutions  of $\mathbb{T}_\ell
(A_r)$ in  $\mathbb{C}$ as well.
In this section, we prove the periodicities
of $\EuScript{T}_\ell(A_r)$ and $\EuScript{Y}_\ell(A_r)$
for any $\ell \geq 2$
by reformulating the determinant method in our setting
to avoid
the projective geometrical arguments used in \cite{V},
with the application to the $C_r$ case also in mind.

\subsection{Level $\ell$ restricted T-system with spiral boundary condition}

Following \cite{V}, we introduce the level $\ell$ restricted T-system
of type $A_r$ with a more general
boundary condition than the
unit boundary condition for $\mathbb{T}_{\ell}
(A_r)$ 
in Definition \ref{defn:RT};
 we call it the {\em spiral boundary condition}.

Let us set
\begin{align}
\begin{split}
H&=\{ (a,m,u)
\mid
a=0,\dots,r+1 ; m=0,\dots, \ell; u\in \mathbb{Z}
\},
\\
H^{\circ}&=\{ (a,m,u)\in H
\mid
a\neq 0,r+1; m\neq0,\ell
\},
\\
\partial H &=H \setminus H^{\circ},\\
H_e&=\{ (a,m,u)\in H
\mid
\mbox{$a+m+u$ is even} \},\\
H_o&=\{ (a,m,u)\in H
\mid
\mbox{$a+m+u$ is odd} \}.
\end{split}
\end{align}
We also use the combined notations,
 $H^{\circ}_o = H^{\circ} \cap H_o$,
$\partial H_e = \partial H \cap H_e$, etc.

\begin{defn}
Fix an integer $\ell \geq 2$.
The {\it level $\ell$ restricted T-system $\widetilde{\mathbb{T}}_{\ell}(A_r)$
of type $A_r$
with the spiral boundary condition}
is the following system of relations for
a family of variables $T=\{T^{(a)}_m(u)
\mid
(a,m,u)\in H
\}$:

(1) T-system:
\begin{align}
\label{eq:TA2}
T^{(a)}_m(u-1)T^{(a)}_m(u+1)
&=
T^{(a)}_{m-1}(u)T^{(a)}_{m+1}(u)
+
T^{(a-1)}_{m}(u)T^{(a+1)}_{m}(u)\\
&
\hskip100pt ((a,m,u)\in H^{\circ}),\notag
\end{align}

(2) Spiral boundary condition:
\begin{align}
\label{eq:spc11}
\begin{split}
T^{(a)}_0(u+1)&=T^{(a-1)}_0(u)\quad (a=1,\dots, r+1),\\
T^{(a)}_{\ell}(u+1)&=T^{(a+1)}_{\ell}(u)\quad (a=0,\dots, r),\\
T^{(0)}_{m}(u+1)&=T^{(0)}_{m+1}(u) \quad (m=0,\dots, \ell-1),\\
T^{(r+1)}_{m}(u+1)&=T^{(r+1)}_{m-1}(u) \quad (m=1,\dots, \ell).
\end{split}
\end{align}
\end{defn}

\begin{defn}
Let 
 $\widetilde{\EuScript{T}}_{\ell}(A_r)$
be the ring with generators
$T^{(a)}_m(u)^{\pm 1}$ ($(a,m,u)\in H$)
and the relations $\widetilde{\mathbb{T}}_{\ell}(A_r)$.
\end{defn}

The condition (\ref{eq:spc11})
means that
$T^{(a)}_m(u)$ is {\em constant along spirals\/}
 on the boundary $\partial H$.
In particular, if we impose $T^{(a)}_m(u)=1$ for any $(a,m,u)\in \partial H$,
then $\widetilde{\mathbb{T}}_{\ell}(A_r)$ reduces
 to ${\mathbb{T}}_{\ell}(A_r)$.
In other words,
$\EuScript{T}_{\ell}(A_r)$
is isomorphic to
$\widetilde{\EuScript{T}}_{\ell}(A_r)/J$,
where $J$ is the ideal of $\widetilde{\EuScript{T}}_{\ell}(A_r)$
generated by
$T^{(a)}_m(u)- 1$ ($(a,m,u)\in \partial H$).

Recall that the dual Coxeter number of $A_r$ is $r+1$.
We will prove

\begin{thm}[{Henriques \cite[Theorem 5]{Hen}}]
\label{thm:TAperiod}
The following relations hold in
 $\widetilde{\EuScript{T}}_{\ell}(A_r)$:
\par
(1) Half periodicity:
$T^{(a)}_m(u+r+1+\ell)=
T^{(r+1-a)}_{\ell-m}(u)$.
\par
(2) Periodicity:
$T^{(a)}_m(u+2(r+1+\ell))=
T^{(a)}_{m}(u)$.
\end{thm}

By the above remark, we obtain
\begin{cor}
\label{cor:TAperiod}
The following relations hold in
 $\EuScript{T}_{\ell}(A_r)$:
\par
(1) Half periodicity:
$T^{(a)}_m(u+r+1+\ell)=
T^{(r+1-a)}_{\ell-m}(u)$.
\par
(2) Periodicity:
$T^{(a)}_m(u+2(r+1+\ell))=
T^{(a)}_{m}(u)$.
\end{cor}

\begin{rem}
Theorem \ref{thm:TAperiod} (1) is
a corollary of a more general theorem by Henriques \cite[Theorem 5]{Hen},
since the spiral boundary condition ensures that $c=1$
for Eq.\ (11) of \cite{Hen}.
However, as we mentioned earlier,
we prove Theorem \ref{thm:TAperiod} by the determinant method of \cite{V}
to have the application to the $C_r$ case in mind.
\end{rem}

\subsection{Proof of Theorem \ref{thm:TAperiod}}

Let $R$ be any ring.
Let us take an arbitrary
 $\ell \times \infty$ 
matrix $M$ over $R$
such that  $M=[ x_k]_{k\in \mathbb{Z}}$,
  $x_k \in R^{\ell}$ with the following periodicity:
\begin{align}
\label{eq:xpc1}
x_{k+r+1+\ell}=(-1)^{\ell-1}x_k.
\end{align}
Let $D_M=\{ D^{(a)}_m(u)\mid (a,m,u)\in H_e\}$ be a family of
minors of $M$ defined by
\begin{align}
D^{(a)}_m(u)&=\det [
\overbrace{x_\beta x_{\beta+1}\ \cdots\ }^{m}
\overbrace{\operatornamewithlimits{\phantom{x}}^{\vee} \ \cdots\
\operatornamewithlimits{\phantom{x}}^{\vee}}^{a}
\overbrace{\phantom{x}\ \cdots \quad}^{\ell-m}
\overbrace{\operatornamewithlimits{\phantom{x} }^\vee \ \cdots \
{\operatornamewithlimits{\mathit{x}}^{\vee}}_{\beta+r+\ell}}^{r+1-a}],\\
\beta&= -\frac{a+m+u}{2}\in \mathbb{Z},
\end{align}
where ${\displaystyle\operatornamewithlimits{\mathit{x}}^{\vee}}_k$
 means the omission of $x_k$ as usual.

\begin{prop}
\label{prop:DM1}
The family $D_M=\{ D^{(a)}_m(u)\mid (a,m,u)\in H_e\}$ satisfies
the following relations in the ring $R$:

(1) T-system:
\begin{align}
\label{eq:D1}
D^{(a)}_m(u-1)D^{(a)}_m(u+1)
&=
D^{(a)}_{m-1}(u)D^{(a)}_{m+1}(u)
+
D^{(a-1)}_{m}(u)D^{(a+1)}_{m}(u)\\
&
\hskip100pt ((a,m,u)\in H^{\circ}_o),\notag
\end{align}

(2) Spiral boundary condition:
\begin{align}
\label{eq:spc1}
\begin{split}
  D^{(a)}_0(u+1)&=D^{(a-1)}_0(u)\quad (a=1,\dots, r+1),\\
  D^{(a)}_{\ell}(u+1)&=D^{(a+1)}_{\ell}(u)\quad (a=0,\dots, r),\\
  D^{(0)}_{m}(u+1)&=D^{(0)}_{m+1}(u) \quad (m=0,\dots, \ell-1),\\ 
 D^{(r+1)}_{m}(u+1)&=D^{(r+1)}_{m-1}(u) \quad (m=1,\dots, \ell).
\end{split}
\end{align}

(3) Half-periodicity:
\begin{align}
D^{(a)}_m(u+r+1+\ell)=
D^{(r+1-a)}_{\ell-m}(u).
\end{align}
\end{prop}
\begin{proof}
(1) They are the Pl\"ucker relations among minors.

(2) The first three relations follow from the definition
of $D^{(a)}_m(u)$.
To show the last relation, we also use the (anti-)periodicity
(\ref{eq:xpc1}) of $x_k$.

(3) It also follows from  the definition
of $D^{(a)}_m(u)$ and (\ref{eq:xpc1}).
\end{proof}

\begin{rem}
This determinant solution of the T-system
(\ref{eq:D1}) is regarded as
the restricted version of Eq.\ (2.25) of \cite{KLWZ}.
\end{rem}

Observe that,
if we divide the family of generators $T=\{T^{(a)}_m(u)\mid (a,m,u)\in H\}$
of $\widetilde{\EuScript{T}}_{\ell}(A_r)$
into two subfamilies,
$T_e=\{T^{(a)}_m(u) \mid
 (a,m,u)\in H_e\}$ and $T_o=\{T^{(a)}_m(u)\mid (a,m,u)\in H_o\}$,
then
$T_e$ and $T_o$ has no mutual relation
 in $\widetilde{\mathbb{T}}_{\ell}(X_r)$.
Therefore, to prove Theorem \ref{thm:TAperiod},
it is enough to consider the half family $T_e$ of $T$
(cf.\ (\ref{eq:Tfact2})).
Then, to prove Theorem \ref{thm:TAperiod},
we have only to show the following:

\begin{prop}
\label{prop:Deta1}
There exists some  $\ell \times \infty$ matrix
$M=[ x_k]_{k\in \mathbb{Z}}$ over $\widetilde{\EuScript{T}}_{\ell}(A_r)$
satisfying
the condition (\ref{eq:xpc1}) such that,
for $D_M=\{ D^{(a)}_m(u)\mid (a,m,u)\in H_e\}$,
the following relation holds in $\widetilde{\EuScript{T}}_{\ell}(A_r)$:
\begin{align}
\label{eq:DT4}
T^{(a)}_m(u)=D^{(a)}_m(u)\quad  ((a,m,u)\in H_e).
\end{align}
\end{prop}
\begin{proof}
We define $x_0,\dots,x_{r+\ell}\in
(\widetilde{\EuScript{T}}_{\ell}(A_r))^{\ell}$ as follows:
Firstly, let us arbitrarily choose $x_1,\dots,x_{\ell-1}$ such that
\begin{align}
\label{eq:DT2}
D^{(0)}_0(0)=\det [ x_0\cdots x_{\ell-1}]
=T^{(0)}_0(0)= T^{(1)}_0(1)
\end{align}
holds. We define $x_{\ell}$ by
\begin{align}
x_{\ell}=
\frac{1}{T^{(1)}_0(1)}
\sum_{m=0}^{\ell-1}(-1)^{\ell-1-m}
T^{(1)}_m(-1-m)x_m.
\end{align}
Then, the following equality holds:
\begin{align}
\label{eq:DT1}
D^{(1)}_m(-1-m)=T^{(1)}_m(-1-m)
\quad (m=0,\dots,\ell-1).
\end{align}
For example,
\begin{align}
\begin{split}
D^{(1)}_0(-1)=\ &\det [ x_1\cdots x_{\ell}]\\
=\ &
\frac{(-1)^{\ell-1}T^{(1)}_0(-1)}{T^{(1)}_0(1)}
\det [ x_1\cdots x_{\ell-1}x_{0}]
=T^{(1)}_0(-1).
\end{split}
\end{align}
Similarly, we recursively define the rest, $x_{\ell+1},\dots, x_{r+\ell}$,
by
\begin{align}
x_{\ell+j}&=
\frac{1}{T^{(1)}_0(1-2j)}
\sum_{m=0}^{\ell-1}(-1)^{\ell-1-m}
T^{(1)}_m(-1-2j-m)x_{j+m}\quad (j=1,\dots,r)
\end{align}
so that the following equality holds (including (\ref{eq:DT1})
as $j=0$):
\begin{align}
\label{eq:DT3}
D^{(1)}_m(-1-2j-m)=T^{(1)}_m(-1-2j-m)
\quad (m=0,\dots,\ell-1; j=0,\dots,r).
\end{align}
Finally, we define the matrix $M=[ x_k ]_{k\in \mathbb{Z}}$
by extending the above $x_0$, \dots, $x_{r+\ell}$ with
 (\ref{eq:xpc1}).

For $D_M$, we claim that
the relation (\ref{eq:DT4}) holds in
$\widetilde{\EuScript{T}}_{\ell}(A_r)$.
This will be shown inductively, based
on the fact that the T-system and
the spiral boundary condition
are satisfied by both $T_e$ and $D_M$.

To proceed the induction,
it is convenient to introduce a prism $P$,
where
\begin{align}
\label{eq:prism}
P=\{ (a,m,u)\in H
\mid a+m+u\leq 0,\
a-m-u \leq 2r+2
\}.
\end{align}
See Figure\ \ref{fig:P}.
We use the notations, $P_e=P\cap H_e$, $P_e[a=1]=\{ (a,m,u)\in P_e \mid a=1\}$, etc.
\begin{figure}
\begin{picture}(100,260)(-37,-235)
\drawline(0,10)(0,0)
\drawline(0,-150)(0,-215)
\dashline{3}(0,0)(0,-150)
\drawline(60,-10)(60,-225)
\drawline(20,-30)(20,-245)
\drawline(-40,-20)(-40,-235)
\drawline(0,0)(84,-14)
\drawline(0,0)(-60,-30)
\drawline(60,-10)(20,-30)
\drawline(-40,-20)(20,-30)
\drawline(0,0)(60,-70)
\drawline(0,0)(-40,-80)
\drawline(-40,-80)(20,-150)
\drawline(60,-70)(20,-150)
\drawline(-40,-80)(0,-150)
\drawline(0,-150)(60,-220)
\drawline(20,-150)(60,-220)
\put(-9,2){$0$}
\put(-3,12){$u$}
\put(90,-17){$m$}
\put(60,-4){$\ell$}
\put(-70,-35){$a$}
\put(-68,-15){$r+1$}
\dottedline{3}(0,-60)(60,-70)
\put(-17,-60){{ $-\ell$}}
\put(-41,-154){{ $-2r-2$}}
\dottedline{3}(0,-210)(60,-220)
\put(-58.5,-212){{$-2r-2-\ell$}}
\end{picture}
\caption{The prism $P$ defined in (\ref{eq:prism}).}
\label{fig:P}
\end{figure}
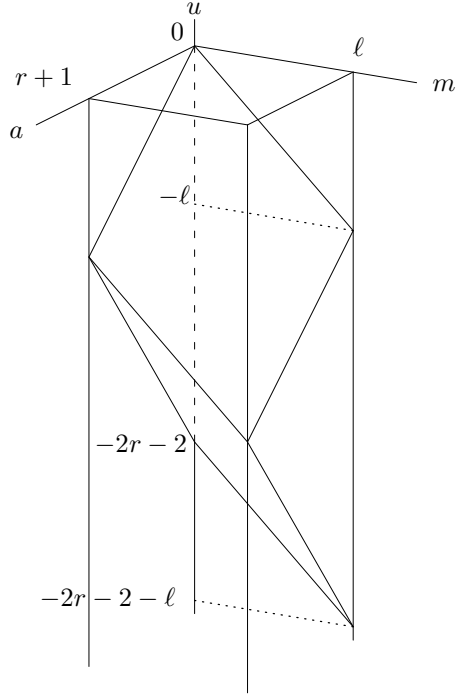

First, we  show that (\ref{eq:DT4}) is true for $(a,m,u)\in P_e$
by the induction on $a$.
By (\ref{eq:DT2}), (\ref{eq:DT3}), and the spiral boundary condition,
we see that (\ref{eq:DT4}) is true for any $(a,m,u)$ in
the set
\begin{align}
P_e[a=0] \cup P_e[a=1] \cup P_e[m=0] \cup P_e[m=\ell].
\end{align}
Assume that (\ref{eq:DT4}) is true up to $a$.
By  (\ref{eq:TA2}), we have
\begin{align}
T^{(a+1)}_{m}(u)
&=
\frac{1}{T^{(a-1)}_{m}(u)}
\left(T^{(a)}_m(u-1)T^{(a)}_m(u+1)
-T^{(a)}_{m-1}(u)T^{(a)}_{m+1}(u)\right).
\end{align}
On the other hand,
by (\ref{eq:D1}) and the induction hypothesis,
we have
\begin{align}
D^{(a+1)}_{m}(u)
&=
\frac{1}{T^{(a-1)}_{m}(u)}
\left(T^{(a)}_m(u-1)T^{(a)}_m(u+1)
-T^{(a)}_{m-1}(u)T^{(a)}_{m+1}(u)\right).
\end{align}
Thus, the relation
$D^{(a+1)}_{m}(u)=
T^{(a+1)}_{m}(u)$ is obtained.
Therefore, (\ref{eq:DT4}) is true  for  $(a,m,u)\in P_e$.

Next,  we show that (\ref{eq:DT4}) is true
 for any $(a,m,u)\in H_e\setminus P_e$.
We first remark that, by the spiral boundary condition,
(\ref{eq:DT4}) is now true for any $(a,m,u)\in \partial H_e$.
Then, using the T-system (\ref{eq:TA2}) once again as
\begin{align}
T^{(a)}_{m}(u\pm 1)
&=
\frac{1}{T^{(a)}_{m}(u\mp 1)}
\left(
T^{(a)}_{m-1}(u)T^{(a)}_{m+1}(u)
+
T^{(a-1)}_m(u)T^{(a+1)}_m(u)
\right),
\end{align}
and repeating the same argument as before,
one can inductively, with respect to $\pm u$,
 conclude that (\ref{eq:DT4}) is true for any $(a,m,u)\in 
H_e\setminus P_e$.
\end{proof}

This completes the proof of Theorem \ref{thm:TAperiod}.

\subsection{Periodicity of Y-system}
The periodicity of $\EuScript{Y}_{\ell}(A_r)$
follows from that of $\widetilde{\EuScript{T}}_{\ell}(A_r)$.
This is due to the following fact:
Unlike the unit boundary condition,
the spiral boundary condition is
{\em fully\/} compatible with the restriction of the Y-system
in view of Theorem \ref{thm:TtoY1} (cf.\ Proposition
\ref{prop:TtoY2}); 
namely, 
\begin{prop}
\label{prop:TtoY4}
(1) There is a ring homomorphism
\begin{align}
\label{eq:TtoY6}
\widetilde{\varphi}_{\ell}: \EuScript{Y}_{\ell}(A_r) \rightarrow
\widetilde{\EuScript{T}}_{\ell}(A_r)
\end{align}
defined by 
\begin{align}
\label{eq:TtoY7}
Y^{(a)}_m(u)\mapsto
\frac{T^{(a-1)}_{m}(u)T^{(a+1)}_{m}(u)}
{T^{(a)}_{m-1}(u)T^{(a)}_{m+1}(u)}
\quad
((a,m,u)\in H^{\circ}).
\end{align}
\par
(2) There is a ring homomorphism
\begin{align}
\psi_{\ell}:
\widetilde{\EuScript{T}}_{\ell}(A_r)
\rightarrow
\EuScript{Y}_{\ell}(A_r)
\end{align}
such that $\psi_{\ell}\circ\widetilde{\varphi}_{\ell}
 = \mathrm{id}_{\EuScript{Y}_{\ell}(X_r)}$.
\par
\end{prop}
\begin{proof}
(1)  It is enough to check the compatibility between the boundary conditions
of $\widetilde{\mathbb{T}}_{\ell}(A_r)$ and $\mathbb{Y}_{\ell}(A_r)$
as Theorem \ref{thm:TtoY1} and Proposition \ref{prop:TtoY2}.
For example, to see the compatibility with $Y^{(a)}_{0}(u)^{-1}=0$ ($a\neq 0$),
 we formally extend  (\ref{eq:TA2}) to
$m=0$ as
\begin{align}
T^{(a)}_0(u-1)T^{(a)}_0(u+1)
&=
T^{(a)}_{-1}(u)T^{(a)}_{1}(u)
+
T^{(a-1)}_{0}(u)T^{(a+1)}_{0}(u)
\notag
\end{align}
and use (\ref{eq:spc11}). Then we have $T^{(a)}_{-1}(u)=0$.
The other cases are checked  similarly.
\par
(2)
We define the image $\psi_{\ell}(T^{(a)}_m(u))$
for the half family $T_e$ of $T$.
The other half $T_o$ is completely parallel.
The construction is in four steps and
 similar to the one for Theorem
\ref{thm:TtoY1}.
For simplicity, let us write the image $\psi_{\ell}(T^{(a)}_m(u))$ as
$T^{(a)}_m(u)$.
Let $P$ be the prism defined in (\ref{eq:prism}).

Step 1. We arbitrarily choose $T^{(0)}_0(-2j)\in
 \EuScript{Y}_{\ell}(A_r)^{\times}$
 ($j=0,1,\dots,r+1$).
Then, we define $T^{(a)}_m(u)$ for the rest of  $(a,m,u)\in P_e[a=0]\cup
P_e[m=0]\cup P_e[m=\ell]$
by (\ref{eq:spc1}).

Step 2. We arbitrarily choose $T^{(1)}_m(-1-m)\in
 \EuScript{Y}_{\ell}(A_r)^{\times}$
 ($m=1,\dots,\ell-1$).
Then, we recursively, with respect to $u$,
define $T^{(a)}_m(u)$ for the rest of $(a,m,u)\in P_e[a=1]$
by
\begin{align}
\label{eq:ty3}
T^{(1)}_{m}(u-1) =
\left(1+Y^{(1)}_m(u)\right)
\frac{T^{(1)}_{m-1}(u)T^{(1)}_{m+1}(u)}
{T^{(1)}_{m}(u+1)}.
\end{align}

Step 3.
We recursively, with respect to $a$,
define $T^{(a)}_m(u)$ for the rest of $(a,m,u)\in P_e$
by
\begin{align}
\label{eq:ty1}
T^{(a+1)}_{m}(u) =
\frac{1}{1+Y^{(a)}_m(u)^{-1}}
\frac{T^{(a)}_m(u-1)T^{(a)}_m(u+1)}
{T^{(a-1)}_{m}(u)}.
\end{align}

Step 4. We define $T^{(a)}_m(u)$ for the rest of $(a,m,u)\in \partial H_e$
by (\ref{eq:spc1}).
Then, we recursively, with respect to $\pm u$,
define $T^{(a)}_m(u)$ for the rest of $(a,m,u)\in H_e$
by
\begin{align}
\label{eq:ty2}
T^{(a)}_{m}(u\pm 1) =
\left(1+Y^{(a)}_m(u)^{-1}\right)
\frac{T^{(a-1)}_m(u)T^{(a+1)}_m(u)}
{T^{(a)}_{m}(u\mp 1)}.
\end{align}

\begin{claim}
 The family $\{ T^{(a)}_m(u)
\mid (a,m,u)\in H_e \}$ defined above
satisfies the following relations in $\EuScript{Y}_{\ell}(A_r)$:
\begin{align}
\label{eq:yt2}
1+Y^{(a)}_m(u)&=
\frac
{T^{(a)}_{m}(u-1)T^{(a)}_{m}(u+1)}
{T^{(a)}_{m-1}(u)T^{(a)}_{m+1}(u)}
\quad((a,m,u)\in H^{\circ}_o),\\
\label{eq:yt1}
1+Y^{(a)}_m(u)^{-1}&=
\frac
{T^{(a)}_{m}(u-1)T^{(a)}_{m}(u+1)}
{T^{(a-1)}_m(u)T^{(a+1)}_m(u)}
\quad((a,m,u)\in H^{\circ}_o).
\end{align}
\end{claim}

The relation (\ref{eq:yt1}) clearly holds by the definition
of  $T^{(a)}_m(u)$ in (\ref{eq:ty1}) and (\ref{eq:ty2}).
With (\ref{eq:ty1}),
one can verify that
(\ref{eq:yt2}) is true for 
any $(a,m,u)\in P^{\circ}_o:=P\cap H^{\circ}_o$ by the induction with
 respect to $a$.
Then, in a similar way, with (\ref{eq:ty2}),
one can prove that
(\ref{eq:yt2}) is true for the rest of  $(a,m,u)\in H^{\circ}_o$
by the induction with respect to $u$.
This completes the proof of the claim.

Now, taking the inverse sum of
(\ref{eq:yt2}) and (\ref{eq:yt1}),
we obtain (\ref{eq:TA2}).
Therefore, $\psi_{\ell}$ is a ring homomorphism.
Furthermore,
taking the ratio of
(\ref{eq:yt2}) and (\ref{eq:yt1}),
we obtain
$Y^{(a)}_m(u)=
{T^{(a-1)}_m(u)T^{(a+1)}_m(u)}/
({T^{(a)}_{m-1}(u)T^{(a)}_{m+1}(u)}).
$
This proves $\psi_{\ell}\circ \widetilde{\varphi}_{\ell}
 = \mathrm{id}_{\EuScript{Y}_{\ell}(A_r)}$.
\end{proof}

By Theorem \ref{thm:TAperiod} and Proposition \ref{prop:TtoY4},
we obtain
\begin{cor}[{cf.\ Volkov  \cite[Theorem 1]{V}}]
\label{cor:YAperiod}
The following relations hold in
 $\EuScript{Y}_{\ell}(A_r)$:
\par
(1) Half periodicity:
$Y^{(a)}_m(u+r+1+\ell)=
Y^{(r+1-a)}_{\ell-m}(u)$.
\par
(2) Periodicity:
$Y^{(a)}_m(u+2(r+1+\ell))=
Y^{(a)}_{m}(u)$.
\end{cor}

By Proposition \ref{prop:TtoY4}, we also obtain
\begin{cor}[cf.\ {Volkov \cite[Proposition 1]{V}}]
For any ring $R$, the map
\begin{align}
\widetilde{\varphi}_{\ell}^*: \mathrm{Hom}\,
 (\widetilde{\EuScript{T}}_{\ell}(A_r),R)
\rightarrow 
\mathrm{Hom}\, (\EuScript{Y}_{\ell}(A_r),R),
\end{align}
induced from the homomorphism $\widetilde{\varphi}_{\ell}$
 in (\ref{eq:TtoY6}), is surjective.
\end{cor}

\section{Determinant method II: $\EuScript{T}_{\ell}(C_r)$
}

\label{sect:TC}

In this section we prove
the periodicity of $\EuScript{T}_{\ell}(C_r)$
for any $\ell\geq 2$.
We do it in three steps.
Firstly, we introduce a ring
$\widetilde{\EuScript{T}}_{\ell}(C_r)$
by slightly generalizing the unit boundary condition
of $\EuScript{T}_{\ell}(C_r)$.
Secondly, we show $\widetilde{\EuScript{T}}_{\ell}(C_r)$
is isomorphic to another ring
$\widehat{\EuScript{T}}_{2\ell}(A_{2r+1})$
which is a variant of $\EuScript{T}_{2\ell}(A_{2r+1})$.
Lastly, we apply the determinant method to 
 $\widehat{\EuScript{T}}_{2\ell}(A_{2r+1})$.

\subsection{Level $\ell$ restricted T-system with quasi-unit
boundary condition}
For $X_r=C_r$,
let $t_a$ ($a\in I$) be the number in (\ref{eq:t1}),
i.e., $t_a=2$ for $a=1,\dots,r-1$ and 1 for $a=r$.

We find that it is convenient to generalize the problem slightly
as follows.
\begin{defn}
Fix an integer $\ell \geq 2$.
The {\it level $\ell$ restricted T-system 
$\widetilde{\mathbb{T}}_{\ell}(C_r)$ of type $C_r$
with the quasi-unit boundary condition}
is the following system of relations
for
a family of variables $T=\{T^{(a)}_m(u)
\mid
a=1,\dots,r ; m=1,\dots, t_a\ell-1; u\in \textstyle\frac{1}{2}\mathbb{Z}
\} \cup
\{ T^{(r)}_{\ell}(u) \mid 
u\in \textstyle\frac{1}{2}\mathbb{Z}
\}$.

(1) T-system  (Eq.\ (\ref{eq:TC1})):

\begin{align}
\label{eq:TA31}
T^{(a)}_m\left(u-\textstyle\frac{1}{2}\right)
T^{(a)}_m\left(u+\textstyle\frac{1}{2}\right)
&=
T^{(a)}_{m-1}(u)T^{(a)}_{m+1}(u)
+T^{(a-1)}_{m}(u)T^{(a+1)}_{m}(u)\\
&\qquad
 (a=1,\dots,r-2;
\ m=1,\dots,2\ell-1;
\ u \in \textstyle\frac{1}{2}\mathbb{Z}),\notag\\
\label{eq:TA32}
T^{(r-1)}_{2m}\left(u-\textstyle\frac{1}{2}\right)
T^{(r-1)}_{2m}\left(u+\textstyle\frac{1}{2}\right)
&=
T^{(r-1)}_{2m-1}(u)T^{(r-1)}_{2m+1}(u)\\
&\qquad
+
T^{(r-2)}_{2m}(u)
T^{(r)}_{m}\left(u-\textstyle\frac{1}{2}\right)
T^{(r)}_{m}\left(u+\textstyle\frac{1}{2}\right)
\notag\\
&\qquad
 (m=1,\dots,\ell-1;
\ u \in \textstyle\frac{1}{2}\mathbb{Z}),\notag\\
\label{eq:TA33}
T^{(r-1)}_{2m+1}\left(u-\textstyle\frac{1}{2}\right)
T^{(r-1)}_{2m+1}\left(u+\textstyle\frac{1}{2}\right)
&=
T^{(r-1)}_{2m}(u)T^{(r-1)}_{2m+2}(u)
+
T^{(r-2)}_{2m+1}(u)
T^{(r)}_{m}(u)T^{(r)}_{m+1}(u)\\
&\qquad
 (m=0,\dots,\ell-1;
\ u \in \textstyle\frac{1}{2}\mathbb{Z}),\notag\\
\label{eq:TA34}
T^{(r)}_{m}(u-1)
T^{(r)}_{m}(u+1)
&=
T^{(r)}_{m-1}(u)T^{(r)}_{m+1}(u)
+
T^{(r-1)}_{2m}(u)\\
&\qquad
 (m=1,\dots,\ell-1;
\ u \in \textstyle\frac{1}{2}\mathbb{Z}).\notag
\end{align}

(2) Quasi-unit boundary condition:
\begin{align}
\label{eq:TA36}
T^{(0)}_m(u)&=T^{(a)}_0(u)=1,\\
\label{eq:TA35}
T^{(a)}_{2\ell}(u)&=1
\quad (a=1,\dots,r-1),
\end{align}
if they occur in the right hand sides of the relations;
and
\begin{align}
\label{eq:TA37}
T^{(r)}_{\ell}(u)^2&=1,\
\quad
T^{(r)}_{\ell}(u+1)=T^{(r)}_{\ell}(u).
\end{align}
\end{defn}

\begin{defn}
Let 
 $\widetilde{\EuScript{T}}_{\ell}(C_r)$
be the ring with generators
$T^{(a)}_m(u)^{\pm 1}$ 
($a=1,\dots,r ;\allowbreak m=1,\dots, t_a\ell-1;
 u\in \textstyle\frac{1}{2}\mathbb{Z}$),
$T^{(r)}_{\ell}(u)$
$(u\in \textstyle\frac{1}{2}\mathbb{Z})$
and the relations $\widetilde{\mathbb{T}}_{\ell}(C_r)$.
\end{defn}

In (\ref{eq:TA37}), if we impose $T^{(r)}_{\ell}(u)=1$
  ($u\in \frac{1}{2}\mathbb{Z}$),
then $\widetilde{\mathbb{T}}_{\ell}(C_r)$ reduces
 to ${\mathbb{T}}_{\ell}(C_r)$.
In other words,
$\EuScript{T}_{\ell}(C_r)$
is isomorphic to
$\widetilde{\EuScript{T}}_{\ell}(C_r)/J$,
where $J$ is the ideal of $\widetilde{\EuScript{T}}_{\ell}(C_r)$
generated by
$T^{(r)}_{\ell}(u)- 1$ ($u\in \frac{1}{2}\mathbb{Z}$).

Recall that the dual Coxeter number of $C_r$ is $r+1$.
We will prove
\begin{thm}
\label{thm:TCperiod}
The following relations hold in $\widetilde{\EuScript{T}}_{\ell}(C_r)$:
\par
(1) Half-periodicity:
$
T^{(a)}_m(u+r+1+\ell)=
\begin{cases}
T^{(a)}_{2\ell-m}(u) & a=1,\dots,r-1, \\
T^{(r)}_{\ell}(u) T^{(r)}_{\ell-m}(u) & a=r.
\end{cases}
$
\par
(2) Periodicity:
$T^{(a)}_m(u+2(r+1+\ell))=
T^{(a)}_{m}(u).$
\end{thm}

By the above remark, we obtain

\begin{cor}
\label{cor:TCperiod}
The following relations hold in $\EuScript{T}_{\ell}(C_r)$:
\par
(1) Half-periodicity:
$
T^{(a)}_m(u+r+1+\ell)=
T^{(a)}_{t_a\ell-m}(u).
$
\par
(2) Periodicity:
$T^{(a)}_m(u+2(r+1+\ell))=
T^{(a)}_{m}(u).$
\end{cor}

\subsection{System  $\widehat{\mathbb{T}}_{2\ell}(A_{2r+1})$.}

To prove Theorem \ref{thm:TCperiod},
we introduce another system of relations
{\em which is equivalent to\/}
 $\widetilde{\mathbb{T}}_{\ell}(C_r)$.

Let us set
\begin{align}
\label{eq:hh1}
\widehat{H}&=\{ (a,m,u)
\mid
a=0,\dots,2r+2 ; m=0,\dots, 2\ell; u\in 
\textstyle\frac{1}{2}
\mathbb{Z}
\},
\\
\widehat{H}^{\circ}&=\{ (a,m,u)\in \widehat{H}
\mid
a\neq 0,2r+2; m\neq 0,2\ell
\},
\\
\partial \widehat{H} &=\widehat{H} \setminus 
\widehat{H}^{\circ},\\
\widehat{H}_e&=\{ (a,m,u)\in \widehat{H}
\mid
\mbox{$a+m+2u$ is even} \},\\
\label{eq:hh2}
\widehat{H}_o&=\{ (a,m,u)\in \widehat{H}
\mid
\mbox{$a+m+2u$ is odd} \}.
\end{align}
We again use the combined notations,
 $\widehat{H}^{\circ}_o = \widehat{H}^{\circ} \cap \widehat{H}_o$,
$\partial \widehat{H}_e = \partial \widehat{H} \cap \widehat{H}_e$, etc.

\begin{defn}
Fix an integer $\ell \geq 2$.
The {\it level $2\ell$ restricted T-system 
$\widehat{\mathbb{T}}_{2\ell}(A_{2r+1})$ of type $A_{2r+1}$
with the quasi-symmetric  condition}
is the following system of relations for
a family of variables $S=\{S^{(a)}_m(u)
\mid
(a,m,u)\in \widehat{H}^{\circ}
\}$:

(1) T-system:
\begin{align}
\label{eq:TA3}
S^{(a)}_m(u-\textstyle\frac{1}{2})S^{(a)}_m(u+\textstyle\frac{1}{2})
&=
S^{(a)}_{m-1}(u)S^{(a)}_{m+1}(u)
+
S^{(a-1)}_{m}(u)S^{(a+1)}_{m}(u).
\end{align}

(2) Quasi-symmetric condition:
\begin{align}
\label{eq:psu1}
S^{(a)}_{m}(u) = (-1)^m S^{(2r+2-a)}_{m}(u).
\end{align}
In particular, $S^{(r+1)}_{m}(u)=0$ for odd $m$.

(3) Quasi-symmetric unit boundary condition:
\begin{align}
\label{eq:spc21}
S^{(a)}_0(u)&=S^{(a)}_{2\ell}(u)=
S^{(0)}_{m}(u)=1,\\
\label{eq:spc24}
S^{(2r+2)}_{m}(u)&=(-1)^m, 
\end{align}
if they occur in the right hand side of
the relations.
\end{defn}

\begin{defn}
Let 
 $\widehat{\EuScript{T}}_{2\ell}(A_{2r+1})$
be the ring with generators
$S^{(a)}_m(u)^{\pm 1}$ ($(a,m,u)\allowbreak\in \widehat{H}^{\circ},
(a,m)\neq (r+1,\text{odd})$)
and the relations $\widehat{\mathbb{T}}_{2\ell}(A_{2r+1})$.
\end{defn}

\begin{rem}
The system  $\widehat{\mathbb{T}}_{2\ell}(A_{2r+2})$
and the map $\rho$ 
below are the restricted versions of those considered
in \cite[Sect.\ 4]{KOSY}.
\end{rem}

\begin{prop}
\label{prop:bij1}
There is a ring isomorphism
\begin{align}
\label{eq:mapf}
\rho:
\widehat{\EuScript{T}}_{2\ell}(A_{2r+1})
\overset{\sim}{\rightarrow}
\widetilde{\EuScript{T}}_{\ell}(C_r)
\end{align}
defined  by
\begin{align}
\label{eq:st1}
S^{(a)}_m(u)&\mapsto T^{(a)}_m(u)
\quad (a=1,\dots,r-1),\\
\label{eq:ST2}
S^{(r)}_{2m}(u)&\mapsto T^{(r)}_m(u-\textstyle\frac{1}{2})
T^{(r)}_m(u+\textstyle\frac{1}{2}),\\
\label{eq:ST1}
S^{(r)}_{2m+1}(u)&\mapsto T^{(r)}_m(u)
T^{(r)}_{m+1}(u),\\
\label{eq:st3}
S^{(r+1)}_{2m}(u)&\mapsto T^{(r)}_m(u)^2,\\
\label{eq:st2}
\rho(S^{(a)}_{m}(u))&=(-1)^m \rho(S^{(2r+2-a)}_{m}(u))
\quad (a=r+2,\dots,2r+1),
\end{align}
where $T^{(r)}_0(u)=1$.
\end{prop}

\begin{proof}
It is easy to check  that the map
$\rho$ is a ring homomorphism by substitution
 (cf.\ \cite{KOSY}).
For simplicity, we write the image $\rho(S^{(a)}_m(u))$ as
$S^{(a)}_m(u)$.
Then, for example,
to show
\begin{align}
\label{eq:SS1}
S^{(r)}_{2m}(u-\textstyle\frac{1}{2})S^{(r)}_{2m}(u+\textstyle\frac{1}{2})
&=
S^{(r)}_{2m-1}(u)S^{(r)}_{2m+1}(u)
+
S^{(r-1)}_{2m}(u)S^{(r+1)}_{2m}(u),
\end{align}
multiply  (\ref{eq:TA34}) by $T^{(r)}_m(u)^2$,
then use (\ref{eq:ST2})--(\ref{eq:st3}).

Let us consider the inverse map $\rho^{-1}:
\widetilde{\EuScript{T}}_{\ell}(C_r)
\rightarrow\widehat{\EuScript{T}}_{2\ell}(A_{2r+1})$.
Looking at (\ref{eq:st1}) and (\ref{eq:ST1}),
it should be given by the following correspondence:
\begin{align}
\label{eq:TS1}
T^{(a)}_m(u)&\mapsto S^{(a)}_m(u)
\quad (a=1,\dots,r-1),\\
\label{eq:TS2}
T^{(r)}_{m}(u)&\mapsto
\begin{cases}
\displaystyle
\frac{S^{(r)}_{2m-1}(u)S^{(r)}_{2m-5}(u)\cdots S^{(r)}_{3}(u)}
{S^{(r)}_{2m-3}(u)S^{(r)}_{2m-7}(u)\cdots S^{(r)}_{1}(u)}
&  \mbox{($m$: even)},\\
\displaystyle
\frac{S^{(r)}_{2m-1}(u)S^{(r)}_{2m-5}(u)\cdots
\cdots  S^{(r)}_{1}(u)}
{S^{(r)}_{2m-3}(u)S^{(r)}_{2m-7}(u)\cdots S^{(r)}_{3}(u)}
& \mbox{($m$: odd)}.
\end{cases}
\end{align}
For simplicity, we write the image $\rho^{-1}(T^{(a)}_m(u))$ as
$T^{(a)}_m(u)$.

\begin{claim}
The family $T=\{ T^{(a)}_m(u)\}$ above satisfies the following relations in
\break 
$\widehat{\EuScript{T}}_{2\ell}(A_{2r+1})$:
\begin{align}
\label{eq:TS3}
 T^{(r)}_m (u-\textstyle\frac{1}{2})
T^{(r)}_m(u+\textstyle\frac{1}{2}) &= S^{(r)}_{2m}(u)\quad
\quad (m=1,\dots,\ell),\\
\label{eq:TS4}
T^{(r)}_m(u)
T^{(r)}_{m+1}(u)
&=S^{(r)}_{2m+1}(u)
\quad (m=1,\dots,\ell-1),\\
\label{eq:TS5}
T^{(r)}_m(u)^2 &=S^{(r+1)}_{2m}(u)
\quad (m=1,\dots,\ell),
\end{align}
where $S^{(r)}_{2\ell}(u)=S^{(r+1)}_{2\ell}(u)=1$.
\end{claim}

Indeed, (\ref{eq:TS4}) follows immediately from (\ref{eq:TS2}),
while (\ref{eq:TS3}) and (\ref{eq:TS5}) are proved by the
induction with respect to $m$.

Now, it is easy to check that $\rho^{-1}$ is 
a ring homomorphism by substitution.

By comparing (\ref{eq:st1})--(\ref{eq:st3})
and (\ref{eq:TS1}), (\ref{eq:TS3})--(\ref{eq:TS5}),
we see that $\rho^{-1}$ is the inverse of $\rho$.
\end{proof}

The following theorem is an analogue of
Theorem \ref{thm:TAperiod}.

\begin{thm}
\label{thm:TCperiod2}
The following relations hold
in  $\widehat{\EuScript{T}}_{2\ell}(A_{2r+1})$:
\par
(1) Half-periodicity:
$ S^{(a)}_m(u+r+1+\ell)
=
(-1)^m
S^{(2r+2-a)}_{2\ell-m}(u)
=S^{(a)}_{2\ell-m}(u).$
\par
(2) Periodicity:
$ S^{(a)}_m(u+2(r+1+\ell))
=S^{(a)}_{m}(u).$
\end{thm}

A proof of Theorem \ref{thm:TCperiod2}
by the determinant method  will be given in
Section \ref{subsec:TCperiod}.

Admitting Theorem \ref{thm:TCperiod2},
let us  prove Theorem \ref{thm:TCperiod} first.

\begin{proof}
[Proof of  Theorem \ref{thm:TCperiod}]
It is enough to prove (1).
For simplicity, we write the image
$\rho(S^{(a)}_m(u))$ as $S^{(a)}_m(u)$.
By (\ref{eq:st1}),
we immediately obtain the relation for $a=1,\dots,r-1$.
Let us verify the case $a=r$.
By Proposition
\ref{prop:bij1} and Theorem \ref{thm:TCperiod2}, we have
(for example, for $\ell,m$: even)
\begin{align}
\begin{split}
&\phantom{=}\ \  T^{(r)}_m(u+r+1+\ell)\\
&=
\frac{S^{(r)}_{2m-1}(u+r+1+\ell)S^{(r)}_{2m-5}(u+r+1+\ell)\cdots
S^{(r)}_{3}(u+r+1+\ell) }
{S^{(r)}_{2m-3}(u+r+1+\ell)S^{(r)}_{2m-7}(u+r+1+\ell)\cdots
S^{(r)}_{1}(u+r+1+\ell) }\\
&=
\frac{S^{(r)}_{2\ell-2m+1}(u)S^{(r)}_{2\ell-2m+5}(u)\cdots 
S^{(r)}_{2\ell-3}(u)
}
{S^{(r)}_{2\ell-2m+3}(u)S^{(r)}_{2\ell-2m+7}(u)\cdots 
S^{(r)}_{2\ell-1}(u)}\\
&=
T^{(r)}_{\ell}(u)
\frac{S^{(r)}_{2\ell-2m-1}(u)S^{(r)}_{2\ell-2m-5}(u)\cdots
S^{(r)}_{3}(u) }
{S^{(r)}_{2\ell-2m-3}(u)S^{(r)}_{2\ell-2m-7}(u)\cdots
S^{(r)}_{1}(u) }\\
&=
T^{(r)}_{\ell}(u)
T^{(r)}_{\ell-m}(u).
\end{split}
\end{align}
\end{proof}

\subsection{Proof of Theorem \ref{thm:TCperiod2}} 
\label{subsec:TCperiod}

The outline of the proof is the same as Theorem
\ref{thm:TAperiod},
but there are some points where extra caution is
necessary due to the singularities $S^{(r+1)}_m(u)=0$ ($m$: odd).

Let $R$ be any ring.
Let us take an arbitrary
 $2\ell \times \infty$  matrix $M$ over
$R$
such that  $M=[ x_k]_{k\in \mathbb{Z}}$,
  $x_k \in (\widehat{\EuScript{T}}_{2\ell}(A_{2r+1}))^{2\ell}$
with the following periodicity:
\begin{align}
\label{eq:xpc2}
x_{k+2r+2+2\ell}=x_k.
\end{align}
Let $D_M=\{ D^{(a)}_m(u)\mid (a,m,u)\in \widehat{H}_e\}$ be a family of
minors of $M$ defined by
\begin{align}
\label{eq:det1}
D^{(a)}_m(u)&=\det [
\overbrace{x_\beta x_{\beta+1}\ \cdots\ }^{m}
\overbrace{\operatornamewithlimits{\phantom{x}}^{\vee} \ \cdots\
\operatornamewithlimits{\phantom{x}}^{\vee}}^{a}
\overbrace{\phantom{x}\ \cdots \quad}^{2\ell-m}
\overbrace{\operatornamewithlimits{\phantom{x} }^\vee \ \cdots \
{\operatornamewithlimits{\mathit{x}}^{\vee}}_{\beta+2r+2\ell+1}}^{2r+2-a}],\\
\beta&= -\frac{a+m+2u}{2}\in \mathbb{Z}.
\end{align}
Then, as Proposition \ref{prop:DM1}, we have the following:
\begin{prop}
\label{prop:DM2}
The family $D_M=\{ D^{(a)}_m(u)\mid (a,m,u)\in \widehat{H}_e\}$ satisfies
the following relations in $R$:

(1) T-system:
\begin{align}
\label{eq:DT}
D^{(a)}_m(u-\textstyle\frac{1}{2})D^{(a)}_m(u+\textstyle\frac{1}{2})
&=
D^{(a)}_{m-1}(u)D^{(a)}_{m+1}(u)
+
D^{(a-1)}_{m}(u)D^{(a+1)}_{m}(u)\\
&
\hskip100pt ((a,m,u)\in \widehat{H}^{\circ}_o),\notag
\end{align}

(2) Boundary condition:
\begin{align}
\label{eq:spc31}
\begin{split}
  D^{(a)}_0(u+\textstyle\frac{1}{2})
&=D^{(a-1)}_0(u)\quad (a=1,\dots, 2r+2),\\
  D^{(a)}_{2\ell}(u+\textstyle\frac{1}{2})
&=D^{(a+1)}_{2\ell}(u)\quad (a=0,\dots, 2r+1),\\
  D^{(0)}_{m}(u+\textstyle\frac{1}{2})
&=D^{(0)}_{m+1}(u) \quad (m=0,\dots, 2\ell-1),\\ 
 D^{(2r+2)}_{m}(u+\textstyle\frac{1}{2})
&=(-1) D^{(2r+2)}_{m-1}(u) \quad (m=1,\dots, 2\ell).
\end{split}
\end{align}

(3) Half-periodicity:
\begin{align}
D^{(a)}_m(u+r+1+\ell)=
(-1)^m
D^{(2r+2-a)}_{2\ell-m}(u).
\end{align}
\end{prop}

We consider an extended family
of generators
$S=\{S^{(a)}_m(u)\mid (a,m,u)\in \widehat{H} \}$
of $\widehat{\EuScript{T}}_{2\ell}(A_{2r+1})$,
where $S^{(r+1)}_m(u)=0$ ($m$: odd)
and $S^{(a)}_m(u)$ ($(a,m,u)\in \partial\widehat{H}$)
are given by (\ref{eq:spc21}) and (\ref{eq:spc24}). 
Again,
we divide the family $S$ into two subfamilies,
$S_e=\{S^{(a)}_m(u) \mid
 (a,m,u)\in \widehat{H}_e\}$ and $S_o=\{S^{(a)}_m(u)\mid (a,m,u)\in
 \widehat{H}_o\}$,
and concentrate on the half family $S_e$.

Theorem \ref{thm:TCperiod2} follows from
Proposition \ref{prop:DM2} and the following proposition:
\begin{prop}
\label{prop:TCperiod3}
There exists some  $2\ell \times \infty$ matrix
$M=\{ x_k\}_{k\in \mathbb{Z}}$ satisfying
the condition (\ref{eq:xpc2}) such that,
for $D_M=\{ D^{(a)}_m(u)\mid (a,m,u)\in \widehat{H}_e\}$,
the following relation holds in $\widehat{\EuScript{T}}_{2\ell}(A_{2r+1})$:
\begin{align}
\label{eq:DT24}
S^{(a)}_m(u)=D^{(a)}_m(u)\quad  ((a,m,u)\in \widehat{H}_e).
\end{align}
\end{prop}

\begin{proof}
We define $x_0$, \dots, $x_{2r+2\ell+1}\in
 (\widehat{\EuScript{T}}_{2\ell}(A_{2r+1}))^{2\ell}$ as follows:
Firstly, let us arbitrarily choose $x_1,\dots,x_{2\ell-1}$ such that
\begin{align}
\label{eq:DT22}
D^{(0)}_0(0):=\det [ x_0\cdots x_{2\ell-1}]
=1=S^{(0)}_0(0)= S^{(1)}_0(\textstyle\frac{1}{2})
\end{align}
holds.
Secondly, we recursively define the rest, $x_{2\ell},\dots,
 x_{2r+2\ell+1}$, by
\begin{align}
x_{2\ell+j}&=
\frac{1}{S^{(1)}_0(\textstyle\frac{1}{2}-j)}
\sum_{m=0}^{2\ell-1}(-1)^{2\ell-1-m}
S^{(1)}_m(-\textstyle\frac{1+2j+m}{2})x_{j+m}\quad (j=0,\dots,2r+1)
\end{align}
so that the following equality holds:
\begin{align}
\label{eq:DT23}
D^{(1)}_m(-\textstyle\frac{1+2j+m}{2})=
S^{(1)}_m(-\textstyle\frac{1+2j+m}{2})
\quad (m=0,\dots,2\ell-1; j=0,\dots,2r+1).
\end{align}
Lastly, we define the matrix $M=[ x_k ]_{k\in \mathbb{Z}}$
by extending the above $x_0$, \dots, $x_{2r+2\ell+1}$ with
the condition (\ref{eq:xpc2}).

For $D_M$, we claim that
the relation (\ref{eq:DT24}) holds in $\widehat{\EuScript{T}}_{2\ell}(A_{2r+1})$.
This will be shown inductively,
based on the fact that
the T-system and
the boundary condition
satisfied by both $S_e$ and $D_M$.

To proceed the induction,
we introduce a prism
\begin{align}
\label{eq:prism2}
\widehat{P}=\{ (a,m,u)\in \widehat{H}
\mid a+m+2u\leq 0,\
a-m-2u \leq 4r+4
\}.
\end{align}
We use the notations, $\widehat{P}_e=\widehat{P}\cap
 \widehat{H}_e$, $\widehat{P}_e[a=1]
=\{ (a,m,u)\in \widehat{P}_e \mid a=1\}$, etc, once again.

First, we show that (\ref{eq:DT24}) is true for $(a,m,u)\in 
\widehat{P}_e$ by the induction on $a$.
By (\ref{eq:DT22}), (\ref{eq:DT23}),
 and (\ref{eq:spc31}),
we see that (\ref{eq:DT24}) is true for any $(a,m,u)$ in
the set
\begin{align}
\widehat{P}_e[a=0] \cup \widehat{P}_e[a=1]
 \cup \widehat{P}_e[m=0] \cup \widehat{P}_e[m=2\ell].
\end{align}
Assume (\ref{eq:DT24}) is true up to $a$.
By the T-system (\ref{eq:TA3}), we have
\begin{align}
S^{(a+1)}_{m}(u)
&=
\frac{1}{S^{(a-1)}_{m}(u)}
\left(S^{(a)}_m(u-\textstyle\frac{1}{2})
S^{(a)}_m(u+\textstyle\frac{1}{2})
-S^{(a)}_{m-1}(u)S^{(a)}_{m+1}(u)\right).
\end{align}
On the other hand, by (\ref{eq:DT}) and the induction
hypothesis, we have
\begin{align}
D^{(a+1)}_{m}(u)
&=
\frac{1}{S^{(a-1)}_{m}(u)}
\left(S^{(a)}_m(u-\textstyle\frac{1}{2})
S^{(a)}_m(u+\textstyle\frac{1}{2})
-S^{(a)}_{m-1}(u)S^{(a)}_{m+1}(u)\right).
\end{align}
Thus, the relation $S^{(a+1)}_{m}(u)=
D^{(a+1)}_{m}(u)$ is obtained.
The induction works up to $a=r+2$.
However, since  $S^{(r+1)}_{2m+1}(u)=D^{(r+1)}_{2m+1}(u)=0$
($m=0,\dots,\ell-1$) by the assumption, 
the relation
$S^{(r+3)}_{2m+1}(u)=D^{(r+3)}_{2m+1}(u)$ is not trivial.
To overcome the point, we have to show the following claim:

\begin{claim1}
For $(r+3,2m+1,u)\in \widehat{P}_e$,
the following relation holds in $\widehat{\EuScript{T}}_{2\ell}(A_{2r+1})$:
\begin{align}
\label{eq:DD1}
D^{(r+3)}_{2m+1}(u)=-D^{(r-1)}_{2m+1}(u).
\end{align}
\end{claim1}

Once the claim is verified, we have
$S^{(r+3)}_{2m+1}(u)=-S^{(r-1)}_{2m+1}(u)=
-D^{(r-1)}_{2m+1}(u)=D^{(r+3)}_{2m+1}(u)$
so that the induction continues and completes.

Let us prove Claim 1.
For simplicity, we write
the ring $\widehat{\EuScript{T}}_{2\ell}(A_{2r+1})$ as
$R$.
Consider a deformation  $M_{\varepsilon}
=[ x_k + \varepsilon x'_k ]$ of the matrix $M
$
with a formal parameter $\varepsilon$
and some $x'_k\in R^{2\ell}$
 ($k\in \mathbb{Z}$, $x'_{k+2r+2+2\ell}=x'_k$).
Let $D^{(a)}_m(u)_{\varepsilon}\in R[\varepsilon]$ be
the corresponding minor for $M_{\varepsilon}$.
Fix $2m+1$ and $u$ in (\ref{eq:DD1}).
We choose  $M_{\varepsilon}$ so that
$D^{(r+1)}_{2m+1}(u)_{\varepsilon}=
\varepsilon D^{(r+1)}_{2m+1}(u)' + o(\varepsilon)$
with $D^{(r+1)}_{2m+1}(u)'\in R^{\times}$.
(This is possible.
For example,
let $x'$ be the $2m+2$th column
for $D^{(r)}_{2m+1}(u+\frac{1}{2})$.
Then, add $\varepsilon x'$ to $M$ at the position of the last
column for $D^{(r+1)}_{2m+1}(u)$.
This deformation yields  $D^{(r+1)}_{2m+1}(u)'=
D^{(r)}_{2m+1}(u+\frac{1}{2})
= T^{(r)}_{2m+1}(u+\frac{1}{2})\in R^{\times}$.)
Now, by (\ref{eq:DT}), we have
\begin{align}
\label{eq:DD2}
\begin{split}
D^{(r-1)}_{2m+1}(u)_{\varepsilon}
D^{(r+1)}_{2m+1}(u)_{\varepsilon}
&=
D^{(r)}_{2m+1}(u-\textstyle\frac{1}{2})_{\varepsilon}
D^{(r)}_{2m+1}(u+\textstyle\frac{1}{2})_{\varepsilon}
-
D^{(r)}_{2m}(u)_{\varepsilon}
D^{(r)}_{2m+2}(u)_{\varepsilon},\\
D^{(r+3)}_{2m+1}(u)_{\varepsilon}
D^{(r+1)}_{2m+1}(u)_{\varepsilon}
&=
D^{(r+2)}_{2m+1}(u-\textstyle\frac{1}{2})_{\varepsilon}
D^{(r+2)}_{2m+1}(u+\textstyle\frac{1}{2})_{\varepsilon}
-
D^{(r+2)}_{2m}(u)_{\varepsilon}
D^{(r+2)}_{2m+2}(u)_{\varepsilon}.
\end{split}
\end{align}
We take the ratio
of the relations in (\ref{eq:DD2}),
\begin{align}
\label{eq:Drel}
\frac{
D^{(r+3)}_{2m+1}(u)_{\varepsilon}
}
{
D^{(r-1)}_{2m+1}(u)_{\varepsilon}
}
=
\frac{D^{(r+2)}_{2m+1}(u-\textstyle\frac{1}{2})_{\varepsilon}
D^{(r+2)}_{2m+1}(u+\textstyle\frac{1}{2})_{\varepsilon}
-
D^{(r+2)}_{2m}(u)_{\varepsilon}
D^{(r+2)}_{2m+2}(u)_{\varepsilon}
}
{
D^{(r)}_{2m+1}(u-\textstyle\frac{1}{2})_{\varepsilon}
D^{(r)}_{2m+1}(u+\textstyle\frac{1}{2})_{\varepsilon}
-
D^{(r)}_{2m}(u)_{\varepsilon}
D^{(r)}_{2m+2}(u)_{\varepsilon}
},
\end{align}
which is in $R[[\varepsilon]]$.
Let us calculate the right hand side of (\ref{eq:Drel}).
Let $\equiv$ mean the equality
in $R[[\varepsilon]]$
 modulo
$\varepsilon R[[\varepsilon]]$.
By (\ref{eq:DT}) and $D^{(r+1)}_{2m+1}(u)_{\varepsilon}\equiv 0$, we have
\begin{align}
\label{eq:De1}
\begin{split}
D^{(r+2)}_{2m}(u)_{\varepsilon}
&\equiv D^{(r+1)}_{2m}(u-\textstyle\frac{1}{2})_{\varepsilon}
D^{(r+1)}_{2m}(u+\textstyle\frac{1}{2})_{\varepsilon}
/
D^{(r)}_{2m}(u)_{\varepsilon},\\
D^{(r+2)}_{2m+2}(u)_{\varepsilon}
&\equiv D^{(r+1)}_{2m+2}(u-\textstyle\frac{1}{2})_{\varepsilon}
D^{(r+1)}_{2m+2}(u+\textstyle\frac{1}{2})_{\varepsilon}
/
D^{(r)}_{2m+2}(u)_{\varepsilon},\\
D^{(r+2)}_{2m+1}(u+\textstyle\frac{1}{2})_{\varepsilon}
&\equiv
- D^{(r+1)}_{2m}(u+\textstyle\frac{1}{2})_{\varepsilon}
D^{(r+1)}_{2m+2}(u+\textstyle\frac{1}{2})_{\varepsilon}
/
D^{(r)}_{2m+1}(u+\textstyle\frac{1}{2})_{\varepsilon},\\
D^{(r+2)}_{2m+1}(u-\textstyle\frac{1}{2})_{\varepsilon}
&\equiv
- D^{(r+1)}_{2m}(u-\textstyle\frac{1}{2})_{\varepsilon}
D^{(r+1)}_{2m+2}(u-\textstyle\frac{1}{2})_{\varepsilon}
/
D^{(r)}_{2m+1}(u-\textstyle\frac{1}{2})_{\varepsilon}.
\end{split}
\end{align}
Thanks  to (\ref{eq:De1}),
the right hand side of (\ref{eq:Drel}) equals,
modulo $\varepsilon R[[\varepsilon]]$,  to
\begin{align}
\label{eq:Drel2}
(-1)\frac
{
D^{(r+1)}_{2m}(u-\textstyle\frac{1}{2})
D^{(r+1)}_{2m+2}(u-\textstyle\frac{1}{2})
D^{(r+1)}_{2m}(u+\textstyle\frac{1}{2})
D^{(r+1)}_{2m+2}(u+\textstyle\frac{1}{2})
}
{D^{(r)}_{2m+1}(u-\textstyle\frac{1}{2})
D^{(r)}_{2m+1}(u+\textstyle\frac{1}{2})
D^{(r)}_{2m}(u)
D^{(r)}_{2m+2}(u)
}.
\end{align}
Then, using the induction hypothesis
(\ref{eq:DT24}) up to $a=r+2$ in $\widehat{P}_e$ 
and the relations (\ref{eq:TA3}) and (\ref{eq:psu1}),
 one can show that
(\ref{eq:Drel2}) is equal to $-1$.
Therefore, we have $
D^{(r+3)}_{2m+1}(u)_{\varepsilon}
\equiv
-D^{(r-1)}_{2m+1}(u)_{\varepsilon}$,
which means (\ref{eq:DD1}).
 This ends the proof of Claim 1.

Next,  we show that (\ref{eq:DT24}) is true
 for any $(a,m,u)\in \widehat{H}_e\setminus \widehat{P}_e$.
We first remark that, by (\ref{eq:spc31}),
(\ref{eq:DT24}) is now true for any $(a,m,u)\in \partial \widehat{H}_e$.
Then, using the T-system (\ref{eq:TA3}) as
\begin{align}
S^{(a)}_{m}(u\pm \textstyle\frac{1}{2})
&=
\frac{1}{S^{(a)}_{m}(u\mp \textstyle\frac{1}{2})}
\left(
S^{(a)}_{m-1}(u)S^{(a)}_{m+1}(u)
+
S^{(a-1)}_m(u)S^{(a+1)}_m(u)
\right),
\end{align}
and repeating the same argument as before,
 one can inductively, with respect to $\pm u$,
 conclude that (\ref{eq:DT24}) is true for any $(a,m,u)\in 
\widehat{H}_e\setminus \widehat{P}_e$.
Once again, in the induction process 
we need the following claim,
which can be verified by
 a similar deformation argument as Claim 1:

\begin{claim2}
For $(r+1,2m+1,u)\in \widehat{H}_e\setminus\widehat{P}_e$,
the following relation holds
in $\widehat{\EuScript{T}}_{2\ell}(A_{2r+1})$:
\begin{align}
D^{(r+1)}_{2m+1}(u)=0.
\end{align}
\end{claim2}

This ends the proof of Proposition \ref{prop:TCperiod3}.
\end{proof}

\begin{rem}
The system $\widetilde{\mathbb{T}}_{\ell}(C_r)$
plays an analogous role of $\widetilde{\mathbb{T}}_{\ell}(A_{r})$
in the determinant method.
Thus, it is natural to expect that the analogue of
Proposition \ref{prop:TtoY4} also holds.
Unfortunately, this is not true.
The simplest case, $C_2$ and $\ell=2$, provides a counterexample
by the similar reason to Example \ref{exmp:surj} (3).
\end{rem}

\section{Direct method: $\EuScript{T}_2(A_r)$,
$\EuScript{T}_2(D_r)$, and  $\EuScript{T}_2(B_r)$}
\label{sect:direct}

In this section we give some explicit formulae
of generators $T^{(a)}_m(u)$
in terms of `initial variables'
 for
$\EuScript{T}_2(A_r)$,
$\EuScript{T}_2(D_r)$, and  $\EuScript{T}_2(B_r)$.
Our goal here is to prove the periodicity
for $\EuScript{T}_2(B_r)$.
The formulae for $\EuScript{T}_{2}(A_r)$ and
$\EuScript{T}_{2}(D_r)$ should be obtained as a specialization
of the more general formulae by 
Caldero-Chapoton \cite[Theorem 3.4]{CC}
and/or by  Yang-Zelevinsky \cite[Theorems 1.10, 1.12]{YZ}.
Nevertheless, we include  these formulae as well,
since they are good examples
showing how the periodicity actually happens in the T-systems.

\subsection{Explicit formula by initial variables for $\EuScript{T}_2(A_r)$}

Throughout this subsection, $I=\{1,\dots,r\}$ is
the set enumerating the diagram $A_r$ in Section
\ref{subsect:unrestrictedT}.
Recall that the (dual) Coxeter number of $A_r$ is $r+1$.

\begin{defn}
\label{defn:gamma}
For a family of variables $X = \{ x_a \mid a \in I \}$,
we define
$\gamma^{(j)}_n(X)$, $\nu^{(j)}_n(X) \in \mathbb{Z}[X^{\pm1}]$ 
($1 \le j \le r, 0\le n \le j$) as follows:
\begin{align}
  \gamma_n^{(j)}(X) &=\frac{1}{x_n}+ \frac{x_{n-1} }{x_n x_{n+1} } + 
  \frac{x_{n-1} }{x_{n+1} x_{n+2} }
  +\cdots  + \frac{x_{n-1} }{x_{j-2} x_{j-1} },\\
  \nu_n^{(j)}(X) &= \frac{x_{n-1}}{x_{j-1}},
  \label{def-gamma-nu}
\end{align}
where we set $x_{-1}=0$ and $x_{0}=1$. 
In particular,
\begin{align}
\label{eq:gn0}
  &\gamma^{(j)}_0(X) = 1, \quad   \nu^{(j)}_0(X) = 0,\\
\label{eq:gn2}
  &\gamma^{(j)}_j(X) = 0, \quad   \nu^{(j)}_j(X) = 1.
\end{align}
\end{defn}

The following lemma is easily checked. 
\begin{lem}
\label{lem:xg1}
The following relations hold
for any $ 1\le j \le r$ and $ 1\le n \le j$.
\begin{align}
  x_{n-1} \gamma_{n-1}^{(j)}(X)  &= 1+ x_{n-2} \gamma_n^{(j)}(X),
  \label{cortgrel}  \\
  x_{j-1}\gamma_{n-1}^{(j)}(X) &= x_{j-1}\gamma^{(j-1)}_{n-1}(X) 
                        +\nu_{n-1}^{(j-1)}(X).&     
  \label{corgamrel}   
  \end{align}
\end{lem}

We define a family $\tau=\{
\tau^{(a)}(u;X)\in \mathbb{Z}[X^{\pm1}]
\mid a\in I, u\in \mathbb{Z},
\text{$a+u$ is odd}
\}$
as follows:
First, we define, for
$a\in I$ and $0\leq k\leq r+2$,
\begin{align}
\label{eq:tau1}
\tau^{(a)}(a-1+2k;X) = 
  \begin{cases}
  \gamma^{(a+k)}_k(X) x_{a+k} + \nu^{(a+k)}_k(X) 
  \quad (0 \le k \le r-a),
  \\
  \gamma^{(r)}_{r+1-a}(X) 
  +  \nu^{(r)}_{r+1-a }(X)\displaystyle{\frac{1+x_{r-1}}{x_r}} 
  \quad (k=r-a+1),
  \\
  \gamma^{(k-1)}_{a+k-(r+2)}(X) x_{k-1} + \nu^{(k-1)}_{a+k-(r+2)}(X)\\
  \hskip100pt (r-a+2 \le k \le r+1),
  \\
  \gamma^{(r)}_{a}(X) 
  +  \nu^{(r)}_a(X) \displaystyle{\frac{1+x_{r-1}}{x_r}} \quad (k=r+2).
  \end{cases}
\end{align}
Then, we extend the definition by
the following periodicity:
\begin{align}
\label{eq:tperiod}
\tau^{(a)}(u+2(r+3);X) = \tau^{(a)}(u;X).
\end{align}

By (\ref{eq:gn0}), we have
\begin{align}
\label{eq:tx1}
  \tau^{(a)}(a-1;X)&= x_a,\\
  \tau^{(a)}(2r+3-a;X)&= x_{r+1-a}.
\end{align}

See Figure \ref{pic:tau} for the fundamental domain 
of $\tau^{(a)}(u;X)$ in the strip
$\{ (a,u) \mid a\in I, u\in \mathbb{Z},
\text{$a+u$ is odd}
\}$
which consists of four parts, corresponding
to the cases in (\ref{eq:tau1}):
\begin{align}
\begin{split}
  &D_1 = \{(a,a-1+2k) \mid a \in I, 0 \le k \le r-a \},
  \\
  &D_2 = \{(a,a-1+2k) \mid a \in I, k=r-a+1 \},
\\
  &D_3 = \{(a,a-1+2k) \mid a \in I, r-a+2 \le k \le r+1 \},
\\
  &D_4 = \{(a,a-1+2k) \mid a \in I, k=r+2\}.
\end{split}
\end{align}

\begin{figure}
\begin{picture}(300,150)(19,0)

\put(-10,20){\vector(1,0){360}}
\put(5,0){\vector(0,1){130}}
\put(340,25){\small$u$}
\put(-5,122){\small$a$}

\put(5,35){\circle*{4}}
\put(-8,35){\small$x_1$}
\put(20,50){\circle*{4}}
\put(7,50){\small$x_2$}
\put(35,65){\circle*{4}}
\put(22,65){\small$x_3$}
\put(65,95){\circle*{4}}
\put(42,95){\small$x_{r-1}$}
\put(80,110){\circle*{4}}
\put(67,110){\small$x_{r}$}

\put(275,35){\circle*{4}}
\put(262,35){\small$x_1$}
\put(290,50){\circle*{4}}
\put(277,50){\small$x_2$}
\put(305,65){\circle*{4}}
\put(292,65){\small$x_3$}
\put(335,95){\circle*{4}}
\put(312,95){\small$x_{r-1}$}
\put(350,110){\circle*{4}}
\put(337,110){\small$x_{r}$}

\drawline(275,35)(305,65)
\dottedline{5}(305,65)(335,95)
\drawline(335,95)(350,110)

\put(80,50){\circle{4}}
\put(65,65){\circle{4}}
\put(95,65){\circle{4}}

\put(35,35){\circle{4}}
\put(65,35){\circle{4}}
\put(95,35){\circle{4}}
\put(155,35){\circle{4}}

\drawline(5,35)(35,65)
\dottedline{5}(35,65)(65,95)
\drawline(65,95)(80,110)

\drawline(5,35)(35,35)
\dottedline{5}(35,35)(65,35)
\drawline(65,35)(95,35)
\dottedline{5}(95,35)(155,35)

\put(95,95){\circle{4}}
\put(155,35){\circle{4}}

\put(110,110){\circle{4}}
\put(125,95){\circle{4}}
\put(155,65){\circle{4}}
\put(170,50){\circle{4}}
\put(185,35){\circle{4}}

\drawline(110,110)(125,95)
\dottedline{5}(125,95)(155,65)
\drawline(155,65)(185,35)

\drawline(80,110)(95,95)
\dottedline{5}(95,95)(155,35)

\put(140,110){\circle*{4}}
\put(127,110){\small$x_1$}
\put(155,95){\circle*{4}}
\put(142,95){\small$x_2$}
\put(185,65){\circle*{4}}
\put(163,65){\small$x_{r-2}$}
\put(200,50){\circle*{4}}
\put(178,50){\small$x_{r-1}$}
\put(215,35){\circle*{4}}
\put(200,35){\small$x_{r}$}

\drawline(140,110)(155,95)
\dottedline{5}(155,95)(185,65)
\drawline(185,65)(215,35)
\drawline(140,110)(170,110)
\dottedline{5}(170,110)(260,110)
\drawline(260,110)(290,110)
\drawline(290,110)(275,95)
\dottedline{5}(275,95)(245,65)
\drawline(245,65)(215,35)

\put(170,110){\circle{4}}
\put(260,110){\circle{4}}
\put(290,110){\circle{4}}
\put(275,95){\circle{4}}
\put(245,65){\circle{4}}
\put(230,50){\circle{4}}
\put(215,35){\circle{4}}

\put(320,110){\circle{4}}
\put(305,95){\circle{4}}
\put(275,65){\circle{4}}
\put(260,50){\circle{4}}
\put(245,35){\circle{4}}

\drawline(320,110)(305,95)
\dottedline{5}(305,95)(275,65)
\drawline(275,65)(245,35)

\put(75,120){\small$D_1$}
\put(105,120){\small$D_2$}
\put(208,120){\small$D_3$}
\put(315,120){\small$D_4$}

\dottedline{3}(5,20)(5,115)
\put(-3,10){\small$0$}
\dottedline{3}(35,20)(35,35)
\put(33,10){\small$2$}
\dottedline{3}(80,20)(80,112)
\put(72,10){\small$r-1$}
\dottedline{3}(140,20)(140,112)
\put(125,10){\small$r+3$}
\dottedline{3}(155,20)(155,35)
\put(150,10){\small$2r-2$}

\dottedline{3}(215,20)(215,35)
\put(200,10){\small$2r+2$}
\dottedline{3}(275,20)(275,35)
\put(255,10){\small$2r+6$}
\dottedline{3}(290,20)(290,40)
\dottedline{3}(290,60)(290,112)
\put(285,10){\small$3r+1$}


\end{picture}
\caption{The fundamental domain of $\tau^{(a)}(u;X)$
for $\EuScript{T}_2(A_r)$.
The filled circles in $D_1$ are identified with
the initial  variables
 of $\EuScript{T}_2(A_r)_+$.
}
\label{pic:tau}
\end{figure}
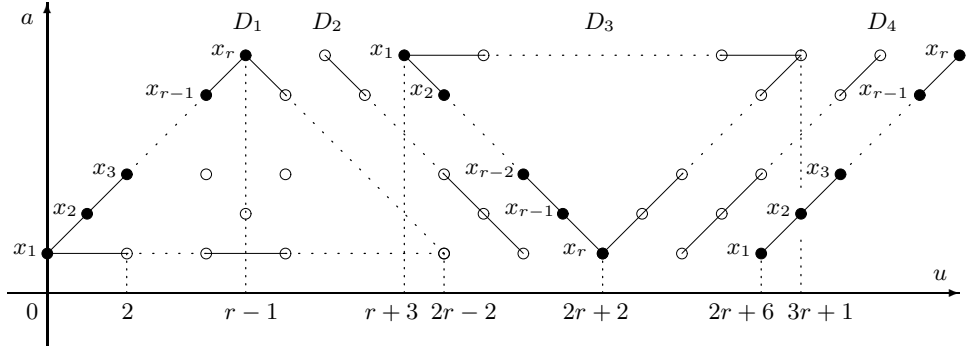

\begin{prop}\label{prop:tau-halfperiod}
The family $\tau$ satisfies the following relations
in $\mathbb{Z}[X^{\pm1}]$:
\par
(1) Half-periodicity:
\begin{align}
\label{eq:thp}
  \tau^{(a)}(u+r+3;X) = \tau^{(r+1-a)}(u;X).
\end{align}

\par
(2) T-system $\mathbb{T}_2(A_r)$:
\begin{align}\label{tau-rel-A_r}
\tau^{(a)}(u-1;X) \, \tau^{(a)}(u+1;X) 
=1+\tau^{(a-1)}(u;X) \, \tau^{(a+1)}(u;X), 
\end{align}
where $\tau^{(0)}(u;X) = \tau^{(r+1)}(u;X)= 1$
if they occur in the right hand side.
\end{prop}

\begin{proof}
For simplicity, in this proof we write 
$\gamma_k^{(j)} = \gamma_k^{(j)}(X)$, 
$\nu_k^{(j)} = \nu_k^{(j)}(X)$ and 
$\tau^{(a)}(u)=\tau^{(a)}(u;X)$.
\par
(1) As an equivalent claim to (\ref{eq:thp}),
we prove
\begin{align}
  \tau^{(a)}(u) = \tau^{(r+1-a)}(u+r+3).
\end{align}
Thanks to (\ref{eq:tperiod}),
it is enough to prove it for $(a,u) \in D_1 \sqcup D_2$.
Suppose that  $(a,u)=(a,a-1+2k) \in D_1$.
Then we have 
\begin{align}
  \tau^{(a)}(a-1+2k) = \gamma^{(a+k)}_k x_{a+k} + \nu^{(a+k)}_k.
\end{align}
On the other hand, $(r+1-a,u+r+3)=(a',a'-1+2k')\in D_3$,
where $a'=r+1-a$ and $k'=a+k+1$.
Therefore, we have
\begin{align}
\begin{split}
  \tau^{(r+1-a)}(u+r+3)
  &= \gamma^{(k'-1)}_{a'+k'-(r+2)} x_{k'-1} + \nu^{(k'-1)}_{a'+k'-(r+2)}
  \\
  &= \gamma^{(a+k)}_k x_{a+k} + \nu^{(a+k)}_k.
\end{split}
\end{align}
Thus the claim follows.
The remaining case $(a,u) \in D_2$ is similar and easier.
\par
(2) Thanks to (1), it is enough to show (\ref{tau-rel-A_r})
for $(a,u-1)=(a,a-1+2k)\in D_1\sqcup D_2$, i.e.,
$0\leq k \leq r-a+1$.
It is proved by a direct calculation using Lemma \ref{lem:xg1}.
\end{proof}

Like (\ref{eq:Tfact1}),
it is enough to consider the subring
$\EuScript{T}_2(A_r)_+$
 of
$\EuScript{T}_2(A_r)$ generated by
$T^{(a)}_1(u)^{\pm1}$ ($a\in I$, $u\in \mathbb{Z}$, $a+u$: odd).
Set
\begin{align}
\label{eq:init}
T^{\mathrm{init}} = \{T^{(a)}_1(a-1) \mid a \in I\}.
\end{align}
We call the elements of $T^{\mathrm{init}}$
the {\em initial variables\/}
of $\EuScript{T}_2(A_r)_+$.

By Proposition
\ref{prop:tau-halfperiod} (2),
we have the following {\em explicit formula\/}
of other variables
$T^{(a)}_1(u)\in \EuScript{T}_2(A_r)_+$ by
Laurent polynomials in the initial variables.

\begin{thm}
\label{thm:explicitsol} 
The following relation holds in $\EuScript{T}_2(A_r)_+$: 
\begin{align}
  \label{T-tau-A_r}
  &T^{(a)}_1(u) =
 \tau^{(a)}(u;T^{\mathrm{init}}),
\end{align}
where we set $x_a = T^{(a)}_1(a-1)$ ($a\in I$) in the right hand side.
\end{thm}

\begin{proof}
In the following we simply write 
$\tau^{(a)}(u)$ for $\tau^{(a)}(u;T^{\mathrm{init}})$.

We rewrite \eqref{T-tau-A_r} as ($n\in \mathbb{Z}$)
\begin{align}\label{T-tau-ind}
  T^{(a)}_1(a-1+2n) = \tau^{(a)}(a-1+2n),
\end{align}
and show it by induction on $n$.
Recall that $\mathbb{T}_2(A_r)$ becomes 
\begin{align}\label{T-rel-A_r}
  T^{(a)}_1(a-1+2n) T^{(a)}_1(a+1+2n) 
  = 1+T^{(a-1)}_1(a+2n) T^{(a+1)}_1(a+2n).
\end{align}
{}By (\ref{eq:tx1}),
the $n=0$ case of \eqref{T-tau-ind}, $\tau^{(a)}(a-1) = T_1^{(a)}(a-1)$, 
is satisfied.
Assume that \eqref{T-tau-ind} holds up to $n$ $(>0)$.
Then \eqref{tau-rel-A_r} with $u=a+2n$ becomes
\begin{align}\label{T-tau-ind2}
\begin{split}
 & T^{(a)}_1(a-1+2n) \tau^{(a)}(a-1+2(n+1)) \\
 &\hskip50pt
 = 1 + \tau^{(a-1)}(a-2+2(n+1)) T^{(a+1)}_1(a+2n).
\end{split}
\end{align}
By setting $a=1$ in  (\ref{T-rel-A_r}) and (\ref{T-tau-ind2}) 
, we obtain the relation
\begin{align}
  \tau^{(1)}(2(n+1)) 
  = \frac{1 + T^{(2)}_1(1+2n)}{T^{(1)}_1(2n)}
= T^{(1)}(2(n+1)).
\end{align}
By increasing $a$ one-by-one up to $r$,
we recursively obtain the relation 
\begin{align}
\begin{split}
  \tau^{(a)}(a-1+2(n+1))  
  &= \frac{1+T^{(a-1)}_1(a+2n) T^{(a+1)}_1(a+2n)}
         {T^{(a)}_1(a-1+2n)}\\
& = T^{(a)}(a-1+2(n+1)).
\end{split}
\end{align}
The case of $n$ $(<0)$ is similar.
\end{proof}

In the above proof, we carefully stay inside the
ring $\EuScript{T}_2(A_r)_+$
by avoiding the division by $\tau^{(a)}(u)$'s.
As a corollary of
Proposition \ref{prop:tau-halfperiod} (1) and
Theorem  \ref{thm:explicitsol}, we once again
obtain the periodicity of
$\EuScript{T}_{2}(A_r)$
(Corollary \ref{cor:SL1} for $X_r=A_r$).

\subsection{Explicit formula by initial variables for $\EuScript{T}_2(D_r)$}

The method is parallel to
the $A_r$ case, so that
we only present  the results for the most part.
Throughout this subsection, $I=\{1,\dots,r\}$ is
the set enumerating the diagram $D_r$ in Section
\ref{subsect:unrestrictedT}.
Recall that the (dual) Coxeter number of $D_r$ is $2r-2$.

For a family of variables $X = \{ x_a~|~ a \in I \}$,
we define 
$\gamma^{(j)}_n(X), \nu^{(j)}_n(X) \in \mathbb{Z}[X^{\pm 1}]$ 
 ($1\leq j \leq r-1, 0\leq  n\leq j $)
by Definition \ref{defn:gamma}.
We set
\begin{align}
\alpha_n =  \gamma^{(r-1)}_{n}(X),\quad
\beta_n = \nu^{(r-1)}_{n}(X),\quad
z = x_{r-1} x_{r}.
\end{align}
Note that $\alpha_{r-1} = \beta_0 = 0$, $\alpha_0 = \beta_{r-1} = 1$,
and $\beta_1=1/x_{r-2}$.

\begin{defn}\label{defbilinear}
Define $\Gamma^{(j)}_n(X), \Pi^{(j)}_n(X), \Omega^{(j)}_n(X) 
\in \mathbb{Z}[X^{\pm 1}]$ 
($0 \le j \le r-2, 0 \le n \le r-1-j$) as
\begin{align}
\Gamma^{(j)}_n(X) &=\alpha_{j}  \alpha_{j+n},  \\
\Pi^{(j)}_n(X) &= \alpha_{j} \beta_{j+n} 
               + \beta_{j} \alpha_{j+n} \frac{2 + \beta_1}{\beta_1}, \\ 
\Omega^{(j)}_n(X) &= \beta_{j} \beta_{j+n} 
                     \bigl(\frac{1+\beta_1}{\beta_1}\bigr)^2.    
\end{align}
In particular,
\begin{align}
  \label{GPO2}
  \Gamma^{(j)}_0(X) = \alpha_j^2, \quad
  \Pi^{(j)}_0(X) = 2 \alpha_j \beta_j \frac{1+\beta_1}{\beta_1}, \quad
  \Omega^{(j)}_0(X) = \beta_j \bigl(\frac{1+\beta_1}{\beta_1}\bigr)^2.
\end{align}
Note that they are independent of $x_{r-1}, x_{r}$.
\end{defn}

\begin{lem}\label{abreduction}
The following relations hold:
\begin{align}
  \label{ab-rel1}
  &\alpha_{n} \beta_{n+1} -  \alpha_{n+1} \beta_{n} = \beta_1 
  \quad (0 \le n \le r-2),
  \\
  \label{ab-rel2}
  \begin{split}
  &\Gamma^{(j)}_n(X)  \Omega^{(j+1)}_n(X) 
     + \Gamma^{(j+1)}_n(X)  \Omega^{(j)}_n(X)
     \\ 
     & \quad - \Gamma^{(j)}_{n+1}(X)  \Omega^{(j+1)}_{n-1}(X)
     - \Gamma^{(j+1)}_{n-1}(X)  \Omega^{(j)}_{n+1}(X)
     = (\beta_1 + 1)^2
  \end{split}\\
  &\hspace{5cm} (0 \le j \le r-3,~1 \le n \le r-2 -j),\notag
  \\
  \label{ab-rel3}
  &\Pi^{(j)}_n(X) \Pi^{(j+1)}_n(X) - \Pi^{(j)}_{n+1}(X) \Pi^{(j+1)}_{n-1}(X) 
      = - \beta_1 (2 + \beta_1)
  \\ 
  &\hspace{5cm} (0 \le j \le r-3,~1 \le n \le r-2 -j). 
  \nonumber
\end{align}
\end{lem}
\begin{proof}
  It is easy to check \eqref{ab-rel1}.
  The rest  are the consequences of Definition \ref{defbilinear}
  and \eqref{ab-rel1}.      
\end{proof}

We define a family
$  \tau = \{\tau^{(a)}(u;X) \in \mathbb{Z}[X^{\pm 1}]
         \mid a \in I, ~ u \in \mathbb{Z},~
 \text{$a+u$  is odd/even\ }
  \allowbreak
 \text{if $a\neq r$/$a=r$} \}
 $
as follows:
First, we define,
 for $1\leq a \leq r-2$ and $0\leq k\leq r-1$,
\begin{align}\label{tau-Dr1}
  \tau^{(a)}(a-1+2k;X) = 
  \begin{cases}
  \gamma^{(a+k)}_k(X) x_{a+k} + \nu^{(a+k)}_k(X) \quad   (0 \le k \le r-a-2),
  \\
  \Gamma^{(a+k-r+1)}_{r-a-1}(X) z + \Pi^{(a+k-r+1)}_{r-a-1}(X) 
  + \displaystyle{\frac{\Omega^{(a+k-r+1)}_{r-a-1}(X)}{z}}
  \\ 
  \hspace*{4.5cm} (r-a-1 \le k \le r-1),
  \\
  \end{cases}
\end{align}
and, for $a=r-1,r$ and $0\leq k\leq r-1$,
\begin{align}\label{tau-Dr2}
  \tau^{(r-1)}(r-2+2k;X) &=  
  \begin{cases}
    \displaystyle{
    \frac{\alpha_k z + \beta_k (1+\beta_1)/\beta_1}{x_{r-1}}} 
    \quad (k:  {\rm odd}),
    \\
    \displaystyle{
    \frac{\alpha_k z + \beta_k (1+\beta_1)/\beta_1}{x_{r}}}   
    \quad (k: {\rm even}),   
  \end{cases}\\
\label{tau-Dr3}
  \tau^{(r)}(r-2+2k;X) &=  
  \begin{cases}
    \displaystyle{
    \frac{\alpha_k z + \beta_k (1+\beta_1)/\beta_1}{x_{r}}}   
    \quad (k:  {\rm odd}),
    \\
    \displaystyle{
    \frac{\alpha_k z + \beta_k (1+\beta_1)/\beta_1}{x_{r-1}}} 
    \quad (k: {\rm even}).
  \end{cases}
\end{align}
Then, we extend the definition by
the following half-periodicity:
\begin{align}
\label{eq:dhalf1}
  \tau^{(a)}(u+2r;X) = \tau^{(\omega(a))}(u;X),
\end{align}
{}from which the periodicity $\tau^{(a)}(u+4r;X) = \tau^{(a)}(u;X)$
 also follows.

By (\ref{eq:gn0}), we have
\begin{align}
  \tau^{(a)}(a-1;X)= x_a
\quad (a=1,\dots,r-1),\quad
  \tau^{(r)}(r-2;X)= x_{r}.
\end{align}

See Figure \ref{pic:tau-Dr} for the half of the fundamental domain
of $\tau^{(a)}(u;X)$ which consists of three parts.
The domains $D_1$ and $D_2$ correspond to
the two cases of \eqref{tau-Dr1}, while
$D_3$ does the cases of \eqref{tau-Dr2} and \eqref{tau-Dr3}:  
\begin{align}
\begin{split}
  &D_1 = \{(a,a-1+2k)  \mid a=1,\dots,r-2;
 ~0 \le k \le r-a-2 \},
  \\
  &D_2 = \{(a,a-1+2k)  \mid a=1,\dots,r-2;
 ~r-a-1 \le k \le r-1 \},
  \\
  &D_3 = \{(r-1,r-2+2k), (r,r-2+2k)  \mid 0 \le k \le r-1 \}.
\end{split}
\end{align}

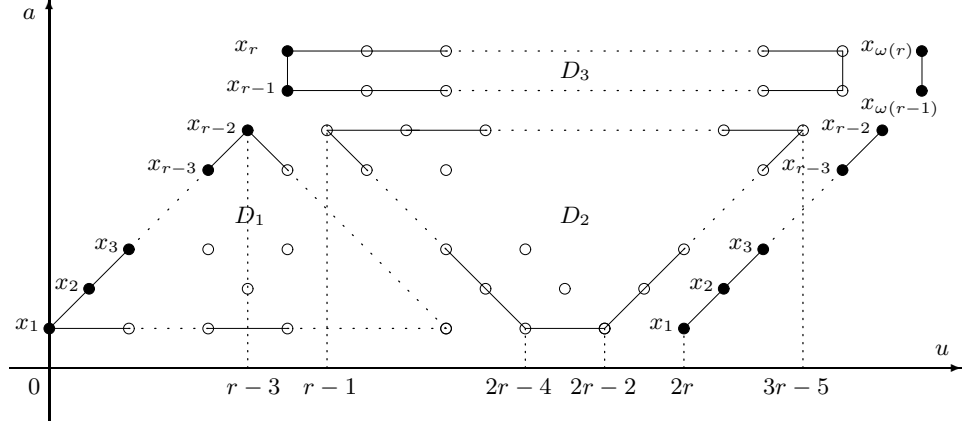
\begin{figure}
\begin{picture}(300,150)(19,0)

\put(-10,20){\vector(1,0){360}}
\put(5,0){\vector(0,1){160}}
\put(340,25){\small$u$}
\put(-5,152){\small$a$}

\put(125,125){\circle{4}}
\put(155,125){\circle{4}}
\put(275,125){\circle{4}}
\put(305,125){\circle{4}}
\put(335,125){\circle*{4}}
\put(125,140){\circle{4}}
\put(155,140){\circle{4}}
\put(275,140){\circle{4}}
\put(305,140){\circle{4}}
\put(335,140){\circle*{4}}

\put(5,35){\circle*{4}}
\put(-8,35){\small$x_1$}
\put(20,50){\circle*{4}}
\put(7,50){\small$x_2$}
\put(35,65){\circle*{4}}
\put(22,65){\small$x_3$}
\put(65,95){\circle*{4}}
\put(42,95){\small$x_{r-3}$}
\put(80,110){\circle*{4}}
\put(57,110){\small$x_{r-2}$}
\put(95,125){\circle*{4}}
\put(72,125){\small$x_{r-1}$}
\put(95,140){\circle*{4}}
\put(75,140){\small$x_{r}$}

\put(232,35){\small$x_1$}
\put(247,50){\small$x_2$}
\put(262,65){\small$x_3$}
\put(282,95){\small$x_{r-3}$}
\put(297,110){\small$x_{r-2}$}
\put(312,118){\small$x_{\omega(r-1)}$}
\put(312,140){\small$x_{\omega(r)}$}

\put(80,50){\circle{4}}
\put(65,65){\circle{4}}
\put(95,65){\circle{4}}

\put(35,35){\circle{4}}
\put(65,35){\circle{4}}
\put(95,35){\circle{4}}
\put(155,35){\circle{4}}

\drawline(5,35)(35,65)
\dottedline{5}(35,65)(65,95)
\drawline(65,95)(80,110)

\drawline(5,35)(35,35)
\dottedline{5}(35,35)(65,35)
\drawline(65,35)(95,35)
\dottedline{5}(95,35)(155,35)

\drawline(95,125)(155,125)
\dottedline{5}(155,125)(275,125)
\drawline(275,125)(305,125)
\drawline(95,140)(155,140)
\dottedline{5}(155,140)(275,140)
\drawline(275,140)(305,140)
\drawline(95,125)(95,140)
\drawline(305,125)(305,140)
\drawline(335,125)(335,140)

\put(95,95){\circle{4}}
\put(155,35){\circle{4}}

\put(110,110){\circle{4}}
\put(125,95){\circle{4}}
\put(155,65){\circle{4}}
\put(170,50){\circle{4}}
\put(185,35){\circle{4}}

\drawline(110,110)(125,95)
\dottedline{5}(125,95)(155,65)
\drawline(155,65)(185,35)

\drawline(80,110)(95,95)
\dottedline{5}(95,95)(155,35)

\put(140,110){\circle{4}}
\put(155,95){\circle{4}}
\put(185,65){\circle{4}}
\put(200,50){\circle{4}}
\put(215,35){\circle{4}}

\drawline(140,110)(170,110)
\dottedline{5}(170,110)(260,110)
\drawline(260,110)(290,110)
\drawline(290,110)(275,95)
\dottedline{5}(275,95)(245,65)
\drawline(245,65)(215,35)

\put(170,110){\circle{4}}
\put(260,110){\circle{4}}
\put(290,110){\circle{4}}
\put(275,95){\circle{4}}
\put(245,65){\circle{4}}
\put(230,50){\circle{4}}
\put(215,35){\circle{4}}

\put(320,110){\circle*{4}}
\put(305,95){\circle*{4}}
\put(275,65){\circle*{4}}
\put(260,50){\circle*{4}}
\put(245,35){\circle*{4}}

\drawline(320,110)(305,95)
\dottedline{5}(305,95)(275,65)
\drawline(275,65)(245,35)

\put(75,75){\small$D_1$}
\put(198,75){\small$D_2$}
\put(198,130){\small$D_3$}

\drawline(185,35)(215,35)
\drawline(110,110)(145,110)

\dottedline{3}(5,20)(5,115)
\put(-3,10){\small$0$}
\dottedline{3}(80,20)(80,112)
\put(72,10){\small$r-3$}
\dottedline{3}(110,20)(110,112)
\put(101,10){\small$r-1$}

\dottedline{3}(185,20)(185,35)
\put(170,10){\small$2r-4$}
\dottedline{3}(215,20)(215,35)
\put(202,10){\small$2r-2$}
\dottedline{3}(245,20)(245,35)
\put(240,10){\small$2r$}
\dottedline{3}(290,20)(290,112)
\put(275,10){\small$3r-5$}

\end{picture}
\caption{The half of the fundamental domain of $\tau^{(a)}(u;X)$
for $\EuScript{T}_2(D_r)$.}
\label{pic:tau-Dr}
\end{figure}
           
Using Lemmas \ref{lem:xg1} and \ref{abreduction},
one can verify, case by case,

\begin{prop}\label{prop:tau-halfperiod-Dr}
The family $\tau$ satisfies the $T$-system $\mathbb{T}_2(D_r)$
in $\mathbb{Z}[X^{\pm 1}]$:
\begin{align}\label{tau-rel-D_r}
\tau^{(a)}(u-1;X) \, \tau^{(a)}(u+1;X) 
=1+\prod_{b \in I: C_{ab} = -1} \tau^{(b)}(u;X). 
\end{align}
\end{prop}

Let
$\EuScript{T}_2(D_r)_+$
be the subring of
$\EuScript{T}_2(D_r)$ generated by
$T^{(a)}_1(u)^{\pm1}$ ($a\in I$, $u\in \mathbb{Z}$,
$a+u$ is  odd/even if $a\neq r$/$a=r$).
Set
\begin{align}
T^{\mathrm{init}} = \{T^{(a)}_1(a-1)\ (a=1,\dots,r-1),
T^{(r)}_1(r-2)
  \}.
\end{align}
We call the elements of  $T^{\mathrm{init}}$
the {\em initial variables\/} of $\EuScript{T}_2(D_r)_+$.

\begin{thm}
\label{thm:Dexplicitsol} 
The following relation holds in $\EuScript{T}_2(D_r)_+$: 
\begin{align}
  &T^{(a)}_1(u) =
 \tau^{(a)}(u;T^{\mathrm{init}}),
\end{align}
where we set $x_a = T^{(a)}_1(a-1)$ ($a\neq r$)
and $x_r = T^{(r)}_1(r-2)$
 in the right hand side.
\end{thm}

As a corollary of 
Theorem  \ref{thm:Dexplicitsol} and (\ref{eq:dhalf1}), we once again
obtain the periodicity of
$\EuScript{T}_{2}(D_r)$
(Corollary \ref{cor:SL1} for $X_r=D_r$).

\subsection{Explicit formula by initial
variables for $\EuScript{T}_2(B_r)$}

Again, the method is parallel to
the former  cases, so that
we only present  the results for the most part.
Throughout this subsection, $I=\{1,\dots,r\}$ is
the set enumerating the diagram $B_r$ in Section
\ref{subsect:unrestrictedT}.
Recall that the dual Coxeter number of $B_r$ is $2r-1$,
and the number $t_a$ in (\ref{eq:t1})
is 2 (resp.\ 1) if $a=r$ (otherwise).

For a family of variables
$X = \{x_a, \overline{x}_a
\ (a=1,\dots,r-1), w_1, w_2, w_3 \}$,
we define $\gamma^{(j)}_n(X), \nu^{(j)}_n(X)
\in \mathbb{Z}[X^{\pm 1}]$ 
($1\leq j \leq r-1$, $0\leq n\leq j$)
by Definition \ref{defn:gamma}.
We set
\begin{align}
\alpha_n =  \gamma^{(r-1)}_{n}(X),\quad
\beta_n = \nu^{(r-1)}_{n}(X),\quad
z = w_1 w_3.
\end{align}
Note that $\alpha_{r-1}=\beta_0=0$
and  $\alpha_0=\beta_{r-1}=1$.

In the following, for $p(X) \in \mathbb{Z}[X^{\pm 1}]$,
we write $\overline{p}(X)$ for what 
obtained from $p(X)$ by replacing $x_i$ with $\overline{x}_i$.
Similarly we write $\widehat{p}(X)$ for 
what obtained from $p(X)$ by swapping $w_1$ and $w_3$ .

\begin{defn} \label{defmudelta}
  We define $\mu_n(X), \delta_n(X) \in \mathbb{Z}[X^{\pm 1}]$
  ($0 \le n \le r$) as
\begin{align}
  \mu_n(X)= \frac{x_{n-1}}{x_{r-1}},
  \quad
  \delta_n(X)= \begin{cases}
               \alpha_n 
               + \displaystyle{\frac{\beta_n}{x_{r-1}}}
&0 \le n \le r-1,               \\
               0  & n=r.
             \end{cases}
\end{align}
where we set $x_0=\overline{x}_0=1, x_{-1}=\overline{x}_{-1}=0$.
In particular, we have
$\mu_0(X) = 0$, $\mu_r(X) = 1$, 
$\delta_0(X)=1$ and $\delta_{r-1}(X)=1/x_{r-1}$. 
\end{defn}

Note that $\mu_n(X)$ and $\delta_n(X)$ depend on $x_{r-1}$,
while $\gamma^{(j)}_n(X)$ and $\nu^{(j)}_n(X)$ do not.

\begin{defn} \label{defetaxiPQR}
We define $P^{(j)}_n(X), Q^{(j)}_n(X), R^{(j)}_n(X), \eta_n(X), \xi_n(X) 
\in \mathbb{Z}[X^{\pm 1}]$ as follows:
For $0 \le j \le r, ~0 \le k \le r, ~1 \le j+k \le r+1$,
\begin{align}
  \begin{split}
    &P^{(j)}_{2k}(X) = \frac{x_{k-1}}{z} \overline{\delta}_{j-1+k}(X), \\
    &Q^{(j)}_{2k}(X) = x_{k-1} \overline{\mu}_{j+k-1}
     + \overline{\delta}_{j+k-1}(X) \overline{x}_{r-1}
     \Bigl( \delta_k(X) +\frac{2 x_{k-1}}{z} \Bigr), \\
    &R^{(j)}_{2k}(X) = \overline{\delta}_{j-1+k}(X) x_{k-1} \overline{x}_{r-1} 
     \Bigl(\frac{\overline{x}_{r-1}}{z}+\frac{1}{x_{r-1}} \Bigr).
  \end{split}
\end{align}
For $0 \le j \le r, ~0 \le k \le r, ~0 \le j+k \le r$,
\begin{align}
  \begin{split}
    &P^{(j)}_{2k+1}(X)  = \overline{\delta}_{k}(X)\delta_{j+k}(X), \\
    &Q^{(j)}_{2k+1}(X)  =  z \delta_{j+k}(X) \overline{\mu}_k(X) 
       + \overline{\delta}_k(X) \mu_{j+k}(X),    \\    
    &R^{(j)}_{2k+1}(X) = z x_{j-1+k} \overline{\mu}_{k}(X)
       \Bigl(  \frac{\overline{x}_{r-1}}{z}+\frac{1}{x_{r-1}} \Bigr).
  \end{split}
\end{align}
And,
\begin{align}
\begin{split}
  &\eta_{2k}(X)  =  w_1 \delta_{k}(X) + \frac{x_{k-1}}{w_3}, \quad 
  \xi_{2k}(X) =  w_1 x_{k-1} 
    \Bigl(  \frac{\overline{x}_{r-1}}{z } + \frac{1}{x_{r-1}}
    \Bigr) 
  \qquad (0 \le k \le r),
  \\
  &\eta_{2k+1}(X) = \frac{\overline{\delta}_k(X) }{w_1 }, \quad
  \xi_{2k+1}(X) = \frac{1}{w_1}
    \bigl( \overline{x}_{r-1} \overline{\delta}_k(X)+z \overline{\mu}_{k}(X) \bigr)
  \qquad (0 \le k \le r-1). 
\end{split}
\end{align}
In particular, we have
$Q_1^{(0)}(X) = R_1^{(0)}(X) = P_{2r+1}^{(0)}(X) = Q_{2r+1}^{(0)}(X) = 0$
and $P_1^{(0)}(X) = 1$.
\end{defn}

\begin{lem}\label{dmreduction} 
The following relations hold:
For $1\le k \le r$,
\begin{align}\label{delta-x}
  \begin{split}
  &\mu_{k-1}(X) x_{k-1} -  \mu_{k}(X)  x_{k-2}=
  \overline{\mu}_{k-1}(X) \overline{x}_{k-1} -  \overline{\mu}_{k}(X)
  \overline{x}_{k-2}=0,
  \\
  &\delta_{k-1}(X) x_{k-1} -  \delta_{k}(X)  x_{k-2}=
  \overline{\delta}_{k-1}(X) \overline{x}_{k-1}
 -  \overline{\delta}_{k}(X)  \overline{x}_{k-2}=1.
  \end{split}
\end{align}
For $1 \le j \le r-1, ~0 \le k \le 2r-2, ~1 \le j+[\frac{k}{2}] \le r$;
\begin{align}
\begin{split}
  &Q_k^{(j)}(X) R_{k+2}^{(j)}(X) + Q_{k+2}^{(j)}(X) R_{k}^{(j)}(X) 
  \\ &\qquad
   - Q_{k+2}^{(j-1)}(X) R_{k}^{(j+1)}(X) - Q_{k}^{(j+1)}(X) R_{k+2}^{(j-1)}(X)
   = 0,
  \\ 
  &Q_k^{(j)}(X) P_{k+2}^{(j)}(X) + Q_{k+2}^{(j)}(X) P_{k}^{(j)}(X)
  \\ &\qquad 
   - Q_{k+2}^{(j-1)}(X) P_{k}^{(j+1)}(X) - Q_{k}^{(j+1)}(X) P_{k+2}^{(j-1)}(X)
   = 0,
  \\ 
  &Q_k^{(j)}(X) Q_{k+2}^{(j)}(X) - Q_{k+2}^{(j-1)}(X) Q_{k}^{(j+1)}(X) = 
  \begin{cases} -z \mu_1(X) \overline{\mu}_1(X) &k : \text{odd}, \\  
                1  & k : \text{even}, \end{cases}
  \\ 
  &P_k^{(j)}(X) R_{k+2}^{(j)}(X) + P_{k+2}^{(j)}(X) R_{k}^{(j)}(X)
   - P_{k+2}^{(j-1)}(X) R_{k}^{(j+1)}(X)\\
&\qquad - P_{k}^{(j+1)}(X) R_{k+2}^{(j-1)}(X)
  = \begin{cases} 1 + z \mu_1(X) \overline{\mu}_1(X) 
                         & k : \text{odd}, \\  
                  0  &k : \text{even}. \end{cases}
\end{split}
\end{align}
For $1 \le k \le 2r$,

\begin{align}
\begin{split}
  &Q_k^{(0)}(X) R_{k+1}^{(0)}(X) + Q_{k+1}^{(0)}(X) R_{k}^{(0)}(X) 
   - Q_{k}^{(1)}(X) R_{k-1}^{(1)}(X) - Q_{k-1}^{(1)}(X) R_{k}^{(1)}(X)
   \\ &\qquad
   = \begin{cases} 0 &k : \text{odd}, \\  
                   \xi_k(X) \widehat{\eta}_k(X) + \widehat{\xi}_k(X) \eta_k(X) 
                   & k : \text{even}, \end{cases}
  \\ 
  &Q_k^{(0)}(X) P_{k+1}^{(0)}(X) + Q_{k+1}^{(0)}(X) P_{k}^{(0)}(X)
   - Q_{k}^{(1)}(X) P_{k-1}^{(1)}(X) - Q_{k-1}^{(1)}(X) P_{k}^{(1)}(X)
     \\ &\qquad 
   = \begin{cases} \xi_k(X) \widehat{\eta}_k(X) + \widehat{\xi}_k(X) \eta_k(X) 
                   & k : \text{odd}, \\
                    0 &  k : \text{even}, \end{cases}
  \\ 
  &P_k^{(0)}(X) R_{k+1}^{(0)}(X) + P_{k+1}^{(0)}(X) R_{k}^{(0)}(X)
   - P_{k}^{(1)}(X) R_{k-1}^{(1)}(X) - P_{k-1}^{(1)}(X) R_{k}^{(1)}(X)
    \\ &\qquad 
    + Q_k^{(0)}(X) Q_{k+1}^{(0)}(X) - Q_{k+1}^{(1)}(X) Q_{k}^{(1)}(X)  
    = \begin{cases} \xi_k(X) \widehat{\xi}_k(X) & k : \text{odd}, \\  
                    \eta_k(X) \widehat{\eta}_k(X) 
                    & k : \text{even}. \end{cases}
\end{split}
\end{align}

\end{lem}
\begin{proof}
  It is easy to check \eqref{delta-x}.
  The rest are the consequence of Definitions
  \ref{defmudelta}, \ref{defetaxiPQR} and \eqref{delta-x}.
\end{proof}

For a triplet $(a,m,u)$\ ($a\in I;
m=1,\dots,2t_a-1; u\in \frac{1}{2}\mathbb{Z}$),
 we set the condition,
\begin{align}
\text{Condition (P):}\quad
          u \in 
            \begin{cases} \mathbb{Z} + \frac{1}{2} 
                          &(a,m) = (r,1), (r,3), \\ 
                          \mathbb{Z} &\text{otherwise}. \end{cases}
\end{align}

We define a family
$  \tau = \{
          \tau^{(a)}_{m}(u;X) \in \mathbb{Z}[X^{\pm 1}] \mid
          a \in I; m=1,\ldots,2t_a-1; u\in \textstyle
\frac{1}{2}\mathbb{Z}; \text{Condition (P)}
         \}
$ 
as follows:
First, we define $\tau^{(a)}_m(u;X)$ in the following region:
\begin{align}
  \label{btau-Dr1}
  &\begin{cases}
  \tau^{(a)}_1(a-1+2k;X) 
  = \gamma^{(a+k)}_k (X) x_{a+k} + \nu^{(a+k)}_k  (X)
  \\ \quad \quad \quad 
  \quad (1\le a \le r-1, ~0\le k \le r-1-a),
  \\ \displaybreak[0]
  \tau^{(a)}_1(a+2k;X) 
  = \overline{\tau}^{(a)}_1(a-1+2k;X)
  \\ 
  \quad \quad \quad 
  \quad (1\le  a \le r-1, ~0\le k \le r-1-a),
  \end{cases} 
  \\ \displaybreak[0]
  \label{btau-Dr2}
  &\begin{cases}
  \tau^{(a)}_1\bigl(a+ 2(r-1-a)+k;X \bigr) 
  =P^{(r-a)}_k(X) w_2 + Q^{(r-a)}_k (X)
   +  \displaystyle{\frac{R^{(r-a)}_k(X)}{w_2}}
  \\
  \quad \quad \quad \quad (1\le a \le r-1, ~1\le k \le 2a+1),
  \\ 
  \tau^{ (r)}_2 (r-1+k;X)   
  = P^{(0)}_{k+1}(X) w_2 + Q^{(0)}_{k+1}(X)
    + \displaystyle{\frac{R^{(0)}_{k+1}(X)}{w_2}}
  \quad (0 \le k \le 2r),
  \end{cases}
  \\ \displaybreak[0]
  \label{btau-Dr3}
  &\begin{cases}
  \tau^{ (r)}_1 (r+2k -\frac{1}{2};X) 
  = \eta_{2k+1}(X) w_2 +\xi_{2k+1}(X)  \quad (0 \le k \le r-1),
  \\
  \tau^{(r)}_1 (r+2k -\frac{3}{2};X) 
  =   \displaystyle{\eta_{2k}(X) + \frac{\xi_{2k}(X)}{w_2} \quad (0 \le k \le r)},
  \\
  \tau^{ (r)}_3 (r+k-\frac{3}{2};X)   
  = \widehat{\tau}^{(r)}_1(r+k -\frac{3}{2};X)
  \quad (0 \le k \le 2r).
  \\
  \end{cases}
\end{align}
Then, we extend the definition by the following 
half-periodicity:
\begin{align}
\label{eq:dhalf2}
  \tau^{(a)}_m(u+2r+1;X) = \widehat{\tau}^{(a)}_m(u;X),
\end{align}
from which the periodicity 
$\tau^{(a)}_m(u+2(2r+1);X) = \tau^{(a)}_m(u;X)$ also follows.

By (\ref{eq:gn0}), we have
\begin{align}
\begin{split}
  &\tau^{(a)}_1(a-1;X)= x_a,\quad
  \tau^{(a)}_1(a;X)= \overline{x}_a
\quad (a=1,\dots,r-1),\\
  &\tau^{(r)}_1(r-\textstyle\frac{3}{2};X)= w_{1},\quad
  \tau^{(r)}_2(r-1;X)= w_{2},\quad
  \tau^{(r)}_3(r-\textstyle\frac{3}{2};X)= w_{3}.
\end{split}
\end{align}

\begin{figure}
\begin{picture}(300,175)(19,0)

\put(-10,20){\vector(1,0){360}}
\put(5,0){\vector(0,1){175}}
\put(340,25){\small$u$}
\put(-5,167){\small$a$}

\put(110,125){\circle{4}}
\put(125,125){\circle{4}}
\put(140,125){\circle{4}}
\put(155,125){\circle{4}}
\put(275,125){\circle{4}}
\put(305,125){\circle{4}}
\put(102.5,140){\circle{4}}
\put(117.5,140){\circle{4}}
\put(132.5,140){\circle{4}}
\put(147.5,140){\circle{4}}
\put(267.5,140){\circle{4}}
\put(282.5,140){\circle{4}}
\put(297.5,140){\circle{4}}
\put(102.5,155){\circle{4}}
\put(117.5,155){\circle{4}}
\put(132.5,155){\circle{4}}
\put(147.5,155){\circle{4}}
\put(267.5,155){\circle{4}}
\put(282.5,155){\circle{4}}
\put(297.5,155){\circle{4}}

\put(5,35){\circle*{4}}
\put(-8,35){\small$x_1$}
\put(20,50){\circle*{4}}
\put(7,50){\small$x_2$}
\put(35,65){\circle*{4}}
\put(22,65){\small$x_3$}

\put(65,95){\circle*{4}}
\put(42,95){\small$x_{r-2}$}
\put(80,110){\circle*{4}}
\put(57,110){\small$x_{r-1}$}

\put(87.5,140){\circle*{4}}
\put(71,142){\small$w_1$}
\put(87.5,155){\circle*{4}}
\put(71,157){\small$w_3$}
\put(95,110){\circle{4}}
\put(80,127){\small$w_2$}
\put(95,125){\circle*{4}}


\put(20,35){\circle*{4}}
\put(23,38){\small$\overline{x}_1$}
\put(35,50){\circle*{4}}
\put(38,53){\small$\overline{x}_2$}
\put(50,65){\circle*{4}}
\put(53,67){\small$\overline{x}_3$}

\put(80,95){\circle*{4}}
\put(95,110){\circle*{4}}

\put(35,50){\circle{4}}
\put(65,50){\circle{4}}
\put(80,50){\circle{4}}
\put(95,50){\circle{4}}
\put(110,50){\circle{4}}
\put(185,50){\circle{4}}
\put(215,50){\circle{4}}
\put(65,65){\circle{4}}
\put(80,65){\circle{4}}
\put(95,65){\circle{4}}
\put(110,65){\circle{4}}
\put(170,65){\circle{4}}
\put(185,65){\circle{4}}
\put(200,65){\circle{4}}
\put(215,65){\circle{4}}
\put(230,65){\circle{4}}
\put(80,95){\circle{4}}
\put(140,95){\circle{4}}
\put(170,95){\circle{4}}
\put(185,95){\circle{4}}
\put(125,110){\circle{4}}
\put(155,110){\circle{4}}
\put(275,110){\circle{4}}
\put(290,125){\circle{4}}

\put(20,35){\circle{4}}
\put(35,35){\circle{4}}
\put(65,35){\circle{4}}
\put(80,35){\circle{4}}
\put(95,35){\circle{4}}
\put(110,35){\circle{4}}
\put(140,35){\circle{4}}
\put(170,35){\circle{4}}
\put(155,35){\circle{4}}
\put(200,35){\circle{4}}

\drawline(5,35)(35,65)
\dottedline{5}(35,65)(65,95)
\drawline(65,95)(80,110)

\drawline(5,35)(35,35)
\dottedline{5}(35,35)(65,35)
\drawline(65,35)(110,35)
\dottedline{5}(110,35)(155,35)

\drawline(87.5,140)(147.5,140)
\dottedline{5}(155,125)(275,125)
\drawline(275,125)(305,125)
\drawline(87.5,155)(147.5,155)
\dottedline{5}(147.5,140)(267.5,140)
\drawline(267.5,140)(297.5,140)
\drawline(87.5,140)(87.5,155)
\drawline(297.5,140)(297.5,155)
\drawline(267.5,155)(297.5,155)
\dottedline{5}(147.5,155)(267.5,155)

\put(95,95){\circle{4}}
\put(155,35){\circle{4}}

\put(110,110){\circle{4}}
\put(125,95){\circle{4}}
\put(155,65){\circle{4}}
\put(170,50){\circle{4}}
\put(185,35){\circle{4}}

\drawline(95,125)(125,95)
\dottedline{5}(125,95)(155,65)
\drawline(155,65)(185,35)
\drawline(95,110)(110,95)
\dottedline{5}(110,95)(140,65)
\drawline(140,65)(170,35)

\put(110,95){\circle{4}}
\put(140,65){\circle{4}}
\put(155,50){\circle{4}}
\put(140,50){\circle{4}}
\drawline(80,110)(95,110)
\drawline(140,35)(170,35)

\put(140,110){\circle{4}}
\put(155,95){\circle{4}}
\put(185,65){\circle{4}}
\put(200,50){\circle{4}}
\put(215,35){\circle{4}}

\drawline(305,125)(275,95)

\dottedline{5}(275,95)(245,65)
\drawline(245,65)(215,35)

\put(170,110){\circle{4}}
\put(290,110){\circle{4}}
\put(275,95){\circle{4}}
\put(245,65){\circle{4}}
\put(230,50){\circle{4}}

\put(320,125){\circle*{4}}
\put(305,110){\circle*{4}}
\put(290,95){\circle*{4}}
\put(260,65){\circle*{4}}
\put(245,50){\circle*{4}}
\put(230,35){\circle*{4}}
\put(312.5,140){\circle*{4}}
\put(312.5,155){\circle*{4}}
\put(237,35){\small$x_1$}
\put(252,50){\small$x_2$}
\put(267,65){\small$x_3$}
\put(297,95){\small$x_{r-2}$}
\put(312,110){\small$x_{r-1}$}
\put(327,125){\small$w_2$}
\put(318,138){\small$w_3$}
\put(318,153){\small$w_1$}

\drawline(312.5,140)(312.5,155)
\drawline(305,110)(290,95)

\dottedline{5}(290,95)(260,65)
\drawline(260,65)(230,35)


\put(83,78){\small$D_1$}
\put(198,78){\small$D_2$}
\put(198,145){\small$D_3$}

\drawline(185,35)(215,35)
\drawline(95,125)(155,125)

\dottedline{3}(5,20)(5,115)
\put(-3,10){\small$0$}
\dottedline{3}(80,20)(80,112)
\put(65,10){\small$r-2$}
\dottedline{3}(95,20)(95,125)
\put(90,10){\small$r-1$}

\dottedline{3}(185,20)(185,35)
\put(170,10){\small$2r-2$}
\dottedline{3}(215,20)(215,35)
\put(212,10){\small$2r$}
\dottedline{3}(230,20)(230,35)
\put(227,10){\small$2r+1$}
\dottedline{3}(305,20)(305,125)
\put(292,10){\small$3r-1$}

\dottedline{3}(5,125)(95,125)
\put(-6,122){\small$r_2$}
\dottedline{3}(5,140)(83,140)
\put(-6,138){\small$r_1$}
\dottedline{3}(5,155)(83,155)
\put(-6,153){\small$r_3$}

\end{picture}
\caption{The half of the fundamental domain of $\tau^{(a)}_m(u;X)$
for $\EuScript{T}_2(B_r)$.}
\label{pic:tau-Br}
\end{figure}

See Figure \ref{pic:tau-Br} for the half of the fundamental domain
of $\tau^{(a)}_m(u;X)$ which consists of three parts.
The domains $D_1$, $D_2$ and $D_3$ correspond to
the cases \eqref{btau-Dr1}, \eqref{btau-Dr2} and \eqref{btau-Dr3},
respectively:
\begin{align}
\begin{split}
  &D_1 = \{(a,a-1+k) \mid 1\leq a\leq  r-1; ~0 \le k \le 2(r-a)-1 \}, 
  \\
  &D_2 = \{(a,a-1+k) \mid  1\leq a\leq r-1; ~2(r-a) \le k \le 2r \}\\
&\qquad\qquad
         \sqcup \{(r_2,r-1+k) \mid 0 \le k \le 2r \},
  \\
  &D_3 = \{(r_1,r+k-\frac{3}{2}), (r_3,r+k-\frac{3}{2}) 
           ~|~ 0 \le k \le 2r \},
\end{split}
\end{align}
where the point  $(r_i,u)$ corresponds to $\tau^{(r)}_i(u)$.

Using Lemmas \ref{lem:xg1} and \ref{dmreduction},
one can verify, case by case,

\begin{prop}\label{explicitsolBr}
  The family $\tau$ satisfies the T-system $\mathbb{T}_2(B_r)$
 in $\mathbb{Z}[X^{\pm 1}]$ (by replacing $T^{(a)}_m(u)$
 in $\mathbb{T}_2(B_r)$ with  $\tau^{(a)}_m(u;X)$).
\end{prop}

Let
$\EuScript{T}_2(B_r)_+$
be the subring of
$\EuScript{T}_2(B_r)$ generated by
$T^{(a)}_m(u)^{\pm1}$ ($a\in I;
m=1,\dots,2t_a -1;
u\in \frac{1}{2}\mathbb{Z}$;
Condition (P)).
Set
\begin{align}
\begin{split}
T^{\mathrm{init}} &= \{T^{(a)}_1(a-1),
T^{(a)}_1(a) \ (a=1,\dots,r-1),\\
&\hskip100pt
T^{(r)}_1(r-\textstyle
\frac{3}{2}),T^{(r)}_2(r-1),T^{(r)}_3(r-\frac{3}{2})
  \}.
\end{split}
\end{align}
We call the elements of  $T^{\mathrm{init}}$
the {\em initial variables\/} of $\EuScript{T}_2(B_r)_+$.

\begin{thm}
\label{thm:Bexplicitsol} 
The following relation holds in $\EuScript{T}_2(B_r)_+$: 
\begin{align}
  &T^{(a)}_m(u) =
 \tau^{(a)}_m(u;T^{\mathrm{init}}),
\end{align}
where we set
$x_a = T^{(a)}_1(a-1)$, $\overline{x}_a = T^{(a)}_1(a)$\  ($a\neq r$),
$w_1 = T^{(r)}_1(r-\frac{3}{2})$, 
$w_2 = T^{(r)}_2(r-1)$, 
and $w_3 = T^{(r)}_3(r-\frac{3}{2})$
 in the right hand side.
\end{thm}

As a corollary of 
Theorem  \ref{thm:Bexplicitsol} and (\ref{eq:dhalf2}), we 
obtain

\begin{cor}
The following relations hold in  $\EuScript{T}_{2}(B_r)$:

(1) Half-periodicity: $T^{(a)}_m(u+2r+1)=
T^{(a)}_{2t_a-m}(u)$.

(2) Periodicity: $T^{(a)}_m(u+2(2r+1))=
T^{(a)}_{m}(u)$.
\end{cor}

\section{Periodicities of restricted T and
Y-systems at levels 1 and 0
}
\label{sect:lev10}

So far, we assumed that the level $\ell$ for the restriction
is greater than or equals to 2.
In this section we extend the periodicity property
of the restricted T and Y-systems at
{\em levels $1$ and $0$}.

\subsection{Periodicities of
restricted T and Y-systems at level 1}
\label{subsect:period1}
In the systems ${\mathbb T}_{\ell}(X_r)$ and 
${\mathbb Y}_{\ell}(X_r)$, 
we treat the variables $T^{(a)}_m(u)$ and 
$Y^{(a)}_m(u)$ with $m=1,\ldots, t_a\ell-1$. 
Thus at $\ell=1$, these systems are {\em void\/} for simply laced $X_r$;
however, Definitions \ref{defn:RT},
\ref{defn:RT2},
\ref{defn:RY}, and \ref{defn:RY2}
 still make sense for nonsimply laced $X_r$.

The level 1 T and Y-systems in these cases
are actually equivalent to the systems of type $A$.
To illustrate, consider ${\mathbb T}_1(F_4)$:
\begin{align}
\label{eq:f41}
\begin{split}
\textstyle T^{(3)}_1(u-\frac{1}{2})T^{(3)}_1(u+\frac{1}{2})
&= 1 + T^{(4)}_1(u),\\
\textstyle T^{(4)}_1(u-\frac{1}{2})T^{(4)}_1(u+\frac{1}{2})
&= 1 + T^{(3)}_1(u),
\end{split}
\end{align}
where we have omitted the first three relations in (\ref{eq:TF1}),
which are void at $\ell=1$.
To be precise, let us introduce another level $\ell$
restricted T-system ${ \mathbb T}'_{\ell}(A_r)$
for $T=\{ T^{(a)}_m(u) \mid a=1,\dots,r;\
m=1,\dots,\ell-1; u\in \frac{1}{\ell}\mathbb{Z} \}$ with
the relations
\begin{align}
\label{eq:t'1}
T^{(a)}_m(u-\textstyle\frac{1}{\ell})
T^{(a)}_m(u+\textstyle\frac{1}{\ell})
=
T^{(a)}_{m-1}(u)T^{(a)}_{m+1}(u)
+
T^{(a-1)}_{m}(u)T^{(a+1)}_{m}(u),
\end{align}
where the left hand side of (\ref{eq:t'1}) differs from (\ref{eq:TA1})
for ${ \mathbb T}_{\ell}(A_r)$.
Then, the relations in (\ref{eq:f41}) are equivalent to
${\mathbb T}'_2(A_2)$.
In other words, $\EuScript{T}_{1}(F_4)\simeq \EuScript{T}'_{2}(A_2)$,
where $\EuScript{T}'_{\ell}(A_r)$ denotes the T-algebra associated 
with ${ \mathbb T}'_{\ell}(A_r)$.
A similar reduction of the  $Y$-system,
$\EuScript{Y}_1(F_4)\simeq \EuScript{Y}'_2(A_2)$,
happens, where ${\mathbb Y}'_{\ell}(A_r)$
and $\EuScript{Y}'_{\ell}(A_r)$ are defined in the same way.

In general,
$\EuScript{T}_1(X_r)\simeq
\EuScript{T}'_t(A_{r'})$
 and 
$\EuScript{Y}_1(X_r)\simeq
\EuScript{Y}'_t(A_{r'})$
hold
 for nonsimply laced $X_r$,
where $t$ is the number in (\ref{eq:t1}) and $r'$ equals the number of 
the short simple roots of $X_r$.

Let us summarize the relevant data for the periodicities
of $\EuScript{T}_1(X_r)$ and $\EuScript{Y}_1(X_r)$:
\begin{align}
\label{eq:XAtab}
\begin{tabular}{c||c|c|c|c}
$X_r$ & $A_{r'}$ & $t$& $h^\vee+1$ & $(r'+1+t)/t$\\
\hline
$B_r$ & $A_1$ & 2& $2r$ & 2\\
$C_r$ & $A_{r-1}$ & 2 & $r+2$ & $(r+2)/2$\\
$F_4$ & $A_2$ & 2 & $10$ & 5/2\\
$G_2$ & $A_1$ & 3 & $5$ & $5/3$
\end{tabular}
\end{align}

\begin{thm}\label{thm:TY1}
(i) For any nonsimply laced $X_r$,
the following relations hold in  $\EuScript{T}_{1}(X_r)$:

(1) Half-periodicity: $T^{(a)}_m(u+h^\vee+1)= T^{(a)}_{t_a-m}(u)$.

(2) Periodicity: $T^{(a)}_m(u+2(h^\vee+1))= T^{(a)}_m(u)$.
\par\noindent
(ii) For any nonsimply laced $X_r$,
the following relations hold in  $\EuScript{Y}_{1}(X_r)$:

(1) Half-periodicity: $Y^{(a)}_m(u+h^\vee+1)= Y^{(a)}_{t_a-m}(u)$.

(2) Periodicity: $Y^{(a)}_m(u+2(h^\vee+1))= Y^{(a)}_m(u)$.
\end{thm}

\begin{proof}
By inspecting (\ref{eq:XAtab}),
one can check that the half-periodicities of 
$\EuScript{T}_1(X_r)$ and $\EuScript{Y}_1(X_r)$
follow from the half- or full-periodicities of 
the corresponding 
$\EuScript{T}'_{t}(A_{r'})$ and $\EuScript{Y}'_{t}(A_{r'})$.
\end{proof}

\subsection{Periodicities of restricted T and Y-systems at level 0}
\label{subsect:period0}
At level 0, one can still introduce, at least formally, a restricted 
T and  Y-system for any $X_r$, and study their periodicity.

\begin{defn}
The {\em level 0 restricted T-system ${\mathbb T}_0(X_r)$
  of type $X_r$}
is the following system of relations for
a family of variables
$T=\{T^{(a)}(u)\mid a \in I,\ u \in U\}$,
where $T^{(0)}(u)=1$ if they occur in the right hand sides
in the relations:

For simply laced $X_r$,
\begin{align}
\label{ak:eq:TADE0}
T^{(a)}(u-1)T^{(a)}(u+1)
=\prod_{b\in I: C_{ab}=-1}T^{(b)}(u).
\end{align}

For $X_r=B_r$,
\begin{align}
\label{ak:eq:TB0}
T^{(a)}(u-1)T^{(a)}(u+1)
&=T^{(a-1)}(u)T^{(a+1)}(u)
\quad (1\leq a\leq r-2),\\
T^{(r-1)}(u-1)T^{(r-1)}(u+1)
&=T^{(r-2)}(u)T^{(r)}(u),\notag\\
T^{(r)}\left(u-\textstyle\frac{1}{2}\right)
T^{(r)}\left(u+\textstyle\frac{1}{2}\right)
&=
T^{(r-1)}\left(u-\textstyle\frac{1}{2}\right)
T^{(r-1)}\left(u+\textstyle\frac{1}{2}\right).
\notag
\end{align}

For $X_r=C_r$,
\begin{align}
\label{ak:eq:TC0}
T^{(a)}\left(u-\textstyle\frac{1}{2}\right)
T^{(a)}\left(u+\textstyle\frac{1}{2}\right)
&=T^{(a-1)}(u)T^{(a+1)}(u)
\quad
 (1\leq a\leq r-2),\\
T^{(r-1)}\left(u-\textstyle\frac{1}{2}\right)
T^{(r-1)}\left(u+\textstyle\frac{1}{2}\right)
&=
T^{(r-2)}(u)
T^{(r)}\left(u-\textstyle\frac{1}{2}\right)
T^{(r)}\left(u+\textstyle\frac{1}{2}\right)
,\notag\\
T^{(r)}(u-1)
T^{(r)}(u+1)
&=
T^{(r-1)}(u).
\notag
\end{align}

For $X_r=F_4$,
\begin{align}
\label{ak:eq:TF0}
T^{(1)}(u-1)T^{(1)}(u+1)
&=T^{(2)}(u),\\
T^{(2)}(u-1)T^{(2)}(u+1)
&=
T^{(1)}(u)T^{(3)}(u),\notag\\
T^{(3)}\left(u-\textstyle\frac{1}{2}\right)
T^{(3)}\left(u+\textstyle\frac{1}{2}\right)
&=
T^{(2)}\left(u-\textstyle\frac{1}{2}\right)
T^{(2)}\left(u+\textstyle\frac{1}{2}\right)
T^{(4)}(u),\notag\\
T^{(4)}\left(u-\textstyle\frac{1}{2}\right)
T^{(4)}\left(u+\textstyle\frac{1}{2}\right)
&=
T^{(3)}(u).
\notag
\end{align}

For $X_r=G_2$,
\begin{align}
\label{ak:eq:TG0}
T^{(1)}(u-1)T^{(1)}(u+1)
&=
T^{(2)}(u),\\
T^{(2)}\left(u-\textstyle\frac{1}{3}\right)
T^{(2)}\left(u+\textstyle\frac{1}{3}\right)
&=
T^{(1)}\left(u-\textstyle\frac{2}{3}\right)
T^{(1)}(u)
T^{(1)}\left(u+\textstyle\frac{2}{3}\right).\notag
\end{align}
\end{defn}

\begin{defn}
The {\em level $0$ restricted T-group $\EuScript{T}_{0}(X_r)$
of type $X_r$} is the abelian group with generators
$T^{(a)}(u)$ ($a\in I, u\in U $)
and the relations $\mathbb{T}_{0}(X_r)$.
\end{defn}

\begin{rem}
${\mathbb T}_0(X_r)$ is obtained from 
the unrestricted T-system ${\mathbb T}(X_r)$
(\ref{eq:TA1})--(\ref{eq:TG1}) by 
setting $T^{(a)}_m(u) = T^{(a)}(u)$ if 
$m=0$ and  $T^{(a)}_m(u) = 0$ otherwise.
It was originally introduced in 
\cite[Sect.\ 2.2]{KNS2} as `bulk T-system'. 
\end{rem}

Similarly,
\begin{defn}
The {\em level 0 restricted Y-system ${\mathbb Y}_0(X_r)$
 of type $X_r$}
is the following system of relations for 
a family of variables
 $Y=\{Y^{(a)}(u) \mid a \in I,\ u \in  U \}$,
where $Y^{(0)}(u)=1$ if they occur in the right hand sides
in the relations:

For simply laced $X_r$,
\begin{align}
\label{ak:eq:YADE0}
Y^{(a)}(u-1)Y^{(a)}(u+1)
=\prod_{b\in I: C_{ab}=-1}Y^{(b)}(u).
\end{align}

For $X_r=B_r$,
\begin{align}
\label{ak:eq:YB0}
Y^{(a)}(u-1)Y^{(a)}(u+1)
&=Y^{(a-1)}(u)Y^{(a+1)}(u)
\quad (1\leq a\leq r-2),\\
Y^{(r-1)}(u-1)Y^{(r-1)}(u+1)
&=Y^{(r-2)}(u)
Y^{(r)}\left(u-\textstyle\frac{1}{2}\right)
Y^{(r)}\left(u+\textstyle\frac{1}{2}\right),\notag\\
Y^{(r)}\left(u-\textstyle\frac{1}{2}\right)
Y^{(r)}\left(u+\textstyle\frac{1}{2}\right)
&=Y^{(r-1)}\left(u\right).
\notag
\end{align}

For $X_r=C_r$,
\begin{align}
\label{ak:eq:YC0}
Y^{(a)}\left(u-\textstyle\frac{1}{2}\right)
Y^{(a)}\left(u+\textstyle\frac{1}{2}\right)
&=Y^{(a-1)}(u)Y^{(a+1)}(u)
\quad
 (1\leq a\leq r-2),\\
Y^{(r-1)}\left(u-\textstyle\frac{1}{2}\right)
Y^{(r-1)}\left(u+\textstyle\frac{1}{2}\right)
&=
Y^{(r-2)}(u)Y^{(r)}(u),\notag\\
Y^{(r)}(u-1)
Y^{(r)}(u+1)
&=
Y^{(r-1)}\left(u-\textstyle\frac{1}{2}\right)
Y^{(r-1)}\left(u+\textstyle\frac{1}{2}\right).
\notag
\end{align}

For $X_r=F_4$,
\begin{align}
\label{ak:eq:YF0}
Y^{(1)}(u-1)Y^{(1)}(u+1)
&=Y^{(2)}(u),\\
Y^{(2)}(u-1)Y^{(2)}(u+1)
&=
Y^{(1)}(u)
Y^{(3)}\left(u-\textstyle\frac{1}{2}\right)
Y^{(3)}\left(u+\textstyle\frac{1}{2}\right),\notag\\
Y^{(3)}\left(u-\textstyle\frac{1}{2}\right)
Y^{(3)}\left(u+\textstyle\frac{1}{2}\right)
&=
Y^{(2)}(u)Y^{(4)}(u),\notag\\
Y^{(4)}\left(u-\textstyle\frac{1}{2}\right)
Y^{(4)}\left(u+\textstyle\frac{1}{2}\right)
&=
Y^{(3)}(u).
\notag
\end{align}

For $X_r=G_2$,
\begin{align}
\label{ak:eq:YG0}
Y^{(1)}(u-1)Y^{(1)}(u+1)
&=
Y^{(2)}\left(u-\textstyle\frac{2}{3}\right)
Y^{(2)}(u)
Y^{(2)}\left(u+\textstyle\frac{2}{3}\right),\\
Y^{(2)}\left(u-\textstyle\frac{1}{3}\right)
Y^{(2)}\left(u+\textstyle\frac{1}{3}\right)
&=
Y^{(1)}(u).\notag
\end{align}
\end{defn}

\begin{defn}
The {\em level $0$ restricted Y-group $\EuScript{Y}_{0}(X_r)$
of type $X_r$} is the abelian group with generators
$Y^{(a)}(u)$ ($a\in I, u\in U $)
and the relations $\mathbb{Y}_{0}(X_r)$.
\end{defn}

\begin{rem}
${\mathbb Y}_0(X_r)$ is obtained from 
unrestricted Y-system ${\mathbb Y}(X_r)$ 
(\ref{eq:YA1})--(\ref{eq:YG1})
by first making the replacement  
$(1+Y^{(a)}_m(u), 1+ Y^{(a)}_m(u)^{-1})
 \rightarrow (Y^{(a)}_m(u),1)$ in the right hand sides
(i.e., taking a formal limit $ Y^{(a)}_m(u) \rightarrow \infty$), 
then setting $Y^{(a)}_m(u) = Y^{(a)}(u)$
if $m=0$ and $Y^{(a)}_m(u)=1$ otherwise.
\end{rem}

{}From (\ref{ak:eq:TADE0}) and (\ref{ak:eq:YADE0}), 
we see that $\EuScript{T}_0(X_r)\simeq \EuScript{Y}_0(X_r)$
if $X_r$ is simply laced.

The following periodicity property justifies
that we call these systems  `level 0'.
Notice that the half-periodicity here  contains the {\em inverse} 
in the right hand sides
in contrast with the
level $\ell \ge 1$ case.

\begin{thm}\label{ak:thm:TY0}
(i) The following relations hold in  $\EuScript{T}_{0}(X_r)$:

(1) Half-periodicity: $T^{(a)}(u+h^\vee)= T^{(\omega(a))}(u)^{-1}$.

(2) Periodicity: $T^{(a)}(u+2h^\vee)= T^{(a)}(u)$.
\par\noindent
(ii) The following relations hold in  $\EuScript{Y}_{0}(X_r)$:

(1) Half-periodicity: $Y^{(a)}(u+h^\vee)= Y^{(\omega(a))}(u)^{-1}$.

(2) Periodicity: $Y^{(a)}(u+2h^\vee)= Y^{(a)}(u)$.
\end{thm}
\begin{proof}
It is enough to show the half-periodicity,
and it can be proved by 
elementary manipulations.
Especially for exceptional $X_r = E_6, E_7, E_8, F_4$ and $G_2$,
it is a matter of a direct check. 
As an illustration 
we present a proof for ${\EuScript T}_0(D_r)\
\left(\simeq{\EuScript Y}_0(D_r)\right)$ and 
${\EuScript Y}_0(B_r)$
below.
The cases ${\EuScript T}_0(B_r), {\EuScript T}_0(C_r), 
{\EuScript T}_0(F_4)$ 
and ${\EuScript T}_0(G_2)$
have been treated in Section 2.2.1 and Appendix A of \cite{KNS2}.

First we consider ${\EuScript T}_0(D_r)$.
{}From (\ref{ak:eq:TADE0}) we have
\begin{align}
T^{(a)}(u)&=\prod_{s=1}^{a}T^{(1)}(u-a-1+2s)\quad (1\le a \le r-2),
\label{ak:eq:d1}\\
T^{(r-1)}(u)T^{(r)}(u)&=\prod_{s=1}^{r-1}T^{(1)}(u-r+2s),
\label{ak:eq:d2}\\
T^{(a)}(u-1)T^{(a)}(u+1)&= T^{(r-2)}(u)\quad(a=r-1,r).
\label{ak:eq:d3}
\end{align} 
In (\ref{ak:eq:d2}) replace $u$ by $u\pm 1$ and take the product.
Using (\ref{ak:eq:d1}) with $a=r-2$ and (\ref{ak:eq:d3}),
one can express all the factors by $T^{(1)}$ only, leading to 
$T^{(1)}(u-r+1)T^{(1)}(u+r-1)=1$.
In view of $h^\vee=2r-2$ and (\ref{ak:eq:d1}), 
this verifies the claim of the theorem 
$T^{(a)}(u)T^{(a)}(u+h^\vee)=1$ 
except for $a=r-1$ and $r$.

Suppose $r$ is even. Then for $a=r-1, r$ we have
\begin{align*}
T^{(a)}(u)T^{(a)}(u+h^\vee) &=
\frac{\prod_{s=0}^{r-1}T^{(a)}(u+2s)}
{\prod_{s=1}^{r-2}T^{(a)}(u+2s)}
=\frac{\prod_{s=0}^{r/2-1}T^{(r-2)}(u+1+4s)}
{\prod_{s=1}^{r/2-1}T^{(r-2)}(u-1+4s)},
\end{align*}
where the second equality is due to (\ref{ak:eq:d3}).
{}From (\ref{ak:eq:d1}) with $a=r-2$, 
the ratio is expressed by  $T^{(1)}$ only, which turns out 
to be 1 owing to $T^{(1)}(u)T^{(1)}(u+h^\vee)=1$.
If $r$ is odd, let $\overline{a}=r-1, r$ according as $a=r, r-1$.
Then we have
\begin{align*}
T^{(a)}(u)T^{(\overline{a})}(u+h^\vee) &=
\frac{\left(\prod_{s=0}^{r-1}T^{(a)}(u+2s)\right)
T^{(\overline{a})}(u+h^\vee)}
{\prod_{s=1}^{r-1}T^{(a)}(u+2s)}\\
&=\frac{\prod_{s=0}^{(r-3)/2}T^{(r-2)}(u+1+4s)}
{\prod_{s=1}^{(r-1)/2}T^{(r-2)}(u-1+4s)}\;\;
T^{(r-1)}(u+h^\vee)T^{(r)}(u+h^\vee),
\end{align*} 
where the second equality is due to (\ref{ak:eq:d3}).
Again this can be shown to be 1 from
(\ref{ak:eq:d1}) with $a=r-2$, (\ref{ak:eq:d2})
and $T^{(1)}(u)T^{(1)}(u+h^\vee)=1$.

Next we consider ${\EuScript Y}_0(B_r)$.
{}From (\ref{ak:eq:YB0}) we find
\begin{align}
Y^{(a)}(u) &= \prod_{s=1}^aY^{(1)}(u-a-1+2s)\quad
(1 \le a \le r-1),\label{ak:eq:b1}\\
Y^{(r-1)}(u-1)Y^{(r-1)}(u+1)
&=Y^{(r-2)}(u)Y^{(r-1)}(u),
\label{ak:eq:b2}\\
Y^{(r)}\left(u-\textstyle\frac{1}{2}\right)
Y^{(r)}\left(u+\textstyle\frac{1}{2}\right)
&=Y^{(r-1)}\left(u\right)
\label{ak:eq:b3}.
\end{align}
Substituting (\ref{ak:eq:b1}) into (\ref{ak:eq:b2}), 
we get $Y^{(r-1)}(u)=\prod_{s=1}^rY^{(1)}(u-r-1+2s)$.
Comparing this with another expression 
(\ref{ak:eq:b1}) with $a=r-1$,
we have
\begin{align}\label{ak:eq:b4}
Y^{(r-1)}(u) = Y^{(r-1)}(u\pm 1)Y^{(1)}(u\mp r\pm 1).
\end{align}
The two relations imply $Y^{(1)}(u-r+1)Y^{(1)}(u+r)=1$.
In view of $h^\vee = 2r-1$ and (\ref{ak:eq:b1}), 
this verifies the claim of the theorem 
$Y^{(a)}(u)Y^{(a)}(u+h^\vee)=1$ 
except for $a=r$.
{}From either of the relations (\ref{ak:eq:b4}) and 
(\ref{ak:eq:b1}) with $a=r-1$, 
one can derive 
$\prod_{s=1}^rY^{(r-1)}(u-r-1+2s) = 
\prod_{s=1}^{r-1}Y^{(r-1)}(u-r+2s)$. 
Substitution of (\ref{ak:eq:b3}) into this gives 
$Y^{(r)}(u)Y^{(r)}(u+h^\vee)=1$.
\end{proof}

For simply laced $X_r$,
a more intrinsic proof of
Theorem \ref{ak:thm:TY0}
by the {\em Coxeter element\/} of the Weyl group
is available,
following the remarkable idea by Fomin-Zelevinsky \cite{FZ3}
used for the proof of  the periodicity of $\EuScript{Y}_2(X_r)$.

\begin{proof}
[Alternative proof for simply laced $X_r$]
Assume that $X_r$ is simply laced.
Let $I=I_+ \sqcup I_-$ be the bipartite decomposition
of the index set $I$, and define $\varepsilon(a)$ by 
$\varepsilon(a) = \pm$ for $a \in I_\pm$. 
Since ${\mathbb T}_0(X_r)$ closes among those
$T^{(a)}(u)$ with fixed `parity' $\varepsilon(a)(-1)^u$,
there is no problem to
impose an additional relation in ${\EuScript T}_0(X_r)$,
\begin{align}
\label{eq:ttrel}
T^{(a)}(u+1)=T^{(a)}(u)^{-1}
\quad \text{whenever}\ \varepsilon(a)(-1)^u=+,
\end{align}
in order to prove its periodicity.
Let $W$ be the Weyl group of type $X_r$ with
the simple reflections $s_a$ ($a\in I$),
which acts on ${\EuScript T}_0(X_r)$ by
\begin{align}
s_b (T^{(a)}(u)) = 
T^{(a)}(u)T^{(b)}(u)^{-C_{ba}}.
\end{align}
Define $\tau_{\pm}=\prod_{a\in I_{\pm}} s_a$.
Then, $\tau_{\pm}\tau_{\mp}$ is the Coxeter element of $W$,
and $\tau_{\varepsilon}$ acts as
\begin{align}\label{ak:eq:tau}
\tau_\varepsilon(T^{(a)}(u)) = 
\begin{cases}
T^{(a)}(u)^{-1} & \varepsilon(a) = \varepsilon,\\
T^{(a)}(u)\prod_{b\neq a}
T^{(b)}(u)^{-C_{ba}}&  \varepsilon(a) \neq \varepsilon.
\end{cases}
\end{align}
By (\ref{eq:ttrel}), (\ref{ak:eq:tau}),
and $\mathbb{T}_0(X_r)$ in (\ref{ak:eq:TADE0}),
we have $T^{(a)}(u+1) =\tau_{(-1)^{u}}(T^{(a)}(u))$.
The following fact  is known (\cite[Ch.\,V, 6.2]{B},
\cite[Lemma 2.1]{FZ3}):
\begin{align}
 \underbrace{
\cdots
 \tau_- \tau_+ \tau_- \tau_+
}_{h\ \text{times}}
=
 \underbrace{
\cdots
 \tau_+ \tau_- \tau_+ \tau_-
}_{h\ \text{times}}
=\omega_0
 \quad (\text{the longest element of $W$}),
\end{align}
where $h$ is the Coxeter number of $X_r$.
Also, $\omega_0(T^{(a)}(u))=T^{(\omega(a))}(u)^{-1}$
due to the remark after (\ref{eq:omega1}).
Using these results, we obtain
\begin{align}
T^{(a)}(u + h ) &=  (\underbrace{
\cdots
 \tau_{\mp} \tau_{\pm} \tau_{\mp} \tau_{\pm}
}_{h\ \text{times}})
(T^{(a)}(u))
=\omega_0(T^{(a)}(u))
=T^{(\omega(a))}(u)^{-1}.
\end{align}
\end{proof}

\begin{rem}
As for the half-periodicity of 
$\EuScript{T}_0(X_r)$, 
a similar result has been obtained in 
Eq.\ (2.8) in \cite{KNS2}.
Compared with Theorem \ref{ak:thm:TY0} here, 
the result there is weaker in that it does not cover 
$A_{r}$ ($r \geq 2$) nor 
individual $T^{(a)}_m(u)$ for $E_6$ and $E_8$. 
Moreover, Eq.\ (2.8a) in \cite{KNS2} should be
corrected for $X_r = D_r$ with 
$r$ odd and $a=r-1, r$.
\end{rem}

\section{Periodicities of restricted T and
Y-systems for twisted quantum affine algebras
}

\label{sect:twisted}

The T and Y-systems considered so far are
associated with the {\em untwisted\/} quantum affine algebra
$U_q(\hat{\mathfrak{g}})$
(when $U=U_{t\hbar}$) as explained in Section \ref{sect:unrest}.
In this section we consider
T and Y-systems associated with the {\em twisted\/}
 quantum affine algebra  $U_q(\hat{\mathfrak{g}}^{\sigma})$ 
following \cite{KS,Her2}.
All the basic results presented for the untwisted case can be
naturally extended to the twisted case as well.
Moreover, the periodicity property of the twisted case reduces to
that of the untwisted case.

\subsection{Dynkin diagrams of twisted affine type}
Throughout this section, we let 
$X_N$ exclusively denote a
Dynkin diagram of type $A_N$ ($N\geq 2$),
 $D_N$ ($N\geq 4$), or $ E_6$.
We keep the enumeration of the nodes of 
$X_N$  by the set $I=\{1,\ldots, N\}$
as in Figure \ref{fig:Dynkin}.
For a pair $(X_N,\kappa)=(A_N,2)$, $(D_N,2)$, $(E_6,2)$, or $(D_4,3)$,
we define
the diagram automorphism $\sigma: I \rightarrow I$
of $X_N$ of order $\kappa$
as follows:
$\sigma(a)=a$ {\em except for the following  cases
(in our enumeration)\/}:
\begin{alignat}{2}
\label{eq:sigma1}
& \sigma(a)=N+1-a\quad (a\in I)&&(X_N,\kappa)=(A_N,2),\\
& \sigma(N-1)=N,\ \sigma(N)=N-1&&(X_N,\kappa)=(D_{N},2), 
 \notag\\
& \sigma(1)=6,\  \sigma(2)=5,\  \sigma(5)=2,\ \sigma(6)=1
&\quad& (X_N,\kappa)=(E_6,2), \notag\\
& \sigma(1)=3,\  \sigma(3)=4,\  \sigma(4)=1
&&(X_N,\kappa)=(D_4,3).
\notag
\end{alignat}
The map $\sigma$ is the same as the involution 
$\omega: I \rightarrow I$ in (\ref{eq:omega1})
except for $X_N = D_N$ ($N$: even).
Let $I/\sigma$ be the set of the $\sigma$-orbits
 of nodes of $X_N$.
We choose, at our discretion,
 a complete set of representatives $I_{\sigma}\subset I$
of $I/\sigma$ as
\begin{align}
\label{eq:sigma2}
I_{\sigma}= 
\begin{cases}
\{ 1,2,\dots, r\} & 
(X_N,\kappa)=(A_{2r-1},2), (A_{2r},2), (D_{r+1},2),\\
\{ 1,2,3,4\} &
(X_N,\kappa)=(E_6,2),\\
 \{ 1,2\} &
(X_N,\kappa)=(D_4,3).
\end{cases}
\end{align}

\begin{figure}
\begin{picture}(283,115)(-15,-90)
%
\put(0,0){\circle{6}}
\put(20,0){\circle{6}}
\put(20,20){\circle{6}}
\put(80,0){\circle{6}}
\put(100,0){\circle*{6}}
\put(45,0){\circle*{1}}
\put(50,0){\circle*{1}}
\put(55,0){\circle*{1}}
\drawline(3,0)(17,0)
\drawline(20,3)(20,17)
\drawline(23,0)(37,0)
\drawline(63,0)(77,0)
\drawline(82,-2)(98,-2)
\drawline(82,2)(98,2)
\drawline(87,0)(93,-6)
\drawline(87,0)(93,6)
\put(-30,-2){$A^{(2)}_{2r-1}$}
\put(-2,-15){\small $1$}
\put(28,17){\small $0$}
\put(18,-15){\small $2$}
\put(70,-15){\small $ r-1$}
\put(98,-15){\small $r$}
%
\put(180,0){
\put(0,0){\circle{6}}
\put(20,0){\circle{6}}
\drawline(3,1)(17,1)
\drawline(3,-1)(17,-1)
\drawline(2,-3)(18,-3)
\drawline(2,3)(18,3)
\drawline(7,6)(13,0)
\drawline(7,-6)(13,0)
\put(-30,-2){$A^{(2)}_{2}$}
\put(-2,-15){\small $0$}
\put(18,-15){\small $1$}
}
%
\put(0,-40){
\put(0,0){\circle{6}}
\put(20,0){\circle{6}}
\put(80,0){\circle{6}}
\put(100,0){\circle{6}}
\put(45,0){\circle*{1}}
\put(50,0){\circle*{1}}
\put(55,0){\circle*{1}}
\drawline(2,-2)(18,-2)
\drawline(2,2)(18,2)
\drawline(13,0)(7,-6)
\drawline(13,0)(7,6)
\drawline(23,0)(37,0)
\drawline(63,0)(77,0)
\drawline(82,-2)(98,-2)
\drawline(82,2)(98,2)
\drawline(93,0)(87,-6)
\drawline(93,0)(87,6)
\put(-30,-2){$A^{(2)}_{2r}$}
\put(-2,-15){\small $0$}
\put(18,-15){\small $1$}
\put(70,-15){\small $ r-1$}
\put(98,-15){\small $r$}
}
%
\put(180,-40){
\put(0,0){\circle{6}}
\put(20,0){\circle*{6}}
\put(80,0){\circle*{6}}
\put(100,0){\circle{6}}
\put(45,0){\circle*{1}}
\put(50,0){\circle*{1}}
\put(55,0){\circle*{1}}
\drawline(2,-2)(18,-2)
\drawline(2,2)(18,2)
\drawline(7,0)(13,-6)
\drawline(7,0)(13,6)
\drawline(23,0)(37,0)
\drawline(63,0)(77,0)
\drawline(82,-2)(98,-2)
\drawline(82,2)(98,2)
\drawline(93,0)(87,-6)
\drawline(93,0)(87,6)
\put(-30,-2){$D^{(2)}_{r+1}$}
\put(-2,-15){\small $0$}
\put(18,-15){\small $1$}
\put(70,-15){\small $ r-1$}
\put(98,-15){\small $r$}
}
%
\put(0,-80){
\put(0,0){\circle{6}}
\put(20,0){\circle{6}}
\put(40,0){\circle{6}}
\put(60,0){\circle*{6}}
\put(80,0){\circle*{6}}
\drawline(3,0)(17,0)
\drawline(23,0)(37,0)
\drawline(63,0)(77,0)
\drawline(42,-2)(58,-2)
\drawline(42,2)(58,2)
\drawline(53,6)(47,0)
\drawline(53,-6)(47,0)
\put(-30,-2){$E^{(2)}_6$}
\put(-2,-15){\small $0$}
\put(18,-15){\small $1$}
\put(38,-15){\small $2$}
\put(58,-15){\small $3$}
\put(78,-15){\small $4$}
}
%
\put(180,-80){
\put(0,0){\circle{6}}
\put(20,0){\circle{6}}
\put(40,0){\circle*{6}}
\drawline(3,0)(17,0)
\drawline(23,0)(37,0)
\drawline(22,-2)(38,-2)
\drawline(22,2)(38,2)
\drawline(33,6)(27,0)
\drawline(33,-6)(27,0)
\put(-30,-2){$D^{(3)}_4$}
\put(-2,-15){\small $0$}
\put(18,-15){\small $1$}
\put(38,-15){\small $2$}
}
\end{picture}
\caption{The Dynkin diagrams 
$X^{(\kappa)}_N$
of twisted affine type and their enumerations
by $I_{\sigma}\cup \{0\}$.
For a filled node $a$, $\sigma(a)=a$
(i.e., $\kappa_a = \kappa$) holds.
}
\label{fig:tDynkin}
\end{figure}
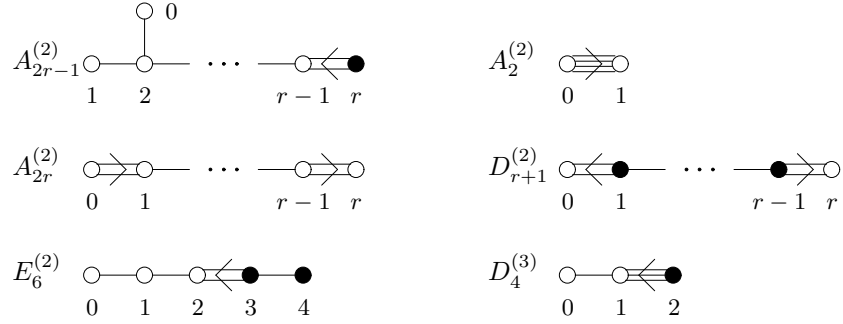

Let $X^{(\kappa)}_N = A^{(2)}_{2r-1}\;(r \ge 2), A^{(2)}_{2r}\; (r \ge 1), 
D^{(2)}_{r+1}\; (r \ge 3), E^{(2)}_6$, or $D^{(3)}_4$
be a Dynkin diagram of  twisted affine type \cite{Ka}.
We enumerate the nodes of $X^{(\kappa)}_N$
with $I_\sigma \cup \{ 0 \}$ as in
Figure \ref{fig:tDynkin},
where $I_{\sigma}$ is the one for $(X_N,\kappa)$. 
By this, 
we  have established the identification of the non-0th  nodes
of the diagram $X^{(\kappa)}_N$
with the nodes of the diagram $X_N$ belonging to
the set $I_{\sigma}$.
For example, for $E^{(2)}_6$, the correspondence is as follows:
\begin{align}
\raisebox{-13pt}
{
\begin{picture}(283,35)(-30,-15)
\put(0,0){
\put(0,0){\circle{6}}
\put(20,0){\circle{6}}
\put(40,0){\circle{6}}
\put(60,0){\circle*{6}}
\put(80,0){\circle*{6}}
\drawline(3,0)(17,0)
\drawline(23,0)(37,0)
\drawline(63,0)(77,0)
\drawline(42,-2)(58,-2)
\drawline(42,2)(58,2)
\drawline(53,6)(47,0)
\drawline(53,-6)(47,0)
\put(-30,-2){$E^{(2)}_6$}
\put(-2,-15){\small $0$}
\put(18,-15){\small $1$}
\put(38,-15){\small $2$}
\put(58,-15){\small $3$}
\put(78,-15){\small $4$}
\dottedline{3}(13,-7)(87,-7)
\dottedline{3}(13,7)(87,7)
\dottedline{3}(13,7)(13,-7)
\dottedline{3}(87,7)(87,-7)
}
%
%
%
\put(160,0){
\put(0,0){\circle{6}}
\put(20,0){\circle{6}}
\put(40,0){\circle{6}}
\put(60,0){\circle{6}}
\put(80,0){\circle{6}}
\put(40,20){\circle{6}}
\drawline(3,0)(17,0)
\drawline(23,0)(37,0)
\drawline(43,0)(57,0)
\drawline(40,3)(40,17)
\drawline(63,0)(77,0)
\dottedline{3}(-7,-7)(47,-7)
\dottedline{3}(-7,-7)(-7,7)
\dottedline{3}(47,-7)(47,27)
\dottedline{3}(-7,7)(33,7)
\dottedline{3}(33,7)(33,27)
\dottedline{3}(33,27)(47,27)
\put(-30,-2){$E_6$}
\put(-2,-15){\small $1$}
\put(18,-15){\small $2$}
\put(38,-15){\small $3$}
\put(58,-15){\small $5$}
\put(78,-15){\small $6$}
\put(50,18){\small $4$}
}
\end{picture}
}
\end{align}
The filled nodes 3,4 in $E^{(2)}_6$ correspond to the
fixed nodes by $\sigma$ in $E_6$.
We use this identification throughout the section.
(The 0th node of $X^{(\kappa)}_N$ is irrelevant in our setting
here.)

We define $\kappa_a$ $(a \in I_{\sigma})$ as
\begin{align}\label{ak:eq:kappaa}
\kappa_a &= \begin{cases} 
1 & \sigma(a) \neq a,\\
\kappa & \sigma(a) = a.
\end{cases}
\end{align}
Note that $X^{(2)}_N=A^{(2)}_{2r}$ is the unique case in which 
$\kappa_a=1$ for any $a \in I_{\sigma}$.

\subsection{Unrestricted T-systems}

Choose  $\hbar\in
\mathbb{C}\setminus 2\pi\sqrt{-1}\mathbb{Q}$
arbitrarily.

\begin{defn}\label{ak:def:uT}
The {\em unrestricted T-system ${\mathbb T}(X^{(\kappa)}_N)$ of type 
$X^{(\kappa)}_N$} 
is the following system of relations for
a family of variables 
$T = \{T^{(a)}_m(u)\mid a \in I_\sigma, m \in {\mathbb N},
u \in {\mathbb C}_{\kappa_a\hbar}\}$,
where $\Omega = 2\pi \sqrt{-1}/\kappa \hbar$,
and  $T^{(0)}_m(u)=T^{(a)}_0(u)=1$ if they 
occur in the right sides in the relations: 

For $X^{(\kappa)}_N=A^{(2)}_{2r-1}$,
\begin{align}
T^{(a)}_m(u-1)T^{(a)}_m(u+1)&=T^{(a)}_{m-1}(u)T^{(a)}_{m+1}(u)
\label{ak:eq:TTAo}\\
&\quad + T^{(a-1)}_m(u)T^{(a+1)}_m(u)
\quad (1 \le a \le r-1),\notag\\
T^{(r)}_m(u-1)T^{(r)}_m(u+1)&=T^{(r)}_{m-1}(u)T^{(r)}_{m+1}(u)
+ T^{(r-1)}_m(u)
T^{(r-1)}_m(u+\Omega).\notag
\end{align}

For $X^{(\kappa)}_N=A^{(2)}_{2r}$,
\begin{align}
T^{(a)}_m(u-1)T^{(a)}_m(u+1)&=T^{(a)}_{m-1}(u)T^{(a)}_{m+1}(u)
\label{ak:eq:TTAe}\\
&\quad + T^{(a-1)}_m(u)T^{(a+1)}_m(u)
\quad (1 \le a \le r-1),\notag\\
T^{(r)}_m(u-1)T^{(r)}_m(u+1)&=T^{(r)}_{m-1}(u)T^{(r)}_{m+1}(u)
+ T^{(r-1)}_m(u)
T^{(r)}_m(u+\Omega).\notag
\end{align}

For $X^{(\kappa)}_N=D^{(2)}_{r+1}$,
\begin{align}
T^{(a)}_m(u-1)T^{(a)}_m(u+1)&=T^{(a)}_{m-1}(u)T^{(a)}_{m+1}(u)
\label{ak:eq:TTD2}\\
&\quad + T^{(a-1)}_m(u)T^{(a+1)}_m(u)
\quad (1 \le a \le r-2),\notag\\
T^{(r-1)}_m(u-1)T^{(r-1)}_m(u+1)&=T^{(r-1)}_{m-1}(u)T^{(r-1)}_{m+1}(u)
\notag\\
&\quad + T^{(r-2)}_m(u)T^{(r)}_m(u)
T^{(r)}_m(u+\Omega),\notag\\
T^{(r)}_m(u-1)T^{(r)}_m(u+1)&=T^{(r)}_{m-1}(u)T^{(r)}_{m+1}(u)
+ T^{(r-1)}_m(u).\notag
\end{align}

For $X^{(\kappa)}_N=E^{(2)}_6$,
\begin{align}
T^{(1)}_m(u-1)T^{(1)}_m(u+1)&=T^{(1)}_{m-1}(u)T^{(1)}_{m+1}(u)
+ T^{(2)}_m(u),\label{ak:eq:TTE}\\
T^{(2)}_m(u-1)T^{(2)}_m(u+1)&=T^{(2)}_{m-1}(u)T^{(2)}_{m+1}(u)
+ T^{(1)}_m(u)T^{(3)}_m(u),\notag\\
T^{(3)}_m(u-1)T^{(3)}_m(u+1)&=T^{(3)}_{m-1}(u)T^{(3)}_{m+1}(u)
+ T^{(2)}_m(u)T^{(2)}_m(u+\Omega)T^{(4)}_m(u),
\notag\\
T^{(4)}_m(u-1)T^{(4)}_m(u+1)&=T^{(4)}_{m-1}(u)T^{(4)}_{m+1}(u)
+ T^{(3)}_m(u).\notag
\end{align}

For $X^{(\kappa)}_N=D^{(3)}_4$,
\begin{align}
T^{(1)}_m(u-1)T^{(1)}_m(u+1)&=T^{(1)}_{m-1}(u)T^{(1)}_{m+1}(u)
+ T^{(2)}_m(u),\label{ak:eq:TTD3}\\
T^{(2)}_m(u-1)T^{(2)}_m(u+1)&=T^{(2)}_{m-1}(u)T^{(2)}_{m+1}(u)
\notag\\
&\qquad
+ T^{(1)}_m(u)
T^{(1)}_m(u-\Omega)
T^{(1)}_m(u+\Omega).\notag
\end{align}
\end{defn}
 The domain $\mathbb{C}_{\kappa_a\hbar}$ of the parameter
$u$ effectively imposes the following periodic condition:
\begin{align}
\label{eq:imperiod2}
T^{(a)}_m(u)=
\begin{cases}
T^{(a)}_m(u+ \kappa\Omega)& \sigma(a)\neq a,\\
T^{(a)}_m(u+ \Omega)& \sigma(a)= a.
\end{cases}
\end{align}

\begin{defn}
The {\em unrestricted T-algebra $\EuScript{T}(X^{(\kappa)}_N)$
of type $X^{(\kappa)}_N$} is the ring with generators
$T^{(a)}_m(u)^{\pm 1}$ 
($a\in I_\sigma, m\in \mathbb{N},
u\in {\mathbb C}_{\kappa_a\hbar} $)
and the relations $\mathbb{T}(X^{(\kappa)}_N)$.
Also, we define the ring $\EuScript{T}^{\circ}(X^{(\kappa)}_N)$
as the subring of $\EuScript{T}(X^{(\kappa)}_N)$
generated by 
$T^{(a)}_m(u)$
($a\in I_\sigma, m\in \mathbb{N},
u\in {\mathbb C}_{\kappa_a\hbar} $).
\end{defn}

Here are some features of
 the T-system $\mathbb{T}(X^{(\kappa)}_N)$
 which are specific to the twisted case:

(i) The relations include the
{\em two\/} basic units of the parameter $u$,
1 and $\Omega$,
which are $\mathbb{Z}$-linearly independent
under our assumption of $\hbar\notin 2\pi \sqrt{-1}\mathbb{Q}$.

(ii)  The domain $\mathbb{C}_{\kappa_a\hbar}$ of the parameter
$u$ and the resulting periodic condition (\ref{eq:imperiod2})
 {\em depend on\/} $a\in I_{\sigma}$.

(iii) We do {\em not\/} consider  the T 
system $\mathbb{T}(X^{(\kappa)}_N)$
whose domain $U$ of the parameter $u$ is  $\mathbb{C}$.
This is because the periodic condition
(\ref{eq:imperiod2}) is now
an integral part of the relations $\mathbb{T}(X^{(\kappa)}_N)$
 due to (i).
This is also natural regarding that no Yangian analogue of
the twisted quantum affine algebra $U_q(\hat{\mathfrak{g}}^{\sigma})$
is known.

(iv) The discrete version of $\mathbb{T}(X^{(\kappa)}_N)$
is available
by taking the domain $U$ of the parameter $u$ as
$U=\mathbb{Z}\times \mathbb{Z}_{\kappa}$,
where $(a,b)\in U$ corresponds to $u=a+b\Omega$,
and imposing the periodic condition (\ref{eq:imperiod2}).

\begin{rem}
The T-system ${\mathbb T}(X^{(\kappa)}_N)$ was introduced in \cite{KS}
as a family of relations in the ring of the commuting  
transfer matrices for solvable lattice models associated with 
the twisted quantum affine algebra $U_q(\hat{\mathfrak{g}}^{\sigma})$
of type $X^{(\kappa)}_N$.
\end{rem}

\begin{rem}
Unifying the untwisted and twisted cases,
the T-system  $\mathbb{T}(X_r)$
and the Y-system  $\mathbb{Y}(X_r)$
of type $X_r$  in Section \ref{sect:unrest}
are also  said to be {\em of type $X^{(1)}_r$} and denoted
by 
 $\mathbb{T}(X^{(1)}_r)$ and  $\mathbb{Y}(X^{(1)}_r)$.
Strictly speaking, this should be applied
only when the domain $U$
of the parameter $u$ is $\mathbb{C}_{t\hbar}$.
However, as we have seen,
such a distinction of $U$ is not so essential
in many aspects of $\mathbb{T}(X_r)$
and $\mathbb{Y}(X_r)$.
\end{rem}

There is a simple relation between the
rings ${\EuScript T}(X^{(\kappa)}_N)$
and ${\EuScript T}(X_N)$.
Let 
$\{\hat{T}^{(a)}_m(u)^{\pm 1}\mid 
a \in I, m \in {\mathbb N}, 
u \in {\mathbb C}_{\hbar}\}$ be the set of generators of  
${\EuScript T}(X_N)$.
Let $\EuScript{J}^\sigma$ be the ideal of 
${\EuScript T}(X_N)$ generated by 
\begin{align}\label{ak:eq:hTT}
\hat{T}^{(a)}_m(u)- \hat{T}^{(\sigma(a))}_m
(u+\Omega)
\quad (a \in I,  m \in {\mathbb N}_{\ge 1}, 
u \in {\mathbb C}_{\hbar}).
\end{align}
Then one can choose a generating set of 
${\EuScript T}(X_N)/\EuScript{J}^\sigma$ as 
$\{\hat{T}^{(a)}_m(u)^{\pm 1}\mid 
a \in I_\sigma, m \in {\mathbb Z}_{\ge 1}, 
u \in {\mathbb C}_{\kappa_a\hbar}\}$.

\begin{prop}\label{ak:prop:TT}
There is a ring isomorphism 
\begin{equation}
\begin{split}
{\EuScript T}(X_N)/\EuScript{J}^\sigma 
&\rightarrow 
{\EuScript T}(X^{(\kappa)}_N)
\\
\hat{T}^{(a)}_m(u)
&\mapsto   T^{(a)}_m(u)\quad (a\in I_\sigma).
\end{split}
\end{equation}
\end{prop}
\begin{proof}
It is easy to check that the relations of the both rings 
are identical under the correspondence. 
\end{proof}

The T-system $\mathbb{T}(X^{(\kappa)}_N)$
plays the same role
in the Grothendieck ring $\mathrm{Rep}\, U_q(\hat{\mathfrak{g}}^{\sigma})$
of the category of the
type 1 finite-dimensional $U_q(\hat{\mathfrak{g}}^{\sigma})$-modules
for the twisted quantum affine algebra
$U_q(\hat{\mathfrak{g}}^{\sigma})$ of type $X^{(\kappa)}_N$
 as the untwisted case.

For arbitrarily chosen $\hbar\in
\mathbb{C}\setminus 2\pi\sqrt{-1}\mathbb{Q}$,
we set the deformation parameter $q$ of
the twisted quantum affine algebras
$U_q(\hat{\mathfrak{g}})$ \cite{J,D1,D2} as
 $q=e^{\hbar}\in \mathbb{C}^{\times}$,
 so that 
$q$ is {\em not a root of unity}.

The $q$-character map $\chi^\sigma_q$ of 
$U_q(\hat{\mathfrak{g}}^{\sigma} )$ is defined by Hernandez
\cite{Her2} as an injective ring
homomorphism 
\begin{align}
\chi^\sigma_q: {\rm Rep} \,U_q(\hat{\mathfrak{g}}^{\sigma})
\rightarrow {\mathbb Z}
[Z_{i, a}^{\pm 1}]_{i \in I_\sigma, a \in {\mathbb C}^\times}.
\end{align}
Consult \cite{Her2} 
for more information on $U_q(\hat{\mathfrak{g}}^{\sigma})$ and
$\chi^\sigma_q$.
The enumeration of $I_\sigma$ in \cite{Her2} 
is the same as the present one
except for $A^{(2)}_{2r}$, where 
$1,2,\ldots, r$ here correspond to $r-1,\ldots, 1,0$ in \cite{Her2}. 
To make the description uniform, 
for $a \in I_\sigma = \{1,\ldots, r\}$ we set  
$\overline{a} = r-a$ for $X^{(\kappa)}_N=A^{(2)}_{2r}$
and $\overline{a}=a$ otherwise.
(This notation $\overline{a}$ will only be used in the rest of this
subsection.)
{}From now on, we employ the  parametrization of the variables
$Z_{\overline{a},q^{\kappa_a u}}$ ($a\in I_{\sigma}$,
$u\in \mathbb{C}_{\kappa_a \hbar}$)
instead of $Z_{i,a}$ ($i\in I_{\sigma}$, $a\in \mathbb{C}^{\times}$)
in \cite{Her2}.
The {\em $q$-character ring\/} 
$\mathrm{Ch}\, U_q(\hat{\mathfrak{g}}^{\sigma})$
of $U_q(\hat{\mathfrak{g}}^{\sigma})$
is defined to be  $\mathrm{Im}\, \chi^{\sigma}_q$.

\begin{defn}
A {\em Kirillov-Reshetikhin module\/} $W^{(a)}_m(u)$
$(a \in I_\sigma, m \in {\mathbb N}, 
u \in {\mathbb C}_{\kappa_a\hbar})$
of $U_q(\hat{\mathfrak{g}}^{\sigma})$ of type $X^{(\kappa)}_N$
is the irreducible finite dimensional 
$U_q(\hat{\mathfrak{g}}^{\sigma})$-module with highest weight 
monomial
\begin{align}
\prod_{j=1}^mZ_{\overline{a},q^{\kappa_a(u+m+1-2j)}}.
\end{align}
\end{defn}
\begin{rem}
The above $W^{(a)}_m(u)$ corresponds to 
$W^{(\overline{a})}_{m, q^{\kappa_a(u-m+1)}}$ in \cite{Her2}.
The $T$-system ${\mathbb T}(X^{(\kappa)}_N)$ 
in Definition \ref{ak:def:uT} agrees with 
the one in \cite[Sect.\ 4.3]{Her2} under the identification
$T^{(a)}_m(u) = X^{(\overline{a})}_{m, q^{\kappa_a(u-m+1)}}$.
\end{rem}

In the same way as (\ref{eq:sa}),
we
define $S_{amu}(T)\in \mathbb{Z}[T]$ ($a\in I, m\in \mathbb{N},
 u\in \mathbb{C}_{\kappa_a\hbar})$,
so that all the relations in $\mathbb{T}(X_N^{(\kappa)})$ are written
in the form $S_{amu}(T)=0$.
Let $I(\mathbb{T}(X_N^{(\kappa)}))$ be the ideal of
$\mathbb{Z}[T]$ generated by $S_{amu}(T)$'s.

\begin{thm}
\label{ak:thm:Tqch}
Let
$\widetilde{T}=\{\widetilde{T}^{(a)}_m(u):=\chi^\sigma_q(W^{(a)}_m(u))
\mid
a\in I_\sigma, m\in \mathbb{N},
u\in \mathbb{C}_{\kappa_a\hbar}
\}$
 be the family of the
$q$-characters of the Kirillov-Reshetikhin modules
of $U_q(\hat{\mathfrak{g}}^{\sigma})$
of type $X^{(\kappa)}_N$.
Then,
\par
(1) The family $\widetilde{T}$ generates the ring
 $\mathrm{Ch}\, U_q(\hat{\mathfrak{g}}^{\sigma})$.
\par
(2) (Hernandez \cite{Her2})
The family $\widetilde{T}$ satisfies the T-system $\mathbb{T}(X^{(\kappa)}_N)$
in $\mathrm{Ch}\, U_q(\hat{\mathfrak{g}}^{\sigma})$
(by replacing $T^{(a)}_m(u)$
in $\mathbb{T}(X^{(\kappa)}_N)$ with $\widetilde{T}^{(a)}_m(u)$).
\par
(3) For any $P(T)\in \mathbb{Z}[T]$,
the relation $P(\widetilde{T})=0$ holds
 in $\mathrm{Ch}\, U_q(\hat{\mathfrak{g}}^{\sigma})$
if and only if there is a nonzero monomial $M(T)\in \mathbb{Z}[T]$
such that  $M(T)P(T)\in I(\mathbb{T}(X_N^{(\kappa)}))$.
\end{thm}

\begin{proof}
(1)  The fundamental character $\chi^\sigma_q(W^{(a)}_1(u))$
has the form,
\begin{align}
\chi^{\sigma}_q(W^{(a)}_1(u))
= Z_{\overline{a},q^{\kappa_a u}} + 
(\text{lower term}),
\end{align}
where `lower' means lower in the
weight lattice for the subalgebra
$U_q(\mathfrak{g}^{\sigma})$ of $U_q(\hat{\mathfrak{g}}^{\sigma})$
\cite{Her2}.
Thus, if there is a nontrivial relation among the
fundamental characters,
then it causes  some nontrivial relation among
 $Z_{\overline{a},q^{\kappa_a u}}$'s.
This is a contradiction.

(2)  This was proved by  \cite[Theorem 4.2]{Her2}.

(3)  The proof is completely parallel with 
Theorem \ref{thm:qch1}\,(3) by setting the height 
as $\mathrm{ht}\, T^{(a)}_m(u)=m$. 
\end{proof}

\begin{cor}
\label{cor:Tqch}
The ring $\EuScript{T}^{\circ}(X^{(\kappa)}_N)$
is isomorphic to
 $\mathrm{Rep}\, U_q(\hat{\mathfrak{g}}^{\sigma})$
by the correspondence $T^{(a)}_m(u)\mapsto W^{(a)}_m(u)$.
\end{cor}

In Appendix \ref{subsect:Qtwisted}
we give parallel results for
the ring associated with the Q-system
and $\mathrm{Rep}\, U_q(\mathfrak{g}^{\sigma})$.

\subsection{Unrestricted Y-systems}

\begin{defn}\label{ak:def:uY}
The {\em unrestricted Y-system ${\mathbb T}(X^{(\kappa)}_N)$ of type 
$X^{(\kappa)}_N$} 
is the following system of relations for
a family of variables 
$Y = \{Y^{(a)}_m(u)\mid a \in I_\sigma, m \in {\mathbb N},
u \in {\mathbb C}_{\kappa_a\hbar}\}$,
where $\Omega = 2 \pi \sqrt{-1}/\kappa\hbar$,
and $Y^{(0)}_m(u)=Y^{(a)}_0(u)^{-1}=0$ if they 
occur in the right sides in the relations: 

For $X^{(\kappa)}_N=A^{(2)}_{2r-1}$,
\begin{align}
Y^{(a)}_m(u-1)Y^{(a)}_m(u+1) & = 
\frac{(1+Y^{(a-1)}_m(u))(1+Y^{(a+1)}_m(u))}
{(1+Y^{(a)}_{m-1}(u)^{-1})(1+Y^{(a)}_{m+1}(u)^{-1})}
\label{ak:eq:TYAo}\\
&\hskip130pt(1 \le a \le r-1),\notag\\
Y^{(r)}_m(u-1)Y^{(r)}_m(u+1) & = 
\frac{(1+Y^{(r-1)}_m(u))
(1+Y^{(r-1)}_m(u+\Omega))}
{(1+Y^{(r)}_{m-1}(u)^{-1})(1+Y^{(r)}_{m+1}(u)^{-1})}.\notag
\end{align}
For $X^{(\kappa)}_N=A^{(2)}_{2r}$,
\begin{align}
Y^{(a)}_m(u-1)Y^{(a)}_m(u+1) & = 
\frac{(1+Y^{(a-1)}_m(u))(1+Y^{(a+1)}_m(u))}
{(1+Y^{(a)}_{m-1}(u)^{-1})(1+Y^{(a)}_{m+1}(u)^{-1})}
\label{ak:eq:TYAe}\\
&\hskip130pt(1 \le a \le r-1),\notag\\
Y^{(r)}_m(u-1)Y^{(r)}_m(u+1) & = 
\frac{(1+Y^{(r-1)}_m(u))
(1+Y^{(r)}_m(u+\Omega))}
{(1+Y^{(r)}_{m-1}(u)^{-1})(1+Y^{(r)}_{m+1}(u)^{-1})}.\notag
\end{align}
For $X^{(\kappa)}_N=D^{(2)}_{r+1}$,
\begin{align}
Y^{(a)}_m(u-1)Y^{(a)}_m(u+1) & = 
\frac{(1+Y^{(a-1)}_m(u))(1+Y^{(a+1)}_m(u))}
{(1+Y^{(a)}_{m-1}(u)^{-1})(1+Y^{(a)}_{m+1}(u)^{-1})}
\label{ak:eq:TYD2}\\
&\hskip130pt(1 \le a \le r-2),\notag\\
Y^{(r-1)}_m(u-1)Y^{(r-1)}_m(u+1) & = 
\frac{(1+Y^{(r-2)}_m(u))(1+Y^{(r)}_m(u))
(1+Y^{(r)}_m(u+\Omega))}
{(1+Y^{(r-1)}_{m-1}(u)^{-1})(1+Y^{(r-1)}_{m+1}(u)^{-1})},
\notag\\
Y^{(r)}_m(u-1)Y^{(r)}_m(u+1) & = 
\frac{1+Y^{(r-1)}_m(u)}
{(1+Y^{(r)}_{m-1}(u)^{-1})(1+Y^{(r)}_{m+1}(u)^{-1})}.\notag
\end{align}
For $X^{(\kappa)}_N=E^{(2)}_6$,
\begin{align}
Y^{(1)}_m(u-1)Y^{(1)}_m(u+1) & = 
\frac{1+Y^{(2)}_m(u)}
{(1+Y^{(1)}_{m-1}(u)^{-1})(1+Y^{(1)}_{m+1}(u)^{-1})},
\label{ak:eq:TYE}\\
Y^{(2)}_m(u-1)Y^{(2)}_m(u+1) & = 
\frac{(1+Y^{(1)}_m(u))(1+Y^{(3)}_m(u))}
{(1+Y^{(2)}_{m-1}(u)^{-1})(1+Y^{(2)}_{m+1}(u)^{-1})},
\notag\\
Y^{(3)}_m(u-1)Y^{(3)}_m(u+1) & = 
\frac{(1+Y^{(2)}_m(u))
(1+Y^{(2)}_m(u+\Omega))
(1+Y^{(4)}_m(u))}
{(1+Y^{(3)}_{m-1}(u)^{-1})(1+Y^{(3)}_{m+1}(u)^{-1})},
\notag\\
Y^{(4)}_m(u-1)Y^{(4)}_m(u+1) & = 
\frac{1+Y^{(3)}_m(u)}
{(1+Y^{(4)}_{m-1}(u)^{-1})(1+Y^{(4)}_{m+1}(u)^{-1})}.\notag
\end{align}
For $X^{(\kappa)}_N=D^{(3)}_4$,
\begin{align}
Y^{(1)}_m(u-1)Y^{(1)}_m(u+1) & = 
\frac{1+Y^{(2)}_m(u)}
{(1+Y^{(1)}_{m-1}(u)^{-1})(1+Y^{(1)}_{m+1}(u)^{-1})},
\label{ak:eq:TYD3}\\
Y^{(2)}_m(u-1)Y^{(2)}_m(u+1) & = 
\frac{(1+Y^{(1)}_m(u))
(1+Y^{(1)}_m(u-\Omega))
(1+Y^{(1)}_m(u+\Omega))}
{(1+Y^{(2)}_{m-1}(u)^{-1})(1+Y^{(2)}_{m+1}(u)^{-1})}.
\notag
\end{align}
\end{defn}

 The domain $\mathbb{C}_{\kappa_a\hbar}$ of the parameter
$u$ effectively imposes the following periodic condition:
\begin{align}
\label{eq:imperiod3}
Y^{(a)}_m(u)=
\begin{cases}
Y^{(a)}_m(u+ \kappa\Omega)& \sigma(a)\neq a,\\
Y^{(a)}_m(u+ \Omega)& \sigma(a)= a.
\end{cases}
\end{align}

\begin{defn}
The {\em unrestricted Y-algebra 
$\EuScript{Y}(X^{(\kappa)}_N)$
of type $X^{(\kappa)}_N$} is the ring with generators
$Y^{(a)}_m(u)^{\pm 1}$, $(1+Y^{(a)}_m(u))^{-1}$
 ($a\in I_\sigma, m\in \mathbb{N},
u\in {\mathbb C}_{\kappa_a\hbar}$)
and the relations $\mathbb{Y}(X^{(\kappa)}_N)$.
\end{defn}

Let 
$\{\hat{Y}^{(a)}_m(u)^{\pm 1}, (1+\hat{Y}^{(a)}_m(u))^{-1} \mid 
a \in I, m \in {\mathbb N}, 
u \in {\mathbb C}_{\hbar}\}$ be the set of generators of  
${\EuScript Y}(X_N)$.
Let $\EuScript{I}^\sigma$ be the ideal of 
${\EuScript Y}(X_N)$ generated by 
\begin{align}\label{ak:eq:hYY}
\hat{Y}^{(a)}_m(u)-
 \hat{Y}^{(\sigma(a))}_m
(u+\Omega)
\quad (a \in I,  m \in {\mathbb N}, 
u \in {\mathbb C}_{\hbar}).
\end{align}
Then one can choose a generating set of 
${\EuScript Y}(X_N)/\EuScript{I}^\sigma$ as 
$\{\hat{Y}^{(a)}_m(u)^{\pm 1},
(1+\hat{Y}^{(a)}_m(u))^{-1}\mid 
a \in I_\sigma, m \in {\mathbb N}, 
u \in {\mathbb C}_{\kappa_a\hbar}\}$.

As Proposition \ref{ak:prop:TT}, we have
\begin{prop}\label{ak:prop:YY}
There is a ring isomorphism 
\begin{equation}
\begin{split}
{\EuScript Y}(X_N)/\EuScript{I}^\sigma 
&\rightarrow {\EuScript Y}(X^{(\kappa)}_N)
\\
\hat{Y}^{(a)}_m(u)
&\mapsto  Y^{(a)}_m(u)\quad (a\in I_\sigma).
\end{split}
\end{equation}
\end{prop}

The following theorem is an analogue of Theorem \ref{thm:TtoY1}.

\begin{thm}
\label{thm:TtoY7}
(1) There is a ring homomorphism
\begin{align}
\label{eq:phi11}
\varphi: \EuScript{Y}(X^{(\kappa)}_N) \rightarrow
\EuScript{T}(X^{(\kappa)}_N)
\end{align}
defined by 
\begin{align}
\label{eq:TtoY11}
Y^{(a)}_m(u)\mapsto
\frac{M^{(a)}_m(u)}
{T^{(a)}_{m-1}(u)T^{(a)}_{m+1}(u)}.
\end{align}
where $T^{(a)}_0(u)=1$.
\par
(2) There is a ring homomorphism
\begin{align}
\psi:
{\EuScript{T}}(X^{(\kappa)}_N)
\rightarrow
\EuScript{Y}(X^{(\kappa)}_N)
\end{align}
such that $\psi\circ\varphi = \mathrm{id}_{\EuScript{Y}(X^{(\kappa)}_N)}$.
\end{thm}

\begin{proof}
We derive the theorem from the results of Theorem \ref{thm:TtoY1}
and Propositions \ref{ak:prop:TT}  and \ref{ak:prop:YY}.

(1)  Let $\hat{\varphi}: \EuScript{Y}(X_N) \rightarrow
\EuScript{T}(X_N)$ be the homomorphism 
in (\ref{eq:phi1}).
Let $\EuScript{J}^{\sigma}$ and $\EuScript{I}^{\sigma}$
be the ideals of $\EuScript{T}(X_N)$ 
and $\EuScript{Y}(X_N)$ in Propositions
\ref{ak:prop:TT}  and \ref{ak:prop:YY},
respectively.
We claim that
$\hat{\varphi}(\EuScript{I}^{\sigma})\subset 
\EuScript{J}^{\sigma}$.
In fact,
\begin{align}
\begin{split}
\hat{\varphi}(\hat{Y}^{(a)}_m(u))
&= 
\frac{\hat{M}^{(a)}_m(u)}
{\hat{T}^{(a)}_{m-1}(u)\hat{T}^{(a)}_{m+1}(u)}\\
& \equiv
\frac{\hat{M}^{(\sigma(a))}_m(u+\Omega)}
{\hat{T}^{(\sigma(a))}_{m-1}(u+\Omega)
\hat{T}^{(\sigma(a))}_{m+1}(u+\Omega)}\quad
\mathrm{mod}\ \EuScript{J}^{\sigma}
\\
&=\hat{\varphi}(\hat{Y}^{(\sigma(a))}_m(u+\Omega)),
\end{split}
\end{align}
where we also used the invariance of $\mathbb{T}(X_N)$ by $\sigma$
in the second equality.
Then, the induced homomorphism
\begin{align}
\hat{\varphi}^{\sigma}:\EuScript{Y}(X_N)/\EuScript{I}^{\sigma}
 \rightarrow
\EuScript{T}(X_N)/\EuScript{J}^{\sigma}
\end{align}
gives the desired homomorphism $\varphi$
under the isomorphisms in
 Propositions \ref{ak:prop:TT}  and \ref{ak:prop:YY}.
\par
(2)  Let $\hat{\psi}: \EuScript{T}(X_N) \rightarrow
\EuScript{Y}(X_N)$ be the homomorphism 
in (\ref{eq:psi}),
where we modify Step 1 of the construction of
$\hat{\psi}$ in the proof of
Theorem \ref{thm:TtoY1} with the following:
(For simplicity, we write $\hat{\psi}(T^{(a)}_1(u))$ as
$T^{(a)}_1(u)$.)

Step 1. We arbitrary choose
$T^{(a)}_1(u)\in \EuScript{Y}(X_N)^{\times}$
($a\in I$) for each $u\in \mathbb{C}_{\hbar}$
in the region $-1\leq \mathrm{Re}\, u < 1$
{\em such that}
\begin{align}
T^{(a)}_1(u) \equiv T^{(\sigma(a))}_1(u+\Omega)
\quad \mathrm{mod}\ \EuScript{I}^{\sigma}.
\end{align}
(For example, just take $T^{(a)}_1(u)=1$.)
\par
Then, one can easily show that 
$T^{(a)}_m(u) \equiv T^{(\sigma(a))}_m(u+\Omega)
\ \mathrm{mod}\ \EuScript{I}^{\sigma}
$
for any $T^{(a)}_m(u)$ constructed in
Steps 2 and 3, again by 
the invariance of $\mathbb{T}(X_N)$ by $\sigma$.
Then, the induced isomorphism
\begin{align}
\hat{\psi}^{\sigma}:\EuScript{T}(X_N)/\EuScript{J}^{\sigma}
 \rightarrow
\EuScript{Y}(X_N)/\EuScript{I}^{\sigma}
\end{align}
gives the desired homomorphism $\psi$
under the isomorphisms in
 Propositions \ref{ak:prop:TT}  and \ref{ak:prop:YY}.
The property 
$\psi\circ\varphi = \mathrm{id}_{\EuScript{Y}(X^{(\kappa)}_N)}$
follows from
$\hat{\psi}^{\sigma}
\circ\hat{\varphi}^{\sigma}
 = \mathrm{id}_{\EuScript{Y}(X_N)/\EuScript{I}^{\sigma}}$.
\end{proof}

\subsection{Restricted T and Y-systems}

\begin{defn}
\label{ak:defn:RTT}
Fix an integer $\ell \geq 2$.
The {\it level $\ell$ restricted T-system 
$\mathbb{T}_{\ell}(X^{(\kappa)}_N)$
of type $X^{(\kappa)}_N$
with the unit boundary condition}
is the system of relations
(\ref{ak:eq:TTAo})--(\ref{ak:eq:TTD3}) naturally restricted to 
a family of variables 
$T=\{T^{(a)}_m(u)\mid
a\in I_\sigma ;\ m=1,\dots, \ell-1;\ 
u\in {\mathbb C}_{\kappa_a\hbar}\}$,
where 
$T^{(0)}_m (u)=T^{(a)}_0 (u)=1$,
and furthermore,  $T^{(a)}_{\ell}(u)=1$
(the {\em unit boundary condition\/}) if they occur
in the right hand sides in the relations.
\end{defn}

\begin{defn}
The {\em level $\ell$ restricted T-algebra 
$\EuScript{T}_{\ell}(X^{(\kappa)}_N)$
of type $X^{(\kappa)}_N$} is the ring with generators
$T^{(a)}_m(u)^{\pm 1}$ ($a\in I_\sigma; 
m=1,\dots,\ell-1; 
u\in {\mathbb C}_{\kappa_a\hbar} $)
and the relations $\mathbb{T}_{\ell}(X^{(\kappa)}_N)$.
Also, we define the ring $\EuScript{T}^{\circ}_{\ell}(X^{(\kappa)}_N)$
as the subring of $\EuScript{T}_{\ell}(X^{(\kappa)}_N)$
generated by 
$T^{(a)}_m(u)$
($a\in I_\sigma; m=1,\dots,\ell-1; 
u\in {\mathbb C}_{\kappa_a\hbar} $).
\end{defn}

\begin{defn}
Fix an integer $\ell \geq 2$.
The {\it level $\ell$ restricted Y-system 
$\mathbb{Y}_{\ell}(X^{(\kappa)}_N)$
of type $X^{(\kappa)}_N$}
is the system of relations
(\ref{ak:eq:TYAo})--(\ref{ak:eq:TYD3})
naturally restricted to a family of variables 
$Y=\{Y^{(a)}_m(u)\mid
a\in I_\sigma ;\ m=1,\dots,\ell-1;\ 
u\in {\mathbb C}_{\kappa_a\hbar}\}$, where 
$Y^{(0)}_m (u)=Y^{(a)}_0 (u)^{-1}=0$,
and furthermore, $Y^{(a)}_{\ell}(u)^{-1}=0$
if they occur in the right hand sides in the relations.
\end{defn}

\begin{defn}
The {\em level $\ell$ restricted Y-algebra 
$\EuScript{Y}_{\ell}(X^{(\kappa)}_N)$
of type $X^{(\kappa)}_N$} is the ring with generators
$Y^{(a)}_m(u)^{\pm 1}$, $(1+Y^{(a)}_m(u))^{-1}$
($a\in I_\sigma; m=1,\dots,\ell-1; 
u\in {\mathbb C}_{\kappa_a\hbar}$)
and the relations $\mathbb{Y}_{\ell}(X^{(\kappa)}_N)$.
\end{defn}

\begin{rem}
The level restrictions of the T and Y-systems 
for the twisted case are introduced here for the first time.
The former is so defined that forgetting the 
parameter $u$, namely the formal replacement 
$T^{(a)}_m(u) \rightarrow Q^{(a)}_m$, coincides with the 
level $\ell$ restricted Q-system introduced in 
\cite[Eq.\ (6.2)]{HKOTT}.
\end{rem}

Proposition \ref{ak:prop:TT} and 
Proposition \ref{ak:prop:YY} have natural counterparts 
in the level restricted situation.
Let 
$\{\hat{T}^{(a)}_m(u)^{\pm 1}\mid 
a \in I, m=1,\ldots, \ell-1, 
u \in {\mathbb C}_{\hbar}\}$ be the set of generators of  
${\EuScript T}_\ell(X_N)$.
Let $\EuScript{J}^\sigma_\ell$ be the ideal of 
${\EuScript T}_\ell(X_N)$ generated by 
\begin{align}\label{ak:eq:hTTl}
\hat{T}^{(a)}_m(u)- \hat{T}^{(\sigma(a))}_m
(u+\Omega)
\quad (a \in I;  m =1,\ldots, l-1; 
u \in {\mathbb C}_{\hbar}).
\end{align}
Then one can choose a generating set of 
${\EuScript T}_\ell(X_N)/\EuScript{J}^\sigma_\ell$ as 
$\{\hat{T}^{(a)}_m(u)^{\pm 1}\mid 
a \in I_\sigma; m=1,\ldots, \ell-1;
u \in {\mathbb C}_{\kappa_a\hbar}\}$.

Similarly, let 
$\{\hat{Y}^{(a)}_m(u)^{\pm 1}, 
(1+\hat{Y}^{(a)}_m(u))^{-1}
\mid 
a \in I; m=1,\ldots, \ell-1; 
u \in {\mathbb C}_{\hbar}\}$ be the set of generators of  
${\EuScript Y}_\ell(X_N)$.
Let $\EuScript{I}^\sigma_\ell$ be the ideal of 
${\EuScript Y}_\ell(X_N)$ generated by 
\begin{align}\label{ak:eq:hYYl}
\hat{Y}^{(a)}_m(u)- \hat{Y}^{(\sigma(a))}_m
(u+\Omega)
\quad (a \in I;  m=1,\ldots, \ell-1;
u \in {\mathbb C}_{\hbar}).
\end{align}
Then one can choose a generating set of 
${\EuScript Y}_\ell(X_N)/\EuScript{I}^\sigma_\ell$ as 
$\{\hat{Y}^{(a)}_m(u)^{\pm 1},
(1+\hat{Y}^{(a)}_m(u))^{-1}\mid 
a \in I_\sigma; m=1,\ldots, \ell-1; 
u \in {\mathbb C}_{\kappa_a\hbar}\}$.

\begin{prop}\label{ak:prop:TTYY}
There is a ring isomorphism
\begin{equation}
\begin{split}
{\EuScript T}_\ell(X_N)/\EuScript{J}^\sigma_\ell
&\rightarrow 
{\EuScript T}_\ell(X^{(\kappa)}_N)
\\
\hat{T}^{(a)}_m(u)
&\mapsto  T^{(a)}_m(u)\quad (a\in I_\sigma).
\end{split}
\end{equation}
Similarly, there is a ring isomorphism
\begin{equation}
\begin{split}
{\EuScript Y}_\ell(X_N)/\EuScript{I}^\sigma_\ell
&\rightarrow 
{\EuScript Y}_\ell(X^{(\kappa)}_N)
\\
\hat{Y}^{(a)}_m(u)
&\mapsto  Y^{(a)}_m(u)\quad (a\in I_\sigma).
\end{split}
\end{equation}
\end{prop}
\begin{proof}
It is compatible to set 
$\hat{T}^{(a)}_\ell(u)=T^{(a)}_\ell(u)=1$
in Propositions \ref{ak:prop:TT}.
It is also compatible to set 
$\hat{Y}^{(a)}_\ell(u)^{-1}=
Y^{(a)}_\ell(u)^{-1}=0$ in 
Proposition \ref{ak:prop:YY}.
\end{proof}

\subsection{Periodicities of restricted T and Y-systems}

Let $h^\vee$ be the dual Coxeter number of 
$X^{(\kappa)}_N$ \cite{Ka}.
It is the same with the (dual) Coxeter number of 
$X_N$ and listed below. 

\begin{align}
\begin{tabular}{c|ccccc}
$X^{(\kappa)}_N$ & $A^{(2)}_{2r-1}$ & 
$A^{(2)}_{2r}$& $D^{(2)}_{r+1}$ & $E^{(2)}_{6}$ & 
$D^{(3)}_{4}$ \\
\hline
$h^{\vee}$& $2r$& $2r+1$&$2r$&$12$&$6$\\
\end{tabular}
\end{align}

The periodicity of ${\EuScript T}_{\ell}(X^{(\kappa)}_N)$
reduces to that of ${\EuScript T}_{\ell}(X_N)$
proved in Corollary \ref{cor:SL3}.

\begin{thm}\label{ak:thm:TTperiod}
The following relations hold in 
${\EuScript T}_{\ell}(X^{(\kappa)}_N)$:

(1) Half-periodicity: 
\begin{align*}
T^{(a)}_m(u+h^\vee+\ell)=
\begin{cases}
T^{(a)}_{\ell-m}(u) 
& \text{if} \; \;X^{(\kappa)}_N = D^{(2)}_{r+1}\;\;
\text{($r+1$: even) or }D^{(3)}_4,\\
T^{(a)}_{\ell-m}(u+\Omega)
& \text{otherwise}.
\end{cases}
\end{align*}

(2) Periodicity: $T^{(a)}_m(u+2(h^\vee+\ell))=
T^{(a)}_{m}(u)$.
\end{thm}
\begin{proof}
It suffices to prove the half-periodicity (1).
First consider the case $X^{(\kappa)}_N \neq 
D^{(2)}_{r+1}$ ($r+1$: even) and $D^{(3)}_4$.
Then we have $\sigma = \omega$ and $\kappa=2$.
Thanks to the first half of Proposition \ref{ak:prop:TTYY},
it is equivalent to showing that 
$\hat{T}^{(a)}_m(u+h^\vee+\ell) = 
\hat{T}^{(a)}_{\ell-m}
(u+\Omega)$
for $a \in I_\sigma$
in ${\EuScript T}_\ell(X_N)/{\EuScript J}^\sigma_\ell$.
Both the left and the right sides coincide with 
$\hat{T}^{(\sigma(a))}_{\ell-m}(u)$
due to the half periodicity of $X_N$
and (\ref{ak:eq:hTT}), respectively.

Next consider the remaining case 
$X^{(\kappa)}_N = 
D^{(2)}_{r+1}$ ($r+1$: even) or $D^{(3)}_4$.
Then we have $\omega={\rm id}$.
By the same reason as before, we are to show
$\hat{T}^{(a)}_m(u+h^\vee+\ell) = 
\hat{T}^{(a)}_{\ell-m}(u)$
for $a \in I_\sigma$
in ${\EuScript T}_\ell(X_N)/{\EuScript J}^\sigma_\ell$.
Again this is guaranteed by 
the half periodicity of $X_N$.
\end{proof}

Similarly, the periodicity of ${\EuScript Y}_{\ell}(X^{(\kappa)}_N)$
reduces to that of ${\EuScript Y}_{\ell}(X_N)$.
Recall that Conjecture \ref{conj:Yperiod1} has been proved
for ${\EuScript Y}_{\ell}(X_N)$,
except for the half-periodicity
for $D_N$ and $E_6$.

\begin{thm}\label{ak:thm:TYperiod}
Suppose that Conjecture \ref{conj:Yperiod1} (1) is also true
for $D_N$ and $E_6$.
Then, the following relations hold in 
${\EuScript Y}_{\ell}(X^{(\kappa)}_N)$:

(1) Half-periodicity: 
\begin{align*}
Y^{(a)}_m(u+h^\vee+\ell)=
\begin{cases}
Y^{(a)}_{\ell-m}(u) 
& \text{if} \; \;X^{(\kappa)}_N = D^{(2)}_{r+1}\;\;
\text{($r+1$: even) or }D^{(3)}_4,\\
Y^{(a)}_{\ell-m}(u+\Omega)
& \text{otherwise}.
\end{cases}
\end{align*}

(2) Periodicity: $Y^{(a)}_m(u+2(h^\vee+\ell))=
Y^{(a)}_{m}(u)$.
\end{thm}

\begin{rem}
\label{rem:nonsim}
By formally setting $\Omega = 0$ (i.e., $\hbar\rightarrow\infty$;
 $q\rightarrow 0,\infty$) or by imposing a further relation
$Y^{(a)}_m(u)=Y^{(a)}_m(u+\Omega)$
in ${\EuScript Y}_{\ell}(X^{(\kappa)}_N)$
for $X^{(\kappa)}_N=A^{(2)}_{2r-1}$ (resp.\ $D^{(2)}_{r+1}$,
 $E^{(2)}_{6}$, $D^{(3)}_{4}$),
one gets the Y-system of the form
\begin{align}
\label{eq:FMY}
Y^{(a)}_m(u-1)Y^{(a)}_m(u+1) & = 
\frac{\prod_{b\in I}
(1+Y^{(b)}_m(u))^{2\delta_{ab}-C_{ab}}
}
{(1+Y^{(a)}_{m-1}(u)^{-1})(1+Y^{(a)}_{m+1}(u)^{-1})}.
\end{align}
Here, $C_{ab}=2(\alpha_a,\alpha_b)/(\alpha_a,\alpha_a)$ is the Cartan matrix for 
 $B_r$ (resp.\ $C_r$, $F_4$, $G_2$)
with the enumeration $I$ in Figure \ref{fig:Dynkin}.
This is the Y-system for nonsimply laced $X_r$
considered in \cite{FZ3,Kel2}.
It was proved,
by \cite{FZ3} for $\ell =2$
and by \cite{Kel2} for any $\ell \geq 2$,
 that the system (\ref{eq:FMY})
has the full-period $2(h+\ell)$,
where $h$ is the {\em Coxeter number\/} of
 $B_r$ (resp.\ $C_r$, $F_4$, $G_2$).
This completely agrees with Theorem \ref{ak:thm:TYperiod},
since the dual Coxeter number $h^{\vee}$ of
$A^{(2)}_{2r-1}$ (resp.\ $D^{(2)}_{r+1}$,
 $E^{(2)}_{6}$, $D^{(3)}_{4}$)
equals to the Coxeter number $h$ of
 $B_r$ (resp.\ $C_r$, $F_4$, $G_2$).
By the same token, one can obtain from
${\EuScript T}_{\ell}(X^{(\kappa)}_N)$
the  T-system of the form
\begin{align}
T^{(a)}_m(u-1)T^{(a)}_m(u+1)&=T^{(a)}_{m-1}(u)T^{(a)}_{m+1}(u)
+
\prod_{b\in I}
T^{(b)}_m(u)^{2\delta_{ab}-C_{ab}},
\end{align}
whose periodicity is the same as (\ref{eq:FMY}).
\end{rem}

\subsection{Periodicities of restricted T and Y-systems at level 0}

Here we introduce the level 0 restricted T-system 
${\mathbb T}_0(X^{(\kappa)}_N)$ and 
T-group ${\EuScript T}_0(X^{(\kappa)}_N)$
of type $X^{(\kappa)}_N$
in a manner similar to Section \ref{sect:restricted}.
The analogous construction of the Y-system and 
Y-group leads to exactly the same objects.
Thus one should understand 
${\EuScript T}_0(X^{(\kappa)}_N)
\simeq 
{\EuScript Y}_0(X^{(\kappa)}_N)$
by
$T^{(a)}(u)\leftrightarrow
Y^{(a)}(u)$
in the sequel.
Such a coincidence has been already encountered 
in the untwisted case between 
${\mathbb T}_0(X_r)$ 
(\ref{ak:eq:TADE0}) and 
${\mathbb Y}_0(X_r)$ 
(\ref{ak:eq:YADE0})
for simply laced $X_r$.

\begin{defn}\label{ak:def:TT0}
The {\em level 0 restricted T-system
${\mathbb T}_0(X^{(\kappa)}_N)$  of type $X^{(\kappa)}_N$}
is the following system of relations for
a family of variables
$T=\{T^{(a)}(u)\mid a \in I_\sigma,\ 
u \in {\mathbb C}_{\kappa_a \hbar}\}$,
where $\Omega=2\pi\sqrt{-1}/\kappa\hbar$,
and  $T^{(0)}(u)=1$ if they occur in the right hand sides
in the relations:

For $X^{(\kappa)}_N=A^{(2)}_{2r-1}$,
\begin{align}
T^{(a)}(u-1)T^{(a)}(u+1)&= T^{(a-1)}(u)T^{(a+1)}(u)
\quad (1 \le a \le r-1),\label{ak:eq:TTAo0}\\
T^{(r)}(u-1)T^{(r)}(u+1)&=
T^{(r-1)}(u)
T^{(r-1)}(u+\Omega).\notag
\end{align}

For $X^{(\kappa)}_N=A^{(2)}_{2r}$,
\begin{align}
T^{(a)}(u-1)T^{(a)}(u+1)&= T^{(a-1)}(u)T^{(a+1)}(u)
\quad (1 \le a \le r-1),\label{ak:eq:TTAe0}\\
T^{(r)}(u-1)T^{(r)}(u+1)&=
T^{(r-1)}(u)
T^{(r)}(u+\Omega).\notag
\end{align}

For $X^{(\kappa)}_N=D^{(2)}_{r+1}$,
\begin{align}
T^{(a)}(u-1)T^{(a)}(u+1)&= T^{(a-1)}(u)T^{(a+1)}(u)
\quad (1 \le a \le r-2),
\label{ak:eq:TTD20}\\
T^{(r-1)}(u-1)T^{(r-1)}(u+1)&=
T^{(r-2)}(u)T^{(r)}(u)
T^{(r)}(u+\Omega),\notag\\
T^{(r)}(u-1)T^{(r)}(u+1)&=T^{(r-1)}(u).\notag
\end{align}

For $X^{(\kappa)}_N=E^{(2)}_6$,
\begin{align}
T^{(1)}(u-1)T^{(1)}(u+1)&=
T^{(2)}(u),\label{ak:eq:TTE0}\\
T^{(2)}(u-1)T^{(2)}(u+1)&=
T^{(1)}(u)T^{(3)}(u),\notag\\
T^{(3)}(u-1)T^{(3)}(u+1)&=
T^{(2)}(u)
T^{(2)}(u+\Omega)T^{(4)}(u),\notag\\
T^{(4)}(u-1)T^{(4)}(u+1)&=
T^{(3)}(u).\notag
\end{align}

For $X^{(\kappa)}_N=D^{(3)}_4$,
\begin{align}
T^{(1)}(u-1)T^{(1)}(u+1)&=
T^{(2)}(u),\label{ak:eq:TTD30}\\
T^{(2)}(u-1)T^{(2)}(u+1)&=
T^{(1)}(u)
T^{(1)}(u-\Omega)
T^{(1)}(u+\Omega).\notag
\end{align}
\end{defn}

\begin{defn}
The {\em level 0 restricted T-group
$\EuScript{T}_{0}(X^{(\kappa)}_N)$
of type $X^{(\kappa)}_N$} is the abelian group with generators 
$T^{(a)}(u)$ 
($a \in I_\sigma, u \in {\mathbb C}_{\kappa_a\hbar}$)
and the relations ${\mathbb T}_0(X^{(\kappa)}_N)$.
\end{defn}

\begin{rem}
$\mathbb{T}_0(X^{(\kappa)}_N)$ is obtained from 
the unrestricted T-system 
$\mathbb{T}(X^{(\kappa)}_N)$
(\ref{ak:eq:TTAo})-(\ref{ak:eq:TTD3})
by setting $T^{(a)}_m(u)=T^{(a)}(u)$ if $m=0$
and $T^{(a)}_m(u)=0$ otherwise.
\end{rem}

Let $\{\hat{T}^{(a)}(u)\mid a \in I, 
u \in {\mathbb C}_{\hbar}\}$ be the set of generators of 
${\EuScript T}_0(X_N)$.
Let ${\EuScript J}^\sigma_0$ be the subgroup of 
${\EuScript T}_0(X_N)$ generated by
\begin{align}\label{ak:eq:hTT0}
\hat{T}^{(a)}(u)\hat{T}^{(\sigma(a))}
(u+\Omega)^{-1}
\quad
(a \in I, u \in {\mathbb C}_{\hbar}).
\end{align}
Then one can choose a generating set of 
${\EuScript T}_0(X_N)/{\EuScript J}^\sigma_0$ as
$\{\hat{T}^{(a)}(u)^{\pm 1}\mid a \in I_\sigma, 
u \in {\mathbb C}_{\kappa_a \hbar}\}$.

\begin{prop}
\label{ak:prop:TT0}
There is a group isomorphism 
\begin{equation}
\begin{split}
{\EuScript T}_0(X_N)/{\EuScript J}^\sigma_0
&\rightarrow 
\EuScript{T}_{0}(X^{(\kappa)}_N)
\\
\hat{T}^{(a)}(u)
&\mapsto 
T^{(a)}(u) \quad (a \in I_\sigma).
\end{split}
\end{equation}
\end{prop}
\begin{proof}
It is easy to check that the relations of the both rings 
are identical under the correspondence. 
\end{proof}

\begin{thm}
The following relations hold in 
$\EuScript{T}_{0}(X^{(\kappa)}_N)$:

(1) Half-periodicity: 
\begin{align*}
T^{(a)}(u+h^\vee)=
\begin{cases}
T^{(a)}(u)^{-1}
& \text{if} \; \;X^{(\kappa)}_N = D^{(2)}_{r+1}\;\;
\text{($r+1$: even) or }D^{(3)}_4,\\
T^{(a)}(u+\Omega)^{-1}
& \text{otherwise}.
\end{cases}
\end{align*}

(2) Periodicity: $T^{(a)}(u+2h^\vee)=
T^{(a)}(u)$.
\end{thm}
\begin{proof}
It suffices to prove the half-periodicity (1).
First consider the case $X^{(\kappa)}_N \neq 
D^{(2)}_{r+1}$ ($r+1$: even) and $D^{(3)}_4$.
Then we have $\sigma = \omega$ and $\kappa=2$.
Thanks to Proposition \ref{ak:prop:TT0},
it is equivalent to showing 
$\hat{T}^{(a)}(u+h^\vee) = \hat{T}^{(a)}
(u+\Omega)^{-1}$
for $a \in I_\sigma$
in ${\EuScript T}_0(X_N)/{\EuScript J}^\sigma_0$.
This equality is verified by Theorem \ref{ak:thm:TY0} 
and (\ref{ak:eq:hTT0}).

Next consider the remaining case 
$X^{(\kappa)}_N = 
D^{(2)}_{r+1}$ ($r+1$: even) or $D^{(3)}_4$.
Then we have $\omega={\rm id}$.
By the same reason as before, we are to show
$\hat{T}^{(a)}(u+h^\vee) = 
\hat{T}^{(a)}(u)$
for $a \in I_\sigma$
in ${\EuScript T}_0(X_N)/{\EuScript J}^\sigma_0$.
Again this is guaranteed by 
Theorem \ref{ak:thm:TY0}.
\end{proof}

\section{Remark on periodicity of $q$-characters
}
\label{sect:qch}

We conclude the paper with  a  remark
on a formal correspondence between the periodicity of 
the T-system  and the  $q$-characters of 
$U_q(\hat{\mathfrak{g}})$ at roots of unity.

Recall that $\EuScript{T}^{\circ}(X_r)
\simeq
\mathrm{Ch}\, U_q(\hat{\mathfrak{g}})$ by Corollary \ref{cor:qch1}.
Let
\begin{align}
\label{eq:piel}
\pi_{\ell}:
\EuScript{T}^{\circ}(X_r)
\rightarrow
\EuScript{T}^{\circ}_{\ell}(X_r)
\end{align}
be the surjective ring homomorphism in \eqref{eq:piel11}.
One can easily see that $T^{(a)}_{t_a\ell}(u)-1$
and $T^{(a)}_{t_a\ell+1}(u)$ are in $\mathrm{Ker}\, \pi_{\ell}$.
Correspondingly,
we define the {\em level $\ell$ restricted
$q$-character ring $\mathrm{Ch}_{\ell}\, U_q(\hat{\mathfrak{g}})$
of  $U_q(\hat{\mathfrak{g}})$}
by 
\begin{align}
\mathrm{Ch}_{\ell}\, U_q(\hat{\mathfrak{g}})
=
\mathrm{Ch}\, U_q(\hat{\mathfrak{g}})/ I_{\ell},
\end{align}
where $I_{\ell}$ is the ideal
of $\mathrm{Ch}\, U_q(\hat{\mathfrak{g}})$
corresponding to $\mathrm{Ker}\, \pi_{\ell}$
under the isomorphism
 $\EuScript{T}^{\circ}(X_r)
\simeq
\mathrm{Ch}\, U_q(\hat{\mathfrak{g}})$.
By construction, we have an isomorphism,
\begin{align}
\begin{split}
\mathrm{Ch}_{\ell}\, U_q(\hat{\mathfrak{g}})
&
\overset{\sim}{\rightarrow}
\EuScript{T}^{\circ}_{\ell}(X_r)\\
\chi_q(W^{(a)}_m(u))&\mapsto
T^{(a)}_m(u)
\quad (1\leq m\leq t_a\ell-1),
\end{split}
\end{align}
and the periodicity of $\EuScript{T}_{\ell}(X_r)$
in Conjecture
\ref{conj:Tperiod1} is rephrased as 
the following  periodicity of
the $q$-characters: For $1\leq m \leq t_a\ell-1$,
\begin{align}
\label{eq:periodT2}
\chi_q(W^{(a)}_m(u+2(h^{\vee} + \ell)))\equiv
\chi_q(W^{(a)}_m(u)) \quad \mathrm{mod}\ I_{\ell}.
\end{align}

In view of (\ref{eq:KR1}),
this implies (but not directly requires)
 $q^{2t(h^{\vee}+\ell)}=1$,
which is natural as mentioned in
the end of Section \ref{subsect:period}.
Let us make this implication, still formal, but more manifest
in the form of a conjecture.
Recall that 
$\mathrm{Ch}\, U_q(\hat{\mathfrak{g}})$
is a subring of
the ring
$\mathbb{Z}[Y^{\pm1}_{a,q^{tu}}]_{a\in I, u\in \mathbb{C}_{t\hbar}}$.
Let $I'_{\ell}$ be the ideal of 
$\mathbb{Z}[Y^{\pm1}_{a,q^{tu}}]_{a\in I, u\in \mathbb{C}_{t\hbar}}$
generated by $I_{\ell}$.

\begin{conj}
\label{conj:periodq}
 The following equality holds in
$\mathbb{Z}[Y^{\pm1}_{a,q^{tu}}]_{a\in I, u\in \mathbb{C}_{t\hbar}}$.
\begin{align}
\label{eq:qch1}
Y_{a,q^{tu+2t(h^{\vee}+\ell)}}\equiv
Y_{a,q^{tu}}
\quad
\mathrm{mod}\ I'_{\ell}.
\end{align}
\end{conj}

\begin{exmp}
\label{exmp:Yper}
(1) $X_r=A_1$.
We set $W_m(u)=W^{(1)}_m(u)$
and $Y_{q^u}=Y_{1,q^{u}}$, for simplicity.
Recall that \cite[Sect.\ 4.1]{FR}
\begin{align}
\label{eq:qch2}
\chi_q(W_m(u))=\left(\prod_{j=1}^m Y_{q^{u+m+1-2j}}\right)
\sum_{i=0}^m \prod_{j=1}^i
A^{-1}_{q^{u+m+2-2j}},
\quad
A_{q^{u}}:=Y_{q^{u-1}}Y_{q^{u+1}}.
\end{align}
Thus, we have
\begin{align}
\chi_q(W_{\ell+1}(u))=
Y_{q^{u-\ell}}\chi_q(W_{\ell}(u+1))
+Y^{-1}_{q^{u-\ell+2}}Y^{-1}_{q^{u-\ell+4}}
\cdots Y^{-1}_{q^{u+\ell+2}}.
\end{align}
Meanwhile,
 $\chi_q(W_{\ell}(u))\equiv 1$
and $\chi_q(W_{\ell+1}(u))\equiv 0$.
Therefore,
\begin{align}
Y_{q^{u-\ell}}Y_{q^{u-\ell+2}}
\cdots Y_{q^{u+\ell+2}}\equiv -1,
\end{align}
{}from which $Y_{q^{u+2(2+\ell)}}\equiv Y_{q^u}$ follows.
\par
(2) $X_r=A_r$ ($r\geq 2$).
We remark that,
in addition to $T^{(a)}_{\ell}(u)-1$ and $T^{(a)}_{\ell+1}(u)$,
$T^{(1)}_{\ell+2}(u)$,\dots,
$T^{(1)}_{\ell+r}(u)$ are also in $\mathrm{Ker}\, \pi_{\ell}$.
Then, generalizing the argument of (1), one can show that
\begin{align}
\label{eq:Yper}
Y_{a,q^{u}}Y_{a,q^{u+2}}
\cdots Y_{a,q^{u+2(r+\ell)}}\equiv (-1)^{ar}
\mod I'_{\ell},
\end{align}
{}from which $Y_{a,q^{u+2(r+1+\ell)}}\equiv Y_{a,q^u}$ follows.
More detail is given in Appendix \ref{sect:Yper}.
\end{exmp}

It is important to establish a precise relation between
the ring $\mathrm{Ch}_{\ell}\, U_q(\hat{\mathfrak{g}})$
and the ring $\mathrm{Rep}\, U^{\mathrm{res}}_{\epsilon}
(\hat{\mathfrak{g}})$ of  \cite{FM2}
for  a primitive
$2t(h^{\vee}+\ell)$th root of unity $\epsilon$.
A similar remark is applicable to the twisted quantum affine algebras
as well.

\begin{appendix}

\section{Q-systems and $\mathrm{Rep}\, U_q(\mathfrak{g})$
}

\subsection{Q-systems for untwisted case}
\label{subsect:Q}

Here we present parallel results to
Theorem \ref{thm:qch1} and Corollary \ref{cor:qch1}
for the Q-system and $\mathrm{Rep}\, U_q(\mathfrak{g})$.
For a Dynkin diagram $X_r$, let $I$ be as in
Section \ref{sect:unrest}.

The following system was introduced by \cite{Ki,KR}.

\begin{defn}
The {\it unrestricted Q-system $\mathbb{Q}(X_r)$
of type $X_r$}
is the following system of relations for
a family of variables $Q=\{Q^{(a)}_m
\mid
a\in I, m\in \mathbb{N} \}$,
where 
$Q^{(0)}_m (u)=Q^{(a)}_0 (u)= 1$ if they occur
in the right hand sides in the relations:

For simply laced $X_r$,
\begin{align}
(Q^{(a)}_m)^2
=
Q^{(a)}_{m-1}Q^{(a)}_{m+1}
+
\prod_{b\in I: C_{ab}=-1}
Q^{(b)}_{m}.
\end{align}

For $X_r=B_r$,
\begin{align}
(Q^{(a)}_m)^2
&=
Q^{(a)}_{m-1}Q^{(a)}_{m+1}
+Q^{(a-1)}_{m}Q^{(a+1)}_{m}
\quad
 (1\leq a\leq r-2),\\
(Q^{(r-1)}_m)^2
&=
Q^{(r-1)}_{m-1}Q^{(r-1)}_{m+1}
+
Q^{(r-2)}_{m}Q^{(r)}_{2m},\notag\\
(Q^{(r)}_{2m})^2
&=
Q^{(r)}_{2m-1}(u)Q^{(r)}_{2m+1}(u)
+
(Q^{(r-1)}_{m})^2,
\notag\\
(Q^{(r)}_{2m+1})^2
&=
Q^{(r)}_{2m}Q^{(r)}_{2m+2}
+
Q^{(r-1)}_{m}Q^{(r-1)}_{m+1}.
\notag
\end{align}

For $X_r=C_r$,
\begin{align}
(Q^{(a)}_m)^2
&=
Q^{(a)}_{m-1}Q^{(a)}_{m+1}
+Q^{(a-1)}_{m}Q^{(a+1)}_{m}
\quad
 (1\leq a\leq r-2),\\
(Q^{(r-1)}_{2m})^2
&=
Q^{(r-1)}_{2m-1}Q^{(r-1)}_{2m+1}
+
Q^{(r-2)}_{2m}
(Q^{(r)}_{m})^2
,\notag\\
(Q^{(r-1)}_{2m+1})^2
&=
Q^{(r-1)}_{2m}Q^{(r-1)}_{2m+2}
+
Q^{(r-2)}_{2m+1}
Q^{(r)}_{m}Q^{(r)}_{m+1},
\notag\\
(Q^{(r)}_{m})^2
&=
Q^{(r)}_{m-1}Q^{(r)}_{m+1}
+
Q^{(r-1)}_{2m}.
\notag
\end{align}

For $X_r=F_4$,

\begin{align}
(Q^{(1)}_m)^2
&=
Q^{(1)}_{m-1}Q^{(1)}_{m+1}
+Q^{(2)}_{m},\\
(Q^{(2)}_m)^2
&=
Q^{(2)}_{m-1}Q^{(2)}_{m+1}
+
Q^{(1)}_{m}Q^{(3)}_{2m},\notag\\
(Q^{(3)}_{2m})^2
&=
Q^{(3)}_{2m-1}Q^{(3)}_{2m+1}
+
(Q^{(2)}_{m})^2
Q^{(4)}_{2m},\notag\\
(Q^{(3)}_{2m+1})^2
&=
Q^{(3)}_{2m}Q^{(3)}_{2m+2}
+
Q^{(2)}_{m}Q^{(2)}_{m+1}
Q^{(4)}_{2m+1},\notag\\
\notag
(Q^{(4)}_{m})^2
&=
Q^{(4)}_{m-1}Q^{(4)}_{m+1}
+Q^{(3)}_m.
\notag
\end{align}

For $X_r=G_2$,
\begin{align}
(Q^{(1)}_m)^2
&=
Q^{(1)}_{m-1}Q^{(1)}_{m+1}
+
Q^{(2)}_{3m},\\
(Q^{(2)}_{3m})^2
&=
Q^{(2)}_{3m-1}Q^{(2)}_{3m+1}
+
(Q^{(1)}_{m})^3,
\notag\\
(Q^{(2)}_{3m+1})^2
&=
Q^{(2)}_{3m}Q^{(2)}_{3m+2}
+
(Q^{(1)}_{m})^2
Q^{(1)}_{m+1},\notag\\
(Q^{(2)}_{3m+2})^2
&=
Q^{(2)}_{3m+1}Q^{(2)}_{3m+3}
+
Q^{(1)}_{m}(u)
(Q^{(1)}_{m+1})^2.
\notag
\end{align}
\end{defn}

\begin{defn}
The {\em unrestricted Q-algebra $\EuScript{Q}(X_r)$
of type $X_r$} is the ring with generators
$Q^{(a)}_m {}^{\pm 1}$ ($a\in I, m\in \mathbb{N}$)
and the relations $\mathbb{Q}(X_r)$.
Also, we define the ring $\EuScript{Q}^{\circ}(X_r)$
as the subring of $\EuScript{Q}(X_r)$
generated by 
$Q^{(a)}_m$ ($a\in I, m\in \mathbb{N}$).
\end{defn}

The system $\mathbb{Q}(X_r)$ is obtained
by $\mathbb{T}(X_r)$ by the reduction
of the spectral parameter $u$.
One can also define the level $\ell$ {\em restricted Q-system}
$\mathbb{Q}_{\ell}(X_r)$ by the reduction
of $\mathbb{T}_{\ell}(X_r)$ \cite{KNS1}.
The system $\mathbb{Q}_{\ell}(X_r)$
plays the central role in the
dilogarithm identities for the
central charges of conformal field theories
(e.g.\ \cite{Ki,Ku,KN,KNS1,KNS2,RTV,FS,GT}, etc.).

Let $\mathfrak{g}$ be the complex simple Lie algebra
of type $X_r$,
and $U_q(\mathfrak{g})$ be the quantized universal enveloping
algebra of $\mathfrak{g}$.
Then,  $U_q(\mathfrak{g})$ is a subalgebra of the
untwisted quantum affine algebra  $U_q(\hat{\mathfrak{g}})$.
Let $\chi$ be the character map
of $U_q(\mathfrak{g})$;
it is an
injective ring homomorphism
$\chi:\mathrm{Rep}\, U_q(\mathfrak{g})
\rightarrow 
\mathbb{Z}[y^{\pm1}_{a}]_{a\in I}$,
where we follow the notation of \cite{FR}.
The {\em character ring\/} $\mathrm{Ch}\, U_q(\mathfrak{g})$
of $U_q(\mathfrak{g})$
is defined to be $\mathrm{Im}\, \chi$.
Thus,  $\mathrm{Ch}\, U_q(\mathfrak{g})$
is an integral domain and
isomorphic to $\mathrm{Rep}\, U_q(\mathfrak{g})$.
Let $\mathrm{Res}:
\mathrm{Rep}\, U_q(\hat{\mathfrak{g}})
\rightarrow
\mathrm{Rep}\, U_q(\mathfrak{g})
$
be the restriction homomorphism
induced from the inclusion
$U_q(\mathfrak{g})\rightarrow
U_q(\hat{\mathfrak{g}})$.

A {\em Kirillov-Reshetikhin module}
$W^{(a)}_m$ ($a\in I$, $m\in \mathbb{N}$)
  of $U_q(\mathfrak{g})$
is the (not necessarily irreducible)  $U_q(\mathfrak{g})$-module
defined by 
$W^{(a)}_m:=\mathrm{Res}\, (W^{(a)}_m(u))$,
where $W^{(a)}_m(u)$ is a Kirillov-Reshetikhin module
of $U_q(\hat{\mathfrak{g}})$.
We remark that $W^{(a)}_m$ is independent of $u$.

In the same way as (\ref{eq:sa}),
we
define $S_{am}(Q)\in \mathbb{Z}[Q]$ ($a\in I, m\in \mathbb{N}$),
so that all the relations in $\mathbb{Q}(X_r)$ are written
in the form $S_{am}(Q)=0$.
Let $I(\mathbb{Q}(X_r))$ be the ideal of
$\mathbb{Z}[Q]$ generated by $S_{am}(Q)$'s.

As Theorem \ref{thm:qch1}, we obtain
\begin{thm}
\label{thm:qch2}
Let
$\widetilde{Q}=\{\widetilde{Q}^{(a)}_m:=\chi(W^{(a)}_m)
\mid
a\in I, m\in \mathbb{N}
\}$
 be the family of the
characters of the Kirillov-Reshetikhin modules
of $U_q(\mathfrak{g})$.
Then,
\par
(1) The family $\widetilde{Q}$ generates the ring
 $\mathrm{Ch}\, U_q(\mathfrak{g})$.
\par
(2) (\cite{N3,Her1})
 The family $\widetilde{Q}$ satisfies the Q-system $\mathbb{Q}(X_r)$
in $\mathrm{Ch}\, U_q(\mathfrak{g})$
(by replacing $Q^{(a)}_m$
in $\mathbb{Q}(X_r)$ with $\widetilde{Q}^{(a)}_m$).
\par
(3) For any $P(Q)\in \mathbb{Z}[Q]$,
the relation $P(\widetilde{Q})=0$ holds
 in $\mathrm{Ch}\, U_q(\mathfrak{g})$
if and only if there is a nonzero monomial $M(Q)\in \mathbb{Z}[Q]$
such that  $M(Q)P(Q)\in I(\mathbb{Q}(X_r))$.
\end{thm}

\begin{cor}
\label{cor:qch2}
The ring
$\EuScript{Q}^{\circ}(X_r)$
is isomorphic to
 $\mathrm{Rep}\, U_q(\mathfrak{g})$
by the correspondence $Q^{(a)}_m\mapsto W^{(a)}_m$.
\end{cor}

By taking $q\rightarrow 1$,
one can also obtain analogous results
for $\mathrm{Rep}\, \mathfrak{g}$.

\subsection{Q-systems for twisted case}
\label{subsect:Qtwisted}

Here we present
parallel results to
Theorem \ref{ak:thm:Tqch} and Corollary \ref{cor:Tqch}
for the Q-system and $\mathrm{Rep}\, U_q(\mathfrak{g^{\sigma}})$.
For a pair $(X_N,\kappa)=(A_N,2)$, $(D_N,2)$, $(E_6,2)$, or $(D_4,3)$,
let $I_{\sigma}$  be as in Section
\ref{sect:twisted}.

The following system was introduced by \cite{HKOTT}.

\begin{defn}
The {\it unrestricted Q-system $\mathbb{Q}(X^{(\kappa)}_N)$
of type $X^{(\kappa)}_N$}
is the following system of relations for
a family of variables $Q=\{Q^{(a)}_m
\mid
a\in I_{\sigma}, m\in \mathbb{N} \}$,
where 
$Q^{(0)}_m (u)=Q^{(a)}_0 (u)= 1$ if they occur
in the right hand sides in the relations:

For $X^{(\kappa)}_N=A^{(2)}_{2r-1}$,
\begin{align}
(Q^{(a)}_m)^2&=Q^{(a)}_{m-1}Q^{(a)}_{m+1}
 + Q^{(a-1)}_m Q^{(a+1)}_m
\quad (1 \le a \le r-1),\\
(Q^{(r)}_m)^2 &=Q^{(r)}_{m-1}Q^{(r)}_{m+1}
+ (Q^{(r-1)}_m)^2.\notag
\end{align}

For $X^{(\kappa)}_N=A^{(2)}_{2r}$,
\begin{align}
(Q^{(a)}_m)^2&=Q^{(a)}_{m-1}Q^{(a)}_{m+1}
+ Q^{(a-1)}_mQ^{(a+1)}_m
\quad (1 \le a \le r-1),\\
(Q^{(r)}_m)^2&=Q^{(r)}_{m-1}Q^{(r)}_{m+1}
+ Q^{(r-1)}_m
Q^{(r)}_m.\notag
\end{align}

For $X^{(\kappa)}_N=D^{(2)}_{r+1}$,
\begin{align}
(Q^{(a)}_m)^2&=Q^{(a)}_{m-1}Q^{(a)}_{m+1}
 + Q^{(a-1)}_m Q^{(a+1)}_m
\quad (1 \le a \le r-2),\\
(Q^{(r-1)}_m)^2&=Q^{(r-1)}_{m-1}Q^{(r-1)}_{m+1}
+ Q^{(r-2)}_m (Q^{(r)}_m)^2,\notag\\
(Q^{(r)}_m)^2 &=Q^{(r)}_{m-1}Q^{(r)}_{m+1}
+ Q^{(r-1)}_m.\notag
\end{align}

For $X^{(\kappa)}_N=E^{(2)}_6$,
\begin{align}
(Q^{(1)}_m)^2 &=Q^{(1)}_{m-1}Q^{(1)}_{m+1}
+ Q^{(2)}_m,\\
(Q^{(2)}_m)^2 &=Q^{(2)}_{m-1} Q^{(2)}_{m+1}
+ Q^{(1)}_m Q^{(3)}_m,\notag\\
(Q^{(3)}_m)^2 &=Q^{(3)}_{m-1} Q^{(3)}_{m+1}
+ (Q^{(2)}_m)^2 Q^{(4)}_m,\notag\\
(Q^{(4)}_m)^2 &=Q^{(4)}_{m-1} Q^{(4)}_{m+1}
+ Q^{(3)}_m.\notag
\end{align}

For $X^{(\kappa)}_N=D^{(3)}_4$,
\begin{align}
(Q^{(1)}_m)^2 &=Q^{(1)}_{m-1}Q^{(1)}_{m+1}
+ Q^{(2)}_m,\\
(Q^{(2)}_m)^2 &=Q^{(2)}_{m-1} Q^{(2)}_{m+1}
+ (Q^{(1)}_m)^3.\notag
\end{align}
\end{defn}

\begin{defn}
The {\em unrestricted Q-algebra $\EuScript{Q}(X^{(\kappa)}_N)$
of type $X^{(\kappa)}_N$} is the ring with generators
$Q^{(a)}_m {}^{\pm 1}$ ($a\in I_{\sigma}, m\in \mathbb{N}$)
and the relations $\mathbb{Q}(X^{(\kappa)}_N)$.
Also, we define the ring $\EuScript{Q}^{\circ}(X^{(\kappa)}_N)$
as the subring of $\EuScript{Q}(X^{(\kappa)}_N)$
generated by 
$Q^{(a)}_m$ ($a\in I_{\sigma}, m\in \mathbb{N}$).
\end{defn}

The system $\mathbb{Q}(X^{(\kappa)}_N)$ is obtained
by $\mathbb{T}(X^{(\kappa)}_N)$ by the reduction
of the spectral parameter $u$.
One can also define the level $\ell$ {\em restricted Q-system}
$\mathbb{Q}_{\ell}(X^{(\kappa)}_N)$ by the reduction
of $\mathbb{T}_{\ell}(X^{(\kappa)}_N)$.

Let $X^{\sigma}_N$ be the subdiagram of $X^{(\kappa)}_N$
obtained by removing the $0$th node.
Explicitly,

\begin{align}
\begin{tabular}{c|cccccc}
$X^{(\kappa)}_N$ & $A^{(2)}_{2r-1}$ & $A^{(2)}_{2}$ & 
$A^{(2)}_{2r}$& $D^{(2)}_{r+1}$ & $E^{(2)}_{6}$ & 
$D^{(3)}_{4}$ \\
\hline
$X^{\sigma}_N$ & $C_r$& $A_1$&$B_r$&$B_r$&$F_4$&$G_2$\\
\end{tabular}
\end{align}

Let $\mathfrak{g}^{\sigma}$ be the complex simple Lie algebra
of type $X^{\sigma}_N$,
and $U_q(\mathfrak{g}^{\sigma})$ be the quantized
universal enveloping
algebra of $\mathfrak{g}^{\sigma}$.
Then, $U_q(\mathfrak{g}^{\sigma})$ is a subalgebra
of the twisted quantum affine algebra
$U_q(\hat{\mathfrak{g}}^{\sigma})$  of type $X^{(\kappa)}_N$
\cite{Her2}.
Let $\mathrm{Res}^{\sigma}:
\mathrm{Rep}\, U_q(\hat{\mathfrak{g}}^{\sigma})
\rightarrow
\mathrm{Rep}\, U_q(\mathfrak{g}^{\sigma})
$
be the restriction homomorphism
induced from the inclusion
$U_q(\mathfrak{g}^{\sigma})\rightarrow
U_q(\hat{\mathfrak{g}}^{\sigma})$.

We define
the (not necessarily irreducible)
$U_q(\mathfrak{g}^{\sigma})$-module
$\dot{W}^{(a)}_m$ ($a\in I_{\sigma}$, $m\in \mathbb{N}$)
by
$\dot{W}^{(a)}_m:=\mathrm{Res}^{\sigma}\, (W^{(a)}_m(u))$,
where $W^{(a)}_m(u)$ is a Kirillov-Reshetikhin module
of $U_q(\hat{\mathfrak{g}}^{\sigma})$.
We remark that $\dot{W}^{(a)}_m$ is independent of $u$.

In the same way as (\ref{eq:sa}),
we
define $S_{am}(Q)\in \mathbb{Z}[Q]$ ($a\in I, m\in \mathbb{N}$),
so that all the relations in $\mathbb{Q}(X_N^{(\kappa)})$ are written
in the form $S_{am}(Q)=0$.
Let
 $I(\mathbb{Q}(X_N^{(\kappa)}))$ be the ideal of
$\mathbb{Z}[Q]$ generated by $S_{am}(Q)$'s.

As Theorem \ref{ak:thm:Tqch}, we obtain
\begin{thm}
\label{thm:qch3}
Let
$\widetilde{Q}=\{\widetilde{Q}^{(a)}_m:=\chi(\dot{W}^{(a)}_m)
\mid
a\in I_{\sigma}, m\in \mathbb{N}
\}$
 be the family of the
characters of $\dot{W}^{(a)}_m(u)$'s.
Then,
\par
(1) The family $\widetilde{Q}$ generates the ring
 $\mathrm{Ch}\, U_q(\mathfrak{g}^{\sigma})$.
\par
(2) (\cite{Her2}) The family $\widetilde{Q}$ satisfies the Q-system
 $\mathbb{Q}(X^{(\kappa)}_N)$
in $\mathrm{Ch}\, U_q(\mathfrak{g}^{\sigma})$
(by replacing $Q^{(a)}_m$
in $\mathbb{Q}(X^{(\kappa)}_N)$ with $\widetilde{Q}^{(a)}_m$).
\par
(3) For any $P(Q)\in \mathbb{Z}[Q]$,
the relation $P(\widetilde{Q})=0$ holds
 in $\mathrm{Ch}\, U_q(\mathfrak{g}^{\sigma})$
if and only if there is a nonzero monomial $M(Q)\in \mathbb{Z}[Q]$
such that  $M(Q)P(Q)\in I(\mathbb{Q}(X_N^{(\kappa)}))$.
\end{thm}

\begin{cor}
\label{cor:qch3}
The ring $\EuScript{Q}^{\circ}(X^{(\kappa)}_N)$
is isomorphic to
 $\mathrm{Rep}\, U_q(\mathfrak{g}^{\sigma})$
by the correspondence $Q^{(a)}_m\mapsto \dot{W}^{(a)}_m$.
\end{cor}

By taking $q\rightarrow 1$,
one can also obtain analogous results
for $\mathrm{Rep}\, \mathfrak{g}^{\sigma}$.

\section{Proof of \eqref{eq:Yper} in Example \ref{exmp:Yper}\ (2)
}
\label{sect:Yper}

First we show that
$T^{(1)}_{\ell+2}(u)$,\dots,
$T^{(1)}_{\ell+r}(u)$ are  in $\mathrm{Ker}\, \pi_{\ell}$.
Let $S^{(a)}_m(u):=\pi_{\ell}(T^{(a)}_m(u))$.
Then, we have
\begin{align}
\label{eq:pi1}
\sum_{a=0}^{r+1}(-1)^a T^{(a)}_1(u+a)
T^{(1)}_{m-a}(u+m+a)&=\delta_{m,0},\\
\label{eq:pi2}
\sum_{j=0}^{m}(-1)^j S^{(m-j)}_{\ell-1}(u+j)
S^{(1)}_{\ell-j}(u+j-m)&=\delta_{m,0},\\
\label{eq:pi3}
\sum_{a=0}^{k-1}(-1)^a S^{(a)}_1(u+a)
S^{(1)}_{\ell+k-a}(u+\ell+k+a)&=0
\quad (1\leq k \leq r),
\end{align}
where $T^{(0)}_m(u)=
T^{(r+1)}_m(u)=S^{(0)}_m(u)=S^{(r+1)}_m(u)=1$
($m\geq 0$)
 and
$T^{(a)}_m(u)=S^{(a)}_m(u)=0$  $(m<0)$.
We obtain \eqref{eq:pi1} from the Jacobi-Trudi type determinant
formula in \cite[Eq.\ (2.21)]{KNS1},
\eqref{eq:pi2} from a similar determinant formula
in $\EuScript{T}^{\circ}_{\ell}(A_r)$,
and  \eqref{eq:pi3} from
\eqref{eq:pi1}, \eqref{eq:pi2}, and the half-periodicity
$S^{(a)}_1(u+\ell+r+1)=S^{(r+1-a)}_{\ell-1}(u)$
in $\EuScript{T}^{\circ}_{\ell}(A_r)$.
It follows from \eqref{eq:pi3} that
$S^{(1)}_{\ell+2}(u)=\dots=S^{(1)}_{\ell+r}(u)=0$.

Next we prove the following statement in Example \ref{exmp:Yper} (2):

\begin{prop}
\label{prop:Yper}
The following relations mod $I'_{\ell}$ hold
in
$\mathbb{Z}[Y^{\pm1}_{a,q^{u}}]_{a\in I, u\in \mathbb{C}_{\hbar}}$:

(1) $Y_{a,q^u}Y_{a,q^{u+2}}\cdots Y_{a,q^{u+2\ell+2r}} \equiv (-1)^{ar}$.

(2) $Y_{a,q^u} \equiv Y_{a,q^{u+2(r+1+\ell)}}$.
\end{prop}

Let 
\begin{align}\label{ak:eq:qqq1}
\fbox{$a$}_u = \frac{Y_{a,q^{u+a-1}}}{Y_{a-1,q^{u+a}}}
\quad (1 \le a \le r+1),
\end{align}
where $Y_{0,q^u}=Y_{r+1,q^u}=1$. 
We introduce the notation
\begin{align}
{}_{u}^{a}\,{\overbrace{\fbox{\phantom{1,2,3,4}}}^m}\,^{b}_{v} 
= \sum \fbox{$a_1$}_{u}\,
\fbox{$a_2$}_{u+2}\cdots \,\fbox{$a_m$}_{v},
\end{align}
where $v=u+2m-2$, and the sum extends over all 
the integers $a_1,\ldots, a_m$ such that 
$a\le a_1 \le \cdots \le a_m \le b$. 
The array of boxes in the right hand side 
is to be understood as 
the product of the monomials (\ref{ak:eq:qqq1}).
By the definition we have the following identities:
\begin{align}
{}_{u}^{a}\,{\overbrace{\fbox{
\phantom{1,2,3,4}}}^{m}}\,^{b}_{v}
&= \fbox{$a$}_u\;
{}^{\;\;\;\;\; a}_{u+2}\,{\overbrace{\fbox{
\phantom{1,2,3}}}^{m-1}}\,^{b}_{v}
+{}_{\;\;\;\;u}^{a+1}\,{\overbrace{\fbox{
\phantom{1,2,3,4}}}^{m}}\,^{b}_{v},
\label{ak:eq:qqq3}\\
{}_{u}^{a}\,{\overbrace{\fbox{
\phantom{1,2,3,4}}}^{m}}\,^{b}_{v}
&= {}_{u}^{a}\,{\overbrace{\fbox{
\phantom{1,2,3}}}^{m-1}}\,^{b}_{v-2}\;
\fbox{$b$}_v
+{}_{u}^{a}\,{\overbrace{\fbox{
\phantom{1,2,3,4}}}^{m}}\,^{b-1}_{v\;\;\;\;}.
\label{ak:eq:qqq4}
\end{align}

It is well known that
the $q$-character of  $W^{(1)}_m(u)$ is given by
\begin{align}
\chi_q(W^{(1)}_m(u)) = \;
{}^{\;\;\;\;\;\;\;\;\;\;\,1}_{u-m+1}
\,{\overbrace{\fbox{\phantom{1,2,3,4}}}^m}
\,_{u+m-1}^{r+1}.
\end{align}

\begin{lem}
Let $1 \le s \le r$ and $1 \le a \le b \le r+1$.
Then,
the following relation
mod $I'_{\ell}$ holds in
$\mathbb{Z}[Y^{\pm1}_{a,q^{u}}]_{a\in I, u\in \mathbb{C}_{\hbar}}$:
\begin{align}\label{ak:eq:qqq2}
{}^{a}_{u}\,{\overbrace{\fbox{
\phantom{1,2,3,4}}}^{\ell+s}}\,_{v}^{b}
\equiv
\begin{cases}
(-1)^s \;{}_u
\fbox{$a\!-\!1,\ldots,2,1$}\;
\fbox{$r\!+\!1,r,\ldots,b\!+\!1$}\,{}_v &
\text{if}\; s=r-b+a,\\
0 & \text{if}\; s > r-b+a.
\end{cases}
\end{align}
Here $v=u+2(\ell+s-1)$ and 
the right hand side of the first case stands for
\begin{align*}
(-1)^s \prod_{\alpha=1}^{a-1}
\fbox{$\alpha$}_{ u+2a-2-2\alpha}
\prod_{\beta=b+1}^{r+1}
\fbox{$\beta$}_{ v+2b+2-2\beta}.
\end{align*}
\end{lem}
\begin{proof}
We employ the induction on $s$.
Suppose $s=1$.
Then $s>r-b+a$ happens only if $(a,b)=(1,r+1)$, therefore
(\ref{ak:eq:qqq2}) is just $\chi_q(W^{(1)}_{\ell+1}(u'))\equiv0$
 for some $u'$.
On the other hand $1=s=r-b+a$ is satisfied 
for $(a,b)=(2,r+1)$ and $(1,r)$. 
Thus we are to show
\begin{align*}
{}^{2}_{u}\,{\overbrace{\fbox{
\phantom{1,2,3,4}}}^{\ell+1}}\,_{v}^{r+1}
\equiv -\, \fbox{$1$}_u,\quad
{}^{1}_{u}\,{\overbrace{\fbox{
\phantom{1,2,3,4}}}^{\ell+1}}\,_{v}^{r}
\equiv -\, \fbox{$r+1$}_v.
\end{align*}
These relations follow from (\ref{ak:eq:qqq3}) and 
(\ref{ak:eq:qqq4}) by setting $(a,b)=(1,r+1), m=\ell+1$
and using $\chi_q(W^{(1)}_\ell(u))\equiv 1$
 and $\chi_q(W^{(1)}_{\ell+1}(u))\equiv 0$.

Now suppose that (\ref{ak:eq:qqq2}) is valid up to $s-1$.
First we consider the case $b-a>r-s$
in (\ref{ak:eq:qqq2}). 
Setting $u'=u+\ell+s-1$ and using (\ref{ak:eq:qqq3})
repeatedly, we have 
\begin{align}\label{ak:eq:qqq6}
0\equiv \chi_q(W^{(1)}_{\ell+s}(u'))
=\sum_{\alpha=1}^{a-1}
\fbox{$\alpha$}_u \; 
{}^{\;\;\;\;\alpha}_{u+2}\,{\overbrace{\fbox{
\phantom{1,2,3}}}^{\ell+s-1}}\,_{v}^{r+1}
+{}^{a}_{u}\,{\overbrace{\fbox{
\phantom{1,2,3,4}}}^{\ell+s}}\,_{v}^{r+1}.
\end{align}
Each term in the $\alpha$ sum is zero
mod $I'_{\ell}$
 due to the 
induction assumption.
Applying (\ref{ak:eq:qqq4}) similarly to the resulting relation, 
we find
\begin{align}\label{ak:eq:qqq5}
0\equiv {}^{a}_{u}\,{\overbrace{\fbox{
\phantom{1,2,3,4}}}^{\ell+s}}\,_{v}^{r+1}
=\sum_{\beta=b+1}^{r+1}
{}^{a}_{u}\,{\overbrace{\fbox{
\phantom{1,2,3}}}^{\ell+s-1}}\,_{v-2}^{\beta}
\;\fbox{$\beta$}_{v}
+{}^{a}_{u}\,{\overbrace{\fbox{
\phantom{1,2,3,4}}}^{\ell+s}}\,_{v}^{b}.
\end{align}
Again each summand in the $\beta$ sum vanishes
mod $I'_{\ell}$ by the 
induction assumption, 
proving the latter case of (\ref{ak:eq:qqq2}).

Next we treat the former case of (\ref{ak:eq:qqq2}),
namely assume that $b-a=r-s$.
If $b\le r$, the same argument 
(\ref{ak:eq:qqq6})--(\ref{ak:eq:qqq5}) as above 
goes through except
that $\beta=b+1$ term in (\ref{ak:eq:qqq5}) is nonvanishing,
leading to 
\begin{align*}
0\equiv {}^{a}_{u}\,{\overbrace{\fbox{
\phantom{1,2,3}}}^{\ell+s-1}}\,_{v-2}^{b+1}\;
\fbox{$b\!+\!1$}_{v}
+{}^{a}_{u}\,{\overbrace{\fbox{
\phantom{1,2,3,4}}}^{\ell+s}}\,_{v}^{b}.
\end{align*}
Applying the induction assumption to the first term, 
we obtain the sought expression for the second term.
If $b=r+1$, then $a=b-r+s=s+1$ and the 
$\alpha$ sum in (\ref{ak:eq:qqq6}) contains 
nonzero summand at $\alpha=a-1$, leading to 
\begin{align*}
0 \equiv  \fbox{$a\!-\!1$}_u\;
{}^{a-1}_{u+2}\,{\overbrace{\fbox{
\phantom{1,2,3}}}^{\ell+s-1}}\,_{v}^{b}
+ {}^{a}_{u}\,{\overbrace{\fbox{
\phantom{1,2,3,4}}}^{\ell+s}}\,_{v}^{b}.
\end{align*}
Again rewriting the first term by using 
the induction assumption yields the sought 
expression for the second term.
\end{proof}

\begin{proof}[Proof of Proposition \ref{prop:Yper}]
(2) is a corollary of (1).
To show (1),  set $a=b$ and $s=r$ in (\ref{ak:eq:qqq2}).
Substituting (\ref{ak:eq:qqq1}) into the resulting relation 
we find 
\begin{align*}
\frac{Y_{a,q^{u+1}}Y_{a,q^{u+3}}\cdots Y_{a,q^{u+2\ell+2r+1}}}
{Y_{a-1,q^u}Y_{a-1,q^{u+2}}\cdots Y_{a-1,q^{u+2\ell+2r}}}
\equiv (-1)^r \quad (1\le a \le r+1)
\end{align*}
for any $u \in {\mathbb Z}$.
{}From $Y_{0,q^u}=1$, the assertion follows.
\end{proof}

\end{appendix}

\end{document}